\documentclass[12pt,reqno]{amsart}

\usepackage[utf8]{inputenc}
\usepackage{listings}
\usepackage{amssymb}
\usepackage{amsmath}
\usepackage{amsthm}
\usepackage{amsfonts}
\usepackage{bbm}
\usepackage{bookmark}
\usepackage{hyperref}
\hypersetup{pdfstartview={FitH}}
\usepackage{color}
\usepackage{xcolor}
\usepackage{mathrsfs}
\usepackage{enumitem}
\usepackage[normalem]{ulem}
\usepackage{tikz}
\usepackage{tikz-3dplot}
\usetikzlibrary{trees,shapes,backgrounds}
\usepackage{caption}
\usetikzlibrary{arrows, arrows.meta, calc, positioning}
\usetikzlibrary{shapes.multipart}
\usepackage{graphicx}
\usepackage{accents}
\usepackage{calc}

\usepackage{fontawesome5}
\newcommand{\auto}{{({\tiny\faCar})}}

\theoremstyle{plain}
\newtheorem{theorem}{Theorem}[section]
\newtheorem{exttheorem}[theorem]{External Theorem}
\newtheorem{lemma}[theorem]{Lemma}

\newtheorem{proposition}[theorem]{Proposition}

\newtheorem{definition}[theorem]{Definition}

\numberwithin{equation}{section}
\newtheorem*{theorem*}{Theorem}

\theoremstyle{remark}
\newtheorem{rem}{Remark}[section]

\newcommand{\uses}[1]{}
\newcommand{\usesdefs}[1]{}
\newcommand{\leanok}{}
\newcommand{\lean}[1]{}

\DeclareRobustCommand{\leanlink}[1]{%
    \href
    {https://pjroos.com/lean-nct/docs/find/?pattern=#1\&strict=false\#doc}
    {\ensuremath{\forall}}
}
\newcommand{\Z}{{\mathbb Z}}
\newcommand{\R}{{\mathbb R}}
\newcommand{\Rm}{{\mathbb R}^m}
\newcommand{\N}{{\mathbb N}}
\newcommand{\C}{{\mathbb C}}

\newcommand{\M}{\mathbf{M}}
\newcommand{\F}{\mathbf{F}}

\DeclareMathOperator{\dist}{dist}
\DeclareMathOperator{\supp}{supp}

\DeclareFontFamily{U}{mathx}{\hyphenchar\font45}
\DeclareFontShape{U}{mathx}{m}{n}{
<5> <6> <7> <8> <9> <10>
<10.95> <12> <14.4> <17.28> <20.74> <24.88>
mathx10
}{}
\DeclareSymbolFont{mathx}{U}{mathx}{m}{n}
\DeclareFontSubstitution{U}{mathx}{m}{n}
\DeclareMathAccent{\widecheck}{0}{mathx}{"71}

\title[Norm-variation of multiple ergodic averages]{A blueprint for the formalization of norm-variation of multiple ergodic averages for commuting transformations}

\author[van Doorn]{Floris van Doorn}
\address{Mathematical Institute, \newline
	University of Bonn,
	Endenicher Allee 60, 53115 Bonn,
	Germany}
\email{vdoorn@math.uni-bonn.de}

\author[Durcik]{Polona Durcik}
\address{Schmid College of Science and Technology, \newline  Chapman University, One University Drive, Orange, CA 92866, USA}
\email{durcik@chapman.edu}

\author[Roos]{Joris Roos}
 \address{Department of Mathematics and Statistics, \newline
	University of Massachusetts Lowell, 1 University Ave,
	Lowell, MA 01854, USA}
\email{joris\_roos@uml.edu}

\author[Slav\'ikov\'a]{Lenka Slav\'ikov\'a}
\address{Department of Mathematical Analysis, Faculty of Mathematics and Physics, \newline Charles University, Sokolovsk\'a 83, 186 75 Praha 8, Czech Republic}
\email{slavikova@karlin.mff.cuni.cz}

\author[Thiele]{Christoph Thiele}
 \address{Mathematical Institute, \newline
	University of Bonn,
	Endenicher Allee 60, 53115 Bonn,
	Germany}
\email{thiele@math.uni-bonn.de}

\date{\today}

\begin{document}
\tdplotsetmaincoords{70}{110}

\begin{abstract}
This blueprint serves as a companion to a forthcoming, shorter traditional mathematical paper. The purpose of this blueprint is two-fold: first, it has served as the foundation for a formalization in Lean 4 of these results. This formalization has been completed largely automatically, making essential use of current frontier large language models. Second, it will serve as a resource to readers of the main paper who are interested in further technical details of the proofs.

The main result concerns norm-variation estimates for multiple ergodic averages associated with $n\ge 2$ commuting measure preserving transformations, providing a quantitative strengthening of Tao's norm-convergence theorem and answering an open question of Avigad and Rute. At the core of the analysis lies an explicit real-variable estimate for twisted multilinear averages that is closely related to certain singular Brascamp--Lieb inequalities.   \end{abstract}

\maketitle

\section{Introduction}

Let $n$ be a natural number, $X$ a measurable space and $\mu$ a measure on $X$. 
Recall that that a map $T: X \to X$ is called a {\em measure preserving transformation} if for every measurable $A\subset X$ the set $T^{-1}(A)$ is measurable and has the
same measure as $A$. For a non-negative integer $n$ let $[n)=\{k\in \N\,:\,0\le k<n\}$.

\begin{definition}[Multiple ergodic average \leanlink{nCT.multipleErgodicAverage}]\label{ergodic averages}\lean{nCT.multipleErgodicAverage}\uses{}\usesdefs{}
Let $\mathbf{f}=(f_j)_{j\in [n)}$ be an $n$-tuple of complex-valued functions on $X$ and let $(T_j)_{j\in [n)}$ be an $n$-tuple of maps $X\to X$. For every positive integer $N$ and every $x\in X$, define the multiple ergodic average
\begin{equation}\label{eq:averages-multiple}
M_{N}(\mathbf{f})(x) = N^{-1} \sum_{i \in [N)} \prod_{j \in [n)} f_j(T_j^i x).
\end{equation}
\end{definition}

\begin{theorem}[Main ergodic theorem \leanlink{nCT.main_ergodic_theorem}]\label{thm:ergodicthm}\lean{nCT.main_ergodic_theorem}\leanok\uses{thm:ergodicthmJ,jump estimates imply variation}\usesdefs{ergodic averages}
Let $n\ge 2$ be an integer and let $r>2^{n-1}$, or $r\ge 2$ if $n=2$. 
Let $(X,\Sigma,\mu)$ be a $\sigma$-finite measure space, 
let $(T_j)_{j\in [n)}$ be an $n$-tuple of mutually commuting measure 
preserving transformations on $X$, 
and let $\mathbf{f}=(f_j)_{j\in [n)}$ be an $n$-tuple 
of complex-valued measurable functions on $X$
such that $\|f_j\|_{L^{2n}(X)}<\infty$ for all $j\in[n)$. 
Then $M_N(\mathbf{f})$ is measurable with 
finite $L^2(X)$ norm for every
positive integer $N$ and
\begin{equation}\label{main ergodic variation estimate}
 \| M_{N} (\mathbf{f}) \|_{V_{r}(L^2(X))}
\le C_{\ref{thm:ergodicthm},n,r}
 \prod_{j\in [n)} \|f_j\|_{2n},
\end{equation}
where for $n=2$,
$C_{\ref{thm:ergodicthm},2,r}=2^{344}$,
and, for $n\ge3$,
$C_{\ref{thm:ergodicthm},n,r}
=
2^{4n+337}
\Big(\frac{r}{r-2^{n-1}}\Big)^{1/r}.$
\end{theorem}
  The definition of the $r$-variation norm is given in Definition \ref{variation seminorm on positive integers}.
Since finite variation of a sequence implies convergence of the sequence and since $L^\infty(X)\subset L^{2n}(X)$ if $X$ is a probability space, Theorem \ref{thm:ergodicthm} immediately recovers Tao's theorem  \cite[Theorem 1.1]{Tao07} for $n\ge 2$, which concerned qualitative norm convergence. Ergodic-theoretic proofs of Tao's theorem were also given by Austin \cite{Austin10}, Host \cite{Host09}, and a generalization to nilpotent groups was proved in \cite{Walsh12}.
Norm-convergence for a single ergodic average with $n=1$ and  $r\ge 2$ was proved by Jones, Ostrovskii, and Rosenblatt \cite{JonesOstrovskiiRosenblatt96}. 
Theorem \ref{thm:ergodicthm} answers a question that was stated by Avigad and Rute in the closing section of \cite{AvigadRute15}. 
Versions of Theorem \ref{thm:ergodicthm} were proved for $n=2$ in \cite{DurcikKovacSkrebThiele19,Kovac16}, and for $n=3$ in \cite{DurcikSlavikovaThiele26}.

Theorem~\ref{thm:ergodicthm} is  reduced  to Theorem~\ref{thm:nct main real} below  by the Calder\'on transference principle; see  \cite{DurcikKovacSkrebThiele19} and  Pernegger \cite{Pernegger2025}, who also formalized this reduction.  Theorem~\ref{thm:nct main real} is an a priori estimate over the real numbers. The proof of this latter estimate is the heart of the matter of this paper. 
It is closely related to the study of cancellation estimates for the simplex Hilbert transform \cite{DurcikKovacThiele19}  and  the singular Brascamp--Lieb estimates  \cite{Kovac12,DurcikSlavikovaThiele22} in real harmonic analysis, but for more general kernels with multi-parameter features. We also refer the reader to   the survey \cite{DurcikThiele21survey}.

\begin{definition}[Twisted average \leanlink{nCT.twistedAverage}, \leanlink{nCT.rescaledTwistedAverage}]\lean{nCT.twistedAverage,nCT.rescaledTwistedAverage}
Let $n\in\mathbb{N}$. For an $n$-tuple of real-valued functions $\mathbf f = (f_i)_{i\in [n)}$ on $\R^n$, a function $\chi:\R\to\R$ and $x\in\mathbb{R}^n$ denote
\begin{equation}\label{A_def}\uses{}\usesdefs{}
A(\chi,\mathbf f)(x) = \int_{\mathbb{R}} \chi(s)\Big(\prod_{i\in [n)} f_i(x+se_i)\Big)\,ds,
\end{equation}
whenever the integrand is measurable and integrable.
Also set $A_t(\chi,\mathbf{f})=A(\chi_{(t)},\mathbf{f})$, where $\chi_{(t)}(x)=t^{-1}\chi(t^{-1}x)$ for $t>0$.
\end{definition}

\begin{theorem}[Main twisted theorem \leanlink{nCT.main_twisted_theorem}]\label{thm:nct main real}\uses{lem:smoothingdecomp,lem:norm_A_sum_le_sum,lem:mainbump1,lem:mainbump2,lem:leftbump,lem:leftbump1,constant main bump one,constant main bump two,constant left bump,constant left bump one}\usesdefs{A_def,defn:window based bump functions}
\lean{nCT.main_twisted_theorem}
\leanok
Let $n\ge 2$ be an integer. For every $J\in\mathbb{N}$, positive real numbers $t_0<t_1<\cdots<t_J$, and every $n$-tuple of real-valued Schwartz functions ${\mathbf f}=(f_i)_{i\in[n)}$ satisfying
\begin{equation}\label{main twisted normalization}
\|f_i\|_{2^{\min(n,i+2)}}=1
\end{equation}
for $i\in[n)$, we have 
\begin{equation}\label{e:ncommuting_real}
\sum_{j\in[J)}\|A_{t_{j+1}}(\mathbf 1_{[0,1]},\mathbf f)-A_{t_j}(\mathbf 1_{[0,1]},\mathbf f)\|_{L^2(\R^n)}^2
\le C_{\ref{thm:nct main real}}J^{1-2^{-n+2}},
\end{equation}
where $C_{\ref{thm:nct main real}}=2^{666}$.
\end{theorem}

\subsection{Use of automation}\label{sec:comments-on-automation}
A first draft of this blueprint was entirely human-written.
In particular, all mathematical decisions about the statements, the organization of the proof, and the dependencies are entirely human.

At the final editing stage when all the statements and definitions, and most of the proofs were in place, parts of the blueprint, and eventually the entire blueprint, were given to \texttt{GPT-5.6-Sol-Pro} with detailed instructions. Over several automation passes with human feedback, the present manuscript was produced.
A major goal of the automation was to compute and optimize the explicit constants and minimize the possibility of mistakes and small inaccuracies in the arguments.
Another goal was to fill in routine arguments, or arguments by analogy that were initially left out.

Consider for example Gaussian domination, Proposition \ref{Gaussian domination combined}, which features several slightly different cases. The proofs of the different cases differ largely only in a routine way (up to switching the function in which mean zero is exploited to obtain decay). Once the statements and the division into cases had been fixed by the authors, there was little value in typing out every local calculation by hand.\\

For transparency, every machine-generated piece of content is tagged with \auto\; throughout the manuscript. Everything not tagged with \auto\; is entirely human-generated, up to possible automatic fixing of typos and minor inaccuracies, tweaking of constants and other minor automated changes.\\

Separate automation passes checked all arguments for local correctness. %This does not replace human mathematical expert review, but has decreased friction during the formalization stage.
Nevertheless, the formalization still surfaced some minor discrepancies in the values of explicit constants in statements of some intermediate propositions, which have then been fixed. Automatically generated informal proofs \auto\, have undergone machine checks, but have not been thoroughly hand-checked . Thus it is possible they contain harmless inaccuracies. This is acceptable since all proofs have been successfully formalized in Lean.

\subsection{Formalization in Lean}

After this blueprint was completed, its contents were formally verified using the proof assistant Lean 4 \cite{LeanPaper} and its mathematical library \texttt{mathlib} \cite{Mathlib}. The formalization is available on GitHub at 
\begin{center}
\href{https://github.com/roos-j/lean-nct}{https://github.com/roos-j/lean-nct}
\end{center}
%and is registered on \href{TODO}{Palomar} and archived at \href{TODO}{Zenodo}.
The main statements and definitions in this introduction are annotated with links to the respective Lean statements (\ensuremath{\forall}). 
The formalization made heavy use of emerging autoformalization capabilities of current frontier models. The bulk of the work was completed using OpenAI's Codex framework with the model \texttt{5.6-terra-ultra}, with some subsequent cleanup done by Anthropic's Claude Code framework with the model \texttt{opus-5-high}. While the statements of Theorem \ref{thm:ergodicthm} and Theorem \ref{thm:nct main real} and its associated definitions have been carefully formalized by hand, virtually all of the remaining formalization was machine-generated across many extended sessions during a one-week period, with a custom instruction set and human supervision. The timeline and required amount of human effort are in striking contrast to those of previous large formalization projects in analysis such as the Carleson project \cite{CarlesonBlueprint}, \cite{CarlesonPaper}, which was predominantly human-generated and concluded its main goals before the current set of highly formalization-capable models was made publicly available.
Both the instruction set and the structure of the blueprint were designed with autoformalization in mind and to mitigate some of the problems with current Codex models during long formalization tasks, such as drifting away from target, avoiding the actual task in favor of creating more and more trivial wrappers, and creating an opaque zoo of auxiliary statements and files.

\subsection{Blueprint choices}
Several convention choices were crucial in guiding both the structure and the content of this blueprint.

\begin{itemize}
\item {\em Strict forward reasoning.} After the introduction, all theorems and definitions must exclusively use theorems that were either previously stated and proved, or are standard theorems. The proof of a theorem must directly follow the corresponding statement (with the exception of the theorems in the introduction).
This naturally aligns the blueprint with the structure of the formalization.
\item {\em Mathlib awareness.} Generally, only facts that are currently available in mathlib in some shape or form can pass as `standard theorems' that require no proof. 
We make two exceptions to this rule: an instance of the Calder\'on transference principle (Theorem \ref{Calderon transference external}), and a multilinear complex interpolation theorem (Theorem \ref{multilinear interpolation external}). These were essentially already formalized prior to this formalization project: the Calder\'on transference principle was first hand-formalized by Pernegger \cite{Pernegger2025}; complex interpolation was autoformalized by one of the authors.
\item {\em Explicit dependencies.} Whenever a named theorem is used in a proof it must be referred to explicitly. The parameters of the theorem should also usually be made explicit, unless it is clear from context (including proofs of hypotheses). This allows for reliable automatic generation of \href{https://pjroos.com/lean-nct/dependency_graph}{dependency graph} metadata.
\item {\em Treatment of constants.} Every claimed inequality must come with an explicitly defined constant, rather than a constant introduced by an existential quantifier. Every constant must be named after the corresponding statement, e.g. $C_{\ref{H kernel derivative estimate Gaussian domination}}$ for the constant of Proposition \ref{H kernel derivative estimate Gaussian domination}. All dependencies on parameters must be explicitly specified and all constants must be recursively defined in terms of previously introduced constants.
Most statements with such constants also have a separate lemma following it that unravels the recursion and provides an explicit estimate not depending on any other constants. The latter are exclusively automatically generated.
\end{itemize}

In this blueprint we also make use of a somewhat unconventional test function space, the Wiener space $W_0$ (see \S \ref{sec:wiener}). This space enjoys convenient invariance properties with respect to the fiberwise constructions common in this project (see e.g. Proposition \ref{W_0 fiber integrals}), and contains certain bump functions lacking rapid decay that play an important role in the argument (see Proposition \ref{diagonal square root}).

\subsection*{Acknowledgments}

FvD and CT were supported by
ERC Synergy Grant 101224275 and Germany's
Excellence Strategy grant
 EXC-2047/1 - 390685813.
% For blueprint, CT supported by ERC. For human paper, CT supported by CRC
PD, JR, and LS acknowledge support by 
the ERC Synergy Grant 101224275 during research stays at Bonn. PD was supported in part by NSF grant DMS-2554859 and a grant from the Simons Foundation. JR was supported in part by NSF grant DMS-2154835 and a grant from the Simons Foundation.
LS was partially supported by the grant no.\ 26-21107S of the Czech Science Foundation.

\section{Preliminaries}\label{sec:prelim}

\newcommand{\g}{\mathfrak{g}}

\subsection{Notation}
We use indexing starting at zero rather than starting at one. For example, if $x$ is a vector in $\R^n$ we write $x=(x_0,\dots,x_{n-1})$. By $\N$ we always denote the natural numbers including zero.
For $n\in\N$ recall that
\begin{equation}\label{auto:index-set-definition} [n)=\{0,\dots,n-1\}. \end{equation}
For a vector $x\in\R^n$ we denote by $|x|$ its standard Euclidean norm. For a multiindex $m\in\N^n$ we denote
\begin{equation}\label{auto:multiindex-length} |m| = \sum_{i\in [n)}m_i. \end{equation}

For a subset $I\subset [n)$ (which inherits the natural order on $[n)$) and $x\in \R^n$ we use the notation
\begin{equation}\label{auto:coordinate-subvector} x_I = (x_{i})_{i\in I} \in \mathbb{R}^{|I|} \end{equation}
and when $I\subset [n)$ and $i\in [n)$ we also write $I\setminus i$ for the set $I$ without the index $i$.

For $x\in \R^n$ let
\begin{equation}\label{auto:coordinate-sum}\Sigma(x)=\sum_{i\in[n)} x_i.\end{equation}

For $y\in(\mathbb{R}^2)^n$ we write $y=(y^m_i)_{i\in [n), m\in [2)}$, $y^m=(y^m_i)_{i\in [n)}$ and similarly if $h\in [2)^n$, then write
\begin{equation}\label{auto:binary-coordinate-selection} y^h = (y_i^{h_i})_{i\in [n)}. \end{equation}
Note there is no ambiguity in having $y^h$, $y^m$ mean different things, because the variables $m$, $h$ have different types.

For $i\in [n)$ we denote by $e_i$ the $i$th standard unit vector in $\R^n$, i.e. the vector with $(e_i)_{i'}=1$ if $i=i'$ and $=0$ otherwise.
Similarly, for $i\in [n), j\in [2)$ denote by $e_{i}^j\in (\R^2)^n$ the standard unit vector with $(e_{i}^j)_{i'}^{j'}=1$ if $i=i', j=j'$ and $=0$ otherwise.

The Fourier transform of an $L^1$ function on $\R^n$ will be defined by
\begin{equation}\label{auto:Fourier-transform-formula}\mathcal{F}f(\xi)= \widehat{f}(\xi) = \int_{\R^n} f(x) e^{-2\pi i x\cdot \xi} dx.\end{equation}
Here $x\cdot \xi=\sum_{i\in [n)} x_i \xi_i$ denotes the standard inner product on $\R^n$.
We have the usual properties of the Fourier transform. For example, $\widehat{f*g}=\widehat{f}\cdot \widehat{g}$.
The Fourier transform extends to an isometry on $L^2$: we have $\|\widehat{f}\|_2 = \|f\|_2$ for $f\in L^2$.
The inverse Fourier transform of an $L^1$ function on $\R^n$ is defined by
\begin{equation}\label{auto:inverse-Fourier-transform-formula} \mathcal{F}^{-1}g(x)= \widecheck{g}(x) = \int_{\R^n} g(\xi) e^{2\pi i x\cdot \xi} d\xi.\end{equation}

If $\phi:\R^n\to \C, \psi:\R^m\to\C$ are functions, then $\phi\otimes \psi:\R^{n+m}\to\C$ denotes their tensor product
\begin{equation}\label{auto:tensor-product-formula} (\phi\otimes \psi)(x,y) = \phi(x)\psi(y), \end{equation}
where we have identified $\R^{n+m}$ with $\R^n\times \R^m$. For $d\in\N$ also define
\begin{equation}\label{auto:tensor-power-formula}\chi^{\otimes d} = \underbrace{\chi \otimes \cdots \otimes \chi}_{d\;\text{times}}.\end{equation}

\auto\ For $r,R>0$ we use the notation
\begin{equation}\label{one dimensional annulus notation}
\mathrm{Ann}_1(r,R)=\{\xi\in\R:R^{-1}r\le|\xi|\le Rr\}.
\end{equation}

\begin{definition}\label{closed ball}\uses{}\usesdefs{}
\lean{Metric.closedBall}
    For $n\in\N$ and $r>0$ and $x\in \R^n$ denote
\begin{equation}\label{auto:closed-ball-definition} B_n(x,r) = \{ \xi\in\mathbb{R}^n\,:\, |\xi-x|\le r \} \end{equation}
\begin{equation}\label{auto:centered-closed-ball-definition} B_n(r) = B_n(0,r) \end{equation}

\end{definition}

\begin{definition}\label{gaussian}\uses{}\usesdefs{}
\lean{Auto.gaussian}
The standard Gaussian on $\R$ is defined by
\begin{equation}\label{standard Gaussian formula}
\g(x)=e^{-\pi x^2}
\end{equation}
for $x\in\R$. This is chosen so that $\int_{\R}\g=1$.
\end{definition}

For $t>0$ and a function $\psi$ on $\R^n$ define the rescaled function
\begin{equation}\label{auto:L1-normalized-dilation}
\psi_{(t)}(x)=t^{-n} \psi(t^{-1}x).
\end{equation}
This is chosen so that $\int_{\R^n} \psi_{(t)} = \int_{\R^n} \psi$.

Denote by $\mathcal{S}(\R^n)$ the space of Schwartz functions on $\R^n$ and by $\mathcal{S}'(\R^n)$ the space of tempered distributions. The Fourier transform of a distribution $T\in\mathcal{S}'(\R^n)$ is defined as
\begin{equation}\label{auto:Fourier-transform-of-distribution} \mathcal{F}(T)(\phi) = T(\mathcal{F}\phi)\quad\text{for}\;\phi\in\mathcal{S}(\R^n). \end{equation}

Unless otherwise specified, all functions can be assumed real-valued.
We shall use freely the fact that if a function, or distribution $f$ on $\R^n$ is real-valued and even, then $\mathcal{F} f$ is also real-valued.

\begin{definition}[bracket bump]\label{bracket bump}\uses{}\usesdefs{}
\lean{Auto.bracketBump,Auto.scaledBracketBump,Auto.scaledBracketBumpReal}
Define
\begin{equation}\label{auto:bracket-weight-definition}
    \left<x\right>=(1+|x|)^{-1}\, .
\end{equation}
We agree on the following parsing, which interprets the power first:
\begin{equation}\label{auto:scaled-bracket-weight}
    \left<x\right>^N_{(s)}=(\left< .\right>^N)_{(s)}(x)=s^{-1}(1+|s^{-1}x|)^{-N}\, .
\end{equation}
\end{definition}

\subsection{The Wiener space \texorpdfstring{$W_0$}{}}\label{sec:wiener}
Let $\mathbb{K}$ be $\R$ or $\C$.
Functions in this section are $\mathbb{K}$-valued.

The  Wiener space $W$ is a space which locally behaves like $L^\infty$ and globally like $L^1$
\cite{Wiener32,Groechenig}. 
In this paper we will only use the space $W_0$, which is the closed subspace of continuous functions in $W$.
The exact definition follows below.

\begin{proposition}\label{auto:local-supremum-measurability}\uses{}\usesdefs{closed ball}
\lean{Auto.localSupremumMeasurability}
\leanok
Let $d\in \N$ and $f:\R^d\to \mathbb{K}$ be continuous.
Then for $r\in (0,\infty)$, the function $g:\R^d\to [0,\infty]$  defined by
\begin{equation}\label{auto:local-supremum-function}
    g(x)=\sup_{|x-y|\le r}|f(y)|
\end{equation}
is finite and upper semi continuous and measurable.
\end{proposition}
\begin{proof}
 Pick $x\in \R^d$. Let $r>0$.
 The function $f$ is continuous on the compact ball $B_d(x,2r)$ (Definition \ref{closed ball}) and thus bounded and uniformly continuous on that ball. It follows that $g$ is finite. By uniform continuity, for every $\lambda>g(x)$ there is an $\epsilon$ such that for $y\in B_d(x,r)$ and $z\in B_d(x,2r)$ with $|y-z|\le \epsilon$ we have
 $|f(z)|<\lambda$. For $z\in B_d(x,r+\epsilon)$
 there is $y\in B_d(x,r)$ on the same ray emanating from $x$ as $z$ such that $|y-z|\le \epsilon$.
 Hence
 \begin{equation}
     \sup_{|y-x|\le r+\epsilon}|f(y)|\le \lambda \, .
 \end{equation}
 Hence, for  every $y\in B_d(x,\epsilon)$, we have
 $g(y)\le \lambda$. Upper semicontinuity of $g$ follows and then also measurability.
 \end{proof}

\begin{definition}\label{auto:Wiener-space-definition}\uses{auto:local-supremum-measurability}\usesdefs{}
\lean{Auto.wienerEnvelope,Auto.MemW0,Auto.wienerNorm,Auto.wienerNormOne}
Let $d\in \N$.  By Proposition \ref{auto:local-supremum-measurability}, the integrand in the following definition is finite and measurable. Define
\begin{equation}\label{auto:Wiener-local-norm}
\|f\|_{W_0,d,r}=    \int_{\R^d} \sup_{|x-y|\le r} |f(y)|\,dx.
\end{equation}
As $d$ can be inferred from the type of $f$, we
will omit it from the notation.
We also set
\begin{equation}\label{auto:Wiener-norm}
    \|f\|_{W_0}=\|f\|_{W_0,1}\, .
\end{equation}
Write $W_0(\R^d)$ or $W_0(d)$ for the set of continuous functions $f:\R^d\to\mathbb{K}$ so that
\begin{equation}\label{auto:Wiener-space-finiteness}  \|f\|_{W_0}< \infty\, . \end{equation}
We also say that $f$ is in $W_0$ as again the type of $f$
determines $d$.
\end{definition}

\begin{proposition}\label{W_0 radius independence}\uses{auto:local-supremum-measurability}\usesdefs{closed ball,auto:Wiener-space-definition}
\lean{Auto.wienerNorm_le_max_one_three_mul_div_pow_mul,Auto.memW0_iff_integrable_wienerEnvelope}
\leanok
Let $d\in \N$.
Let $0<r,s<\infty$ and set
\begin{equation}\label{auto:Wiener-radius-comparison-constant}
    C_{\ref{W_0 radius independence},d,r,s}=\max(1,(3s/r)^d)\, .
\end{equation}

Then for all continuous
$f:\R^d\to \mathbb{K}$ we have
\begin{equation}\label{W_0(r) property}
    \|f\|_{W_0,s}\le C_{\ref{W_0 radius independence},d,r,s}\|f\|_{W_0,r}\, .
\end{equation}
In particular, for every $0<r<\infty$, a  continuous
$f:\R^d\to \mathbb{K}$ is in $W_0$ if and only if $\|f\|_{W_0,r}<\infty$.
\end{proposition}

\begin{proof}
By Proposition \ref{auto:local-supremum-measurability}, all local suprema integrated below are measurable. If $0<s\le r<\infty$, the statement \eqref{W_0(r) property} follows with factor $1$  by nesting of the suprema. Now assume $0<r<s<\infty.$
Let $X$ be a set of points in $B_d(s)$ of
mutual distance at least $r$. As the open balls of radius $r/2$ about these points are disjoint and contained in $B_d(3s/2)$, volume comparison shows there
are at most $(3s/r)^d$ many points in $X$.
We may therefore pick a maximal set $X$ of such points. The closed balls of radius $r$ about these
points cover $B_d(s)$, because any uncovered point would contradict maximality of $X$.

By translation, for every $x\in \R^d$,
the closed ball $B_d(x,s)$ is covered by the closed balls
$B_d(x+z,r)$ with $z\in X$. It follows that
\begin{equation*}
    \|f\|_{W_0,s}\le \int_{\R^d} \sup_{z\in X} \left(\sup_{|x+z-y|\le r} |f(y)|\right)\,dx\, .
\end{equation*}
Estimating the supremum in $X$ by the sum in $X$ and pulling the sum out of the integral, we estimate the last display using translation invariance of the integral by
\begin{equation*}
         \sum_{z\in X} \int_{\R^d} \sup_{|x+z-y|\le r} |f(y)|\,dx \le \sum_{z\in X} \int_{\R^d} \sup_{|x-y|\le r} |f(y)|\,dx\le (3s/r)^d \|f\|_{W_0,r}\, .
\end{equation*}
This completes the proof of \eqref{W_0(r) property}.

The characterization of $W_0(X)$ by arbitrary $r$ follows.

\end{proof}

\begin{proposition}\label{P:lp-embedding}\uses{}\usesdefs{closed ball,auto:Wiener-space-definition}
\lean{Auto.MemW0.memLp}
\leanok
Let $d\in \N$ be given and set
\begin{equation}\label{auto:Wiener-Lp-embedding-constant}
    C_{\ref{P:lp-embedding},d}=\pi^{d/2}\Gamma(d/2+1)^{-1}\, .
\end{equation}
If $f\in W_0(d)$ and $1\le p\le \infty$, then
\begin{equation}\label{Lp W_0}
\|f\|_p\le C_{\ref{P:lp-embedding},d}^{\tfrac 1p -1}\|f\|_{W_0}    \, .
\end{equation}
\end{proposition}

\begin{proof}
Let $f\in W_0(d)$. As $|f(x)|\le \sup_{|x-y|\le 1}|f(y)|$ for every $x\in \R^d$, integration yields
\begin{equation}
    \|f\|_1\le \|f\|_{W_0}\, .
\end{equation}
Moreover, for every $\lambda<\|f\|_\infty$  there is a point $x$ with $|f(x)|\ge \lambda$.
For every $y$ in  $B_d(x,1)$ (Definition \ref{closed ball}), we have
\begin{equation*}
    \sup_{|z-y|\le 1}
|f(z)|\ge \lambda \, .\end{equation*}
With $C_{\ref{P:lp-embedding},d}$  the volume of  $B_d(x,1)$, we obtain by integration
\begin{equation}
    \lambda C_{\ref{P:lp-embedding},d}\le \|f\|_{W_0}\, .
\end{equation}
Taking supremum over all $\lambda<
\|f\|_\infty$ gives
\begin{equation}
    \|f\|_\infty\le C_{\ref{P:lp-embedding},d}^{-1}\|f\|_{W_0}\, .
\end{equation}
Inequality \eqref{Lp W_0} follows by logarithmic convexity of the $L^p$ norms.
\end{proof}

\begin{proposition}\label{W_0 fiber integrals}\uses{P:lp-embedding,W_0 radius independence,auto:local-supremum-measurability}\usesdefs{auto:Wiener-space-definition}
\lean{Auto.exists_wienerEnvelope_fiber_and_integral_comp_injective_continuousLinearMap_bound}
\leanok
For every injective linear map $\pi:\R^{m}\times \R^{l}\to \R^n$ with $m,l\in\N$, $1\le m+l\le n$
there are $0<C_{\ref{W_0 fiber integrals},\pi,1},C_{\ref{W_0 fiber integrals},\pi,2}<\infty$ such that the following holds. For every $f\in W_0(n)$ and $u\in\R^m$, the function $h_u$ defined by
\begin{equation}\label{auto:Wiener-fiber-function}
    h_u(v)=f(\pi(u,v))
\end{equation}
is in $W_0(l)$ with
\begin{equation}\label{W_0 on fiber}
    \|h_u\|_{W_0}\le C_{\ref{W_0 fiber integrals},\pi,1} \|f\|_{W_0}\, .
\end{equation}
The function $g:\Rm\to \C$ defined by
\begin{equation}\label{auto:Wiener-fiber-integral} g(u) = \int_{\R^l} f(\pi(u,v)) \,dv= \int_{\R^l} h_u(v) \,dv \end{equation}
is in $W_0(m)$ with
\begin{equation}\label{W_0 across fibers}
    \|g\|_{W_0}\le C_{\ref{W_0 fiber integrals},\pi,2}  \|f\|_{W_0}\, .
\end{equation}
\end{proposition}
 Note that $l,m,n$ are determined through the type of $\pi$, thus the two constants in the theorem
 may in particular depend on $l,m,n$.

\begin{proof}[Proof \auto]
By Proposition \ref{auto:local-supremum-measurability}, the local suprema of the continuous functions occurring below are measurable. Choose a linear map $\tau:\R^{n-m-l}\to\R^n$ such that
\[
(u,v,z)\longmapsto \pi(u,v)+\tau z
\]
is an isomorphism from $\R^m\times\R^l\times\R^{n-m-l}$ onto $\R^n$. Let $D>0$ be the absolute value of its determinant and set
\[
R=\sqrt{2}\|\pi\|+\|\tau\|.
\]
We use the convention that the volume of the unit ball in $\R^0$ is $1$.

Fix $u\in\R^m$. If $|v-v'|\le1$, $|u-u'|\le1$, and $|z|\le1$, then
\[
|\pi(u,v')-(\pi(u',v)+\tau z)|\le R.
\]
Averaging first over the unit balls in the $u'$ and $z$ variables, integrating in $v$, and applying the change of variables above therefore gives
\[
\|h_u\|_{W_0,l,1}
\le
D^{-1}C_{\ref{P:lp-embedding},m}^{-1}C_{\ref{P:lp-embedding},n-m-l}^{-1}
\|f\|_{W_0,n,R}.
\]
Proposition \ref{W_0 radius independence} gives \eqref{W_0 on fiber} with
\[
C_{\ref{W_0 fiber integrals},\pi,1}
=
D^{-1}C_{\ref{P:lp-embedding},m}^{-1}C_{\ref{P:lp-embedding},n-m-l}^{-1}
C_{\ref{W_0 radius independence},n,1,R}.
\]

The same estimate, before integration in $u'$, supplies an integrable majorant when $u'$ ranges in a neighborhood of $u$. Dominated convergence therefore shows that $g$ is continuous. Moreover,
\[
\int_{\R^m}\sup_{|u-u'|\le1}|g(u')|\,du
\le
D^{-1}C_{\ref{P:lp-embedding},n-m-l}^{-1}\|f\|_{W_0,n,R}.
\]
A second application of Proposition \ref{W_0 radius independence} proves \eqref{W_0 across fibers} with
\[
C_{\ref{W_0 fiber integrals},\pi,2}
=
D^{-1}C_{\ref{P:lp-embedding},n-m-l}^{-1}
C_{\ref{W_0 radius independence},n,1,R}.
\]
\end{proof}

\begin{proposition}\label{tensor Wiener}\uses{auto:local-supremum-measurability}\usesdefs{auto:Wiener-space-definition}
\lean{Auto.MemW0.fintype_tensor,Auto.fintype_tensor_wienerNorm_le}
\leanok
    Let $J\in \N$ with $J\ge 1$. For $j\in [J)$, let $I_j\subset \N$ and assume the $I_j$ are pairwise disjoint.
    Set $I=\bigcup_{j\in [J)}I_j$. Let $f_j:\R^{|I_j|}\to \mathbb{K}$ in $W_0(|I_j|)$ and define
    \begin{equation}\label{auto:Wiener-product-function}
        f(x_I)=\prod_{j\in [J)}f_j(x_{I_j})\, .
    \end{equation}
    Then $f\in W_0(|I|)$ and
    \begin{equation}\label{product W_0}
        \|f\|_{W_0}\le \prod_{j\in [J)}\|f_j\|_{W_0}\, .
    \end{equation}
\end{proposition}
\begin{proof}
    Define $g_j(x_I)=f_j(x_{I_j})$. Then $g_j$ is continuous as composition of the continuous
    projection $x_I\to x_{I_j}$ and the continuous map $f_j$. The function $f$ as product of
    continuous functions is also continuous.
We have
\begin{equation}
     \sup_{j\in [J)} |x_{I_j}-y_{I_j}|\le |x_I-y_I|\, ,
\end{equation}
and hence
\begin{equation}
    \sup_{|x_I-y_I|\le 1}|f(y_I)|\le \prod_{j\in [J)} \sup_{|x_{I}-y_{I}|\le 1}|f_j(y_{I_j})|\le
    \prod_{j\in [J)} \sup_{|x_{I_j}-y_{I_j}|\le 1}|f_j(y_{I_j})|\, .
\end{equation}
By Proposition \ref{auto:local-supremum-measurability}, each factor on the right is measurable. It follows by Tonelli that
\begin{equation}
    \int_{\R^{|I|}}\sup_{|x_I-y_I|\le 1}|f(y_I)|dx_I\le
    \prod_{j\in [J)} \int_{\R^{|I_j|}}\sup_{|x_{I_j}-y_{I_j}|\le 1}|f_j(y_{I_j})|dx_{I_j}\, ,
\end{equation}
    where each factor is finite by assumption and we have $f\in W_0$ and \eqref{product W_0}.
\end{proof}

\begin{proposition}\label{W_0 Brascamp Lieb}\uses{tensor Wiener,W_0 fiber integrals}\usesdefs{auto:Wiener-space-definition}
\lean{Auto.exists_brascamp_lieb_memW0}
\leanok
 Let  $m,l\in\N$, $m+l\le n$ and assume $l\neq 0$.
Let $J\in \N$ with $J\ge 1$. For $j\in [J)$, let $l_j\in \N$ and let  $\Pi_j:\R^m\times \R^l \to \R^{l_j}$
be a linear map such that
\begin{equation}\label{pi injective}
    \bigcap_{j\in [J)} {\rm ker}\Pi_j=\{0\}\, .
\end{equation}
Then there is $C>0$ such that the following holds.
For $j\in [J)$, let $f_j\in W_0(\R^{l_j})$. Then for each
 $u\in \R^m$, the function
 \begin{equation}\label{auto:Wiener-Brascamp-Lieb-fiber-integrand}
     v\to \prod_{j\in [J)}f_j(\Pi_{j}(u,v))
 \end{equation}
 is in $W_0(\R^l)$ and with $f$ defined by
    \begin{equation}\label{auto:Wiener-Brascamp-Lieb-integral}
        f(u)=\int_{\R^l}\prod_{j\in [J)}f_j(\Pi_{j}(u,v))\, dv \, ,
    \end{equation}
    we have $f\in W_0(\R^{m})$ and
    \begin{equation}\label{auto:Wiener-Brascamp-Lieb-bound}
        \|f\|_{W_0}\le C\prod_{j\in[J)}\|f_j\|_{W_0}\, .
    \end{equation}
\end{proposition}
Note that $m=0$ is a case in the above proposition.
\begin{proof}
Let
\begin{equation}
    I_j=[\sum_{h\le j}l_h)\setminus [\sum_{h<j}l_h)
\end{equation}
and $I=[\sum_{h<J} l_h)$.
Define
$\Pi:\R^{m}\times \R^l\to \R^{|I|}$ by
\begin{equation}
    (\Pi(u,v))_{I_j}=(\Pi_{j}(u,v))\, .
\end{equation}
Then $\Pi$ is injective by \eqref{pi injective}.

Define $g:\R^{|I|}\to \C$ by
\begin{equation}
    g(x_I)=\prod_{j\in [J)}f_j(x_{I_j})\, .
\end{equation}
By Proposition \ref{tensor Wiener}, $g\in W_0(\R^{|I|})$ with
\begin{equation}
    \|g\|_{W_0}\le \prod_{j\in[J)}\|f_j\|_{W_0}\ .
\end{equation}
By Proposition \ref{W_0 fiber integrals}, $f\in W_0$ with
\begin{equation}
    \|f\|_{W_0}\le C_{\ref{W_0 fiber integrals},\Pi,2}\|g\|_{W_0}\, .
\end{equation}

This proves Proposition \ref{W_0 Brascamp Lieb} with
\[
C=C_{\ref{W_0 fiber integrals},\Pi,2}.
\]
\end{proof}

\begin{proposition}\label{P:schwartz-into-wiener}\uses{auto:local-supremum-measurability}\usesdefs{auto:Wiener-space-definition}
\lean{Auto.SchwartzMap.memW0}
\leanok
We have $\mathcal{S}(\R^n)\subset W_0(\R^n)$.
\end{proposition}

\begin{proof}[Proof \auto]
Let $f\in\mathcal S(\R^n)$. Choose $N>n$. There is a finite constant $C$ such that
\[
|f(y)|\le C(1+|y|)^{-N}
\]
for every $y\in\R^n$. If $|x-y|\le1$, then $1+|x|\le2(1+|y|)$, and hence
\[
|f(y)|\le 2^NC(1+|x|)^{-N}.
\]
Proposition \ref{auto:local-supremum-measurability} shows that the local supremum is measurable. The right-hand side is integrable in $x$, so $f\in W_0(\R^n)$.
\end{proof}

\subsubsection{Convolution along a vector}

\begin{definition}[Convolution along a vector]\label{auto:convolution-along-vector-definition}\uses{}\usesdefs{auto:Wiener-space-definition}
\lean{Auto.convolutionAlongVector}
For $\rho\in W_0(\R^n), \varphi\in W_0(\R)$ and a vector $\alpha\in\mathbb{R}^n$ define for $x\in\mathbb{R}^n$,
\begin{equation}\label{auto:convolution-along-vector-formula} (\rho *_\alpha \varphi)(x) = \int_{\R} \rho(x-p\alpha) \varphi(p)\,dp \end{equation}
\end{definition}

\begin{proposition}[Properties of convolution along a vector]
\label{convolution vector}\uses{W_0 Brascamp Lieb}\usesdefs{auto:Wiener-space-definition,auto:convolution-along-vector-definition}
\lean{Auto.memW0_convolutionAlongVector,Auto.fourier_convolutionAlongVector}
\leanok
For $\rho \in W_0(\R^n), \varphi\in W_0(\R), \alpha,\xi\in\mathbb{R}^n$ we have $\rho *_\alpha \varphi\in W_0(\R^n)$ and
\begin{equation}\label{auto:convolution-along-vector-Fourier-transform} \widehat{\rho *_\alpha \varphi}(\xi) = \widehat{\rho}(\xi)\widehat{\varphi}(\alpha\cdot\xi). \end{equation}
\end{proposition}

\begin{proof}
Apply Proposition \ref{W_0 Brascamp Lieb}. The left-hand side equals
\[\int_{\R^n} (\rho *_\alpha \varphi)(x)e^{-2\pi ix\cdot \xi} dx=\int_{\R^n}  \int_{\R} \rho(x-p\alpha) \varphi(p)\,dp\,e^{-2\pi ix\cdot \xi} dx\]
\[=\int_{\R^n}  \int_{\R} \rho(x) \varphi(p) e^{-2\pi i (x+p\alpha)\cdot \xi} dp\,dx,\]
which equals the right-hand side.
\end{proof}

\subsection{\texorpdfstring{$K$ kernels}{K kernels}}

\begin{definition}[normalized function tuples]\label{normalized function tuples}\uses{}\usesdefs{}
\lean{Auto.normalizedFunctionTuples}
 Let $\mathfrak{F}$ be the set of tuples
$\F=(F_i)_{i\in [n)}$ of real-valued functions in $\mathcal{S}(\R^n)$

satisfying for $i\in [n)$ the normalization
 \begin{equation}\label{auto:normalized-function-tuple} \|F_{i}\|_{2^{i+\min(n-i,2)}}=1. \end{equation}
In other words,
 \begin{equation}\label{f-normalization}
  \|F_{i}\|_{2^{i+2}}=1\;\text{if}\;i\le n-2\;\text{and}\;\|F_{n-1}\|_{2^{n}}=1  \, .
 \end{equation}

\end{definition}

The word cube in the following definition stands for the parameter set $[2)^k$, which can be interpreted as a $k$ dimensional cube.

\begin{definition}[cube Brascamp--Lieb form]\label{cube Brascamp--Lieb}\uses{W_0 Brascamp Lieb}\usesdefs{auto:Wiener-space-definition}
\lean{Auto.cubeBrascampLiebForm}
Define for $k\in\N$ with $1 \le k\le  n$
and  $L\in W_0(\R^k)$ and $F\in W_0(\R^n)$
\begin{equation}\label{auto:cube-Brascamp-Lieb-form} \Theta_{{\rm cube}}(k,L)(F) = \int_{(\R^2)^k} \int_{\R^{n-k}} L(y^1-y^0) \Big(\prod_{h\in [2)^k} F(y^h, x_{[k,n)})\Big)^{2^{\min(n-k,1)}}\;dx_{[k,n)}\,dy.\end{equation}
\end{definition}
Note that the integrand in $\Theta_{{\rm cube}}(k, L)(F)$ is in $W_0$ by Proposition \ref{W_0 Brascamp Lieb}.

\begin{proposition}[cube BL inequality]\label{cube BL inequality}\uses{}\usesdefs{auto:Wiener-space-definition,cube Brascamp--Lieb}
\lean{Auto.cubeBLInequality}
\leanok
Let $k\in\N$ satisfy $1 \le k\le n$.  Let $L\in W_0(\R^k)$. Then, for all $F\in W_0(\R^n)$ with $\|F\|_{2^{k+\min(n-k,1)}}=1$,
\begin{equation}\label{cube-k-m-bound}
   |\Theta_{\rm cube}(k,L)(F)| \le\|L\|_1 .
\end{equation}
\end{proposition}

\begin{proof}
This follows from H\"older's inequality.
Changing variables $y^1\mapsto y^0 + u$ we have $y^h_i=y^0_i+h_i u_i$, so
\[ \Theta_{{\rm cube}}(k,L)(F) = \int_{\R^k} L(u) \Big(\int_{\R^k}\int_{\R^{n-k}} \Big(\prod_{h\in [2)^k} F((y^0_i+h_i u_i)_{i\in [k)}, x_{[k,n)})\Big)^{2^{\min(n-k,1)}}\;dx_{[k,n)}\,dy^0\Big)\,du \]
By the triangle inequality, $|\Theta_{{\rm cube}}(k,L)(F)|$ is no greater than $\|L\|_1$ times the supremum over $u\in\R^k$ of
\[ \int_{\R^n} \Big(\prod_{h\in [2)^k} |F(\pi_h(x))|\Big)^{2^{\min(n-k,1)}}\;dx, \]
where we have renamed $y^0$ to $x_{[0,k-1]}$
and set $\pi_h(x)=(x_i+h_i u_i)_{i\in [k)}, x_{[k,n)}$. By H\"older's inequality this is
\[ \le \Big(\prod_{h\in [2)^k} \|F\circ \pi_h\|_{2^{k+\min(n-k,1)}}\Big)^{2^{\min(n-k,1)}}, \]
which is equal to $1$ because each $\pi_h$ is an isometry and $\|F\|_{2^{k+\min(n-k,1)}}=1$.
\end{proof}

The word {\em prism} in what follows stands for the Cartesian product of a cube with a simplex.
The $2^k$ corners of the $k$ dimensional cubes are parameterized by $[2)^k$,
while the $n-k+1$ corners of the simplex are
parameterized by integers in the interval $[k-1,n)$.
Note that the simplex with two corners is an interval,
hence both cases $k=n-1$ and $k=n$ can be identified with a cube.

\begin{definition}[prism Brascamp--Lieb form]\label{prism Brascamp--Lieb}\uses{W_0 Brascamp Lieb}\usesdefs{auto:Wiener-space-definition,normalized function tuples}
\lean{Auto.prismBrascampLiebForm}
Define for integers $1 \le k\le  n$
and  $K\in W_0(\R^{k+1})$ and $\F\in \mathfrak{F}$,

\begin{equation}\label{auto:K-kernel-form} \Theta_k(K)(\F) = \int_{(\R^{2})^k} \int_{\R^{n-k+1}} K(y^1-y^0, \Sigma(y^0) + \Sigma(x_{[k-1,n)}))
\prod_{\substack{h\in [2)^k,\\i\in [k-1,n)}} F_{i}( y^h, x_{[k-1,n)\setminus i})\;dx_{[k-1,n)}\;dy \end{equation}
\end{definition}

Note that the integrand in $\Theta_k(K)(\F)$ is in $W_0$ by Proposition \ref{W_0 Brascamp Lieb}.
The argument of $K$ in
this definition is split as $\R^{k}\times \R$
while
the argument of $F_{i-1}$  is split as $\R^{k}\times \R^{n-k}$. Recall that superscripts $0$ and $1$ as in $y^0$ and $y^1$ are not powers but just superscripts.

\begin{proposition}[prism BL inequality]\label{prism BL inequality}\uses{}\usesdefs{auto:Wiener-space-definition,normalized function tuples,prism Brascamp--Lieb}
\lean{Auto.prismBLInequality}
\leanok

Let $k\in\N$ satisfy $1\le k\le n$ and $K\in W_0(\R^{k+1})$ and $\F\in \mathfrak{F}$. Then
\begin{equation}\label{prism-k-bound}
   |\Theta_k(K)(\F)| \le \|K\|_1.
\end{equation}
\end{proposition}

\begin{proof}
This follows from H\"older's inequality.
Changing variables $(y^1,x_{k-1})\mapsto (u,s)$, where
\[
y^1=y^0+u,\qquad x_{k-1}=s-\Sigma(y^0)-\Sigma(x_{[k,n)}),
\]
we have $y_i^h=y_i^0+h_iu_i$ and
$\Sigma(y^0)+\Sigma(x_{[k-1,n)})=s$. For fixed $(u,s)$, rename
$(y^0,x_{[k,n)})$ as $x\in\R^n$ and set
\[
\pi_{h,i}(x)
=
\Big((x_l+h_lu_l)_{l\in [k)},
\big(s-\Sigma(x),x_{[k,n)}\big)_{[k-1,n)\setminus i}\Big).
\]
Then
\[
\Theta_k(K)(\F)
=
\int_{\R^{k+1}} K(u,s)
\int_{\R^n}
\prod_{\substack{h\in [2)^k\\ i\in [k-1,n)}}
F_i(\pi_{h,i}(x))
\,dx\,du\,ds.
\]
By the triangle inequality, $|\Theta_k(K)(\F)|$ is no greater than $\|K\|_1$ times the supremum over $(u,s)\in\R^{k+1}$ of
\[
\int_{\R^n}
\prod_{\substack{h\in [2)^k\\ i\in [k-1,n)}}
|F_i(\pi_{h,i}(x))|
\,dx.
\]
By H\"older's inequality this is no greater than
\[
\prod_{\substack{h\in [2)^k\\ i\in [k-1,n)}}
\|F_i\circ\pi_{h,i}\|_{2^{i+\min(n-i,2)}}.
\]
Each $\pi_{h,i}$ is invertible and measure preserving, so this is equal to one by the definition of $\mathfrak{F}$.
\end{proof}

The following proposition concerns the special case $k=n$.
\begin{proposition}[simplification 1 prism]\label{simplification 1 prism}\uses{}\usesdefs{auto:Wiener-space-definition,normalized function tuples,prism Brascamp--Lieb}
\lean{Auto.simplificationOnePrism}
\leanok
    Let $K\in W_0(\R^{n+1})$. If for all $z\in \R^{n+1}$ we have
    \begin{equation}\label{1 prism mean zero assumption}
        \int_\R K(z+qe_n)\, dq=0\, ,
    \end{equation}
    then for all $\F\in \mathfrak{F}$,
    \begin{equation}\label{auto:terminal-K-form-zero}
         \Theta_n(K)(\F)= 0\, .
    \end{equation}
\end{proposition}
\begin{proof}
The form is
\[ \Theta_{n}(K)(\F) = \int_{(\R^{2})^n} \int_{\R} K(y^1-y^0, \Sigma(y^0) + x_n)
\prod_{\substack{h\in [2)^n}} F_{n-1}( y^h)\;dx_n\;dy. \]
Since the product $\prod_{\substack{h\in [2)^k}} F_{n-1}( y^h)$ does not depend on $x_n$, this is equal to
\[\Theta_n(K)(\F) = \int_{(\R^{2})^n} \Big( \prod_{\substack{h\in [2)^n}} F_{n-1}( y^h) \Big) \Big(\int_{\R} K(y^1-y^0, \Sigma(y^0) + x_n)\,dx_n\Big)\,dy. \]
Setting $z=(y^1-y^0, \Sigma(y^0))\in\R^{n+1}$ for every fixed $y\in (\R^2)^n$, we see that the inner integral over $x_n$ is equal to zero by the assumption \eqref{1 prism mean zero assumption}.
\end{proof}

In the next proposition we apply the Cauchy-Schwarz inequality as in the proof of Proposition \ref{prism BL inequality}, but this time carefully preserving in one of the terms on the right-hand side
a cancellation if the function $\phi_j$ has integral zero.

\begin{proposition}[singly cancellative Cauchy-Schwarz]\label{single cancellative Cauchy-Schwarz}\uses{cube BL inequality,W_0 Brascamp Lieb}\usesdefs{auto:Wiener-space-definition,normalized function tuples,cube Brascamp--Lieb,prism Brascamp--Lieb}
\lean{Auto.singlyCancellativeKernel_memW0,Auto.singlyCancellativeLift_memW0,Auto.singlyCancellativeCauchySchwarz_bound}
\leanok
Let $J\ge 1$ and $1\le k<n$ be integers. For
$(\rho_j)_{j\in [J)}\subset W_0(\R^{k+1})$, and
$(\phi_j)_{j\in [J)}\subset W_0(\R)$, define, for $u\in \R^{k+1}$,
\begin{equation}\label{auto:Cauchy-Schwarz-kernel-K}
    K(u)=\sum_{j\in[J)}\int_\R \rho_j(u_{[k)},u_k-q)\phi_j(q)\,dq\,.
\end{equation}

Then $K\in W_0(\R^{k+1})$ by Proposition \ref{W_0 Brascamp Lieb}.
For $u\in \R^{k+2}$, define
\begin{equation}\label{auto:Cauchy-Schwarz-kernel-tilde-K}
    \tilde{K}(u)=\sum_{j\in[J)}\int_\R
    |\rho_j|(u_{[k)},u_{k+1}-q)\phi_j(u_k+q)\phi_j(q)\,dq\, .
\end{equation}
Then $\tilde{K}\in W_0(\R^{k+2})$.
Moreover, for all $\F\in\mathfrak{F}$,
\begin{equation}\label{auto:Cauchy-Schwarz-at-k-bound}
    |\Theta_k(K)(\F)|^2
    \le \Theta_{k+1}(\tilde{K})(\F)\sum_{j\in[J)}\|\rho_j\|_1\, .
\end{equation}
\end{proposition}

\begin{proof}
We compute, substituting $x_{k-1}$ by $x_{k-1}+q$ and then $q$ by $r-x_{k-1}$ in combination with suitable orders of integration in $dx_{k-1}$ and $dq$, $\Theta_k(K)(\F)$ equals
\begin{equation}\label{sccs0}
\Theta_k(K)(\F)=\sum_{j\in[J)} \int_{(\R^2)^k}\int_{\R^{n-k+1}}
\rho_j(y^1-y^0,\Sigma(y^0)+\Sigma(x_{[k-1,n)}))
\end{equation}
\[
{}\times
\Big(
\int_\R
\Big(
\prod_{\substack{h\in[2)^k\\ i\in[k,n)}}
F_{i}(y^h,r,x_{[k,n)\setminus i})
\Big)
\phi_j(r-x_{k-1})\,dr
\Big)
\]
\[
{}\times
\Big(
\prod_{h\in[2)^k}F_{k-1}(y^h,x_{[k,n)})
\Big)
dx_{[k-1,n)}\,dy\, .
\]
We estimate $\rho_j$ by $|\rho_j|$ and also the term in outmost brackets in \eqref{sccs0} by its absolute value $A$
and the last bracket in \eqref{sccs0} by its absolute value $B$.
We interpret the resulting expression as a pairing $\left< A , B \right>$ with respect to a positive measure.
Applying Cauchy-Schwarz, we estimate the modulus squared of \eqref{sccs0} by the product of two terms.
The first term $\left< A , A \right>$ is no greater than
\[
\sum_{j\in[J)}\int_{(\R^2)^k}\int_{\R^{n-k+1}}
|\rho_j|(y^1-y^0,\Sigma(y^0)+\Sigma(x_{[k-1,n)}))
\]
\[
{}\times
\Big|
\int_\R
\Big(
\prod_{\substack{h\in[2)^k\\ i\in[k,n)}}
F_{i}(y^h,r,x_{[k,n)\setminus i})
\Big)
\phi_j(r-x_{k-1})\,dr
\Big|^2
dx_{[k-1,n)}\,dy,
\]
Using that the $F_i$ are real and renaming the integration variables $r$, which appear twice
through the square, once $y_k^0$ and once $y_k^1$,

\[
\sum_{j\in[J)}\int_{(\R^2)^{k+1}}\int_{\R^{n-k}}
\int_\R
|\rho_j|((y^1-y^0)_{[k)},\Sigma(y^0_{[k)})+x_{k-1}+\Sigma(x_{[k,n)}))
\]
\[
{}\times
\Big(
\prod_{\substack{h\in[2)^{k+1}\\ i\in[k,n)}}
F_{i}(y^h,x_{[k,n)\setminus i})
\Big)
\phi_j(y_k^1-x_{k-1})\phi_j(y_k^0-x_{k-1})
\,dx_{k-1}\,dx_{[k,n)}\,dy.
\]
This is equal to $\Theta_{k+1}(\tilde{K})(\F)$.
In the second identity, we substitute
$x_{k-1}$ by $x_{k-1}+y_k^0$ and then name the variable $x_{k-1}$ as $-q$.

The second term $\left< B , B \right>$, in Cauchy-Schwarz, after integrating the variable $x_{k-1}$ which only appears in the last argument of $\rho_j$,
is identified as
\begin{equation}
    \sum_{j\in[J)} \Theta_{\rm cube}(k,L_j)(F_{k-1})
\end{equation}
with
\begin{equation}
    L_j(u)=\int_\R |\rho_j(u,q)|\,dq
\end{equation}
and estimated by Proposition \ref{cube BL inequality}.
\end{proof}

If $k=n-1$, one may do an even more careful Cauchy-Schwarz, preserving cancellation in both terms on the right-hand-side.

\begin{proposition}[doubly cancellative Cauchy-Schwarz, $k=n-1$]\label{doubly cancellative Cauchy-Schwarz}\uses{}\usesdefs{auto:Wiener-space-definition,normalized function tuples,prism Brascamp--Lieb}
\lean{Auto.doublyCancellativeKernel_memW0,Auto.doublyCancellativeLift_memW0,Auto.doublyCancellativeCauchySchwarz_bound}
\leanok
Let $J\ge 1$ be an integer. For $(\rho_j)_{j\in[J)}\subset W_0(\R^n)$ and
$(\phi_j)_{j\in [J)}\subset W_0(\R)$ define for $u\in \R^n$,
\begin{equation}\label{auto:Cauchy-Schwarz-terminal-kernel-K}
    K(u)=\sum_{j\in [J)}\int_{\R^2} \rho_j(u_{[n-1)}, u_{n-1}-r-s) \phi_j(r) \phi_j(s) \, ds dr\, .
\end{equation}
Then $K\in W_0(\R^n)$.
For $u\in\R^{n+1}$ let
\begin{equation}\label{auto:Cauchy-Schwarz-terminal-kernel-tilde-K}
    \tilde{K}(u)=\sum_{j\in [J)}\int_\R |\rho_j|(u_{[n-1)}, u_{n}-q) \phi_j(u_{n-1}+q) \phi_j(q) \, dq\, .
\end{equation}
Then $\tilde{K}\in W_0(\R^{n+1})$ and for all $\F\in\mathfrak{F}$,
\begin{equation}\label{auto:Cauchy-Schwarz-at-n-minus-one-bound}
    |\Theta_{n-1}(K)(\F)| \le \sup_{\mathbf{\tilde{F}}\in\mathfrak{F}} |\Theta_{n}(\tilde{K})(\mathbf{\tilde{F}})|.
\end{equation}
\end{proposition}

\begin{proof}
We expand $\Theta_{n-1}(K)(\F)$, replacing $x_{n-2}$ by $x_{n-2}+r$ and $x_{n-1}$ by $x_{n-1}+s$
and then $r$ by $r-x_{n-2}$ and $s$ by $s-x_{n-1}$. We obtain that $\Theta_{n-1}(K)(\F)$ equals
\begin{equation}\label{ncs1}
\Theta_{n-1}(K)(\F)=\sum_{j\in[J)} \int_{(\R^2)^{n-1}} \int_{\R^2}
\rho_j(y^1-y^0, \Sigma(y^0)+\Sigma(x_{[n-2,n)}))
\end{equation}
\[
{}\times \Big( \int_{\R} \Big(\prod_{h\in [2)^{n-1}} F_{n-1}(y^h,r)\Big)
\phi_j(r-x_{n-2})\,dr\Big)
\]
\[
{}\times \Big( \int_{\R} \Big(\prod_{h\in [2)^{n-1}} F_{n-2}(y^h,s)\Big)
\phi_j(s-x_{n-1})\,ds\Big)
    dx_{[n-2,n)}\,dy .
\]
We estimate $\rho_j$ by $|\rho_j|$ and also the two bracketed terms in \eqref{ncs1} by their absolute values.
We interpret the resulting expression as a pairing with respect to a positive measure.
Applying Cauchy-Schwarz, we estimate the modulus squared of \eqref{ncs1} by the product of two terms.

The first term is no greater than
\[
\sum_{j\in[J)} \int_{(\R^2)^{n-1}} \int_{\R^2}
|\rho_j|(y^1-y^0, \Sigma(y^0)+\Sigma(x_{[n-2,n)}))
\]
\[
{}\times
\Big|
 \int_{\R} \Big(\prod_{h\in [2)^{n-1}} F_{n-1}(y^h,r)\Big)
\phi_j(r-x_{n-2})\,dr
\Big|^2
    dx_{[n-2,n)}\,dy .
\]
Using that the $F_i$ are real and renaming the integration variables $r$, which appear twice
through the square, once $y_{n-1}^0$ and once $y_{n-1}^1$, this is equal to
\[
\sum_{j\in[J)} \int_{(\R^2)^n}\int_{\R^2}
|\rho_j|((y^1-y^0)_{[n-1)},\Sigma(y^0_{[n-1)})+x_{n-2}+x_{n-1})
\]
\[
{}\times
\Big(\prod_{h\in[2)^n}F_{n-1}(y^h)\Big)
\phi_j(y_{n-1}^1-x_{n-2})\phi_j(y_{n-1}^0-x_{n-2})
\,dx_{n-2}\,dx_{n-1}\,dy .
\]
In this identity, we substitute $x_{n-2}$ by $x_{n-2}+y_{n-1}^0$ and then name the variable $x_{n-2}$ as $-q$.
This identifies the last display as $\Theta_n(\tilde K)(\F)$.

The second term is no greater than
\[
\sum_{j\in[J)} \int_{(\R^2)^{n-1}} \int_{\R^2}
|\rho_j|(y^1-y^0, \Sigma(y^0)+\Sigma(x_{[n-2,n)}))
\]
\[
{}\times
\Big|
 \int_{\R} \Big(\prod_{h\in [2)^{n-1}} F_{n-2}(y^h,s)\Big)
\phi_j(s-x_{n-1})\,ds
\Big|^2
    dx_{[n-2,n)}\,dy .
\]
Proceeding analogously, using that the $F_i$ are real and renaming the integration variables $s$, which appear twice
through the square, once $y_{n-1}^0$ and once $y_{n-1}^1$, this becomes
\[
\sum_{j\in[J)} \int_{(\R^2)^n}\int_{\R^2}
|\rho_j|((y^1-y^0)_{[n-1)},\Sigma(y^0_{[n-1)})+x_{n-2}+x_{n-1})
\]
\[
{}\times
\Big(\prod_{h\in[2)^n}F_{n-2}(y^h)\Big)
\phi_j(y_{n-1}^1-x_{n-1})\phi_j(y_{n-1}^0-x_{n-1})
\,dx_{n-2}\,dx_{n-1}\,dy .
\]
In this identity, we substitute $x_{n-1}$ by $x_{n-1}+y_{n-1}^0$ and then name the variable $x_{n-1}$ as $-q$.
This identifies the last display as $\Theta_n(\tilde K)(\mathbf{\tilde F})$, where $\mathbf{\tilde F}$ arises from $\F$ by interchanging the functions $F_{n-2}$ and $F_{n-1}$.
Thus
\begin{equation}
    |\Theta_{n-1}(K)(\F)|^2
    \le \Theta_n(\tilde K)(\F)\Theta_n(\tilde K)(\mathbf{\tilde F}),
\end{equation}
which implies the claim since $\tilde{F}$ is again in $\mathfrak{F}$.
\end{proof}

The terms on the right-hand-side of Cauchy-Schwarz are by nature positive. This is
stated independently in the following proposition.
\begin{proposition}[Positivity]\label{Positivity K}\uses{}\usesdefs{auto:Wiener-space-definition,normalized function tuples,prism Brascamp--Lieb}
\lean{Auto.positivityKernel_memW0,Auto.positivityKernel_nonnegative}
\leanok
 Let $1\le k\le n-1$. Let $\rho\in W_0(\R^{k+1})$ be non-negative and
$\phi\in W_0(\R)$. Set for $u\in\R^{k+2}$ and $i\in [k+1)$
\begin{equation}\label{auto:positive-kernel-K}
    K(u)=\int_\R \rho(u_{[k+1)\setminus i}, u_{k+1}-q) \phi(u_{i}+q) \phi(q) \, dq\, .
\end{equation}
Then $K\in W_0(\R^{k+2})$ and for all $\mathbf{F}\in\mathfrak{F}$
\begin{equation}\label{auto:positive-K-form}
    \Theta_{k+1}(K)(\F)\ge 0\, .
\end{equation}
\end{proposition}

\begin{proof}
We have
\[
\Theta_{k+1}(K)(\F)
= \int_{(\R^2)^{k+1}}\int_{\R^{n-k}}\int_\R
\rho((y^1-y^0)_{[k+1)\setminus i},\Sigma(y^0)+\Sigma(x_{[k,n)})-q)
\]
\[
{}\times
\phi(y_i^1-y_i^0+q)\phi(q)
\prod_{\substack{h\in [2)^{k+1}\\j\in[k,n)}}
F_j(y^h,x_{[k,n)\setminus j})
\,dq\,dx_{[k,n)}\,dy .
\]

Substituting \(q\) by \(y_i^0-q\) and integrating first in \(y_i^0,y_i^1\), this is equal to
\[
\int_{(\R^2)^{k}}\int_{\R^{n-k}}\int_\R
\rho((y^1-y^0)_{[k+1)\setminus i},
\Sigma(y^0_{[k+1)\setminus i})+q+\Sigma(x_{[k,n)}))
\]
\[
{}\times
\Big|
\int_\R
\Big(
\prod_{\substack{h\in [2)^{k+1}\\ h_i=0\\ j\in[k,n)}}
F_j(y^h,x_{[k,n)\setminus j})
\Big)
\phi(y_i^0-q)\,dy_i^0
\Big|^2
\,dq\,dx_{[k,n)}\,dy_{[k+1)\setminus i},
\]
which is non-negative.
\end{proof}

\begin{proposition}[Monotonicity]\label{Monotonicity K}\uses{Positivity K}\usesdefs{auto:Wiener-space-definition,normalized function tuples,prism Brascamp--Lieb}In the situation of Proposition \ref{Positivity K},
\lean{Auto.monotonicityK}
\leanok
in particular, if $\rho\le \tilde{\rho}$ and
\begin{equation}\label{auto:kernel-monotonicity-K-rho} K_\rho(u)=\int_\R \rho(u_{[k)}, u_{k+1}-q) \phi(u_{k}+q) \phi(q) \, dq\end{equation}
then
\begin{equation}\label{auto:kernel-monotonicity-bound} \Theta_{k+1}(K_\rho)(\F) \le \Theta_{k+1}(K_{\tilde{\rho}})(\F) \end{equation}
\end{proposition}
\begin{proof}
    Apply Proposition \ref{Positivity K} to $K_{\tilde{\rho}-\rho}$ and use linearity.
\end{proof}

\subsection{\texorpdfstring{$M$}{M} kernels}

\begin{proposition}[$M$ to $K$]\label{M to K}\uses{W_0 Brascamp Lieb}\usesdefs{auto:Wiener-space-definition}
\lean{Auto.mToK,Auto.mToK_integrand_memW0,Auto.mToK_memW0,Auto.mToK_eLpNorm_one_le}
\leanok
Let $k\in\N$ satisfy $1\le k\le n$. Let $M\in W_0((\R^2)^k)$.
Define for $z\in \R^{k+1}$,
\begin{equation}\label{kviam_2}
    K_k(M)(z)=\int_{\R^{k-1}}  M((z_{l}+p_{l},p_{l})_{l\in [k-1)}, z_{k-1}+z_{k}-\Sigma(p),z_{k}-\Sigma(p))\,dp.
\end{equation}
The integrand is in $W_0(\R^{k-1})$ for every $z\in\R^{k+1}$.
The argument of $M$ is identified with an element of $(\R^2)^{k}$.
Then $K_k(M)\in W_0(\R^{k+1})$ and
\begin{equation}\label{auto:K-map-L1-bound} \|K_k(M)\|_1 \le \|M\|_1. \end{equation}
\end{proposition}

\begin{proof}
The claims on membership in $W_0$ follow from
Proposition \ref{W_0 Brascamp Lieb}.
The claimed inequality follows by the triangle inequality and Fubini-Tonelli's theorem.
\end{proof}

\begin{rem}
For example, for $k=1$,
\begin{equation}\label{auto:K-one-explicit-formula}K_1(M)(z)=M(z_0+z_1,z_1).\end{equation}
\end{rem}

\begin{rem}
Note that $K_k(M)$ can also be written as
\begin{equation}\label{kviam_2'}
    K_k(M)(z)=\int_{\R^{k}}  M((z_{l}+p_{l},p_{l})_{l\in[k)}) \delta( z_{k}-\Sigma(p))\,dp.
\end{equation}
\end{rem}

\begin{definition}[prism form]\label{auto:prism-form-definition}\uses{M to K}\usesdefs{auto:Wiener-space-definition,normalized function tuples,prism Brascamp--Lieb}
\lean{Auto.prismForm}
Let $1\le k\le n$. By Proposition \ref{M to K}, $K_k(M)\in W_0(\R^{k+1})$ for every $M\in W_0((\R^2)^k)$. We define for such $M$ and $\mathbf{F}\in\mathfrak{F}$,
\begin{equation}\label{auto:prism-form-formula}\Lambda_k(M)(\mathbf{F}) = \Theta_k(K_k(M))(\mathbf{F}).\end{equation}
\end{definition}

\begin{proposition}[Positivity]\label{Positivity M}\uses{tensor Wiener,Positivity K}\usesdefs{auto:Wiener-space-definition,normalized function tuples,prism Brascamp--Lieb,auto:prism-form-definition}
\lean{Auto.positivityM_memW0,Auto.positivityM_nonnegative}
\leanok
Let $1\le k\le n$ and $i\in [k)$. Let $\tilde{M}\in W_0((\R^2)^{k-1})$ be nonnegative
and let $\phi\in W_0(\R)$ be real valued. Define for $y\in (\R^2)^k$,
\begin{equation}\label{auto:positive-prism-product-kernel}
    M(y)=\tilde{M}(y_{[k)\setminus i}) \phi^{\otimes 2}(y_{i}).
\end{equation}
Then $M\in W_0((\R^2)^k)$ and for all $\F\in {\mathfrak{F}}$,
\begin{equation}\label{auto:positive-prism-form}
    \Lambda_k(M)(\F)\ge 0.
\end{equation}
\end{proposition}

\begin{proof}
    Apply Proposition \ref{tensor Wiener} and Proposition \ref{Positivity K} after a change of variables.
\end{proof}

\begin{proposition}[Cauchy-Schwarz at $k$]\label{Cauchy-Schwarz at k}\uses{single cancellative Cauchy-Schwarz,Monotonicity K,M to K}\usesdefs{auto:Wiener-space-definition,auto:convolution-along-vector-definition,normalized function tuples,auto:prism-form-definition}
\lean{Auto.cauchySchwarzKernel_memW0,Auto.cauchySchwarzLift_memW0,Auto.cauchySchwarzAtK_bound}
\leanok
Let $k,J\in\N$ with $1\le k<n-1$ and $J\ge1$. For $j\in[J)$, let $\rho_j\in W_0((\R^2)^k)$ and let $\varphi_j\in W_0(\R)$ be real valued. Let $i\in[k)$. Let $M\in W_0((\R^2)^k)$ and $\widetilde M\in W_0((\R^2)^{k+1})$ be defined by
\begin{equation}\label{before CS}
M=\sum_{j\in[J)}\rho_j*_{e_i^0+e_i^1}\varphi_j,
\end{equation}
and, for $y\in(\R^2)^{k+1}$,
\begin{equation}\label{Cauchy Schwarz tilde M definition}
\widetilde M(y)=\sum_{j\in[J)}|\rho_j(y_{[k)})|\varphi_j^{\otimes2}(y_k).
\end{equation}
Then for all $\F\in\mathfrak F$,
\begin{equation}\label{Cauchy Schwarz at k estimate}
|\Lambda_k(M)(\F)|
\le
|\Lambda_{k+1}(\widetilde M)(\F)|^{1/2}
\Big(\sum_{j\in[J)}\|\rho_j\|_1\Big)^{1/2}.
\end{equation}
\end{proposition}
\begin{proof}
Without loss of generality we may set $i=k-1$ by applying the coordinate permutation which switches the $i$th pair with the $(k-1)$st pair and preserves the order inside each pair. Then
\[
K_k(M)(z)
=
\sum_{j\in[J)}\int_\R K_k(\rho_j)(z_{[k)},z_k-q)\varphi_j(q)\,dq.
\]
Apply Proposition \ref{single cancellative Cauchy-Schwarz} with $K_k(\rho_j)$ in place of $\rho_j$, and then Proposition \ref{Monotonicity K} with $K_k(|\rho_j|)$ in place of $|K_k(\rho_j)|$. Proposition \ref{M to K} gives
\[
\|K_k(\rho_j)\|_1\le\|\rho_j\|_1.
\]
The resulting kernel is
\[
\sum_{j\in[J)}\int_\R
K_k(|\rho_j|)(z_{[k)},z_{k+1}-q)
\varphi_j(z_k+q)\varphi_j(q)\,dq,
\]
which equals $K_{k+1}(\widetilde M)(z)$. This proves \eqref{Cauchy Schwarz at k estimate}.
\end{proof}

\begin{proposition}[Cauchy-Schwarz at $n-1$]\label{Cauchy-Schwarz at n-1}\uses{doubly cancellative Cauchy-Schwarz,Monotonicity K}\usesdefs{auto:Wiener-space-definition,auto:convolution-along-vector-definition,normalized function tuples,prism Brascamp--Lieb,auto:prism-form-definition}
\lean{Auto.doublyCauchySchwarzKernel_memW0,Auto.cauchySchwarzLift_memW0,Auto.cauchySchwarzAtNMinusOne_bound,Auto.cauchySchwarzAtNMinusOne}
\leanok
Let $J\in\N$ with $J\ge1$ and, for $j\in [J)$, let
$\rho_j\in W_0((\R^2)^{n-1})$  and $\varphi_j\in W_0(\R)$ be real valued.
Let $i \in [n-1)$.
Define
\begin{equation}\label{before CS 2}
    M=\sum_{j\in [J)} \rho_j *_{e_i^0+e_i^1} (\varphi_j * \varphi_j),
\end{equation}
and for $y\in (\R^2)^n$,
\begin{equation}\label{auto:terminal-prism-majorant-kernel}
    \tilde{M}(y)=\sum_{j\in [J)} |\rho_j(y_{[n-1)})| \varphi_j^{\otimes 2}(y_{n-1}).
\end{equation}
Then $M\in W_0((\R^2)^{n-1})$ and $\tilde{M}\in W_0((\R^2)^n)$ and for all $\F\in\mathfrak{F}$
\begin{equation}\label{auto:terminal-prism-majorant-bound}
|\Lambda_{n-1}(M)(\F)| \le \sup_{\mathbf{\tilde{F}}\in\mathfrak{F}}|\Lambda_{n}(\tilde{M})(\mathbf{\tilde{F}})|.
\end{equation}
\end{proposition}

\begin{proof}Without loss of generality we may set $i=n-2$ (replace $\rho_j$ with $\rho_j\circ \sigma$ where $\sigma:(\R^2)^{n-1}\to (\R^2)^{n-1}$ is the coordinate permutation switching the  $i$th pair with the $(n-2)$th pair, keeping the order inside the pair unchanged).
Recall
\[ K_{n-1}(M)(z) = \int_{\R^{n-2}} M((z_l+p_l,p_l)_{l\in [n-2)}, z_{n-2} + z_{n-1} - \Sigma(p), z_{n-1} - \Sigma(p))\,dp. \]
For $z\in\R^n$,

$K_{n-1}(M)(z)$ equals
\[ \sum_{j\in [J)} \int_{\R^{n-2}} \int_\R \int_\R \rho_j((z_l+p_l,p_l)_{l\in [n-2)}, z_{n-2} + z_{n-1} - \Sigma(p) - q, z_{n-1} - \Sigma(p) - q) \varphi_j(q-r)\varphi_j(r)\,dr\,dq\,dp.  \]

Changing variables $q\mapsto q+r$ this is equal to
\[ \sum_{j\in [J)} \int_{\R^2} K_{n-1}(\rho_j)(z_{[n-1)}, z_{n-1} - q - r)\varphi_j(q)\varphi_j(r)\,d(q,r). \]

Applying Proposition \ref{doubly cancellative Cauchy-Schwarz} with $K_{n-1}(\rho_j)$ in place of $\rho_j$ gives
\[ |\Lambda_{n-1}(M)(\F)| \le \sup_{\mathbf{\tilde{F}}\in\mathfrak{F}}|\Theta_n(\tilde{K})(\mathbf{\tilde{F}})|,  \]
where for $u\in \R^{n+1}$,
\[
\tilde{K}(u)=\sum_{j\in [J)}\int_\R |K_{n-1}(\rho_j)|(u_{[n-1)}, u_{n}-q)| \varphi_j(u_{n-1}+q) \varphi_j(q) \, dq.
\]
By Proposition \ref{Monotonicity K} we may replace $|K_{n-1}(\rho_j)|$ by $K_{n-1}(|\rho_j|)$ and then recover $K_n(\tilde{M})(u)$.
\end{proof}

\subsection{Multiplicatively spaced monotone sequences}

\begin{definition}[Multiplicatively spaced monotone sequences]\label{multiplicatively spaced monotone sequences}\uses{}\usesdefs{}
\lean{Auto.SpacedSequence,Auto.A}
A sequence $a:\Z\to\R$ with $a(j)>0$ for every $j\in\Z$ is called multiplicatively spaced if, for each $j\in\Z$
\begin{equation}\label{auto:spaced-sequence-growth}
    2a(j)\le a(j+1)\, .
\end{equation}
Let ${\rm A}$ be the set of such sequences.
\end{definition}

\begin{proposition}[Extension of sequences]\label{Extension of sequences}\uses{}\usesdefs{multiplicatively spaced monotone sequences}
\lean{Auto.extensionOfSequences}
\leanok
Let $J\in\N$ with $J\ge1$ and let $a:[J)\to\R$ satisfy $a(j)>0$ for every $j\in[J)$ and, for all $j\in[J-1)$, we have $a(j+1)\ge 2a(j)$.
Then there is a unique $b\in A$ such that for all $j\in [J)$ we have $a(j)=b(j)$ and for $j\ge J$ we have $b(j)=2^{j-J+1} a(J-1)$ and for $j<0$ we have $b(j)=2^j a(0)$\, .
\end{proposition}
\begin{proof}[Proof \auto]
Define $b:\Z\to\R$ by
\[
b(j)=2^ja(0)\qquad(j<0),
\]
\[
b(j)=a(j)\qquad(j\in[J)),
\]
and
\[
b(j)=2^{j-J+1}a(J-1)\qquad(j\ge J).
\]
Then clearly $b(j)=a(j)$ for every $j\in[J)$, while the desired formulas for
$j<0$ and $j\ge J$ hold by definition.

It remains to check that $b\in {\rm A}$. Let $j\in\Z$. If $j<-1$, then
\[
 b(j+1)=2^{j+1}a(0)=2b(j).
\]
If $j=-1$, then
\[
 b(0)=a(0)=2\cdot 2^{-1}a(0)=2b(-1).
\]
If $j\in [J-1)$, then by assumption
\[
 b(j+1)=a(j+1)\ge 2a(j)=2b(j).
\]
If $j=J-1$, then
\[
 b(J)=2a(J-1)=2b(J-1).
\]
Finally, if $j\ge J$, then
\[
 b(j+1)=2^{j+1-J+1}a(J-1)=2b(j).
\]
Thus for every $j\in\Z$ we have $2b(j)\le b(j+1)$, so $b\in{\rm A}$.

The conditions in the statement prescribe the value of $b(j)$ for every
$j\in\Z$, since the three ranges $j<0$, $j\in[J)$, and $j\ge J$ cover
$\Z$ and are disjoint. Hence such a $b$ is unique.
\end{proof}

\begin{proposition}[Operations on spaced sequences]\label{Operations on spaced sequences}\uses{}\usesdefs{multiplicatively spaced monotone sequences}
\lean{Auto.max_mem_A,Auto.smul_mem_A,Auto.shift_mem_A,Auto.sqrt_sq_add_sq_mem_A}
\leanok
    Let $a,b\in {\rm A}$. Then the following hold:

    (i) $\max(a,b)\in {\rm A}$

    (ii) $t\cdot a\in {\rm A}$ for every $t>0$.

    (iii) Let $n\in\mathbb{Z}$. If $c:\Z\to(0,\infty)$ is defined by $c(j)=a(j+n)$ for all $j\in\mathbb{Z}$, then $c\in {\rm A}$.

    (iv) If $c:\Z\to (0,\infty)$ is defined by $c(j)=\sqrt{a(j)^2+b(j)^2}$ for all $j\in\Z$, then $c\in {\rm A}$.
\end{proposition}
\begin{proof}[Proof \auto]
For (i), let $j\in\Z$. If $\max(a,b)(j)=a(j)$, then
\[
2\max(a,b)(j)=2a(j)\le a(j+1)\le\max(a,b)(j+1).
\]
The case $\max(a,b)(j)=b(j)$ is identical.

For (ii), for every $j\in\Z$,
\[
2(ta)(j)=t(2a(j))\le ta(j+1)=(ta)(j+1).
\]

For (iii), for every $j\in\Z$,
\[
2c(j)=2a(j+n)\le a(j+n+1)=c(j+1).
\]

For (iv), for every $j\in\Z$,
\[
2c(j)=\sqrt{(2a(j))^2+(2b(j))^2}
\le\sqrt{a(j+1)^2+b(j+1)^2}=c(j+1).
\]
Thus each sequence belongs to ${\rm A}$.
\end{proof}

\begin{definition}[Distance of spaced sequences]\label{Distance of spaced sequences}\uses{}\usesdefs{multiplicatively spaced monotone sequences}
\lean{Auto.WithinSequenceDistance,Auto.SequenceDistance}
Given $a,b\in {\rm A}$, define
\begin{equation}\label{auto:spaced-sequence-distance}
    \dist(a,b)=\min\{k\in\mathbb{Z}_{\ge 0}\,:\,\forall j, a(j-k)\le b(j)\le a(j+k)\}
\end{equation}
with the understanding that the value is $\infty$ if there is no such $k$.
\end{definition}

\begin{proposition}[Properties of distance of sequences]\label{Properties of distance of sequences}\uses{}\usesdefs{multiplicatively spaced monotone sequences,Distance of spaced sequences}
\lean{Auto.sequenceDistance_zero_eq,Auto.sequenceDistance_comm,Auto.sequenceDistance_triangle,Auto.sequenceDistance_shift_le,Auto.sequenceDistance_smul,Auto.sequenceDistance_pow_two_smul_le}
\leanok
    Let $a,b,c\in {\rm A}$.
    Then

    (i) If $\dist(a,b)=0$, then $a=b$.

    (ii) $\dist(a,b)=\dist(b,a)$

    (iii) $\dist(a,c)\le \dist(a,b)+\dist(b,c)$

    (iv) If $c(j)=b(j+1)$, then $\dist(a,c)\le \dist(a,b)+1$.

    (v) For all $t>0$, $\dist(t\cdot a,t\cdot b)=\dist(a,b)$

    (vi) For all $h\in\mathbb{Z}$,
    $\dist(a,2^h a)\le |h|$.
\end{proposition}
\begin{proof}[Proof \auto]
For (i), if $\dist(a,b)=0$, then for every $j\in\Z$,
\[
 a(j)\le b(j)\le a(j),
\]
and hence $a(j)=b(j)$.

For (ii), suppose first that $\dist(a,b)\le k<\infty$. Then for every $j\in\Z$,
\[
 a(j-k)\le b(j)\le a(j+k).
\]
Applying this with $j-k$ in place of $j$ gives
\[
 b(j-k)\le a(j),
\]
and applying it with $j+k$ in place of $j$ gives
\[
 a(j)\le b(j+k).
\]
Therefore
\[
 b(j-k)\le a(j)\le b(j+k)
\]
for every $j\in\Z$, so $\dist(b,a)\le k$. By symmetry, $\dist(a,b)=\dist(b,a)$.

For (iii), if either $\dist(a,b)$ or $\dist(b,c)$ is infinite, there is nothing to prove. Otherwise let
\[
 \dist(a,b)\le k,\qquad \dist(b,c)\le l.
\]
Then for every $j\in\Z$,
\[
 b(j-l)\le c(j)\le b(j+l).
\]
Using the inequalities for $a$ and $b$ on the two outside terms gives
\[
 a(j-l-k)\le b(j-l)\le c(j)\le b(j+l)\le a(j+l+k).
\]
Hence
\[
 \dist(a,c)\le k+l.
\]
Taking the minimum over admissible $k$ and $l$ gives
\[
 \dist(a,c)\le \dist(a,b)+\dist(b,c).
\]

For (iv), if $\dist(a,b)$ is infinite, there is nothing to prove. Otherwise let $\dist(a,b)\le k$. Since $c(j)=b(j+1)$, for every $j\in\Z$,
\[
 a(j+1-k)\le c(j)\le a(j+1+k).
\]
Because $a\in{\rm A}$ is increasing, we have
\[
 a(j-k-1)\le a(j+1-k)
\]
and
\[
 a(j+1+k)=a(j+k+1).
\]
Therefore
\[
 a(j-(k+1))\le c(j)\le a(j+(k+1)),
\]
so $\dist(a,c)\le k+1$. Taking the minimum over admissible $k$ gives
\[
 \dist(a,c)\le \dist(a,b)+1.
\]

For (v), for every $k\in\Z_{\ge 0}$, the inequalities
\[
 a(j-k)\le b(j)\le a(j+k)
\]
hold for every $j\in\Z$ if and only if
\[
 t a(j-k)\le t b(j)\le t a(j+k)
\]
hold for every $j\in\Z$. Hence
\[
 \dist(t\cdot a,t\cdot b)=\dist(a,b).
\]

For (vi), first suppose $h\ge 0$. Iterating the defining inequality for $a\in{\rm A}$ gives
\[
 2^h a(j)\le a(j+h)
\]
for every $j\in\Z$. Also, since $a$ is increasing,
\[
 a(j-h)\le a(j)\le 2^h a(j).
\]
Thus
\[
 a(j-h)\le 2^h a(j)\le a(j+h),
\]
so $\dist(a,2^h a)\le h=|h|$.

If $h<0$, write $-h>0$. Iterating the defining inequality gives
\[
 2^{-h}a(j+h)\le a(j),
\]
and hence
\[
 a(j+h)\le 2^h a(j).
\]
Also, since $a$ is increasing,
\[
 2^h a(j)\le a(j)\le a(j-h).
\]
Therefore
\[
 a(j-(-h))\le 2^h a(j)\le a(j+(-h)),
\]
so $\dist(a,2^h a)\le -h=|h|$.

\end{proof}

\begin{definition}[Closed balls in $\mathrm{A}$]\label{closed balls in A}\uses{}\usesdefs{multiplicatively spaced monotone sequences,Distance of spaced sequences}
\lean{Auto.sequenceDistanceBall}
For $a\in\mathrm{A}$ and $r>0$ denote
\begin{equation}\label{auto:spaced-sequence-distance-ball}B_{\mathrm{dist}}(a,r)=\{b\in\mathrm{A}\,:\,\dist(a,b)\le r\}.\end{equation}
\end{definition}

\subsection{Gaussians}

Recall that $\g(x)=e^{-\pi x^2}$ for $x\in\R$.

\begin{proposition}\label{square root one minus Gaussian}\uses{}\usesdefs{gaussian}
\lean{Auto.sqrtOneMinusGaussian,Auto.sqrtOneMinusGaussian_wellDefined,Auto.continuous_sqrtOneMinusGaussian,Auto.sqrtOneMinusGaussian_lower,Auto.sqrtOneMinusGaussian_bounds}
\leanok
    The function
    \begin{equation}\label{auto:square-root-one-minus-Gaussian}
        f(x)=\sqrt{1-\g(x)}
    \end{equation}
    is well-defined
    using the nonnegative square root and is continuous on $\R$.
    For $|x|\le \frac 12$, we have
    \begin{equation}\label{auto:square-root-one-minus-Gaussian-small-lower-bound}
        f(x)\ge \tfrac 12 |x|
    \end{equation}
    and for $|x|\ge \frac 12$ we have
 \begin{equation}\label{auto:square-root-one-minus-Gaussian-large-bounds}
        1-\g(x)\le f(x)\le 1.
    \end{equation}
\end{proposition}
\begin{proof}[Proof \auto]
For every $x\in\R$ we have $0<\g(x)\le 1$, and hence
$1-\g(x)\ge 0$. Thus $f$ is well defined using the nonnegative square
root. Since $\g$ is continuous and the square root is continuous on
$[0,\infty)$, the function $f$ is continuous.

Let $|x|\le 1/2$ and put $u=\pi x^2$. Then $0\le u\le \pi/4<1$. Using
$e^{-u}\le 1-u+u^2/2$, we get
\[
1-e^{-u}\ge u-\tfrac{u^2}{2}\ge \tfrac u2 .
\]
Therefore
\[
f(x)^2=1-e^{-\pi x^2}\ge \tfrac{\pi x^2}{2}\ge \tfrac{x^2}{4},
\]
and hence $f(x)\ge \frac12 |x|$.

Finally, if $|x|\ge 1/2$, then $0\le 1-\g(x)\le 1$. For every
$a\in[0,1]$ we have $a\le \sqrt a\le 1$, so with $a=1-\g(x)$ this gives
\[
1-\g(x)\le f(x)\le 1 .
\]
\end{proof}

\begin{proposition}\label{Gaussian bump decay}\uses{}\usesdefs{gaussian,bracket bump}
\lean{Auto.gaussianBumpDecay}
\leanok
For every $x\in\R$, $m,N\in\N$,
\begin{equation}\label{auto:Gaussian-derivative-decay} |\g^{(m)}(x)|\le C_{\ref{Gaussian bump decay},m,N} \langle x\rangle^N, \end{equation}
where
\begin{equation}\label{auto:Gaussian-bump-decay-constant}
C_{\ref{Gaussian bump decay},m,N}
=\bigl(64(m+1)\bigr)^{m/2}\bigl(4(N+1)\bigr)^{N/2}.
\end{equation}

\end{proposition}
\begin{proof}[Proof \auto]
We first prove that, for every $m\in\N$ and $x\in\R$,
\[
|\g^{(m)}(x)|
\le
\bigl(64(m+1)\bigr)^{m/2}e^{-\pi x^2/2}.
\]
For $m=0$ this is immediate. For $m\ge 1$, repeated differentiation gives
\[
\g^{(m)}(x)=P_m(x)e^{-\pi x^2},
\]
where
\[
P_m(x)
=
m!\sum_{j=0}^{\lfloor m/2\rfloor}
\tfrac{(-1)^{m+j}2^{m-2j}\pi^{m-j}x^{m-2j}}
{j!(m-2j)!}.
\]
Indeed, one proves this using induction and
$P_{m+1}(x)=P_m'(x)-2\pi xP_m(x).$
Set $d=m-2j$. We use
\[
|x|^d e^{-\pi x^2/2}
\le
\left(\tfrac{d+1}{\pi e}\right)^{d/2}
\]
with the convention that the right hand side is $1$ when $d=0$.
Also,
\[
\tfrac{m!}{j!d!}
=
\binom{m}{d}\tfrac{(2j)!}{j!}
\le
2^m m^j .
\]
Hence, using $\pi<4$ and $e>2$, each summand in $|P_m(x)|e^{-\pi x^2/2}$ is at most
\[
2^m(\pi m)^j
\left(\tfrac{4\pi(d+1)}{e}\right)^{d/2}
\le
\bigl(32(m+1)\bigr)^{m/2}.
\]
Since there are at most $2^{m/2}$ summands, we get
\[
|P_m(x)|e^{-\pi x^2/2}
\le
\bigl(64(m+1)\bigr)^{m/2}.
\]
This proves the displayed Gaussian derivative bound.

Next, for every $N\in\N$,
\[
(1+|x|)^N e^{-\pi x^2/2}
\le
2^N\max\bigl(1,|x|^Ne^{-\pi x^2/2}\bigr)
\le
\bigl(4(N+1)\bigr)^{N/2}.
\]
Equivalently,
\[
e^{-\pi x^2/2}
\le
\bigl(4(N+1)\bigr)^{N/2}\langle x\rangle^N .
\]
Therefore
\[
|\g^{(m)}(x)|
\le
\bigl(64(m+1)\bigr)^{m/2}
\bigl(4(N+1)\bigr)^{N/2}
\langle x\rangle^N.
\]
This proves the claim.
\end{proof}

\begin{proposition}[Elementary Gaussian properties]\label{Elementary Gaussian properties}\uses{Gaussian bump decay}\usesdefs{gaussian,auto:Wiener-space-definition}
\lean{Auto.gaussian_memSchwartz,Auto.gaussian_memW0,Auto.gaussian_fourier_fixed,Auto.gaussianRescale_convolution,Auto.gaussianRescale_fourier}
\leanok

(i) $\g\in \mathcal{S}(\R)\subset W_0(\R)$

(ii) $\widehat{\g}=\g$.

(iii) We have for $\lambda,\mu>0$
    \begin{equation}\label{eq:gaussconv}  \g_{(\lambda)}*\g_{(\mu)}=\g_{(\sqrt{\lambda^2+\mu^2})}\, .
    \end{equation}

(iv) We have for $\lambda>0$
 \begin{equation}\label{eq:scaleFT}
     \widehat{\g_{(\lambda)}}(\xi)=\g(\lambda \xi)
 \end{equation}
\end{proposition}
\begin{proof}[Proof \auto]
By Proposition \ref{Gaussian bump decay}, every derivative of $\g$ decays
faster than any polynomial. Hence $\g\in \mathcal S(\R)$, and therefore
$\g\in W_0(\R)$.

We next prove the Fourier transform identity. Set
\[
I(\xi)=\int_\R e^{-\pi x^2}e^{-2\pi i x\xi}\,dx.
\]
Differentiation under the integral and integration by parts give
\[
I'(\xi)=i\int_\R \g'(x)e^{-2\pi i x\xi}\,dx=-2\pi\xi I(\xi).
\]
The standard Gaussian integral gives $I(0)=1$. Hence the unique solution of this differential equation is $I(\xi)=e^{-\pi\xi^2}$, and therefore $\widehat{\g}=\g$.

Let $\lambda>0$ and write
\[
\g_{(\lambda)}(x)=\lambda^{-1}\g(x/\lambda).
\]
Then, by the change of variables $x=\lambda y$,
\[
\widehat{\g_{(\lambda)}}(\xi)
=
\int_\R \lambda^{-1}\g(x/\lambda)e^{-2\pi i x\xi}\,dx
=
\int_\R \g(y)e^{-2\pi i y(\lambda \xi)}\,dy
=
\g(\lambda \xi).
\]
This proves \eqref{eq:scaleFT}.

Finally, using \eqref{eq:scaleFT} and $\widehat{\g}=\g$, we have
\[
\widehat{\g_{(\lambda)}*\g_{(\mu)}}(\xi)
=
\widehat{\g_{(\lambda)}}(\xi)\widehat{\g_{(\mu)}}(\xi)
=
\g(\lambda \xi)\g(\mu \xi)
=
e^{-\pi(\lambda^2+\mu^2)\xi^2}.
\]
Since
\[
e^{-\pi(\lambda^2+\mu^2)\xi^2}
=
\g(\sqrt{\lambda^2+\mu^2}\,\xi)
=
\widehat{\g_{(\sqrt{\lambda^2+\mu^2})}}(\xi),
\]
Fourier inversion gives
\[
\g_{(\lambda)}*\g_{(\mu)}
=
\g_{(\sqrt{\lambda^2+\mu^2})}.
\]
\end{proof}

\begin{proposition}\label{poisson to abel}\uses{}\usesdefs{}
\lean{Auto.poissonKernel,Auto.poissonKernel_fourier}
\leanok
Let
\begin{equation}\label{auto:Gaussian-comparison-function}
    p(x)=2(1+(2\pi x)^2)^{-1}.
\end{equation}
Then
\begin{equation}\label{eq: p hat}
    \widehat{p}(\xi)=e^{- |\xi|}.
\end{equation}
\end{proposition}
\begin{proof}[Proof of Proposition \ref{poisson to abel}]
For \(a\in\R\), set
\[
    I(a)=\int_0^\infty e^{-t}\cos(at)\,dt,
    \qquad
    J(a)=\int_0^\infty e^{-t}\sin(at)\,dt.
\]
Integration by parts, differentiating the trigonometric functions and integrating the exponential function gives
\[
    I(a)=1-aJ(a),
    \qquad
    J(a)=aI(a).
\]
Hence
\[
    I(a)=\tfrac{1}{1+a^2}.
\]
Using this for $a=2\pi x$ for \(x\in\R\)
\begin{equation}
    p(x)=2\int_0^\infty e^{-\xi}\cos(2\pi x\xi)\,d\xi  =\int_{\R}e^{-|\xi|}e^{2\pi i x\xi}\,d\xi
\end{equation}
 The Fourier inversion formula gives \eqref{eq: p hat}. Namely, it is applicable because
 both $p$ and the right hand side of $\eqref{eq: p hat}$ are continuous and integrable.

\end{proof}

\begin{proposition}\label{auxiliary function B}\uses{square root one minus Gaussian,poisson to abel}\usesdefs{gaussian}
\lean{Auto.auxiliaryFunctionB_properties}
\leanok
    With $p$ as in Proposition \ref{poisson to abel}, define $B:\R\to \R$ by
\begin{equation}\label{auto:Gaussian-square-root-remainder}
    B(\xi)=1-\sqrt{1-\g(\xi)}-\widehat{p}(\sqrt{\pi}\xi)\, .
\end{equation}
Then $B$ is smooth on $\R\setminus \{0\}$ and continuous on $\R$. The function $B'$
has a continuous extension to $\R$ which is integrable and the
function $B''$, extended to be $0$ at $0$, is Borel measurable and integrable on $\R$.
Moreover,
\begin{equation}\label{eq:B-explicit-norms}
    \|B\|_1\le 6,\qquad
    \|B'\|_1\le 12,\qquad
    \|B''\|_1\le 68.
\end{equation}
\end{proposition}
\begin{proof}[Proof \auto]
Set
\[
    \alpha=\sqrt{\pi},\qquad
    E(\xi)=e^{-\alpha|\xi|},\qquad
    s(\xi)=\sqrt{1-\g(\xi)}.
\]
Then \(B=1-s-E\). By Proposition \ref{square root one minus Gaussian}, the function $s$ is continuous. Hence \(B\) is smooth on \(\R\setminus\{0\}\), even, and continuous on \(\R\), since \(s(0)=0\) and \(E(0)=1\).

We first control the origin. For \(\xi\ge0\), put \(u=\alpha \xi\). Since
\[
    1-e^{-u^2}=u^2\int_0^1 e^{-tu^2}\,dt,
\]
we can write
\[
    s(\xi)=uR(u),\qquad
    R(u)=\left(\int_0^1 e^{-tu^2}\,dt\right)^{1/2}.
\]
For \(0\le u\le1\), elementary differentiation under the integral sign gives
\[
    \tfrac12\le R(u)\le1,\qquad |R'(u)|\le1,\qquad |R''(u)|\le5.
\]
Thus, with
\[
    F(u)=1-uR(u)-e^{-u},
\]
we have \(B(\xi)=F(\alpha\xi)\) for \(\xi\ge0\), \(F(0)=F'(0)=0\), and
\[
    |F''(u)|
    =
    |-2R'(u)-uR''(u)-e^{-u}|
    \le 8
    \qquad (0\le u\le1).
\]
Since \(\alpha^{-1}>1/2\), it follows that
\[
    |B''(\xi)|\le 8\alpha^2=8\pi<32
    \qquad (0<|\xi|\le1/2),
\]
and \(B'\) extends continuously to \(0\) by setting \(B'(0)=0\).

It remains to estimate the tails. By Proposition \ref{square root one minus Gaussian},
$s(\xi)\ge1-\g(\xi)$ when $|\xi|\ge1/2$. Since $\pi>3$ and
\[
e^{3/4}>1+\tfrac34+\tfrac12\left(\tfrac34\right)^2=\tfrac{65}{32}>2,
\]
we have $\g(\xi)\le e^{-\pi/4}<1/2$, and hence $s(\xi)\ge1/2$. For \(\xi>0\),
\[
    B'(\xi)
    =
    -\tfrac{\pi\xi\g(\xi)}{s(\xi)}
    +
    \alpha e^{-\alpha\xi},
\]
and
\[
    B''(\xi)
    =
    \tfrac{-\pi\g(\xi)+2\pi^2\xi^2\g(\xi)}{s(\xi)}
    +
    \tfrac{\pi^2\xi^2\g(\xi)^2}{s(\xi)^3}
    -
    \pi e^{-\alpha\xi}.
\]
By evenness, the same absolute value bounds hold on the negative half-line. Hence, for
\(|\xi|\ge1/2\),
\[
    |B(\xi)|\le \g(\xi)+E(\xi),
\]
\[
    |B'(\xi)|\le 2\pi|\xi|\g(\xi)+\alpha E(\xi),
\]
and
\[
    |B''(\xi)|
    \le
    2\pi\g(\xi)+12\pi^2\xi^2\g(\xi)+\pi E(\xi).
\]

We now integrate. On \(|\xi|\le1/2\), the trivial bound \(|B|\le3\) gives contribution at
most \(3\). On the complement,
\[
    \int_\R(\g+E)
    =
    1+\tfrac{2}{\alpha}
    \le3.
\]
Thus \(\|B\|_1\le6\).

For \(B'\), the local bound \(|B'(\xi)|<32|\xi|\) on \(|\xi|\le1/2\) gives local
contribution at most \(8\). The tail contributes at most
\[
    2\pi\int_\R |\xi|\g(\xi)\,d\xi
    +
    \alpha\int_\R E(\xi)\,d\xi
    =
    2+2=4.
\]
Thus \(\|B'\|_1\le12\).

Finally, for \(B''\), the local contribution is at most \(32\). The tail contributes at most
\[
\int_\R \bigl(2\pi\g(\xi)+12\pi^2\xi^2\g(\xi)+\pi E(\xi)\bigr)\,d\xi = 2\pi+12\pi^2\tfrac{1}{2\pi}+\pi\tfrac{2}{\alpha} = 8\pi+2\sqrt{\pi}<36.
\]
Therefore
\[
    \|B''\|_1\le 32+36=68.
\]
This proves the claimed integrability and the explicit estimates. The stated Borel measurability
of the extension of \(B''\) is immediate from continuity on each half-line and the definition at
the single point \(0\).
\end{proof}

\begin{proposition}[square root of Gaussian decay]\label{square root of Gaussian decay}\uses{auxiliary function B,poisson to abel,Elementary Gaussian properties}\usesdefs{gaussian,bracket bump,auto:Wiener-space-definition}
\lean{Auto.sqrtGaussianDecay}
\leanok
For $x\in\R$ let
\begin{equation}\label{auto:inverse-Fourier-square-root-Gaussian} \rho(x) = \mathcal{F}^{-1}(\xi \mapsto 1- \sqrt{1-\g(\xi)})(x). \end{equation}
This is a well-defined function in $W_0(\R)$ satisfying for all $x\in\R$,
\begin{equation} \label{sqr gauss C}
    0\le \rho(x) \le C_{\ref{square root of Gaussian decay}} \langle x\rangle^2,
\end{equation}
where $C_{\ref{square root of Gaussian decay}}=17$.
\end{proposition}
\begin{proof}[Proof \auto]
Set
\[
    m(\xi)=1-\sqrt{1-\g(\xi)}.
\]
We first construct \(\rho\). By Proposition \ref{auxiliary function B},
\[
    m(\xi)=B(\xi)+e^{-\sqrt{\pi}|\xi|}.
\]
Both terms on the right-hand side are integrable, so \(m\in L^1(\R)\). Define
\[
    \rho=\widecheck m.
\]
We will prove that \(\rho\in W_0(\R)\), that it satisfies the claimed decay estimate, and that
\(\widehat\rho=m\).

Let \(\alpha=\sqrt{\pi}\) and \(E_\alpha(\xi)=e^{-\alpha|\xi|}\). By Proposition
\ref{poisson to abel},
\[
    E_\alpha(\xi)=\widehat p(\alpha \xi).
\]
By the scaling rule for the Fourier transform,
\[
    \widecheck{E_\alpha}(x)
    =
    \alpha^{-1}p(x/\alpha)
    =
    \tfrac{2\alpha}{\alpha^2+(2\pi x)^2}
    =
    \tfrac{2/\sqrt{\pi}}{1+4\pi x^2}.
\]
In particular, since $\pi>3$,
\[
\frac{2}{\sqrt\pi}\frac{(1+|x|)^2}{1+4\pi x^2}
<\frac76\frac{(1+|x|)^2}{1+12x^2}
\le\frac76\frac{13}{12}<\frac32.
\]
Here the middle inequality follows from
$12(1+|x|)^2\le13(1+12x^2)$.
Thus
\begin{equation}\label{eq:abel-decay}
    0\le \widecheck{E_\alpha}(x)\le \tfrac32\langle x\rangle^{2}.
\end{equation}

It remains to estimate \(\widecheck B\). By Proposition \ref{auxiliary function B},
\[
    \|B\|_1\le 6,
    \qquad
    \|B''\|_1\le 68.
\]
Therefore
\[
    |\widecheck B(x)|\le 6
    \qquad (x\in\R).
\]
For \(x\neq 0\), integration by parts twice and $\pi>3$ give
\[
    |\widecheck B(x)|
    \le
    \tfrac{\|B''\|_1}{4\pi^2x^2}
    <
    \tfrac{17}{9x^2}.
\]
If $|x|\le5/9$, then
$6(1+|x|)^2\le6(14/9)^2<15$.
If $|x|>5/9$, then
\[
\tfrac{17}{9x^2}(1+|x|)^2
<\tfrac{17}{9}(14/5)^2<15.
\]
Thus
\begin{equation}\label{eq:Bcheck-decay}
|\widecheck B(x)|\le15\langle x\rangle^2.
\end{equation}
Combining these estimates, we obtain
\[
|\rho(x)|\le\tfrac{33}{2}\langle x\rangle^2
<17\langle x\rangle^2.
\]
The function \(\rho\) is continuous because \(m\in L^1\). The preceding decay estimate implies
\(\rho\in W_0(\R)\): indeed, for \(|x|\le 2\) the local \(W_0\)-supremum is bounded, while for
\(|x|>2\) and \(|x-y|\le 1\) we have \(|y|\ge |x|/2\), so the local supremum is bounded by a
constant multiple of \(\langle x\rangle^{2}\), which is integrable on \(\R\).
It remains to prove non-negativity. For \(0\le t\le 1\), the binomial series gives
\[
    1-\sqrt{1-t}
    =
    \sum_{k=1}^\infty a_k t^k,
    \qquad
    a_k=(-1)^{k+1}\binom{1/2}{k}>0.
\]
Moreover \(\sum_{k=1}^\infty a_k=1\). Hence, for every \(\xi\in\R\),
\[
    m(\xi)
    =
    \sum_{k=1}^\infty a_k \g(\xi)^k
    =
    \sum_{k=1}^\infty a_k \g(\sqrt{k}\xi).
\]
By Proposition \ref{Elementary Gaussian properties},
\[
    \g(\sqrt{k}\xi)=\widehat{\g_{(\sqrt{k})}}(\xi).
\]
The series
\[
\sum_{k=1}^\infty a_k\g_{(\sqrt{k})}(x)
\]
converges in $L^1(\R)$, because $\|\g_{(\sqrt{k})}\|_1=1$ and $\sum_ka_k=1$. Its Fourier transform is $m$. Since both this series and $\rho=\widecheck m$ are in $L^1$, Fourier inversion and injectivity of the Fourier transform give
\[
    \rho(x)
    =
    \sum_{k=1}^\infty a_k \g_{(\sqrt{k})}(x).
\]
Every term on the right-hand side is non-negative, so \(\rho(x)\ge 0\) for every \(x\in\R\).
\end{proof}

\subsection{Bumps and their estimates}\label{bump section}

\begin{lemma}
\label{lem:smoothdecay}\uses{}\usesdefs{auto:Wiener-space-definition}
\lean{Auto.smoothDecay}
\leanok
Let $N\geq 2$ be an integer. Let $\zeta:\R\to\C$ be an $N$ times continuously differentiable function with compact support. Then the function $\phi=\mathcal F^{-1}\zeta$ belongs to $W_0(\R)$ and we have for all $x\in\R$, $x\neq0$, the estimate

\begin{equation}\label{auto:Fourier-decay-minimum-bound}|\phi(x)|\le \min(\|\widehat{\phi}\|_1, \|\widehat{\phi}^{(N)}\|_1 (2\pi)^{-N} |x|^{-N}).\end{equation}
\end{lemma}

\begin{proof}
By the assumptions on $\zeta$, we deduce that $\zeta\in L^1(\R)$. Thus, $\phi$ is a well-defined continuous function.

By definition,
\begin{equation}
    \label{fourierinvphi}
    \phi(x) = \int_{\R} \zeta(\xi)e^{2\pi i x\xi} d\xi,
\end{equation}
and  by the triangle inequality
\begin{equation}
    \label{xleone}
    |\phi(x)| = \Big | \int_{\R} \zeta(\xi)e^{2\pi i x\xi} d\xi \Big | \le \int_{\R} |\zeta(\xi)e^{2\pi i x\xi}| d\xi \le \|\zeta\|_1 .
\end{equation}

For the second estimate we integrate by parts  in \eqref{fourierinvphi} to obtain
\[\phi(x) = (2\pi i x)^{-N} \int_{\R} \zeta^{(N)}(\xi) e^{2\pi i x \xi} d\xi, \]
and thus by the triangle inequality,
\[|\phi(x)| \le (2\pi)^{-N}|x|^{-N} \int_{\R}|\zeta^{(N)}(\xi) e^{2\pi i x \xi}| d\xi \]
\begin{equation}
    \label{xgeone}
    \le  (2\pi)^{-N} |x|^{-N} \|\zeta^{(N)}\|_1.
\end{equation}
Combining \eqref{xleone} and \eqref{xgeone} gives that $\phi \in W_0(\R)$. By the Fourier inversion formula, we then conclude that $\zeta=\widehat{\phi}$.
\end{proof}

\begin{lemma}
    \label{lem: min and bracket}\uses{}\usesdefs{bracket bump} Let $N\ge 1$ be an integer, then
\lean{Auto.min_and_bracket}
\leanok
\begin{equation}\label{auto:min-power-bracket-bound}\min(1, |x|^{-N})\le 2^N   \langle x\rangle^{N}\, .\end{equation}
\end{lemma}
\begin{proof}
 We have, using that the mean is less than the maximum
    \[\min(1, |x|^{-N})=\max(1, |x|)^{-N}\le (2^{-1}(1+|x|))^{-N}=2^N  \langle x\rangle^{N}\, .\]
\end{proof}

We will mainly use these lemmas when we have control on the $L^\infty$ norm and compact support of   the Fourier transform, so we state this case separately.

\begin{lemma}
   \label{lem:smoothdecay2}\uses{lem:smoothdecay,lem: min and bracket}\usesdefs{bracket bump,auto:Wiener-space-definition}
\lean{Auto.smoothDecay2}
\leanok
Let $N\geq 2$ be an integer. Let $\zeta:\R\to\C$ be an $N$ times continuously differentiable function with compact support.
Then the function $\phi=\mathcal F^{-1} \zeta$ belongs to $W_0(\R)$ and we have for all $x\in \R$ the estimate
\begin{equation}\label{E:smoothdecay2}
|\phi(x)|\leq C_{\ref{lem:smoothdecay2}, N} \max(\|\widehat{\phi}\|_\infty, (2\pi )^{-N}\|\widehat{\phi}^{(N)}\|_\infty)|\supp (\widehat{\phi})|   \langle x\rangle^{N}
\end{equation}
with $C_{\ref{lem:smoothdecay2}, N}=2^{N} $.
\end{lemma}

\begin{proof}
The fact that $\phi \in W_0(\R)$ and $\widehat{\phi}=\zeta$ follows from Lemma~\ref{lem:smoothdecay}.
We have for any compactly supported bounded function $f:\R\to\C$
  \[\|f\|_1 \le \|f\|_\infty |\textup{supp}(f)| .\]
Using this twice and Lemma \ref{lem:smoothdecay} estimates $|\phi(x)|$ by
\[\le \min(\|\widehat{\phi}\|_1, \|\widehat{\phi}^{(N)}\|_1 (2\pi)^{-N} |x|^{-N})
\le |\supp (\widehat{\phi})| \min(\|\widehat{\phi}\|_\infty, \|\widehat{\phi}^{(N)}\|_\infty (2\pi)^{-N} |x|^{-N})\, . \]
  The claim now follows from Lemma \ref{lem: min and bracket}.
\end{proof}

\begin{proposition}[mean value bump estimate]\label{mean value bump estimate 2}\uses{lem:smoothdecay2,lem: min and bracket}\usesdefs{bracket bump,auto:Wiener-space-definition}
\lean{Auto.meanValueBumpEstimate}
\leanok
Let $N\in\mathbb N$ with $N\ge1$ and let $\rho$ be a $W_0(1)$ function with
$\widehat{\rho}$ supported in $[-1,1]$ and $\widehat{\rho}$ being $N$ times continuously differentiable.

Then for all $x,y\in\mathbb{R}$
\begin{equation}\label{auto:mean-value-bump-bound} |\rho(x+y)-\rho(x)| \le C_{\ref{mean value bump estimate 2},N}\max_{0\le \nu\le N}\|(2\pi)^{-\nu}\widehat{\rho}^{(\nu)}\|_\infty \min(1, 2\pi |y|) (\langle x+y\rangle^N + \langle x\rangle^N). \end{equation}
where
\begin{equation}\label{auto:mean-value-bump-constant-definition}
C_{\ref{mean value bump estimate 2},N}=2^N\max(2,(1+(2\pi)^{-1})^N).
\end{equation}

\end{proposition}
\begin{proof}[Proof \auto]
For $N=1$, Fourier inversion and one integration by parts show directly that the pointwise estimate in Lemma \ref{lem:smoothdecay2}, with coefficient $2$, remains valid for functions whose Fourier transform is supported in $[-1,1]$. Indeed, if $\zeta$ has this support and
\[
B=\max(\|\zeta\|_\infty,(2\pi)^{-1}\|\zeta'\|_\infty),
\]
then, by Lemma \ref{lem: min and bracket},
\[
|\mathcal F^{-1}\zeta(x)|\le2B\min(1,|x|^{-1})\le4B\langle x\rangle.
\]
Thus the pointwise estimate from that lemma may be used below for every $N\ge1$.

Set
\[
\delta=2\pi|y|,
\qquad
B=\max_{0\le \nu\le N}\|(2\pi)^{-\nu}\widehat\rho^{(\nu)}\|_\infty.
\]
If $\delta>1$, the triangle inequality and the pointwise estimate just discussed give
\[
|\rho(x+y)-\rho(x)|
\le2^{N+1}B(\langle x+y\rangle^N+\langle x\rangle^N),
\]
which is bounded by the right-hand side of \eqref{auto:mean-value-bump-bound}.

Assume now that $\delta\le1$ and set $\widetilde\rho(x)=\rho(x+y)-\rho(x)$. Then
\[
\widehat{\widetilde\rho}(\xi)=\widehat\rho(\xi)(e^{2\pi i\xi y}-1)
\]
and
\[
\widehat{\widetilde\rho}^{(N)}(\xi)
=\widehat\rho^{(N)}(\xi)(e^{2\pi i\xi y}-1)
+\sum_{\nu=0}^{N-1}\binom N\nu\widehat\rho^{(\nu)}(\xi)(2\pi iy)^{N-\nu}e^{2\pi i\xi y}.
\]
Since $|\xi|\le1$ on the support and $\delta^{N-\nu}\le\delta$ for $\nu<N$, these identities give
\[
\|\widehat{\widetilde\rho}\|_\infty\le B\delta
\]
and
\[
(2\pi)^{-N}\|\widehat{\widetilde\rho}^{(N)}\|_\infty
\le B\delta(1+(2\pi)^{-1})^N.
\]
The pointwise estimate gives
\[
|\widetilde\rho(x)|
\le2^{N+1}B\delta(1+(2\pi)^{-1})^N\langle x\rangle^N.
\]
Applying the same estimate with $-y$ in place of $y$, at the point $x+y$, gives the same bound with $\langle x+y\rangle^N$ in place of $\langle x\rangle^N$. Averaging the two bounds yields
\[
|\rho(x+y)-\rho(x)|
\le2^NB\delta(1+(2\pi)^{-1})^N(\langle x+y\rangle^N+\langle x\rangle^N),
\]
and \eqref{auto:mean-value-bump-bound} follows.
\end{proof}

\begin{definition}[standard bump]\label{standard bump}\uses{}\usesdefs{}
\lean{Auto.standardBumpFinite,Auto.standardBump,Auto.standardBumpRescale}
Define, for \(l\in\N\) with \(l\ge1\),
\begin{equation}\label{auto:standard-bump-finite-product}
 \Phi_l
 =
 \widehat{1_{[-3/4,3/4]}}
 \prod_{i\in [l)}
 2^{i+2}
 \widehat{1_{[-2^{-i-3},2^{-i-3}]}} .
\end{equation}
For $x\in\R$, define
\begin{equation}\label{auto:standard-bump-limit}\Phi(x)=\lim_{l\to\infty}\Phi_l(x).\end{equation}
(The limit exists for every $x\in\R$.)
\end{definition}

\begin{proposition}[standard bump]\label{standard bump properties}\uses{lem:smoothdecay2}\usesdefs{bracket bump,standard bump}
\lean{Auto.standardBumpProperties,Auto.standardBumpProperties_l1Convergence,Auto.standardBumpProperties_linfConvergence,Auto.standardBumpProperties_schwartz,Auto.standardBumpProperties_fourierShape,Auto.standardBumpProperties_fourierDerivativeEstimate,Auto.standardBumpProperties_derivativeDecay}
\leanok
The limit $\lim_{l\to \infty} \Phi_l$ exists in $L^\infty$  and $L^1$ sense and is a Schwartz function $\Phi$
whose Fourier transform takes values in $[0,1]$, is supported in $[-1,1]$, and constant $1$
on $[-1/2,1/2]$.

It satisfies for every $m,N\in\N$ and every $x\in\R$,
\begin{equation}
    \label{Phi_derbound}
    \|  (\widehat{\Phi^{(m)}})^{(N)}  \|_\infty \le \tilde{C}_{\ref{standard bump properties},m,N}
\end{equation}
and
\begin{equation}\label{auto:standard-bump-derivative-decay} |\Phi^{(m)}(x)| \le C_{\ref{standard bump properties},m,N} \langle x\rangle^N, \end{equation}
where $\tilde{C}_{\ref{standard bump properties},m,N}=2^{4m + 2N^2 + 5N}$ and $C_{\ref{standard bump properties},m,N}=2^{4m+2N^2+6N+2}.$
\end{proposition}

\begin{proof}[Proof \auto]
Let
\[
a_l=\sum_{i\in[l)}2^{-i-3}=\tfrac14-2^{-l-2}.
\]
The Fourier transform of $\Phi_l$ is the convolution of $1_{[-3/4,3/4]}$ with probability densities supported on the intervals $[-2^{-i-3},2^{-i-3}]$, $i\in[l)$. Hence it takes values in $[0,1]$, is supported in
\[
[-3/4-a_l,3/4+a_l]=[-1+2^{-l-2},1-2^{-l-2}],
\]
and is equal to one on
\[
[-3/4+a_l,3/4-a_l]=[-1/2-2^{-l-2},1/2+2^{-l-2}].
\]

Every factor in the product defining $\Phi_l$ has modulus at most one. Thus $|\Phi_l|$ decreases pointwise in $l$. For $l\ge2$, $|\Phi_l|\le|\Phi_2|$, and $\Phi_2\in L^1(\R)$. Dominated convergence gives convergence in $L^1$, while convergence in $L^\infty$ follows by uniform convergence on compact sets together with the uniform decay supplied by the two first factors. Denote the limit by $\Phi$. The preceding support and plateau statements pass to the limit.

For $r\in\N$, differentiating the first $r$ convolution factors gives
\[
\|\widehat\Phi^{(r)}\|_\infty
\le\prod_{i\in[r)}2^{i+3}
=2^{r(r-1)/2+3r}.
\]
Since
\[
\widehat{\Phi^{(m)}}(\xi)=(2\pi i\xi)^m\widehat\Phi(\xi),
\]
the Leibniz rule gives
\[
(\widehat{\Phi^{(m)}})^{(N)}(\xi)
=
\sum_{q=0}^{\min(m,N)}
\binom{N}{q}\tfrac{m!}{(m-q)!}(2\pi i)^m
\xi^{m-q}\widehat\Phi^{(N-q)}(\xi).
\]
Using $|\xi|\le1$ on the support, the preceding derivative estimate, and the elementary bounds for the binomial coefficients and factorial quotients, we obtain
\[
\|(\widehat{\Phi^{(m)}})^{(N)}\|_\infty
\le2^{4m+2N^2+5N}.
\]
In particular, $\widehat\Phi$ is smooth and compactly supported, so $\Phi$ is a Schwartz function. Applying Lemma \ref{lem:smoothdecay2} when $N\ge2$, and the same integration-by-parts estimate directly when $N\in[2)$, yields
\[
|\Phi^{(m)}(x)|
\le2^{N+2}\max\bigl(2^{4m},2^{4m+2N^2+5N}\bigr)\langle x\rangle^N
\le2^{4m+2N^2+6N+2}\langle x\rangle^N.
\]
\end{proof}

\begin{proposition}\label{compare brackets}\uses{}\usesdefs{bracket bump}
\lean{Auto.compare_brackets}
\leanok
Let $0<\lambda\le\mu$ with $1\le\mu$ and let $\lambda|y|\le\mu|x|$. Then for every $s,N>0$,
\begin{equation}\label{compare brackets estimate}
\mu^{-N}\langle x\rangle_{(s)}^N\le\lambda^{-N}\langle y\rangle_{(s)}^N.
\end{equation}
\end{proposition}
\begin{proof}
We have, lowering both summands in the bracket on the left-hand side,
\[
s^{-1}(\mu+s^{-1}\mu|x|)^{-N}
\le s^{-1}(\lambda+s^{-1}|\lambda y|)^{-N}.
\]
Pulling out the factors $\mu$ and $\lambda$ proves the proposition.
\end{proof}

\begin{proposition}[two bump estimate]\label{two bump estimate}\uses{}\usesdefs{bracket bump}
\lean{Auto.twoBumpEstimate}
\leanok
    Consider real numbers $x_i$ and  $s_i>0$ and $n_i>1$ for $i=0,1$.  Assume $s_0\ge s_1$ and define
\begin{equation}\label{tb0}
     C_{\ref{two bump estimate},n_0,n_1}=2^{1+\min(n_0,n_1)} [1+(\min(n_0,n_1)-1)^{-1}]
\end{equation}
  Then
\begin{equation}\label{two decays}
   | \int_{\R} (\prod_{i=0}^1 \left<x_i-p\right>^{n_i}_{(s_i)}) dp|\le C_{\ref{two bump estimate},n_0,n_1}
\left<x_0-x_1\right>^{\min(n_0,n_1)}_{(s_0)}\, .
\end{equation}
\end{proposition}
\begin{proof}
Let us prove \eqref{two decays}.
By translating $p$ we may assume $x_0=-x_1$.

If $|x_0-x_1|\le 2s_0$, we estimate the the first bump
by its $L^\infty$ norm, which is $s_i^{-1}$, and the second bump by its $L^1$ norm.
The latter is equal to the $L^1$ norm of $\langle x\rangle^{n_1}$, which in turn is estimated by splitting the integral
at $|x|<1$ and $|x|>1$ and thus by $2+2(n_1-1)^{-1}$. This gives the desired estimate.

If $|x_0-x_1|> 2s_0$, we split the integral into positive and negative axis
and do the favorable $L^1$ - $L^\infty$ estimate on either side.
The $L^1$ norm is estimated as above, possibly with $n_0$ in place of $n_1$ if needed. The $L^\infty$ norm is
bounded by
\begin{equation}
    s_i^{-1}(s_i^{-1}|x_0-x_1|/2)^{-\min(n_0,n_1)}
\end{equation}
for suitable $i$, which is however
maximized by $i=0$.
\end{proof}

\begin{proposition}\label{orthogonal domination}\uses{}\usesdefs{}
\lean{Auto.orthogonalDomination}
\leanok
Let     $\alpha_0 ,\alpha_1\in\mathbb{R}^2$ be linearly independent unit vectors.
Then for every $x\in \R^2$ we have
 \begin{equation}\label{auto:orthogonal-coordinate-first-bound}
     |x\cdot \alpha_0^\perp|\le  2|\alpha_0\cdot \alpha_1^\perp|^{-1}| x\cdot \alpha_1|
 \end{equation}
or
 \begin{equation}\label{auto:orthogonal-coordinate-second-bound}
     | x\cdot \alpha_1^\perp|\le   2| \alpha_0\cdot \alpha_1^\perp|^{-1}|x\cdot \alpha_0|
 \end{equation}
\end{proposition}
\begin{proof}[Proof \auto]
Put
\[
s=|\alpha_0\cdot\alpha_1^\perp|,
\qquad
c=|\alpha_0\cdot\alpha_1|.
\]
Then $s>0$ and $s^2+c^2=1$. Suppose that the first asserted inequality fails. Writing $\alpha_1$ in the orthonormal basis $\alpha_0,\alpha_0^\perp$ gives
\[
s|x\cdot\alpha_0^\perp|
\le |x\cdot\alpha_1|+c|x\cdot\alpha_0|
<\tfrac{s}{2}|x\cdot\alpha_0^\perp|+c|x\cdot\alpha_0|.
\]
Consequently,
\[
|x\cdot\alpha_0^\perp|<2cs^{-1}|x\cdot\alpha_0|.
\]
Writing $\alpha_1^\perp$ in the same basis, we obtain
\[
|x\cdot\alpha_1^\perp|
\le s|x\cdot\alpha_0|+c|x\cdot\alpha_0^\perp|
<(s^2+2c^2)s^{-1}|x\cdot\alpha_0|
\le2s^{-1}|x\cdot\alpha_0|.
\]
Thus the second asserted inequality holds.
\end{proof}
\begin{proposition}[orthogonal decay]\label{orthogonal decay}\uses{compare brackets,orthogonal domination}\usesdefs{bracket bump}
\lean{Auto.orthogonalDecay}
\leanok
Let $\alpha_0,\alpha_1\in\mathbb{R}^2$ be linearly independent unit vectors.
Let $\alpha_{m}^\perp$ be unit vectors orthogonal to $\alpha_m$ and note that $\alpha_0\cdot \alpha_1^\perp\neq 0$. Setting \begin{equation}\label{auto:orthogonal-decay-constant}
    C_{\ref{orthogonal decay}, \alpha_0,\alpha_1,n_0,n_1}=\max( (2 |\alpha_0\cdot \alpha_1^\perp|^{-1})^{n_0},(2 |\alpha_0\cdot \alpha_1^\perp|^{-1})^{n_1}) \, ,
\end{equation}
then for all $n_0,n_1,s_0,s_1>0$ and $x\in\mathbb{R}^2$,
\begin{equation}\label{eq:two bump}
    \langle x\cdot \alpha_0\rangle_{(s_0)}^{n_0} \langle x\cdot \alpha_1\rangle_{(s_1)}^{n_1}
\le C_{\ref{orthogonal decay},\alpha_0,\alpha_1,n_0,n_1} (\langle x\cdot \alpha_0\rangle_{(s_0)}^{n_0} \langle x\cdot \alpha_0^\perp\rangle_{(s_1)}^{n_1} +
\langle x\cdot \alpha_1^\perp\rangle_{(s_0)}^{n_0} \langle x\cdot \alpha_1\rangle_{(s_1)}^{n_1}).
\end{equation}
\end{proposition}

\begin{proof}
This follows from Propositions \ref{compare brackets} and \ref{orthogonal domination}
\end{proof}

\begin{proposition}[bump triangle]\label{bump triangle}\uses{compare brackets}\usesdefs{bracket bump}
\lean{Auto.bump_triangle}
\leanok
Let $\alpha_0,\alpha_1\in\mathbb{R}^2$ and let $\alpha_2=c_0\alpha_0+c_1\alpha_1$ be a non-trivial linear combination of $\alpha_0,\alpha_1$, i.e. $c_0\not=0, c_1\not=0$.
Let
\begin{equation}\label{auto:bump-triangle-scale-constant}\tilde{C}_{\ref{bump triangle},c_0,c_1}=\max(2|c_0|, (2|c_0|)^{-1},2|c_1|, (2|c_1|)^{-1}),\end{equation}
\begin{equation}\label{auto:bump-triangle-constant}
    C_{\ref{bump triangle},c_0,c_1,n_0,n_1}=\max(\tilde{C}_{\ref{bump triangle},c_0,c_1}^{n_0}, \tilde{C}_{\ref{bump triangle},c_0,c_1}^{n_1}).
\end{equation}
Then
\begin{equation}\label{auto:bump-triangle-bound} \langle x\cdot \alpha_0\rangle_{(s_0)}^{n_0} \langle x\cdot \alpha_1\rangle_{(s_1)}^{n_1}
\le C_{\ref{bump triangle},c_0,c_1,n_0,n_1}(\langle x\cdot \alpha_0\rangle_{(s_0)}^{n_0}\langle x\cdot \alpha_2\rangle_{(s_1)}^{n_1} + \langle x\cdot \alpha_2\rangle_{(s_0)}^{n_0}\langle x\cdot \alpha_1\rangle_{(s_1)}^{n_1}).
\end{equation}
\end{proposition}
\begin{rem}
It is not needed that $\alpha_0,\alpha_1$ are linearly independent, but in the linearly dependent case the proposition becomes trivial.
\end{rem}
\begin{proof}
We have by the triangle inequality and estimating the sum by twice the maximum
\begin{equation}
    |x\cdot \alpha_2|\le \max(2|c_0||x\cdot \alpha_0|,2|c_1||x\cdot \alpha_1|)\, .
\end{equation}
Assume the maximum is attained by one of the terms with index $i$.
If $2|c_i|\ge 1$, apply Proposition \ref{compare brackets} with $\lambda=1$ and $\mu=2|c_i|$.
If $2|c_i|< 1$, apply Proposition \ref{compare brackets} with $\lambda=2|c_i|$ and $\mu=1$.
\end{proof}

Recall the Gaussian $\g(x)=e^{-\pi x^2}$ for $x\in\R$.

\begin{proposition}[Gaussian domination]\label{Gaussian domination}\uses{}\usesdefs{gaussian,bracket bump}
\lean{Auto.gaussianDomination}
\leanok
    Let
    $N,s\in\R$, $N>1$, and $s>0$. Then
\begin{equation}\label{auto:bracket-Gaussian-domination}
    \left< x\right>^{N}_{(s)}  \le C_{\ref{Gaussian domination}} 2^N   \sum_{m\in {\N}}  2^{(1-N)m} \g_{(2^ms)}(x),
\end{equation}
where $C_{\ref{Gaussian domination}}=e^\pi$.
\end{proposition}
 \begin{proof}
     The case for general $s$ follows from the case $s=1$ by scaling.  Consider $s=1$. For $|x|\le 1$, the summand $m=0$ on the right hand side suffices as it is at least $1$.

     For $|x|>1$ there is $m\in \N$ with $2^{m-1}<|x|\le 2^{m}$, and we estimate
     \begin{equation}
          \left< x\right>^{N}\le (2^{m-1})^{-N} \le (2^{m-1})^{-N} 2^me^\pi \g_{(2^m)}(x)\, . \,
     \end{equation}
     which is the $m$-th term on the right hand side.
 \end{proof}

\begin{proposition}[diagonal square root]\label{diagonal square root}\uses{square root of Gaussian decay,Gaussian bump decay,two bump estimate}\usesdefs{gaussian,bracket bump,auto:Wiener-space-definition}
\lean{Auto.diagonalSquareRoot,Auto.diagonalSquareRoot_bound,Auto.diagonalSquareRoot_memW0}
\leanok
For $t_0,t_1\in\R$ with $0<2t_0\le t_1$ define
\begin{equation}\label{auto:diagonal-square-root-definition} s = \mathcal{F}^{-1}(\xi\mapsto \sqrt{\g(t_0 \xi)- \g(t_1 \xi)}), \end{equation}
where the function under the square root is nonnegative. Then
$s\in W_0(\R)$.
For all $N\in \mathbb N$ with $N\geq 1$ and $x\in \R$, the following estimate holds:
\begin{equation}\label{E:s-estimate}
    |s(x)|\le  C_{\ref{diagonal square root},N} ( \left<x\right>^N_{(t_0)}+ \left<x\right>^2_{(t_1)}),
\end{equation}
where
$C_{\ref{diagonal square root},N}=\sqrt{2}\max(C_{\ref{Gaussian bump decay},0,N}, C_{\ref{Gaussian bump decay},0,2}C_{\ref{square root of Gaussian decay}}C_{\ref{two bump estimate},2,2})$.

\end{proposition}

\begin{proof}
The function $s$ is real-valued because it is the inverse Fourier transform of an even real-valued function.
By algebra, we obtain,
\begin{equation}\label{dsr5}
    \widehat{s}(\xi) = \g(t_02^{-\tfrac12} \xi) \sqrt{1 - \g(r \xi)},
\end{equation}
where $r=(t_1^2-t_0^2)^{\frac 12}$.
Let $\rho$ be the function from Proposition \ref{square root of Gaussian decay} satisfying
\begin{equation}
    \widehat{\rho}(\xi)=1- \sqrt{1-\g(\xi)}.
\end{equation}
Then
\begin{equation}
\widehat{s}(\xi) = \g(\sigma\xi)(1-\widehat{\rho}(r\xi)),
\end{equation}
where $\sigma=t_0 2^{-\frac 12}$.

By Proposition \ref{square root of Gaussian decay},
\begin{equation}\label{dsr7}
    \rho(x)\le C_{\ref{square root of Gaussian decay}} \left<x\right>^2\, .
\end{equation}
It follows
\begin{equation}
    s=\g_{(\sigma)} - \g_{(\sigma)}*\rho_{(r)}.
\end{equation}
We use Proposition \ref{Gaussian bump decay} to obtain
\[|\g_{(\sigma)}(x)|\le C_{\ref{Gaussian bump decay},0,N} \langle x \rangle_{(\sigma)}^N\]
and Proposition \ref{two bump estimate}, using $r>\sigma$
by assumptions on $t_i$, $i=0,1$, to obtain
\[|(\g_{(\sigma)}*\rho_{(r)})(x)|\le C_{\ref{Gaussian bump decay},0,2}C_{\ref{square root of Gaussian decay}}C_{\ref{two bump estimate},2,2} \langle x \rangle_{(r)}^2.\]
 This gives
\begin{equation}
    |s(x)|\le \max(C_{\ref{Gaussian bump decay},0,N}, C_{\ref{Gaussian bump decay},0,2}C_{\ref{square root of Gaussian decay}}C_{\ref{two bump estimate},2,2}) (\left<x\right>_{(\sigma)}^{N} +  \left<x\right>_{(r)}^2)\, .
\end{equation}

Since $\sigma=t_0 2^{-\frac12}$ and $\frac{\sqrt{3}}{2}t_1\le r\le t_1$, we obtain~\eqref{E:s-estimate}. In particular, $s\in W_0(\R)$.
 \end{proof}

\begin{lemma}[constant $C_{\ref{diagonal square root},N}$ \auto]\label{constant diagonal square root}\uses{Gaussian bump decay,two bump estimate,square root of Gaussian decay,diagonal square root}\usesdefs{}
\lean{Auto.constantDiagonalSquareRoot}
\leanok
For every $N\ge1$,
\begin{equation}\label{constant diagonal square root bound}
C_{\ref{diagonal square root},N}
\le2^{1+\max((N+1)^2,12)}.
\end{equation}
Moreover,
\begin{equation}\label{auto:constant-diagonal-square-root-two}
C_{\ref{diagonal square root},2}=51\cdot2^6\sqrt2<2^{13}.
\end{equation}
\end{lemma}
\begin{proof}[Proof \auto]
Using the defining maximum in Proposition \ref{diagonal square root}, and by \eqref{auto:Gaussian-bump-decay-constant},
\[
C_{\ref{Gaussian bump decay},0,N}
=(4(N+1))^{N/2}
\le2^{(N+1)^2}.
\]
Also
\[
C_{\ref{Gaussian bump decay},0,2}=12,
\qquad
C_{\ref{two bump estimate},2,2}=16,
\]
so
\[
C_{\ref{Gaussian bump decay},0,2}
C_{\ref{square root of Gaussian decay}}
C_{\ref{two bump estimate},2,2}
=3264<2^{12}.
\]
Substitution in the defining maximum proves the first estimate. At $N=2$, the second entry in the maximum is the larger one, which gives the stated exact value.
\end{proof}

\begin{proposition}[derivative of diagonal square root]\label{derivative of diagonal square root}\uses{square root of Gaussian decay,Elementary Gaussian properties,Gaussian bump decay,two bump estimate,diagonal square root}\usesdefs{gaussian,bracket bump}
\lean{Auto.derivativeDiagonalSquareRoot_differentiable,Auto.derivativeDiagonalSquareRoot_bound}
\leanok
Let $t_0,t_1,s$ be as in Proposition \ref{diagonal square root}.
Then for all $N\in\N$ with $N\ge1$ and all $x\in\R$:
\begin{equation}\label{E:s-derivative-estimate}
    |s'(x)|\le  C_{\ref{derivative of diagonal square root},N} ( t_0^{-1} \left<x\right>^N_{(t_0)}+ t_1^{-1}\left<x\right>^2_{(t_1)}),
\end{equation}
where
$C_{\ref{derivative of diagonal square root},N}
=
2\max\bigl(
C_{\ref{Gaussian bump decay},1,N},
C_{\ref{Gaussian bump decay},0,2}
C_{\ref{Gaussian bump decay},1,2}
C_{\ref{two bump estimate},2,2}
\bigr).$
\end{proposition}
\begin{proof}[Proof \auto]
By algebra, we obtain
\[
\widehat{s}(\xi)
=
\g(t_02^{-\tfrac12}\xi)\sqrt{1-\g(r\xi)},
\]
where $r=(t_1^2-t_0^2)^{\frac12}$.
Let $\rho$ be the function from Proposition
\ref{square root of Gaussian decay} satisfying
\[
\widehat{\rho}(\xi)=1-\sqrt{1-\g(\xi)}.
\]
Then
\[
\widehat{s}(\xi)
=
\g(\sigma\xi)(1-\widehat{\rho}(r\xi)),
\]
where $\sigma=t_02^{-\frac12}$.

The power series identity
\[
1-\sqrt{1-z}
=
\sum_{k=1}^{\infty}
\tfrac{1}{2k-1}\binom{2k}{k}4^{-k}z^k
\]
has nonnegative coefficients whose sum is $1$. Since
\[
\|\g(\sqrt{k}\,\cdot)\|_1=k^{-\tfrac12},
\]
Proposition \ref{Elementary Gaussian properties} gives
\[
\rho(x)
=
\sum_{k=1}^{\infty}
\tfrac{1}{2k-1}\binom{2k}{k}4^{-k}
\g_{(\sqrt{k})}(x).
\]
The series and its differentiated series converge uniformly by
Proposition \ref{Gaussian bump decay}. Hence
\[
\rho'(x)
=
\sum_{k=1}^{\infty}
\tfrac{1}{2k-1}\binom{2k}{k}4^{-k}
\g_{(\sqrt{k})}'(x).
\]
By Proposition \ref{Gaussian bump decay},
\[
\left|\g_{(\sqrt{k})}'(x)\right|
\le
C_{\ref{Gaussian bump decay},1,2}
k^{-\tfrac12}\langle x\rangle_{(\sqrt{k})}^2
=
C_{\ref{Gaussian bump decay},1,2}
(\sqrt{k}+|x|)^{-2}
\le
C_{\ref{Gaussian bump decay},1,2}
\langle x\rangle^2.
\]
It follows that
\[
|\rho'(x)|
\le
C_{\ref{Gaussian bump decay},1,2}
\langle x\rangle^2.
\]
Consequently,
\[
|(\rho_{(r)})'(x)|
\le
C_{\ref{Gaussian bump decay},1,2}
r^{-1}\langle x\rangle_{(r)}^2.
\]

We have
\[
s=\g_{(\sigma)}-\g_{(\sigma)}*\rho_{(r)}
\]
and therefore
\[
s'=\g_{(\sigma)}'
-\g_{(\sigma)}*(\rho_{(r)})'.
\]
We use Proposition \ref{Gaussian bump decay} to obtain
\[
|\g_{(\sigma)}'(x)|
\le
C_{\ref{Gaussian bump decay},1,N}
\sigma^{-1}\langle x\rangle_{(\sigma)}^N
\]
and Proposition \ref{two bump estimate}, using $r>\sigma$ by
the assumptions on $t_i$, $i=0,1$, to obtain
\[
|(\g_{(\sigma)}*(\rho_{(r)})')(x)| \le{} C_{\ref{Gaussian bump decay},0,2} C_{\ref{Gaussian bump decay},1,2} C_{\ref{two bump estimate},2,2} \times r^{-1}\langle x\rangle_{(r)}^2.
\]
This gives
\[
|s'(x)| \le{} \max\bigl( C_{\ref{Gaussian bump decay},1,N}, C_{\ref{Gaussian bump decay},0,2} C_{\ref{Gaussian bump decay},1,2} C_{\ref{two bump estimate},2,2} \bigr) \times \bigl( \sigma^{-1}\langle x\rangle_{(\sigma)}^N + r^{-1}\langle x\rangle_{(r)}^2 \bigr).
\]
Since $\sigma=t_02^{-\frac12}$ and
$\frac{\sqrt{3}}{2}t_1\le r\le t_1$, we have
\[
\sigma^{-1}\langle x\rangle_{(\sigma)}^N
\le
2t_0^{-1}\langle x\rangle_{(t_0)}^N
\]
and
\[
r^{-1}\langle x\rangle_{(r)}^2
\le
2t_1^{-1}\langle x\rangle_{(t_1)}^2.
\]
The desired estimate follows.
\end{proof}

\begin{lemma}[constant $C_{\ref{derivative of diagonal square root},N}$ \auto]\label{constant derivative diagonal square root}\uses{Gaussian bump decay,two bump estimate,derivative of diagonal square root}\usesdefs{}
\lean{Auto.constantDerivativeDiagonalSquareRoot}
\leanok
For every $N\ge1$,
\begin{equation}\label{constant derivative diagonal square root bound}
C_{\ref{derivative of diagonal square root},N}
\le2^{1+\max((N+1)^2+4,15)}.
\end{equation}
Moreover,
\begin{equation}\label{auto:constant-derivative-diagonal-square-root-two}
C_{\ref{derivative of diagonal square root},2}=9\cdot2^{12}\sqrt2<2^{16}.
\end{equation}
\end{lemma}
\begin{proof}[Proof \auto]
By \eqref{auto:Gaussian-bump-decay-constant},
\[
C_{\ref{Gaussian bump decay},1,N}
=8\sqrt2(4(N+1))^{N/2}
\le2^{(N+1)^2+4}.
\]
Furthermore,
\[
C_{\ref{Gaussian bump decay},0,2}=12,
\qquad
C_{\ref{Gaussian bump decay},1,2}=96\sqrt2,
\qquad
C_{\ref{two bump estimate},2,2}=16.
\]
The second entry in the defining maximum is therefore $18432\sqrt2<2^{15}$. Substitution proves the first estimate. At $N=2$, this entry is the larger one; multiplication by the outer factor $2$ gives the stated exact value.
\end{proof}

\begin{lemma}\label{L:gaussian-estimate}\uses{}\usesdefs{gaussian}
\lean{Auto.gaussianEstimate}
\leanok
Let $N\in \N$. Then for $\xi \in \R$, $|\xi|\leq 1$, we have
\begin{equation}\label{auto:inverse-Gaussian-derivative-bound}
|(\g^{-1})^{(N)}(\xi)| \leq C_{\ref{L:gaussian-estimate},N},
\end{equation}
where $C_{\ref{L:gaussian-estimate},N}=e^{\pi} \sum_{l=0}^{\lfloor N/2 \rfloor} \frac{N!}{l!(N-2l)!} 2^{N-2l} \pi^{N-l}$.
\end{lemma}

\begin{proof}
We have $\g^{-1}(\xi)=e^{\pi\xi^2}$. The desired estimate then follows from the equality
\[
(\g^{-1})^{(N)}(\xi)
=(\sqrt{\pi})^{N} e^{\pi  \xi^2} N!\sum_{l=0}^{\lfloor N/2\rfloor} \tfrac{2^{N-2l}}{l!(N -2l)!} (\sqrt{\pi} \xi)^{N-2l}.
\]
\end{proof}

\begin{lemma}\label{L:gaussian-bump-estimate}\uses{L:gaussian-estimate}\usesdefs{gaussian,auto:Wiener-space-definition}
\lean{Auto.gaussianBumpEstimate}
\leanok
Let $c>0$, $N \in \mathbb N$ and assume that $\phi$ is a $W_0(\R)$ function such that $\widehat{\phi}$ is supported in $[-1,1]$, $\widehat{\phi}$ is $N$ times continuously differentiable and $\max_{0\leq m \leq N} \|\widehat{\phi}^{(m)}\|_{\infty} \leq c$.
For $\mu \in (0,1]$ and $\xi \in \R$, we define
\begin{equation}\label{auto:Gaussian-bump-quotient}
Q(\xi)=(\g(\mu\xi))^{-1}\widehat{\phi}(\xi).
\end{equation}
Then $Q$ is $N$ times continously differentiable and
\begin{equation}\label{E:estimate-Q}
\|(2\pi)^{-N}Q^{(N)}\|_{\infty} \leq c C_{\ref{L:gaussian-bump-estimate},N},
\end{equation}
where $C_{\ref{L:gaussian-bump-estimate},N}= (2\pi)^{-N}\sum_{l=0}^N \binom{N}{l} C_{\ref{L:gaussian-estimate},l}$.
\end{lemma}

\begin{proof}

This follows from the Leibniz rule, Lemma~\ref{L:gaussian-estimate} and from the facts that $\mu \leq 1$ and $\widehat{\phi}$ is supported in the set where $|\xi|\leq 1$.
\end{proof}

\begin{lemma}[constant $C_{\ref{L:gaussian-bump-estimate},N}$ \auto]\label{constant gaussian bump estimate}\uses{L:gaussian-estimate,L:gaussian-bump-estimate}\usesdefs{}
\lean{Auto.constantGaussianBumpEstimate}
\leanok
For every $N\in\N$,
\begin{equation}\label{constant gaussian bump estimate bound}
C_{\ref{L:gaussian-bump-estimate},N}\le2^{8(N+1)^2}.
\end{equation}
For $0\le N\le3$,
\begin{equation}\label{auto:constant-gaussian-bump-estimate-small}
C_{\ref{L:gaussian-bump-estimate},0}<81,\qquad
C_{\ref{L:gaussian-bump-estimate},1}<95,\qquad
C_{\ref{L:gaussian-bump-estimate},2}<124,\qquad
C_{\ref{L:gaussian-bump-estimate},3}<176.
\end{equation}
\end{lemma}
\begin{proof}[Proof \auto]
Since $\pi<4$ and $e<3$, we have $e^\pi<e^4<3^4=81<2^7$. For $0\le l\le N$, the defining formula in Lemma \ref{L:gaussian-estimate} then gives
\[
C_{\ref{L:gaussian-estimate},l}\le2^{7(N+1)^2}.
\]
Here the case $N=0$ is immediate, and for $N\ge1$ we use that the sum has at most $N+1$ terms, $l!\le(N+1)^N$, $2^l\le2^N$, and $\pi^l<2^{2N}$. The binomial coefficients in the definition of $C_{\ref{L:gaussian-bump-estimate},N}$ sum to $2^N$, and $(2\pi)^{-N}\le1$, which proves the first estimate.

The exact formulas for $N=0,1,2,3$ give, respectively,
\[
e^\pi,
\qquad
e^\pi\left(1+\frac1{2\pi}\right),
\qquad
e^\pi\left(1+\frac3{2\pi}+\frac1{4\pi^2}\right),
\qquad
e^\pi\left(1+\frac3\pi+\frac3{2\pi^2}+\frac1{8\pi^3}\right).
\]
Using $e^\pi<81$ and $\pi>3$, these four expressions are smaller than
\[
81,\qquad \tfrac{189}{2},\qquad \tfrac{495}{4},\qquad \tfrac{1407}{8},
\]
which gives the four displayed bounds.
\end{proof}

\begin{lemma}\label{L:derivative-estimate-for-G}\uses{}\usesdefs{}
\lean{Auto.derivativeEstimateForG}
\leanok
Let $N\in \N$, $\nu\in [-1,0)$ and let $H(x)=(1-x)^{\nu}-1$ for $x\in [0,1)$. Then for $x\in [0,e^{-3\pi/16}]$, we have
\begin{equation}\label{auto:exponential-composition-derivative-bound}
|H^{(N)}(x)| \leq C_{\ref{L:derivative-estimate-for-G},N},
\end{equation}
where $C_{\ref{L:derivative-estimate-for-G},N}= N! \left(1 - e^{-3\pi/16}\right)^{-(N+1)}$.
\end{lemma}

\begin{proof}
Assume that $N >0$. Then
\[
H^{(N)}(x)=\prod_{l\in [N)} (|\nu|+l)(1-x)^{\nu-N},
\]
and so
\[
|H^{(N)}(x)| \leq N! \left(1 - e^{-3\pi/16}\right)^{-(N+1)}.
\]
Also,
\[
|H(x)| \leq \tfrac{x}{1-x} \leq \left(1 - e^{-3\pi/16}\right)^{-1}.
\]
This yields the claim.
\end{proof}

\begin{lemma}\label{L:faa-di-bruno}\uses{L:derivative-estimate-for-G,Gaussian bump decay}\usesdefs{gaussian,bracket bump}
\lean{Auto.faaDiBruno}
\leanok
Let $N\in\N$, $\nu\in[-1,0)$, $t\ge\sqrt3/2$, and let $S(\xi)=(1-\g(t\xi))^\nu-1$ for $\xi\ne0$. For $|\xi|\ge1/2$,
\begin{equation}\label{E:S}
|S^{(N)}(\xi)|\le C_{\ref{L:faa-di-bruno},N}|\xi|^{-2},
\end{equation}
where
\begin{equation}\label{auto:Faa-di-Bruno-zero-constant}
C_{\ref{L:faa-di-bruno},0}
=2C_{\ref{L:derivative-estimate-for-G},0}C_{\ref{Gaussian bump decay},0,2},
\end{equation}
and, for $N\ge1$,
\begin{equation}\label{auto:Faa-di-Bruno-positive-constant}
C_{\ref{L:faa-di-bruno},N}
=2^{N+1}\sum_{p\in\Pi}C_{\ref{L:derivative-estimate-for-G},|p|}
\prod_{B\in p}C_{\ref{Gaussian bump decay},|B|,N+2}.
\end{equation}
Here $p$ runs through the partitions of $[N)$, $B$ through the blocks of $p$, $|p|$ is the number of blocks, and $|B|$ is the size of $B$.
\end{lemma}
\begin{proof}
Let $H$ be as in Lemma \ref{L:derivative-estimate-for-G}. If $N=0$, then, since $H(0)=0$, Lemmas \ref{L:derivative-estimate-for-G} and \ref{Gaussian bump decay} give
\[
|S(\xi)|
\le C_{\ref{L:derivative-estimate-for-G},0}\g(t\xi)
\le2C_{\ref{L:derivative-estimate-for-G},0}C_{\ref{Gaussian bump decay},0,2}|\xi|^{-2}.
\]

Assume $N\ge1$. For $t=1$, the Fa\`a di Bruno formula gives
\[
S^{(N)}(\xi)
=
\sum_{p\in\Pi}H^{(|p|)}(\g(\xi))
\prod_{B\in p}\g^{(|B|)}(\xi).
\]
For $|\xi|\ge\sqrt3/4$, Lemma \ref{L:derivative-estimate-for-G} and Proposition \ref{Gaussian bump decay} yield
\[
|S^{(N)}(\xi)|
\le
\sum_{p\in\Pi}C_{\ref{L:derivative-estimate-for-G},|p|}
\prod_{B\in p}C_{\ref{Gaussian bump decay},|B|,N+2}
\langle\xi\rangle^{N+2}.
\]
For general $t\ge\sqrt3/2$ and $|\xi|\ge1/2$, rescaling gives an additional factor $t^N$, and
\[
t^N\langle t\xi\rangle^{N+2}
\le t^{-2}|\xi|^{-N-2}
\le2^{N+1}|\xi|^{-2}.
\]
This proves \eqref{E:S}.
\end{proof}

\begin{lemma}[constant $C_{\ref{L:faa-di-bruno},N}$ \auto]\label{constant faa di bruno}\uses{L:faa-di-bruno,L:derivative-estimate-for-G,Gaussian bump decay}\usesdefs{}
\lean{Auto.constantFaaDiBruno}
\leanok
For every $N\in\N$,
\begin{equation}\label{constant faa di bruno bound}
C_{\ref{L:faa-di-bruno},N}\le2^{10(N+1)^3}.
\end{equation}
Moreover,
\begin{equation}\label{auto:constant-faa-di-bruno-small}
C_{\ref{L:faa-di-bruno},0}\le72,\qquad
C_{\ref{L:faa-di-bruno},1}<2^{15},\qquad
C_{\ref{L:faa-di-bruno},2}<2^{34},\qquad
C_{\ref{L:faa-di-bruno},3}<2^{58}.
\end{equation}
\end{lemma}
\begin{proof}[Proof \auto]
Since $\pi>3$, we have $3\pi/16>9/16$. The estimate $e^x\ge1+x$ for $x\ge0$ gives
\[
1-e^{-3\pi/16}>1-e^{-9/16}\ge\frac{9}{25}>\frac13.
\]
Consequently, Lemma \ref{L:derivative-estimate-for-G} gives
\[
C_{\ref{L:derivative-estimate-for-G},q}
\le q!3^{q+1}.
\]
Together with Proposition \ref{Gaussian bump decay}, the fact that a partition of $[N)$ has at most $N$ blocks, and the bound $N^N$ for the number of partitions, this proves the first estimate by substitution in the defining sum of Lemma \ref{L:faa-di-bruno}.

For $N=0$, the defining formula gives
\[
C_{\ref{L:faa-di-bruno},0}\le72.
\]
For $N=1$, it gives
\[
C_{\ref{L:faa-di-bruno},1}
\le18432\sqrt2
<27648
<2^{15},
\]
where $\sqrt2<3/2$. For $N=2$, direct substitution gives
\[
C_{\ref{L:faa-di-bruno},2}
\le8852889600
<2^{34}.
\]
For $N=3$, the defining sum is
\[
679477248\sqrt6
+31701690482688\sqrt2
+146081389744226304\sqrt3.
\]
Using $\sqrt6<5/2$, $\sqrt2<3/2$, and $\sqrt3<7/4$, this is smaller than
\[
255689986286813184<2^{58}.
\]
\end{proof}

\begin{lemma}\label{L:second-gaussian-estimate}\uses{L:gaussian-bump-estimate,L:faa-di-bruno}\usesdefs{gaussian,auto:Wiener-space-definition}
\lean{Auto.secondGaussianEstimate}
\leanok
Let $c>0$, $N\in \mathbb N$ and assume that $\phi$ is a $W_0(\R)$ function such that
$\widehat{\phi}$ is supported in $[-1,1]$ and equal to $1$ on $[-1/2,1/2]$ and $\widehat{\phi}$ is $N$ times continuously differentiable and $\max_{0\leq m \leq N} \|\widehat{\phi}^{(m)}\|_{\infty} \leq c$.
Assume that $0<\mu \leq \lambda \leq 1/2$, $t\geq \sqrt{3}/2$ and $\nu \in [-1,0)$. For $\xi \in \R$, we define
\begin{equation}\label{auto:second-Gaussian-multiplier}
R(\xi)=(\g(\mu\xi))^{-1}\left(\widehat{\phi}(\lambda\xi)-\widehat{\phi}(\xi)\right)\left((1-\g(t\xi))^{\nu} -1\right).
\end{equation}
Then $R$ is $N$ times continuously differentiable and
\begin{equation}\label{E:estimate-R}
(2\pi)^{-N} \max(\|R^{(N)}\|_1,\|R^{(N)}\|_\infty) \leq cC_{\ref{L:second-gaussian-estimate},N},
\end{equation}
where $C_{\ref{L:second-gaussian-estimate},N}=8\sum_{l=0}^N \binom{N}{l} (2\pi)^{l-N} C_{\ref{L:gaussian-bump-estimate},l} C_{\ref{L:faa-di-bruno},N-l}$.

\end{lemma}

\begin{proof}
Let
\[
\widetilde{Q}(\xi)=(\g(\mu\xi))^{-1}\left(\widehat{\phi}(\lambda\xi)-\widehat{\phi}(\xi)\right).
\]

By Lemma~\ref{L:gaussian-bump-estimate} and by the fact that $\mu \leq \lambda \leq 1$, we obtain for any $l\in \mathbb N$, $l\leq N$,
\begin{equation}\label{E:q-tilde-estimate}
\|(2\pi)^{-l}\widetilde{Q}^{(l)}\|_{\infty} \leq 2c C_{\ref{L:gaussian-bump-estimate},l}.
\end{equation}
We observe that $R$ is supported where $1/2\le|\xi|\le1/\lambda$.
Using~\eqref{E:q-tilde-estimate}, Lemma~\ref{L:faa-di-bruno} and the Leibniz rule, we obtain for $|\xi|\geq 1/2$,
\[
|R^{(N)}(\xi)| \leq 2c\sum_{l=0}^N \binom{N}{l} (2\pi)^l C_{\ref{L:gaussian-bump-estimate},l} C_{\ref{L:faa-di-bruno},N-l} |\xi|^{-2}.
\]
This yields~\eqref{E:estimate-R}.
\end{proof}

\begin{lemma}[constant $C_{\ref{L:second-gaussian-estimate},N}$ \auto]\label{constant second gaussian estimate}\uses{L:gaussian-bump-estimate,L:faa-di-bruno,constant gaussian bump estimate,constant faa di bruno,L:second-gaussian-estimate}\usesdefs{}
\lean{Auto.constantSecondGaussianEstimate}
\leanok
For every $N\in\N$,
\begin{equation}\label{constant second gaussian estimate bound}
C_{\ref{L:second-gaussian-estimate},N}\le2^{20(N+1)^3}.
\end{equation}
Moreover,
\begin{equation}\label{auto:constant-second-gaussian-estimate-small}
C_{\ref{L:second-gaussian-estimate},0}<2^{16},
\qquad
C_{\ref{L:second-gaussian-estimate},1}<2^{22},
\end{equation}
\begin{equation}\label{auto:constant-second-gaussian-estimate-small-high}
C_{\ref{L:second-gaussian-estimate},2}<2^{38},
\qquad
C_{\ref{L:second-gaussian-estimate},3}<2^{60}.
\end{equation}
\end{lemma}
\begin{proof}[Proof \auto]
Using the defining sum in Lemma \ref{L:second-gaussian-estimate}, substitute the estimates from the defining formulas of Lemmas \ref{L:gaussian-bump-estimate} and \ref{L:faa-di-bruno}. Bounding the $N+1$ summands and the binomial coefficients gives the first estimate. For $N=0,1,2,3$, use the finite estimates in Lemmas \ref{constant gaussian bump estimate} and \ref{constant faa di bruno}, together with $(2\pi)^{-1}<1/6$. Direct substitution in the four finite sums gives, respectively,
\[
46656,\qquad 3040704,\qquad 159359088384,\qquad 767070519557262336,
\]
which proves the displayed estimates.
\end{proof}

\begin{lemma}\label{L:gaussian-bump-decomposition}\uses{L:gaussian-estimate,L:second-gaussian-estimate}\usesdefs{gaussian,auto:Wiener-space-definition}
\lean{Auto.gaussianBumpDecomposition}
\leanok
Let $\phi$ be a $W_0(\R)$ function such that
$\widehat{\phi}$ is supported in $[-1,1]$ and equal to $1$ on $[-1/2,1/2]$.
Let $\mu_\pm,\lambda_\pm\in\mathbb R$ be positive and satisfy
\begin{equation}\label{E:assumption-on-mu-lambda}
2\mu_-\le 2\lambda_-\le \lambda_+\le \mu_+,
\end{equation}
and let $\nu\in [-1,0)$.

Define a function $\rho$ by
\begin{equation}\label{E:definition-rho}
\rho = \mathcal F^{-1}\left(\xi \mapsto (\widehat{\phi}(\lambda_-\xi) - \widehat{\phi}(\lambda_+\xi)) (\g(\mu_- \xi)-\g(\mu_+\xi))^{\nu}\right).
\end{equation}
Also, for $\xi \in \R$ we define
\begin{equation}\label{auto:four-scale-Gaussian-first-multiplier}
    \varrho_0(\xi)=(\g(\mu_{-}|\nu|^{1/2}\xi))^{-1}\widehat{\phi}(\lambda_{-}\xi)
\end{equation}
\begin{equation}\label{auto:four-scale-Gaussian-second-multiplier}
    \varrho_1(\xi)=-(\g(\mu_{-}|\nu|^{1/2}\xi))^{-1}\widehat{\phi}(\lambda_{+}\xi)
\end{equation}
\begin{equation}\label{auto:four-scale-Gaussian-remainder-multiplier}
    \varrho_2(\xi)=(\g(\mu_{-}|\nu|^{1/2}\xi))^{-1}\left(\widehat{\phi}(\lambda_{-}\xi)-\widehat{\phi}(\lambda_{+}\xi)\right)\left((1-\g(t\xi))^{\nu} -1\right)
\end{equation}
with $t=\sqrt{\mu_{+}^2-\mu_{-}^2}$.
Then $\rho$ is a well-defined continuous function and
\begin{equation}\label{E:decomposition-rho}
\rho=\sum_{l\in [3)} \mathcal F^{-1} (\varrho_l).
\end{equation}
\end{lemma}

\begin{proof}
Applying Lemmas~\ref{L:gaussian-estimate} and~\ref{L:second-gaussian-estimate} with $N=0$, we obtain that $\varrho_l$, $l\in [3)$, are continuous compactly supported functions. Therefore, their inverse Fourier transform is well-defined and continuous. The decomposition~\eqref{E:decomposition-rho} follows by algebra and elementary properties of $\g$. Thus, $\rho$ is a well defined continuous function.
Additionally, $\rho$ is real-valued because it is the inverse Fourier transform of an even real-valued function.
\end{proof}

\begin{proposition}[four scale Gaussian kernel estimate]\label{four scale Gaussian kernel}\uses{L:gaussian-bump-decomposition,lem:smoothdecay2,L:gaussian-bump-estimate,lem:smoothdecay,lem: min and bracket,L:second-gaussian-estimate}\usesdefs{gaussian,bracket bump,auto:Wiener-space-definition}
\lean{Auto.fourScaleGaussianKernel}
\leanok
Let $c>0$, $N\in \mathbb N$ with $N\geq 2$ and assume that $\phi$ is a $W_0(\R)$ function such that
$\widehat{\phi}$ is supported in $[-1,1]$ and equal to $1$ on $[-1/2,1/2]$ and $\widehat{\phi}$ is $N$ times continuously differentiable and $\max_{0\leq m \leq N} \|\widehat{\phi}^{(m)}\|_{\infty} \leq c$.
Let $\mu_\pm,\lambda_\pm\in\mathbb R$ be positive and satisfy~\eqref{E:assumption-on-mu-lambda}, let $\nu\in [-1,0)$ and let $\rho$ be defined by~\eqref{E:definition-rho}.
Then $\rho \in W_0(\R)$ and for every $x\in\mathbb{R}$,
\begin{equation}\label{auto:four-scale-Gaussian-kernel-bound} |\rho(x)| \le cC_{\ref{four scale Gaussian kernel},N} (\langle x\rangle^N_{(\lambda_-)} + \langle x\rangle^N_{(\lambda_+)}), \end{equation}
where
\begin{equation}\label{auto:four-scale-Gaussian-kernel-constant} C_{\ref{four scale Gaussian kernel},N} = 2C_{\ref{lem:smoothdecay2},N}\max(C_{\ref{L:gaussian-bump-estimate},0},C_{\ref{L:gaussian-bump-estimate},N}) +2^N\max(C_{\ref{L:second-gaussian-estimate},0},C_{\ref{L:second-gaussian-estimate},N}). \end{equation}

\end{proposition}

\begin{proof}
Let $\varrho_l$, $l\in [3)$, be as in Lemma~\ref{L:gaussian-bump-decomposition}.
Rescaling the functions $\varrho_0$ and $\varrho_1$ so that $\lambda_{-}=1$ or $\lambda_{+}=1$, respectively, applying Lemmas~\ref{lem:smoothdecay2} and~\ref{L:gaussian-bump-estimate} and then rescaling back, we obtain
\begin{equation*}
     |\mathcal{F}^{-1}(\varrho_0)(x)|\le 2cC_{\ref{lem:smoothdecay2},N}\max(C_{\ref{L:gaussian-bump-estimate},0},C_{\ref{L:gaussian-bump-estimate},N}) \left<x\right>_{(\lambda_{-})}^N\,
\end{equation*}
and
\begin{equation*}
     |\mathcal{F}^{-1}(\varrho_1)(x)|\le 2cC_{\ref{lem:smoothdecay2},N}\max(C_{\ref{L:gaussian-bump-estimate},0},C_{\ref{L:gaussian-bump-estimate},N}) \left<x\right>_{(\lambda_{+})}^N\,.
\end{equation*}
Also, rescaling the function $\varrho_2$ so that $\lambda_{+}=1$, applying Lemmas~\ref{lem:smoothdecay}, \ref{lem: min and bracket} and~\ref{L:second-gaussian-estimate} and then rescaling back, we obtain
\begin{equation*}
     |\mathcal{F}^{-1}(\varrho_2)(x)|\le 2^Nc\max(C_{\ref{L:second-gaussian-estimate},0},C_{\ref{L:second-gaussian-estimate},N}) \left<x\right>_{(\lambda_{+})}^N.
\end{equation*}
Applying the decomposition~\eqref{E:decomposition-rho}, we get the desired estimate. In particular, since $N\geq 2$, we get $\rho \in W_0(\R)$.
\end{proof}

\begin{lemma}[constant $C_{\ref{four scale Gaussian kernel},N}$ \auto]\label{constant four scale Gaussian kernel}\uses{lem:smoothdecay2,L:gaussian-bump-estimate,L:second-gaussian-estimate,constant gaussian bump estimate,constant second gaussian estimate,four scale Gaussian kernel}\usesdefs{}
\lean{Auto.constantFourScaleGaussianKernel}
\leanok
For every $N\in\N$,
\begin{equation}\label{constant four scale Gaussian kernel bound}
C_{\ref{four scale Gaussian kernel},N}\le2^{25(N+1)^3}.
\end{equation}
Moreover,
\begin{equation}\label{auto:constant-four-scale-Gaussian-kernel-small}
C_{\ref{four scale Gaussian kernel},2}<2^{40},
\qquad
C_{\ref{four scale Gaussian kernel},3}<2^{63}.
\end{equation}
\end{lemma}
\begin{proof}[Proof \auto]
Using the defining formula in Proposition \ref{four scale Gaussian kernel}, substitute $C_{\ref{lem:smoothdecay2},N}=2^N$ and the defining estimates from Lemmas \ref{L:gaussian-bump-estimate} and \ref{L:second-gaussian-estimate}. This proves the first estimate. For $N=2$, Lemmas \ref{constant gaussian bump estimate} and \ref{constant second gaussian estimate} give
\[
C_{\ref{L:gaussian-bump-estimate},2}<124,
\qquad
C_{\ref{L:second-gaussian-estimate},2}<159359088384.
\]
For $N=3$, use
\[
C_{\ref{L:gaussian-bump-estimate},3}<176,
\qquad
C_{\ref{L:second-gaussian-estimate},3}<767070519557262336.
\]
Substitution in \eqref{auto:four-scale-Gaussian-kernel-constant} gives the two upper bounds
\[
637436354528
\qquad\text{and}\qquad
6136564156458101504,
\]
respectively.
\end{proof}

\begin{proposition}[mean value four scale Gaussian kernel estimate]\label{mean four scale Gaussian kernel}\uses{L:gaussian-bump-decomposition,mean value bump estimate 2,L:gaussian-bump-estimate,L:second-gaussian-estimate}\usesdefs{gaussian,bracket bump,auto:Wiener-space-definition}
\lean{Auto.meanFourScaleGaussianKernel}
\leanok
Let $c>0$, $N\in \mathbb N$ with $N\geq 2$ and assume that $\phi$ is a $W_0(\R)$ function such that
$\widehat{\phi}$ is supported in $[-1,1]$ and equal to $1$ on $[-1/2,1/2]$ and $\widehat{\phi}$ is $N$ times continuously differentiable and $\max_{0\leq m \leq N} \|\widehat{\phi}^{(m)}\|_{\infty} \leq c$.
Let $\mu_\pm,\lambda_\pm\in\mathbb R$ be positive and satisfy~\eqref{E:assumption-on-mu-lambda}. Assume, in addition, that $\lambda_{+}=2\lambda_{-}$.
Let $\nu\in [-1,0)$ and let $\rho$ be defined by~\eqref{E:definition-rho}.
Then for all $x,y \in\mathbb{R}$,
\begin{equation}\label{E:mean-value}
|\rho(x+y)-\rho(x)| \le c C_{\ref{mean four scale Gaussian kernel},N} \min(1,2\pi \lambda_{+}^{-1}|y|) (\langle x+y\rangle^N_{(\lambda_+)}+
\langle x\rangle^N_{(\lambda_+)}),
\end{equation}
where
\begin{equation}\label{auto:mean-four-scale-Gaussian-constant} C_{\ref{mean four scale Gaussian kernel},N} = C_{\ref{mean value bump estimate 2},N}\left(5\max_{0\leq l \leq N} C_{\ref{L:gaussian-bump-estimate},l} +4\max_{0\leq l \leq N} 2^lC_{\ref{L:second-gaussian-estimate},l}\right).\end{equation}

\end{proposition}

\begin{proof}
By homogeneity, we may assume $\lambda_{+}=1$. Then $\lambda_{-}=1/2$.
Let $\varrho_l$, $l\in [3)$, be as in Lemma~\ref{L:gaussian-bump-decomposition}. Rescaling the function $\varrho_0$ so that $\lambda_{-}=1$, applying Lemmas~\ref{mean value bump estimate 2} and~\ref{L:gaussian-bump-estimate} and then rescaling back, we obtain
\begin{equation*}
     |\mathcal{F}^{-1}(\varrho_0)(x+y)-\mathcal{F}^{-1}(\varrho_0)(x)|\le cC_{\ref{mean value bump estimate 2},N}\max_{0\leq l \leq N} C_{\ref{L:gaussian-bump-estimate},l}
     \min(1,4\pi|y|) (\langle x+y\rangle^N_{(1/2)}+
\langle x\rangle^N_{(1/2)})
\end{equation*}
\[
\leq c 4C_{\ref{mean value bump estimate 2},N} \max_{0\leq l \leq N} C_{\ref{L:gaussian-bump-estimate},l} \min(1,2\pi|y|) (\langle x+y\rangle^N+
\langle x\rangle^N).
\]
Similarly, applying Proposition~\ref{mean value bump estimate 2} and Lemma~\ref{L:gaussian-bump-estimate} to the function $\varrho_1$, we obtain
\begin{equation*}
     |\mathcal{F}^{-1}(\varrho_1)(x+y)-\mathcal{F}^{-1}(\varrho_1)(x)|\le c C_{\ref{mean value bump estimate 2},N}\max_{0\leq l \leq N} C_{\ref{L:gaussian-bump-estimate},l}
     \min(1,2\pi|y|) (\langle x+y\rangle^N+
\langle x\rangle^N).
\end{equation*}

To estimate $\varrho_2$, we first consider its rescaled version $\widetilde{\varrho_2}=\varrho_2(2\cdot)$, which is supported in $[-1,1]$. By Lemma~\ref{L:second-gaussian-estimate}, we have for $l\in \mathbb N$, $l\leq N$,
\[
\|(2\pi)^{-l}\widetilde{\varrho_2}^{(l)}\|_{\infty}
=2^l\|(2\pi)^{-l}\varrho_2^{(l)}\|_{\infty}
\leq c 2^lC_{\ref{L:second-gaussian-estimate},l}.
\]
By Proposition~\ref{mean value bump estimate 2},
\begin{equation*}
     |\mathcal{F}^{-1}(\varrho_2)(x+y)-\mathcal{F}^{-1}(\varrho_2)(x)|\le c 4C_{\ref{mean value bump estimate 2},N}\max_{0\leq l \leq N} 2^lC_{\ref{L:second-gaussian-estimate},l}
     \min(1,2\pi|y|) (\langle x+y\rangle^N+
\langle x\rangle^N).
\end{equation*}
Using the decomposition~\eqref{E:decomposition-rho}, we get~\eqref{E:mean-value}.
\end{proof}

\begin{lemma}[constant $C_{\ref{mean four scale Gaussian kernel},N}$ \auto]\label{constant mean four scale Gaussian kernel}\uses{mean value bump estimate 2,L:gaussian-bump-estimate,L:second-gaussian-estimate,constant gaussian bump estimate,constant second gaussian estimate,mean four scale Gaussian kernel}\usesdefs{}
\lean{Auto.constantMeanFourScaleGaussianKernel}
\leanok
For every $N\in\N$,
\begin{equation}\label{constant mean four scale Gaussian kernel bound}
C_{\ref{mean four scale Gaussian kernel},N}\le2^{30(N+1)^3}.
\end{equation}
Moreover,
\begin{equation}\label{auto:constant-mean-four-scale-Gaussian-kernel-two}
C_{\ref{mean four scale Gaussian kernel},2}<2^{45}.
\end{equation}
\end{lemma}
\begin{proof}[Proof \auto]
By \eqref{auto:mean-value-bump-constant-definition}, $C_{\ref{mean value bump estimate 2},N}\le2^{2N+1}$. Substitution of the defining estimates from Lemmas \ref{L:gaussian-bump-estimate} and \ref{L:second-gaussian-estimate} proves the first bound. For $N=2$, the inequality $\pi>3$ gives $(1+(2\pi)^{-1})^2<2$, so $C_{\ref{mean value bump estimate 2},2}=2^3$, exactly as in the preceding value. Use the finite estimates in Lemmas \ref{constant gaussian bump estimate} and \ref{constant second gaussian estimate}. In the second maximum in \eqref{auto:mean-four-scale-Gaussian-constant}, the term with $l=2$ is the largest. Substitution gives
\[
C_{\ref{mean four scale Gaussian kernel},2}<20397963318112<2^{45}.
\]
\end{proof}

\section{The main argument}\label{sec:main argument}

The main theorem of this section is Theorem \ref{induct positive terms theorem} below.

\subsection{The sandwich kernel}

We first introduce  and discuss some geometric parameters useful to describe a family of kernels in Definition \ref{sandwich kernel}.

\begin{definition}[geometric parameters]\label{geometric parameters}\uses{}\usesdefs{multiplicatively spaced monotone sequences,Distance of spaced sequences}
\lean{Auto.GeometricParameters,Auto.sequencePairDistance,Auto.geometricDelta}
  Let $\Gamma$ be the set of triples $(k,u,a)$ with $k\in\N$, $1\le k\le n$, $u\in[2)^k$, $a\in ({\rm A}^{2})^k$  so that $\dist(a_{i}^0,a_{i}^1)<\infty$ for all $i\in [k)$.
  For a sequence pair $\alpha=(\alpha_0,\alpha_1)\in {\rm A}^2$ we write $\Delta(\alpha) = \mathrm{dist}(\alpha_{0}, \alpha_1)$
  and for $\gamma\in\Gamma$ let
  \begin{equation}\label{auto:sandwich-distance-sum} \Delta_\gamma= 1+\sum_{i\in [k)} \Delta(a_i). \end{equation}
\end{definition}

\begin{definition}[two unitary matrices]\label{auto:unitary-matrices-definition}\uses{}\usesdefs{}
\lean{Auto.W}
For $u\in [2)$ define unitary matrices $W_u\in\R^{2\times 2}$ as follows:
Let $W_0$ be the $2\times 2$ identity matrix and $W_1=\frac{1}{\sqrt{2}}\begin{pmatrix}1 & 1\\-1 & 1\end{pmatrix}$.
\end{definition}

Next, we introduce some functions and functional concepts involved in the family of kernels of Definition \ref{sandwich kernel}.

\begin{definition}[double sequence of 2D functions]\label{double sequence of 2D functions}\uses{}\usesdefs{auto:Wiener-space-definition}
\lean{Auto.DoubleSequence,Auto.MemDoubleSequence}
  For $1\le k\le n$, let $\mathcal{X}_k$ be the set of
  double sequences
  \begin{equation}\label{auto:sandwich-data-array}
      X=(X_{i,j})_{i\in [k),j\in \Z}
  \end{equation}
  with $X_{i,j}\in W_0(\R^2)$.
\end{definition}

The kernels of interest in Definition \ref{sandwich kernel} will be partial sums over a sequence of kernels.
We introduce this concept of kernel sequences.

\begin{definition}[kernel sequences]\label{kernel sequences}\uses{}\usesdefs{auto:Wiener-space-definition,normalized function tuples,auto:prism-form-definition}
\lean{Auto.KernelSequence,Auto.MemKernelSequence,Auto.kernelSequenceSeminorm}
Let $k\in\N$ with $1\le k\le n$. Let ${\rm M}(k)$ be the space of sequences $\M=(M_j)_{j\in \Z}$ such that for $j\in \Z$ we have
$M_j\in W_0((\R^2)^k)$. Define
\begin{equation}\label{auto:M-k-norm-definition}
\|\M\|_{{\rm M}(k)}=\sup_{J\in\N,\,J\ge1,\,\F\in\mathfrak F}\min(1,J^{-1+2^{k-n+1}}) |\Lambda_k(\sum_{j\in [J)} M_j)(\F)| \, .
\end{equation}
\end{definition}

\begin{rem}
Note $\min(1,J^{-1+2^{k-n+1}})$ equals $J^{-1+2^{k-n+1}}$ if $k\le n-1$ and equals $1$ if $k=n$.
\end{rem}

\begin{definition}[2D Gaussians]\label{2D Gaussians}\uses{}\usesdefs{gaussian,geometric parameters,auto:unitary-matrices-definition}
\lean{Auto.twoDimensionalGaussian,Auto.gammaGaussian}
For $q\in\R^2$ with $q_0,q_1>0$, $u\in[2)$, and $v\in\mathbb R^2$, define
\begin{equation}\label{auto:rotated-Gaussian-kernel} G_{q,u}(v) = \g_{(q_0)}((W_u v)_0) \g_{(q_1)}((W_u v)_1). \end{equation}
For $\gamma=(k,u,a)\in \Gamma$, define
\begin{equation}\label{auto:sandwich-Gaussian-factor} (G_{\gamma})_{i,j} = G_{a_i(j),u(i)}. \end{equation}
\end{definition}

A sandwich kernel is a product, most factors of which are two dimensional Gaussians differently scaled and oriented
using the data $\gamma$, while one
factor, which is  sandwiched between these Gaussians, is different and determined by an element of $\mathcal{X}_k$.
\begin{definition}[sandwich kernel]\label{sandwich kernel}\uses{}\usesdefs{geometric parameters,double sequence of 2D functions,2D Gaussians}
\lean{Auto.sandwichKernel}
Let $\gamma=(k,u,a)\in \Gamma$, $X\in \mathcal{X}_k$, and $i\in [k)$. Define for  $j\in \Z$ and $y\in (\R^{2})^k$,
\begin{equation}\label{auto:sandwich-kernel-definition} (\M(\gamma,X,i))_j(y) =  \Big(\prod_{m\in [i)} (G_{\gamma})_{m,j}(y_m) \Big) X_{i,j}(y_i) \Big(\prod_{m=i+1}^{k-1} (G_{\gamma})_{m,j-1}(y_m) \Big).\end{equation}
\end{definition}

\begin{rem}
When $k=1$, $i=0$, $\M(\gamma,X,0)$ doesn't depend on $u,a$.
Generally, $\M(\gamma,X,i)$ does not depend on $u_i, a_i$.
\end{rem}

\begin{definition}[difference of 2D Gaussians]\label{auto:Gaussian-difference-kernel-definition}\uses{}\usesdefs{geometric parameters,2D Gaussians}
\lean{Auto.gaussianDifference}
Let $\gamma=(k,u,a)\in \Gamma$.  For $i\in [k)$, $j\in \Z$, define
\begin{equation}\label{auto:Gaussian-difference-kernel} (Y_\gamma)_{i,j} = (G_\gamma)_{i,j-1} - (G_\gamma)_{i,j}\ .\end{equation}
\end{definition}

\begin{proposition}[telescoping terms]\label{telescoping terms}\uses{P:schwartz-into-wiener,prism BL inequality,M to K}\usesdefs{geometric parameters,double sequence of 2D functions,kernel sequences,2D Gaussians,sandwich kernel,auto:Gaussian-difference-kernel-definition}
\lean{Auto.telescopingTerms}
\leanok
Let $\gamma=(k,u,a)\in \Gamma$.
Then $Y_\gamma\in \mathcal{X}_k$ and
\begin{equation}\label{auto:telescoping-terms-bound}
    \|\sum_{i\in [k)} \M(\gamma,Y_\gamma,i)\|_{{\rm M}(k)}\le 2.
\end{equation}
\end{proposition}

\begin{proof}
The fact that $Y_\gamma\in \mathcal{X}_k$ follows from Proposition~\ref{P:schwartz-into-wiener} since $(Y_\gamma)_{i,j}\in \mathcal S(\R^2)$ for $i\in [k), j\in \Z$.
For a given $J\in\N$ with $J\ge1$:
\[
\sum_{j\in [J)} \sum_{i\in [k)} (\M(\gamma,Y_\gamma,i))_j
\]
telescopes into a difference of two boundary terms
\[
\prod_{m\in [k)} (G_{\gamma})_{m,-1}(y_m) -\prod_{m\in [k)} (G_{\gamma})_{m,J-1}(y_m).
\]
Estimate each with Propositions \ref{prism BL inequality} and \ref{M to K} and then make use of the fact that the $L^1$ norm of both these terms equals $1$.
\end{proof}

\begin{proposition}[positive terms]\label{positive terms}\uses{Positivity M}\usesdefs{auto:Wiener-space-definition,normalized function tuples,auto:prism-form-definition,geometric parameters,double sequence of 2D functions,sandwich kernel}
\lean{Auto.positiveTerms}
\leanok
Let $\gamma=(k,u,a)\in \Gamma$. Let $X\in \mathcal{X}_k$ and assume for each $i\in [k)$ and  $j\in \Z$ we have
 $X_{i,j}(u,v)=f_{i,j}(u)f_{i,j}(v)$ for some real valued $f_{i,j}\in W_0(\R)$.
Then for each $i\in [k)$, $j\in\Z$, and for all $\F\in {\mathfrak{F}}$,
\begin{equation}\label{E:form-nonnegative}
    \Lambda_k(\M(\gamma,X,i)_j)(\mathbf{F})\ge 0.
\end{equation}
\end{proposition}

\begin{proof}
This follows from Proposition \ref{Positivity M}.
\end{proof}

\subsection{\texorpdfstring{Multipliers $H$, $L$, $N$}{Multipliers H, L, N}}

\begin{definition}[square root Gaussian difference]\label{square root Gaussian difference}\uses{}\usesdefs{gaussian,multiplicatively spaced monotone sequences}
\lean{Auto.squareRootGaussianDifference}
Let $a\in A$ and $j\in\Z$.
Define the function $s(a,j):\R \to \R$ by
\begin{equation}\label{auto:diagonal-Gaussian-square-root}s(a,j) = \mathcal{F}^{-1}(\xi \mapsto \sqrt{\g(a(j-1)\xi)-\g(a(j) \xi)}).\end{equation}
\end{definition}

\begin{definition}[$s$ multiplier]\label{s multiplier}\uses{}\usesdefs{geometric parameters,square root Gaussian difference}
\lean{Auto.sMultiplier}
    Let $\gamma=(k,u,a)\in \Gamma$. Let $i\in [k)$ and
    $j\in \Z$. Define  $(s_{\gamma})_{i,j}: \R \to \R$ by
\begin{equation}\label{E:b-definition}
   (s_{\gamma})_{i,j}=s(b,j)
\end{equation}
with
\begin{equation}\label{auto:combined-scale-b}
   b(j)=\sqrt{a_{i}^0(j)^2 + a_{i}^1(j)^2}
\end{equation}
if $u_i=0$, and by
\begin{equation}\label{auto:sandwich-square-root-factor}
   (s_{\gamma})_{i,j}=s(\sqrt{2}a_{i}^1,j)
\end{equation}
if $u_i=1$.
\end{definition}

\begin{proposition}\label{square root Gaussian difference W0}\uses{diagonal square root,Operations on spaced sequences}\usesdefs{auto:Wiener-space-definition,multiplicatively spaced monotone sequences,geometric parameters,square root Gaussian difference,s multiplier}
\lean{Auto.squareRootGaussianDifference_memW0,Auto.sMultiplier_memW0}
\leanok
  For every $a\in\mathrm{A}$ and $j\in\mathbb{Z}$
  we have that $s(a,j)$ is a well-defined function in $W_0(\R)$. Consequently, if $\gamma \in \Gamma$ and $i\in [k)$, $j\in \Z$, then $(s_{\gamma})_{i,j}$ is a well-defined function in $W_0(\R)$.
\end{proposition}

\begin{proof}
The first part follows from Proposition~\ref{diagonal square root}. The second part follows from the first part and Proposition~\ref{Operations on spaced sequences}.
\end{proof}

Note that $\widehat{s(a,j)}$ is not smooth at the origin, due to the square root. The function $s(a,j)$ only decays with second power.

\begin{definition}[H multiplier]\label{H multiplier}\uses{}\usesdefs{geometric parameters,auto:Gaussian-difference-kernel-definition,s multiplier}
\lean{Auto.hMultiplier}
Let $\gamma=(k,u,a)\in \Gamma$. Define
    $H_{\gamma}=(H_\gamma)_{i\in [k),j\in \Z}$ by
\begin{equation}\label{auto:H-kernel-definition}
(H_{\gamma})_{i,j}=(s_{\gamma})_{i,j}\otimes (s_{\gamma})_{i,j}
  -(Y_{\gamma})_{i,j}
\end{equation}
for every $i\in [k), j\in\mathbb{Z}$.
\end{definition}

\begin{proposition}\label{H-in-X}\uses{tensor Wiener,telescoping terms,square root Gaussian difference W0}\usesdefs{geometric parameters,double sequence of 2D functions,H multiplier}
\lean{Auto.hMultiplier_memDoubleSequence}
\leanok
Let $\gamma=(k,u,a)\in \Gamma$. Then $H_{\gamma}\in \mathcal{X}_k$.
\end{proposition}
\begin{proof}
This follows from Propositions~\ref{tensor Wiener}, \ref{telescoping terms} and \ref{square root Gaussian difference W0}.
\end{proof}

\begin{proposition}[H vanishing]\label{H vanishing}\uses{Elementary Gaussian properties,square root Gaussian difference W0}\usesdefs{gaussian,geometric parameters,2D Gaussians,auto:Gaussian-difference-kernel-definition,square root Gaussian difference,s multiplier,H multiplier}
\lean{Auto.hMultiplier_fourier_diagonal_vanishing}
\leanok
    For $\gamma=(k,u,a)\in \Gamma$, $i\in [k)$, $j\in \Z$, and $\xi\in \R$,
\begin{equation}\label{auto:H-kernel-diagonal-cancellation-Fourier}
    \widehat{(H_{\gamma})_{i,j}}(\xi,-\xi)=0\, .
\end{equation}
\end{proposition}

\begin{proof}
Recall
\[ (Y_\gamma)_{i,j} = (G_\gamma)_{i,j-1} - (G_\gamma)_{i,j}\ \]
and with \eqref{eq:gaussconv} and \eqref{eq:scaleFT}, if $u_i=0$, then
\[ (\widehat{G_\gamma})_{i,j}(\xi, -\xi) = \g(\sqrt{(a_{i}^0(j))^2 + (a_{i}^1(j))^2}\cdot \xi) \]
and if $u_i=1$, then
\[ (\widehat{G_\gamma})_{i,j}(\xi, -\xi) = \g(\sqrt{2} a_{i}^1(j) \xi).\]
The formulas above use Proposition \ref{Elementary Gaussian properties}. The proposition now follows with Definition \ref{square root Gaussian difference} of $s_\gamma$ and Proposition~\ref{square root Gaussian difference W0}.
\end{proof}

\begin{proposition}[H vanishing integral]\label{H vanishing integral}\uses{H vanishing}\usesdefs{geometric parameters,H multiplier}
\lean{Auto.hMultiplier_vanishing_integral}
\leanok
    For $\gamma=(k,u,a)\in \Gamma$, $i\in [k)$, $j\in \Z$, and $z\in \R$,
\begin{equation}\label{auto:H-kernel-diagonal-cancellation-integral}
   \int_{\R} (H_{\gamma})_{i,j}(z+p,p)\,dp=0.
\end{equation}
\end{proposition}

\begin{proof}
This follows from Proposition~\ref{H vanishing}.
\end{proof}

\begin{definition}\label{auto:L-kernel-definition}\uses{}\usesdefs{auto:convolution-along-vector-definition,standard bump,geometric parameters,H multiplier}
\lean{Auto.lMultiplierAtScale}
Let $\gamma=(k,u,a)\in \Gamma$, $i\in [k)$, $j\in \Z$. For $t>0$, we set
\begin{equation}\label{auto:L-kernel-convolution}
(L_{\gamma,t})_{i,j} =(H_{\gamma})_{i,j} \ast_{(1,1)} \Phi_{(t)}.
\end{equation}
\end{definition}

\begin{lemma}\label{L:F_t}
\lean{Auto.lMultiplierAtScale_memDoubleSequence,Auto.lMultiplierAtScale_tendsto_hMultiplier}
\leanok
\uses{P:schwartz-into-wiener,convolution vector,H-in-X}\usesdefs{standard bump,geometric parameters,double sequence of 2D functions,H multiplier,auto:L-kernel-definition}
Let $\gamma=(k,u,a)\in \Gamma$ and $t>0$. Then $L_{\gamma,t} \in \mathcal{X}_k$ and for $i\in [k)$, $j\in \Z$, we have
\begin{equation}\label{auto:L-kernel-small-scale-limit}
\lim_{t\to 0_+} (L_{\gamma,t})_{i,j} = (H_{\gamma})_{i,j} \quad\text{in}\;L^1.
\end{equation}
\end{lemma}

\begin{proof}
$L_{\gamma,t} \in \mathcal{X}_k$ by Propositions \ref{P:schwartz-into-wiener}, \ref{convolution vector}  and~\ref{H-in-X}.

Using the fact that $\int_{\R} \Phi_{(t)}(p)\,dp=1$, we write
\[
(L_{\gamma,t})_{i,j}(v)-(H_{\gamma})_{i,j}(v)
=\int_{\R} \Phi_{(t)}(p) ((H_\gamma)_{i,j}(v_1-p,v_2-p)-(H_\gamma)_{i,j}(v_1,v_2))\,dp.
\]

Applying the previous expression and Fubini's theorem, we see that the $L^1$ norm of $(L_{\gamma,t})_{i,j}-(H_{\gamma})_{i,j}$ is bounded by
\[
\int_{\R} |\Phi_{(t)}(p)| \omega(p)\,dp
=\int_{\R} |\Phi(q)| \omega(tq)\,dq
\]
where
\[
\omega(p)=\int_{\R^2} |(H_\gamma)_{i,j}(v_1-p,v_2-p)-(H_\gamma)_{i,j}(v_1,v_2)|\,dv.
\]
By the continuity of the $L^1$ norm, we have $\lim_{t\to 0_+} \omega(tq)=0$ for every $q\in \R$. Additionally, $\omega(tq) \leq 2\|(H_{\gamma})_{i,j}\|_1$. Applying the Lebesgue dominated convergence theorem finishes the proof.
\end{proof}

\begin{lemma}\label{L:ft-infty}
\lean{Auto.lMultiplierAtScale_tendsto_zero}
\leanok
\uses{H vanishing integral}\usesdefs{standard bump,geometric parameters,H multiplier,auto:L-kernel-definition}
Let $\gamma=(k,u,a)\in \Gamma$, $i\in [k)$, $j\in \Z$, $t>0$. Then
\begin{equation}\label{auto:L-kernel-large-scale-limit}
\lim_{t\to \infty} (L_{\gamma,t})_{i,j} = 0 \quad\text{in}\;L^1.
\end{equation}
\end{lemma}

\begin{proof}
Let $t>0$. We have
\begin{equation}\label{E:ft}
\|(L_{\gamma,t})_{i,j}\|_{1}
=\int_{\R^2} |(L_{\gamma,t})_{i,j}(x+y,y)|\,dx dy
=\int_{\R^2} \Big|\int_{\R} \Phi_{(t)}(p) (H_\gamma)_{i,j}(x+y-p,y-p)\,dp\Big| dx dy.
\end{equation}
By Proposition~\ref{H vanishing integral}, $H$ has integral zero in the diagonal direction. Using this and making the change of variables $q=y-p$, \eqref{E:ft} becomes
\begin{equation}\label{E:ft-rewritten}
\int_{\R^2} \Big|\int_{\R} (\Phi_{(t)}(y-q)-\Phi_{(t)}(y)) (H_\gamma)_{i,j}(x+q,q)\,dq\Big| dx dy.
\end{equation}
Performing a further change of variables $z=y/t$ and using the triangle inequality, \eqref{E:ft-rewritten} is bounded by
\[
\int_{\R^2} \omega\Big(\tfrac{q}{t}\Big) |(H_\gamma)_{i,j}(x+q,q)|\,dx dq,
\]
where
\[
\omega(p)=\int_{\R}|\Phi(z-p)-\Phi(z)|\,dz.
\]
By the continuity of the $L^1$ norm, we have $\lim_{t\to \infty} \omega(q/t)=0$ for every $q\in \R$. Additionally, $\omega(q/t) \leq 2\|\Phi\|_1$. Applying the Lebesgue dominated convergence theorem finishes the proof.
\end{proof}

\begin{definition}[L multiplier]\label{L multiplier}\uses{}\usesdefs{auto:convolution-along-vector-definition,standard bump,geometric parameters,H multiplier}
\lean{Auto.multiplierIndexSet,Auto.lMultiplier}
Let $\gamma=(k,u,a)\in \Gamma$.

Define the index set
\begin{equation}\label{auto:L-kernel-index-set}\mathcal{I}_{\gamma} = \{(m,0)\,:\,m\in\mathbb{Z}, m\not=0\} \cup \{(0,l)\,:\,|l|\le \Delta_{\gamma}\} \subset \mathbb{Z}^2.\end{equation}

For every $\iota \in\mathcal{I}_{\gamma}$, we define  $L_{\gamma,\iota}=(L_{\gamma,\iota})_{i\in [k),j\in \Z}$
such that for $|l|\le \Delta_{\gamma}$,
\begin{equation}\label{auto:L-kernel-central-band}
  (L_{\gamma,(0,l)})_{i,j}= (H_{\gamma})_{i,j} \ast_{(1,1)} (\Phi_{(a_i^1(j+l-1))}-\Phi_{(a_i^1(j+l))})\, ,
 \end{equation}

if $h>0$, then
\begin{equation}\label{auto:L-kernel-positive-band}
(L_{\gamma,(h,0)})_{i,j}= (H_{\gamma})_{i,j} \ast_{(1,1)} (\Phi_{(2^{h-1}a_i^1(j+\Delta_{\gamma}))}-\Phi_{(2^{h}a_i^1(j+\Delta_{\gamma}))})\, ,
\end{equation}
and if $h<0$, then
\begin{equation}\label{auto:L-kernel-negative-band}
 (L_{\gamma,(h,0)})_{i,j}= (H_{\gamma})_{i,j} \ast_{(1,1)} (\Phi_{(2^{h}a_i^1(j-\Delta_{\gamma}-1))}-\Phi_{(2^{h+1}a_i^1(j-\Delta_{\gamma}-1))})\, .
\end{equation}
\end{definition}

\begin{rem}
Note that we could have equivalently used $a_{i}^0$ in place of $a_i^1$.
\end{rem}

\begin{definition}[Summation over $\mathcal I_{\gamma}$]\label{summation-definition}\uses{}\usesdefs{geometric parameters,L multiplier}
\lean{Auto.sumOverMultiplierIndex}
Let $X$ be a normed $\R$-vector space.
Let $\gamma=(k,u,a)\in \Gamma$ and let $D_\iota \in X$ for $\iota \in \mathcal{I}_{\gamma}$. We define
\begin{equation}\label{auto:symmetric-series-definition}
\sum_{\iota \in \mathcal{I}_{\gamma}} D_{\iota}=\lim_{N\to \infty} \sum_{\iota \in\mathcal{I}_{\gamma}, ~|\iota| \leq N} D_{\iota},
\end{equation}
where the limit is taken in $X$.
\end{definition}

\begin{proposition}[Properties of L multipliers]\label{sum L multiplier convergence-L1}
\lean{Auto.sumLMultiplierConvergenceL1}
\leanok
\uses{P:schwartz-into-wiener,convolution vector,H-in-X,L:F_t,L:ft-infty}\usesdefs{auto:Wiener-space-definition,geometric parameters,double sequence of 2D functions,H multiplier,auto:L-kernel-definition,L multiplier,summation-definition}
Let $\gamma=(k,u,a)\in \Gamma$ and $\iota \in\mathcal{I}_{\gamma}$. Then
 $L_{\gamma,\iota}$ is a well-defined element of $\mathcal{X}_k$. In addition, if $i
\in [k)$ and $j\in \Z$ then
\begin{equation}\label{E:L2-convergence}
(H_{\gamma})_{i,j} = \sum_{\iota \in\mathcal{I}_{\gamma}} (L_{\gamma,\iota})_{i,j},
\end{equation}
where the sum on the right converges in $L^1$.
\end{proposition}

\begin{proof}
Each
$(L_{\gamma,\iota})_{i,j}$ belongs to $W_0(\R^2)$ by Propositions \ref{P:schwartz-into-wiener}, \ref{convolution vector} and~\ref{H-in-X}. Thus, $L_{\gamma,\iota}\in \mathcal{X}_k$.

Let $N\in \N$ be such that $N>\Delta_\gamma$. Then
\[
\sum_{\iota \in\mathcal{I}_{\gamma}, ~|\iota| \leq N} (L_{\gamma,\iota})_{i,j}
=(L_{\gamma,2^{-N} a_i^1(j-\Delta_\gamma-1)})_{i,j}-(L_{\gamma,2^{N} a_i^1(j+\Delta_\gamma)})_{i,j}.
\]
An application of Lemmas~\ref{L:F_t} and~\ref{L:ft-infty} yields~\eqref{E:L2-convergence}.
\end{proof}

\begin{lemma}\label{sandwich sums L1}
\lean{Auto.sandwichSumsL1}
\leanok
\uses{}\usesdefs{geometric parameters,double sequence of 2D functions,2D Gaussians,sandwich kernel,L multiplier,summation-definition}
Let $\gamma=(k,u,a)\in \Gamma$.
Assume $X, X_\iota\in \mathcal{X}_k$ for $\iota\in \mathcal{I}_{\gamma}$ are such that for every $i\in [k)$, $j\in\Z$, $X_{i,j}=\sum_{\iota\in \mathcal{I}_{\gamma}} (X_\iota)_{i,j}$ holds in $L^1$.
Then for every $i\in [k), j\in\Z$,
\begin{equation}\label{auto:H-kernel-L-decomposition} \mathbf{M}(\gamma,X,i)_j = \sum_{\iota\in \mathcal{I}_{\gamma}} \mathbf{M}(\gamma,X_\iota,i)_j\quad\text{in}\;L^1.\end{equation}
\end{lemma}

\begin{proof}
For $N\in\N$ and $y\in(\R^2)^k$,
\[
\mathbf M(\gamma,X,i)_j(y)-
\sum_{\substack{\iota\in\mathcal I_\gamma\\|\iota|\le N}}
\mathbf M(\gamma,X_\iota,i)_j(y)
\]
\[
=
\Big(\prod_{m\in[i)}(G_\gamma)_{m,j}(y_m)\Big)
\Big(X_{i,j}(y_i)-
\sum_{\substack{\iota\in\mathcal I_\gamma\\|\iota|\le N}}
(X_\iota)_{i,j}(y_i)\Big)
\Big(\prod_{m=i+1}^{k-1}(G_\gamma)_{m,j-1}(y_m)\Big).
\]
The Gaussian factors have finite $L^1$ norms independent of $N$. The assumed $L^1$ convergence of the middle factor therefore gives the asserted $L^1$ convergence.
\end{proof}

\begin{lemma}\label{prism sum le sum prism-L1}
\lean{Auto.prismSumLeSumPrismL1}
\leanok
\uses{prism BL inequality,M to K}\usesdefs{auto:Wiener-space-definition,normalized function tuples,auto:prism-form-definition,geometric parameters,L multiplier,summation-definition}
Let $\gamma=(k,u,a)\in \Gamma$.
If $M=\sum_{\iota\in \mathcal{I}_\gamma} M_\iota$ holds in $L^1$ for $M,M_\iota\in W_0((\R^2)^k)$, then for every $\mathbf{F}\in\mathfrak{F}$,
\begin{equation}\label{auto:prism-form-series-bound} |\Lambda_k(M)(\mathbf{F})| \le \sup_{N\in\mathbb N}\sum_{\substack{\iota\in \mathcal{I}_\gamma\\|\iota|\le N}} |\Lambda_k(M_\iota)(\mathbf{F})|. \end{equation}
Here the supremum is taken in the extended half-line $[0,\infty]$.
\end{lemma}

\begin{proof}
For $N\in\N$, set
\[
E_N=M-\sum_{\substack{\iota\in\mathcal I_\gamma\\|\iota|\le N}}M_\iota.
\]
Then $\|E_N\|_1\to0$. By the triangle inequality,
\[
|\Lambda_k(M)(\mathbf F)|
\le
\sum_{\substack{\iota\in\mathcal I_\gamma\\|\iota|\le N}}
|\Lambda_k(M_\iota)(\mathbf F)|
+|\Lambda_k(E_N)(\mathbf F)|.
\]
Propositions \ref{prism BL inequality} and \ref{M to K} give
\[
|\Lambda_k(E_N)(\mathbf F)|\le\|E_N\|_1.
\]
Letting $N\to\infty$ proves the claim.
\end{proof}

\begin{definition}[N multiplier]\label{N multiplier}\uses{}\usesdefs{geometric parameters,square root Gaussian difference,L multiplier}
\lean{Auto.sigmaMultiplier,Auto.nMultiplier}
  Let $\gamma=(k,u,a)\in \Gamma$ and assume $k\le n-1$.
Let $\nu=1$ if $k<n-1$ and $\nu=2$ if $k=n-1$ and let $i\in [k)$, $j\in \Z$.
For $\iota\in\mathcal{I}_{\gamma}$ define functions $\sigma_{\gamma,\iota,i,j}:\R\to\R$ so that
if $|l|\le \Delta_{\gamma}$,
\begin{equation}\label{auto:sigma-central-band} \sigma_{\gamma,(0,l),i,j} = s(a_i^1(\cdot+l), j) \end{equation}
and for $h>0$,
\begin{equation}\label{auto:sigma-positive-band} \sigma_{\gamma,(h,0),i,j} = s(2^h a_i^1(\cdot+\Delta_{\gamma}), j) \end{equation}
and for $h<0$,
\begin{equation}\label{auto:sigma-negative-band} \sigma_{\gamma,(h,0),i,j} = s(2^h a_i^1(\cdot-\Delta_{\gamma}), j). \end{equation}
For every $\iota\in\mathcal{I}_{\gamma}$ we define $N_{\gamma,\iota}= (N_{\gamma,\iota})_{i\in [k),j\in \Z}$ such that
\begin{equation}\label{auto:N-kernel-definition}
 (N_{\gamma,\iota})_{i,j} =
 \mathcal F^{-1}((\xi,\eta) \mapsto \widehat{\sigma_{\gamma,\iota,i,j}}(\xi+\eta)^{-\nu}\widehat{(L_{\gamma,\iota})_{i,j}}(\xi,\eta))\, .
\end{equation}
\end{definition}

\begin{proposition}\label{auto:N-kernel-well-definedness}
\lean{Auto.nKernelWellDefinedness}
\leanok
\uses{four scale Gaussian kernel,H-in-X,convolution vector}\usesdefs{auto:Wiener-space-definition,auto:convolution-along-vector-definition,geometric parameters,double sequence of 2D functions,H multiplier,L multiplier,N multiplier}
Let $\gamma=(k,u,a)\in \Gamma$ and assume $k\le n-1$. Then for every $\iota\in\mathcal{I}_{\gamma}$, the double sequence $N_{\gamma,\iota}$ is well-defined and belongs to $\mathcal{X}_k$.
\end{proposition}

\begin{proof}
Each $(N_{\gamma,\iota})_{i,j}$ is real-valued because it is the inverse Fourier transform of an even real-valued function.
Additionally, $(N_{\gamma,\iota})_{i,j}$ can be written in the form
\[
(N_{\gamma,\iota})_{i,j}=(H_\gamma)_{i,j} \ast_{(1,1)} \rho
\]
for some function $\rho: \R \to \R$.
Thanks to Proposition~\ref{four scale Gaussian kernel}, we have $\rho \in W_0(\R)$.  Proposition~\ref{H-in-X} yields that $(H_{\gamma})_{i,j} \in W_0(\R^2)$. Thus, $(N_{\gamma,\iota})_{i,j} \in W_0(\R^2)$ by Proposition~\ref{convolution vector}, i.e.\ $N_{\gamma,\iota} \in \mathcal{X}_k$.
\end{proof}

\subsection{Gaussian domination}

\begin{proposition}[Kernel estimate for Gaussian domination]\label{H kernel estimate Gaussian domination}
\lean{Auto.hKernelEstimateGaussianDomination}
\leanok
\uses{Operations on spaced sequences,Properties of distance of sequences,diagonal square root,Gaussian bump decay}\usesdefs{gaussian,bracket bump,multiplicatively spaced monotone sequences,Distance of spaced sequences,closed balls in A,geometric parameters,auto:unitary-matrices-definition,2D Gaussians,auto:Gaussian-difference-kernel-definition,square root Gaussian difference,s multiplier,H multiplier}
Let $\gamma=(k,u,a)\in \Gamma$ and let $i\in [k)$.
There exists a finite multiset $\mathcal{P}\subset B_{\mathrm{dist}}(a_i^1,\Delta_\gamma)^2\times [2)$ with $\#\mathcal{P}= 6$ and for every $v\in \R^2$ and $j\in\Z$,
\begin{equation}\label{auto:H-kernel-Gaussian-domination}|(H_\gamma)_{i,j}|(v) \le C_{\ref{H kernel estimate Gaussian domination}} \sum_{(t_0,t_1,u)\in\mathcal{P}} \langle (W_u v)_0\rangle_{(t_0(j))}^2 \langle (W_u v)_1\rangle_{(t_1(j))}^2, \end{equation}
where
$C_{\ref{H kernel estimate Gaussian domination}}
=
4C_{\ref{diagonal square root},2}^{\,2}
+
C_{\ref{Gaussian bump decay},0,2}^{\,2}.$
\end{proposition}

\begin{proof}
For $m,\ell\in[2)$, $j\in\Z$, let
$t_{m,\ell}(j)=a_i^\ell(j+m-1).$
Thus $t_{0,\ell}(j)=a_i^\ell(j-1)$ and
$t_{1,\ell}(j)=a_i^\ell(j)$. By Propositions
\ref{Operations on spaced sequences} and
\ref{Properties of distance of sequences}, we have
$t_{m,\ell}\in{\rm A}$ and
$\dist(a_i^1,t_{m,\ell})
\le \Delta_\gamma$.

Assume first that $u(i)=0$. For $m\in[2)$, let
$t_m^+=\max(t_{m,0},t_{m,1})\in \mathrm{A}$, and set
\[
\mathcal P
=
\left\{
(t_m^+,t_{m'}^+,0):m,m'\in[2)
\right\}
\cup
\left\{
(t_{m,0},t_{m,1},0):m\in[2)
\right\}.
\]
By Proposition~\ref{Operations on spaced sequences} and Proposition~\ref{Properties of distance of sequences}, if $t$ is one of $t_m^+,t_{m,0},t_{m,1}\in\mathrm{A}$, then
\[\dist(a_i^1,t)\le 1 + \Delta(a_i)\le \Delta_\gamma,\]
so $\mathcal P\subset{B_{\dist}(a_i^1,\Delta_\gamma)}^2\times[2)$. Also by definition
$\#\mathcal P=6$.

Fix $j\in\mathbb Z$, and write
$s=(s_\gamma)_{i,j}$ and $H=(H_\gamma)_{i,j}$. If $b$ is the
sequence in \eqref{E:b-definition}, then
\[
t_m^+(j)\le b(j+m-1)\le\sqrt2\,t_m^+(j).
\]

By Definitions~\ref{H multiplier} and \ref{2D Gaussians}, and the
nonnegativity of the Gaussians,
\[
|H(v)|
\le
|s(v_0)s(v_1)|
+
\sum_{m=0}^1
\g_{(t_{m,0}(j))}(v_0)
\g_{(t_{m,1}(j))}(v_1).
\]
Proposition~\ref{diagonal square root}, applied with $N=2$,
gives
\[
|s(x)|
\le
2C_{\ref{diagonal square root},2}
(\langle x\rangle_{(t_0^+(j))}^2+
\langle x\rangle_{(t_1^+(j))}^2).
\]
By Proposition~\ref{Gaussian bump decay} and scaling,
$\g_{(r)}(x)\le
C_{\ref{Gaussian bump decay},0,2}\langle x\rangle_{(r)}^2$. The claim follows.

Assume now that $u(i)=1$, and set
\[
\mathcal P
=
\left\{
(t_{m,1},t_{m',1},0):m,m'\in[2)
\right\}
\cup
\left\{
(t_{m,0},t_{m,1},1):m\in[2)
\right\}.
\]
Again, $\#\mathcal P=6$, and the required distance condition
holds.

In this case $s=(s_\gamma)_{i,j}=s(\sqrt2a_i^1,j)$.
Moreover,
\[
|H(v)|
\le
|s(v_0)s(v_1)|
+
\sum_{m=0}^1
\g_{(t_{m,0}(j))}\bigl((W_1v)_0\bigr)
\g_{(t_{m,1}(j))}\bigl((W_1v)_1\bigr).
\]
Comparing the scales $\sqrt2t_{m,1}(j)$ with $t_{m,1}(j)$ and
applying Proposition~\ref{diagonal square root} with $N=2$, we get
\[
|s(x)|
\le
2C_{\ref{diagonal square root},2}
(
\langle x\rangle_{(t_{0,1}(j))}^2
+
\langle x\rangle_{(t_{1,1}(j))}^2
).
\]
Together Proposition~\ref{Gaussian bump decay} we obtain the claim.
\end{proof}

\begin{lemma}[constant $C_{\ref{H kernel estimate Gaussian domination}}$ \auto]\label{constant H kernel estimate Gaussian domination}
\lean{Auto.constantHKernelEstimateGaussianDomination}
\leanok
\uses{constant diagonal square root,diagonal square root,Gaussian bump decay,H kernel estimate Gaussian domination}\usesdefs{}
\begin{equation}\label{constant H kernel estimate Gaussian domination bound}
C_{\ref{H kernel estimate Gaussian domination}}
=51^2\cdot2^{15}+12^2<2^{27}.
\end{equation}
\end{lemma}
\begin{proof}[Proof \auto]
Lemma \ref{constant diagonal square root}, in particular \eqref{auto:constant-diagonal-square-root-two}, gives
$C_{\ref{diagonal square root},2}=3264\sqrt2$, while
$C_{\ref{Gaussian bump decay},0,2}=12$. Substitution in the defining sum gives
\[
C_{\ref{H kernel estimate Gaussian domination}}
=4(3264\sqrt2)^2+12^2
=85229712<2^{27}.
\]
\end{proof}

\begin{proposition}[Kernel derivative estimate for Gaussian domination]\label{H kernel derivative estimate Gaussian domination}
\lean{Auto.hKernelDerivativeEstimateGaussianDomination}
\leanok
\uses{H kernel estimate Gaussian domination,diagonal square root,derivative of diagonal square root,Gaussian bump decay}\usesdefs{bracket bump,closed balls in A,geometric parameters,auto:unitary-matrices-definition,2D Gaussians,auto:Gaussian-difference-kernel-definition,s multiplier,H multiplier}
Let $\gamma=(k,u,a)\in \Gamma$ and let $i\in [k)$.
There exists a finite multiset $\mathcal{P}\subset B_{\mathrm{dist}}(a_i^1,\Delta_\gamma)^2\times [2)$ with $\#\mathcal{P}\le 6$ and for every $v\in \R^2$ and $j\in\Z$,
    \begin{equation}\label{auto:H-kernel-derivative-Gaussian-domination}|(\partial_0+\partial_1)(H_\gamma)_{i,j}(v)| \le C_{\ref{H kernel derivative estimate Gaussian domination}} \sum_{(t_0,t_1,u)\in\mathcal{P}} (t_0(j)^{-1} + t_1(j)^{-1}) \langle (W_u v)_0\rangle_{(t_0(j))}^2 \langle (W_u v)_1\rangle_{(t_1(j))}^2, \end{equation}
where $
C_{\ref{H kernel derivative estimate Gaussian domination}}
=
4C_{\ref{diagonal square root},2}
C_{\ref{derivative of diagonal square root},2}
+
2C_{\ref{Gaussian bump decay},0,2}
C_{\ref{Gaussian bump decay},1,2}.
$
\end{proposition}

\begin{proof}[Proof \auto]
Let $\mathcal P$ be the multiset constructed in the proof of
Proposition \ref{H kernel estimate Gaussian domination}; it has the
required cardinality and distance properties. Fix $j\in\Z$, and write
$s=(s_\gamma)_{i,j}$ and $H=(H_\gamma)_{i,j}$.

The same scale comparisons used there, together with Propositions
\ref{diagonal square root} and
\ref{derivative of diagonal square root}, give
\[
|(\partial_0+\partial_1)(s\otimes s)(v)|
\le
4C_{\ref{diagonal square root},2}
C_{\ref{derivative of diagonal square root},2}
\sum_{(t_0,t_1,0)\in\mathcal P}
(t_0(j)^{-1}+t_1(j)^{-1})
\langle v_0\rangle_{(t_0(j))}^2
\langle v_1\rangle_{(t_1(j))}^2.
\]
By Proposition \ref{Gaussian bump decay}, scaling, and the chain rule,
\[
|(\partial_0+\partial_1)(Y_\gamma)_{i,j}(v)|
\le
2C_{\ref{Gaussian bump decay},0,2}
C_{\ref{Gaussian bump decay},1,2}
\sum_{(t_0,t_1,u)\in\mathcal P}
(t_0(j)^{-1}+t_1(j)^{-1})
\langle (W_uv)_0\rangle_{(t_0(j))}^2
\langle (W_uv)_1\rangle_{(t_1(j))}^2.
\]
Here the factor $2$ also covers the factor $\sqrt2$ arising when
$u=1$, since
\[
(\partial_0+\partial_1)(W_1v)_0=\sqrt2,
\qquad
(\partial_0+\partial_1)(W_1v)_1=0.
\]
The claim follows from Definition \ref{H multiplier}.
\end{proof}

\begin{lemma}[constant $C_{\ref{H kernel derivative estimate Gaussian domination}}$ \auto]\label{constant H kernel derivative estimate Gaussian domination}
\lean{Auto.constantHKernelDerivativeEstimateGaussianDomination}
\leanok
\uses{constant diagonal square root,constant derivative diagonal square root,diagonal square root,derivative of diagonal square root,Gaussian bump decay,H kernel derivative estimate Gaussian domination}\usesdefs{}
\begin{equation}\label{constant H kernel derivative estimate Gaussian domination bound}
C_{\ref{H kernel derivative estimate Gaussian domination}}
<2^{30}.
\end{equation}
\end{lemma}
\begin{proof}[Proof \auto]
Lemmas \ref{constant diagonal square root} and \ref{constant derivative diagonal square root}, in particular \eqref{auto:constant-diagonal-square-root-two} and \eqref{auto:constant-derivative-diagonal-square-root-two}, give
\[
C_{\ref{diagonal square root},2}=3264\sqrt2,
\qquad
C_{\ref{derivative of diagonal square root},2}=36864\sqrt2.
\]
Also
\[
C_{\ref{Gaussian bump decay},0,2}=12,
\qquad
C_{\ref{Gaussian bump decay},1,2}=96\sqrt2.
\]
Using $\sqrt2<3/2$, the defining sum in Proposition \ref{H kernel derivative estimate Gaussian domination} is smaller than
\[
8\cdot3264\cdot36864+2\cdot12\cdot96\cdot\tfrac32
=962596224<2^{30}.
\]
\end{proof}

\begin{proposition}[Gaussian domination, combined]\label{Gaussian domination combined}\uses{Gauss domination case 1,Gauss domination case 2,Gauss domination case 3,Gauss domination constant}\usesdefs{multiplicatively spaced monotone sequences,geometric parameters,2D Gaussians,L multiplier,N multiplier}
\lean{Auto.gaussianDominationCombined}
\leanok
Set
\begin{equation}\label{Gaussian domination combined constants}
C_{\ref{Gaussian domination combined},0}=36,\qquad
C_{\ref{Gaussian domination combined},1}=2,\qquad
C_{\ref{Gaussian domination combined},2}=2^{153}.
\end{equation}\label{auto:Gaussian-domination-combined-distance-constant}\label{auto:Gaussian-domination-combined-majorant-constant}
Let $\gamma=(k,u,a)\in\Gamma$ and assume $k\le n-1$. Let $i\in[k)$. For every $\iota\in\mathcal I_\gamma$ there exists a finite set $\mathcal B=\mathcal B_{\gamma,i,\iota}$ with
\begin{equation}\label{auto:Gaussian-domination-index-cardinality}
\#\mathcal B\le C_{\ref{Gaussian domination combined},0},
\end{equation}
and, for every $b\in\mathcal B$, an element $u_b\in[2)$, and, for every $(b,m)\in\mathcal B\times\N^2$, a sequence pair $p_{b,m}\in{\rm A}^2$, such that
\begin{equation}\label{Gaussian domination combined distance}
\Delta(p_{b,m})\le C_{\ref{Gaussian domination combined},1}(\Delta_\gamma+|\iota_0|+|m|)
\end{equation}
and, for every $j\in\Z$ and $v\in\R^2$,
\begin{equation}\label{Gaussian domination main estimate}
|(N_{\gamma,\iota})_{i,j}(v)|
\le C 2^{-|\iota_0|/2}
\sum_{m\in\N^2}2^{-|m|/2}
\sum_{b\in\mathcal B}G_{p_{b,m}(j),u_b}(v)
\end{equation}
holds with $C=C_{\ref{Gaussian domination combined},2}$.
The nonnegative series on the right converges for every $v\in\R^2$ and has finite $L^1$ norm.
\end{proposition}

The proof of this proposition is split into three cases: Propositions \ref{Gauss domination case 1}, \ref{Gauss domination case 2}, and \ref{Gauss domination case 3}, below. Their common constant is verified in Lemma \ref{Gauss domination constant}.

\begin{proposition}[Gaussian domination, near case]\label{Gauss domination case 1}\uses{standard bump properties,mean four scale Gaussian kernel,H kernel estimate Gaussian domination,two bump estimate,bump triangle,Operations on spaced sequences,H vanishing integral,Properties of distance of sequences,orthogonal decay,Gaussian domination}\usesdefs{gaussian,bracket bump,multiplicatively spaced monotone sequences,Distance of spaced sequences,closed balls in A,standard bump,geometric parameters,auto:unitary-matrices-definition,2D Gaussians,square root Gaussian difference,H multiplier,L multiplier,N multiplier}
\lean{Auto.gaussDominationCase1}
\leanok
The conclusion of Proposition \ref{Gaussian domination combined} holds in the case $\iota=(h,0)$, $h>0$, with the constant in \eqref{Gaussian domination main estimate} equal to
\begin{equation}\label{auto:Gaussian-domination-case-one-constant}
C_{\ref{Gauss domination case 1}}
=2^8\pi e^{2\pi}
\tilde C_{\ref{standard bump properties},0,2}
C_{\ref{mean four scale Gaussian kernel},2}
C_{\ref{H kernel estimate Gaussian domination}}
\max\bigl(
C_{\ref{two bump estimate},\tfrac32,\tfrac32}+4,
4C_{\ref{bump triangle},-\tfrac12,\tfrac12,\tfrac32,\tfrac32}
C_{\ref{two bump estimate},\tfrac32,\tfrac32}
\bigr).
\end{equation}
\end{proposition}

\begin{proof}[Proof]
\[ C = C_{\ref{Gauss domination case 1}} =
2^8\pi e^{2\pi}
\tilde{C}_{\ref{standard bump properties},0,2}
C_{\ref{mean four scale Gaussian kernel},2}
C_{\ref{H kernel estimate Gaussian domination}}
\max(
C_{\ref{two bump estimate},\tfrac32,\tfrac32}+4,
4C_{\ref{bump triangle},-\tfrac12,\tfrac12,\tfrac32,\tfrac32}
C_{\ref{two bump estimate},\tfrac32,\tfrac32}
). \]
Let $\iota=(h,0)\in\mathcal{I}_{\gamma}$ with $h>0$.
Let $\rho$ be the Schwartz function defined as
\[\rho=\mathcal{F}^{-1}\Big(\frac{\widehat{\Phi_{(2^{h-1}a_i^1(j+\Delta_{\gamma}))}}-\widehat{\Phi_{(2^{h}a_i^1(j+\Delta_{\gamma}))}}}{\widehat{s(2^h a_i^1(\cdot+\Delta_{\gamma}), j)}^{\nu}}\Big).\]
Let $\lambda=\lambda(j)=2^h a_i^1(j+\Delta_{\gamma})$.
Note $\lambda\in{\rm A}$ by Proposition~\ref{Operations on spaced sequences}. We will omit the argument $j$ where clear from context.
Moreover, letting $F = (N_{\gamma,\iota})_{i,j}$ and $H=(H_\gamma)_{i,j}$, we have for $v\in\mathbb{R}^2$,
\[ F(v) = \int_{\mathbb{R}} \rho(p) H(v_0-p, v_1-p)\, dp. \]
By Proposition \ref{H vanishing integral}, $H$ has mean zero in the diagonal direction.
Setting $v_0=w_0-w_1$ and $v_1=w_0+w_1$ (then $w_m = 2^{-\frac12} (W_1 v)_m$) and changing variables in the integral and subtracting zero gives
\begin{equation}\label{eqn:Gaussianbumppf1}
F(v) = \int_{\mathbb{R}} \Big(\rho(w_0+p)-\rho(w_0)\Big) H(-w_1-p, w_1-p)\,dp.
\end{equation}
By Proposition \ref{mean four scale Gaussian kernel} and Proposition \ref{standard bump properties} (each applied with $N=2$),
\begin{equation}\label{eqn:Gaussianbumppf1rhomvt}
|\rho(w_0+p)-\rho(w_0)| \leq 2\pi \tilde{C}_{\ref{standard bump properties},0,2}
C_{\ref{mean four scale Gaussian kernel},2}\min(1, \lambda^{-1}|p|) (\langle w_0+p\rangle^2_{(\lambda)}+
\langle w_0\rangle^2_{(\lambda)})
\end{equation}
By Proposition \ref{H kernel estimate Gaussian domination} we obtain $\mathcal{P}\subset B_{\mathrm{dist}}(a_i^1,\Delta_\gamma)^2\times [2)$ with $\#\mathcal{P}\le 6$
such that for all $v\in\R^2$,
\[ |H(v)| \le C_{\ref{H kernel estimate Gaussian domination}} \sum_{(t_0,t_1,u)\in\mathcal{P}} \langle (W_u v)_0\rangle_{(t_0)}^2 \langle (W_u v)_1\rangle_{(t_1)}^2, \]
where we omitted the argument $j$.
Note that by Proposition \ref{Properties of distance of sequences}, if $t\in B_{\dist}(a_i^1,\Delta_\gamma)$, then
\begin{equation}\label{eqn:Gaussianbumppf1rfac}
\tfrac{t}{\lambda} \le 2^{-h}.
\end{equation}
In the following we will use the inequality $\min(1,a)\le a^\frac12$, valid for all $a\ge 0$. For every $p,w_1\in\R$ we have
\[ \min(1, \lambda^{-1}|p|) \le \sum_{\pm} \min(1, \tfrac12\lambda^{-1}|p\pm w_1|) \le \sum_{\pm} (\lambda^{-1}|p\pm w_1|)^{\tfrac12}\]
and if $t$ satisfies \eqref{eqn:Gaussianbumppf1rfac}, then the right-hand side is
\begin{equation}\label{Gaussianbumppf1u0}
\le 2^{-\tfrac{h}{2}} \sum_{\pm} (1+ t^{-1} |p\pm w_1|)^{\tfrac12}.
\end{equation}
Altogether, the integrand in \eqref{eqn:Gaussianbumppf1} is no greater than
\[ C_1 2^{-\tfrac{h}{2}} \sum_{\alpha\in[2)} \langle w_0+\alpha p\rangle_{(\lambda)}^2 \Big( 2\sum_{(t_0,t_1,0)\in\mathcal{P}} \langle -w_1-p\rangle_{(t_0)}^{\tfrac32} \langle w_1-p\rangle_{(t_1)}^{\tfrac32} +\sum_{(t_0,t_1,1)\in\mathcal{P}} \langle p\rangle^{\tfrac32}_{(t_0)} \langle w_1\rangle_{(t_1)}^2  \Big),\]
where we have used \eqref{Gaussianbumppf1u0} for the terms with $u=0$ and $\min(1,\lambda^{-1}|p|)\le 2^{-\frac{h}{2}}(1+t_0^{-1}|p|)^{\frac12}$ for the terms with $u=1$, and
$C_1 = 2\pi \tilde{C}_{\ref{standard bump properties},0,2}
C_{\ref{mean four scale Gaussian kernel},2} C_{\ref{H kernel estimate Gaussian domination}}.$
Integrating over $p\in\R$ we obtain that $|F(v)|$ is no greater than $2C_1 2^{-\frac{h}{2}}$ times
\begin{equation}\label{eqn:Gaussianbumppf1bump}
\sum_{(t_0,t_1,1)\in\mathcal{P}}\langle w_0\rangle_{(\lambda)}^2 \langle w_1\rangle^2_{(t_1)} \int_{\R}\langle p\rangle_{(t_0)}^{\tfrac32}\,dp
\end{equation}
\begin{equation}\label{eqn:Gaussianbumppf2bump}
 +\langle w_0\rangle_{(\lambda)}^2 \sum_{(t_0,t_1,0)\in\mathcal{P}} \int_{\R} \langle w_1+p\rangle_{(t_0)}^{\tfrac32} \langle w_1-p\rangle_{(t_1)}^{\tfrac32}\,dp  + \sum_{(t_0,t_1,1)\in\mathcal{P}} \langle w_1\rangle^2_{(t_1)} \int_\R \langle w_0+p\rangle_{(\lambda)}^2 \langle p\rangle_{(t_0)}^{\tfrac32}  \, dp
\end{equation}
\begin{equation}\label{eqn:Gaussianbumppf3bump}
+ \sum_{(t_0,t_1,0)\in\mathcal{P}}\int_{\R} \langle w_0+p\rangle_{(\lambda)}^2 \langle w_1+p\rangle_{(t_0)}^{\tfrac32} \langle w_1-p\rangle_{(t_1)}^{\tfrac32}\,dp
\end{equation}

By the two bump estimate, Proposition \ref{two bump estimate}, we may bound \eqref{eqn:Gaussianbumppf1bump} and \eqref{eqn:Gaussianbumppf2bump} by
\[
C_{2} \sum_{r\in \mathcal{M}_1} \langle (W_1 v)_0\rangle_{(\lambda)}^{\tfrac32}\langle (W_1 v)_1\rangle_{(r)}^{\tfrac32},
\]
where $\mathcal{M}_1=\{\max(t_0,t_1)\,:\,(t_0,t_1,0)\in \mathcal{P}\}\cup \{ t_1 \,:\,(t_0,t_1,1)\in \mathcal{P}\}$ as multiset, and
$C_2 = 4(C_{\ref{two bump estimate},\frac32,\frac32}+ \int_{\R}\langle p\rangle^{\frac32}\,dp).$

To estimate \eqref{eqn:Gaussianbumppf3bump} we first use Proposition \ref{bump triangle} to obtain
\[ \langle w_1+p\rangle_{(t_0)}^{\tfrac32} \langle w_1-p\rangle_{(t_1)}^{\tfrac32}
\leq C_{\ref{bump triangle},-\tfrac12,\tfrac12,\tfrac32,\tfrac32}(\langle w_1\rangle_{(t_0)}^{\tfrac32} \langle w_1-p\rangle_{(t_1)}^{\tfrac32} + \langle w_1+p\rangle_{(t_0)}^{\tfrac32} \langle w_1\rangle_{(t_1)}^{\tfrac32}). \]
Together with Proposition \ref{two bump estimate} this allows us to bound \eqref{eqn:Gaussianbumppf3bump} by
\[ C_3 \sum_{(t_0,t_1,0)\in\mathcal{P}} \langle w_0+w_1\rangle_{(\lambda)}^{\tfrac32} \langle w_1\rangle_{(t_0)}^{\tfrac32} + \langle w_0-w_1\rangle_{(\lambda)}^{\tfrac32} \langle w_1\rangle_{(t_1)}^{\tfrac32},   \]
where $C_3=C_{\ref{bump triangle},-\frac12,\frac12,\frac32,\frac32} C_{\ref{two bump estimate},\frac32,\frac32}$.
By Proposition \ref{orthogonal decay} the previous is no greater than
\[
16C_3\sum_{(t_0,t_1,0)\in\mathcal P}
\langle v_0\rangle_{(t_0)}^{\tfrac32}
\langle v_1\rangle_{(\lambda)}^{\tfrac32} +
\langle v_0\rangle_{(\lambda)}^{\tfrac32}
\langle v_1\rangle_{(t_1)}^{\tfrac32} +
\langle (W_1v)_0\rangle_{(\lambda)}^{\tfrac32}
\langle (W_1v)_1\rangle_{(t_0)}^{\tfrac32}+
\langle (W_1v)_0\rangle_{(\lambda)}^{\tfrac32}
\langle (W_1v)_1\rangle_{(t_1)}^{\tfrac32}
\]
which we can write as
\[ 16C_3 \sum_{(r_0,r_1,u)\in\mathcal{M}_2}\langle (W_u v)_0\rangle_{(r_0)}^{\tfrac32}
\langle (W_u v)_1\rangle_{(r_1)}^{\tfrac32} \]
with $\mathcal{M}_2
=
\{(t_0,\lambda,0):(t_0,t_1,0)\in\mathcal{P}\}
\cup
\{(\lambda,t_1,0):(t_0,t_1,0)\in\mathcal{P}\}
\cup
\{(\lambda,t_0,1):(t_0,t_1,0)\in\mathcal{P}\}
\cup
\{(\lambda,t_1,1):(t_0,t_1,0)\in\mathcal{P}\}$ as a multiset.
To summarize, we showed that
\[ |F(v)|\le C_4 2^{-\tfrac h2} \sum_{b\in\mathcal{B}}\langle (W_{u_b} v)_0\rangle_{(p_{b,0})}^{\tfrac32}
\langle (W_{u_b} v)_1\rangle_{(p_{b,1})}^{\tfrac32},  \]
where $\mathcal{B}$ is a finite set in bijection with the multiset $\mathcal{M}_1\sqcup \mathcal{M}_2$ and $u_b=1$, $p_{b}=(\lambda, r)\in\mathrm{A}^2$ when $b$ corresponds to $r\in \mathcal{M}_1$, and $u_b=u$, $p_b=(r_0,r_1)$ when $b$ corresponds to $(r_0,r_1,u)\in\mathcal{M}_2$, and
$C_4 = 2C_1 \max(C_2,16C_3).$
Finally we use Gaussian domination, Proposition \ref{Gaussian domination}, to estimate the right hand side of the previous display by
\[ 8 C_{\ref{Gaussian domination}}^2 C_4  2^{-\tfrac h2}
\sum_{b\in\mathcal B}
\sum_{m\in\mathbb{N}^2}
2^{-\tfrac{|m|}{2}}
\g_{(2^{m_0}p_{b,0})}\big((W_{u_b}v)_0\big)
\g_{(2^{m_1}p_{b,1})}\big((W_{u_b}v)_1\big). \]
For $b\in\mathcal B$ and $m\in\N^2$, set
\[
p_{b,m}=(2^{m_0}p_{b,0},2^{m_1}p_{b,1}).
\]
The explicit multisets $\mathcal M_1$ and $\mathcal M_2$ contain at most $6$ and $24$ elements, respectively, so $\#\mathcal B\le30$. Propositions \ref{Operations on spaced sequences} and \ref{Properties of distance of sequences} give
\[
\Delta(p_{b,m})\le2(\Delta_\gamma+h+|m|).
\]
The series converges pointwise and in $L^1$ by comparison with the product of two geometric series.

\medskip
\end{proof}

\begin{proposition}[Gaussian domination, far case]\label{Gauss domination case 2}\uses{standard bump properties,four scale Gaussian kernel,H kernel derivative estimate Gaussian domination,bump triangle,two bump estimate,Operations on spaced sequences,orthogonal decay,H kernel estimate Gaussian domination,Properties of distance of sequences,Gaussian domination}\usesdefs{bracket bump,multiplicatively spaced monotone sequences,Distance of spaced sequences,closed balls in A,standard bump,geometric parameters,auto:unitary-matrices-definition,2D Gaussians,square root Gaussian difference,H multiplier,L multiplier,N multiplier}
\lean{Auto.gaussDominationCase2}
\leanok
The conclusion of Proposition \ref{Gaussian domination combined} holds in the case $\iota=(h,0)$, $h<0$, with the constant in \eqref{Gaussian domination main estimate} equal to
\begin{equation}\label{auto:Gaussian-domination-case-two-constant}
C_{\ref{Gauss domination case 2}}
=3\cdot2^7e^{2\pi}
\tilde C_{\ref{standard bump properties},0,3}
C_{\ref{four scale Gaussian kernel},3}
C_{\ref{H kernel derivative estimate Gaussian domination}}
C_{\ref{bump triangle},1,1,2,2}
C_{\ref{two bump estimate},2,2}.
\end{equation}
\end{proposition}

\begin{proof}[Proof \auto]
\[ C = C_{\ref{Gauss domination case 2}} =
3\cdot 2^7 e^{2\pi}
\tilde{C}_{\ref{standard bump properties},0,3}
C_{\ref{four scale Gaussian kernel},3}
C_{\ref{H kernel derivative estimate Gaussian domination}}
C_{\ref{bump triangle},1,1,2,2}
C_{\ref{two bump estimate},2,2} \]
Let $\iota=(h,0)\in\mathcal I_\gamma$ with $h<0$, and let
\[
\rho=\mathcal F^{-1}\Bigg(
\frac{
\widehat{\Phi_{(2^h a_i^1(j-\Delta_\gamma-1))}}
-
\widehat{\Phi_{(2^{h+1}a_i^1(j-\Delta_\gamma-1))}}
}{
\widehat{s(2^h a_i^1(\cdot-\Delta_\gamma),j)}^\nu
}
\Bigg).
\]
Set $\lambda(j)=2^{h+1}a_i^1(j-\Delta_\gamma-1)$. Then $\lambda\in{\rm A}$ by Proposition \ref{Operations on spaced sequences}. As before, write
$F=(N_{\gamma,\iota})_{i,j}$ and $H=(H_\gamma)_{i,j}$, so that
\[
F(v)=\int_\R \rho(p)H(v_0-p,v_1-p)\,dp.
\]

By Propositions \ref{four scale Gaussian kernel} and
\ref{standard bump properties}, applied with $N=3$,
\[
|\rho(p)|
\le
\tilde{C}_{\ref{standard bump properties},0,3}
C_{\ref{four scale Gaussian kernel},3}
\bigl(
\langle p\rangle_{(\lambda/2)}^3+
\langle p\rangle_{(\lambda)}^3
\bigr)
\le
3\tilde{C}_{\ref{standard bump properties},0,3}
C_{\ref{four scale Gaussian kernel},3}
\langle p\rangle_{(\lambda)}^3.
\]
The multiplier defining $\rho$ vanishes in a neighborhood of the origin, hence $\int_\R\rho=0$. Consequently,
\[
\left|\int_{-\infty}^p\rho(q)\,dq\right|
\le
\tfrac32
\tilde{C}_{\ref{standard bump properties},0,3}
C_{\ref{four scale Gaussian kernel},3}
\lambda\langle p\rangle_{(\lambda)}^2.
\]
Integration by parts gives
\[
F(v)=
\int_\R
\left(\int_{-\infty}^p\rho(q)\,dq\right)
(\partial_0+\partial_1)H(v_0-p,v_1-p)\,dp.
\]

Let $\mathcal P$ be the multiset from Proposition
\ref{H kernel derivative estimate Gaussian domination}, and omit the argument $j$ from the scales. Setting
$v_0=w_0-w_1$ and $v_1=w_0+w_1$, Propositions
\ref{bump triangle}, \ref{two bump estimate}, and
\ref{orthogonal decay} bound, for $(t_0,t_1,0)\in\mathcal P$,
\[
\int_\R
\langle p\rangle_{(\lambda)}^2
\langle v_0-p\rangle_{(t_0)}^2
\langle v_1-p\rangle_{(t_1)}^2\,dp
\]
by
\[
9C_{\ref{bump triangle},1,1,2,2}
C_{\ref{two bump estimate},2,2}
\Bigl(
\langle v_0\rangle_{(t_0)}^2
\langle v_1\rangle_{(\max(\lambda,t_1))}^2
+
\langle v_0\rangle_{(\lambda)}^2
\langle v_1\rangle_{(\max(t_0,t_1))}^2
+
\langle (W_1v)_0\rangle_{(\lambda)}^2
\langle (W_1v)_1\rangle_{(\max(t_0,t_1))}^2
\Bigr).
\]
For $(t_0,t_1,1)\in\mathcal P$, the corresponding integral is bounded by
\[
4C_{\ref{two bump estimate},2,2}
\langle (W_1v)_0\rangle_{(\max(\lambda,t_0))}^2
\langle (W_1v)_1\rangle_{(t_1)}^2.
\]

If $t\in B_{\dist}(a_i^1,\Delta_\gamma)$, then
\[
\tfrac{\lambda}{t}\le 2^h.
\]
Thus
\[
\lambda(t_0^{-1}+t_1^{-1})\le 2^{h+1}.
\]
Let $\mathcal B$ index the bracket products displayed above, after collecting repetitions, and let $u_b$ and $p_b$ denote their orientations and sequence pairs. Using the explicit multiset $\mathcal P$ constructed in the proof of Proposition
\ref{H kernel estimate Gaussian domination}, we have
$\#\mathcal B\le30$. We obtain
\[
|F(v)|
\le
27
\tilde{C}_{\ref{standard bump properties},0,3}
C_{\ref{four scale Gaussian kernel},3}
C_{\ref{H kernel derivative estimate Gaussian domination}}
C_{\ref{bump triangle},1,1,2,2}
C_{\ref{two bump estimate},2,2}
2^{h/2}
\sum_{b\in\mathcal B}
\langle (W_{u_b}v)_0\rangle_{(p_{b,0})}^2
\langle (W_{u_b}v)_1\rangle_{(p_{b,1})}^2.
\]

For $m\in\N^2$, set
\[
p_{b,m}=(2^{m_0}p_{b,0},2^{m_1}p_{b,1}).
\]
Propositions \ref{Operations on spaced sequences} and
\ref{Properties of distance of sequences} give
\[
\Delta(p_{b,m})
\le
2(\Delta_\gamma+|h|+|m|).
\]
Finally, Proposition \ref{Gaussian domination}, applied with $N=2$ to both factors, gives
\[
|F(v)|
\le
C_{\ref{Gauss domination case 2}}
2^{-\tfrac{|h|}{2}}
\sum_{m\in\N^2}2^{-\tfrac{|m|}{2}}
\sum_{b\in\mathcal B}
G_{p_{b,m}(j),u_b}(v).
\]
The series converges pointwise and has finite $L^1$ norm by comparison with the product of two geometric series.

\medskip
\end{proof}

\begin{proposition}[Gaussian domination, intermediate case]\label{Gauss domination case 3}\uses{standard bump properties,four scale Gaussian kernel,H kernel estimate Gaussian domination,bump triangle,two bump estimate,Operations on spaced sequences,orthogonal decay,Properties of distance of sequences,Gaussian domination}\usesdefs{bracket bump,multiplicatively spaced monotone sequences,standard bump,geometric parameters,auto:unitary-matrices-definition,2D Gaussians,square root Gaussian difference,H multiplier,L multiplier,N multiplier}
\lean{Auto.gaussDominationCase3}
\leanok
The conclusion of Proposition \ref{Gaussian domination combined} holds in the case $\iota=(0,l)$, $|l|\le\Delta_\gamma$, with the constant in \eqref{Gaussian domination main estimate} equal to
\begin{equation}\label{auto:Gaussian-domination-case-three-constant}
C_{\ref{Gauss domination case 3}}
=2^7e^{2\pi}
\tilde C_{\ref{standard bump properties},0,2}
C_{\ref{four scale Gaussian kernel},2}
C_{\ref{H kernel estimate Gaussian domination}}
C_{\ref{bump triangle},1,1,2,2}
C_{\ref{two bump estimate},2,2}.
\end{equation}
\end{proposition}

\begin{proof}[Proof \auto]
\[ C = C_{\ref{Gauss domination case 3}} = 2^7e^{2\pi}
\tilde{C}_{\ref{standard bump properties},0,2}
C_{\ref{four scale Gaussian kernel},2}
C_{\ref{H kernel estimate Gaussian domination}}
C_{\ref{bump triangle},1,1,2,2}
C_{\ref{two bump estimate},2,2} \]
Let $\iota=(0,l)\in\mathcal I_\gamma$ with $|l|\le\Delta_\gamma$, and let
\[
\rho=\mathcal F^{-1}\Bigg(
\frac{
\widehat{\Phi_{(a_i^1(j+l-1))}}
-
\widehat{\Phi_{(a_i^1(j+l))}}
}{
\widehat{s(a_i^1(\cdot+l),j)}^\nu
}
\Bigg).
\]
Set
\[
\lambda_-(j)=a_i^1(j+l-1),
\qquad
\lambda_+(j)=a_i^1(j+l).
\]
Then $\lambda_-,\lambda_+\in{\rm A}$ by Proposition
\ref{Operations on spaced sequences}. Writing
$F=(N_{\gamma,\iota})_{i,j}$ and $H=(H_\gamma)_{i,j}$, we have
\[
F(v)=\int_\R\rho(p)H(v_0-p,v_1-p)\,dp.
\]
By Propositions \ref{four scale Gaussian kernel} and
\ref{standard bump properties}, applied with $N=2$,
\[
|\rho(p)|
\le
\tilde{C}_{\ref{standard bump properties},0,2}
C_{\ref{four scale Gaussian kernel},2}
\bigl(
\langle p\rangle_{(\lambda_-)}^2+
\langle p\rangle_{(\lambda_+)}^2
\bigr).
\]

Let $\mathcal P$ be the multiset from Proposition
\ref{H kernel estimate Gaussian domination}, and omit the argument $j$ from the scales. Set $v_0=w_0-w_1$ and $v_1=w_0+w_1$. For each
$\lambda\in\{\lambda_-,\lambda_+\}$ and
$(t_0,t_1,0)\in\mathcal P$, Propositions
\ref{bump triangle}, \ref{two bump estimate}, and
\ref{orthogonal decay} bound
\[
\int_\R
\langle p\rangle_{(\lambda)}^2
\langle v_0-p\rangle_{(t_0)}^2
\langle v_1-p\rangle_{(t_1)}^2\,dp
\]
by
\[
9C_{\ref{bump triangle},1,1,2,2}
C_{\ref{two bump estimate},2,2}
\Bigl(
\langle v_0\rangle_{(t_0)}^2
\langle v_1\rangle_{(\max(\lambda,t_1))}^2
+
\langle v_0\rangle_{(\lambda)}^2
\langle v_1\rangle_{(\max(t_0,t_1))}^2
+
\langle (W_1v)_0\rangle_{(\lambda)}^2
\langle (W_1v)_1\rangle_{(\max(t_0,t_1))}^2
\Bigr).
\]
For $(t_0,t_1,1)\in\mathcal P$, the corresponding integral is bounded by
\[
4C_{\ref{two bump estimate},2,2}
\langle (W_1v)_0\rangle_{(\max(\lambda,t_0))}^2
\langle (W_1v)_1\rangle_{(t_1)}^2.
\]

Let $\mathcal B$ index the resulting bracket products after collecting repetitions, and let $u_b$ and $p_b$ denote their orientations and sequence pairs. From the explicit form of $\mathcal P$ in the proof of Proposition
\ref{H kernel estimate Gaussian domination}, the resulting multiset has at most $36$ elements. Hence
\[
\#\mathcal B\le36.
\]
We obtain
\[
|F(v)|
\le
9
\tilde{C}_{\ref{standard bump properties},0,2}
C_{\ref{four scale Gaussian kernel},2}
C_{\ref{H kernel estimate Gaussian domination}}
C_{\ref{bump triangle},1,1,2,2}
C_{\ref{two bump estimate},2,2}
\sum_{b\in\mathcal B}
\langle (W_{u_b}v)_0\rangle_{(p_{b,0})}^2
\langle (W_{u_b}v)_1\rangle_{(p_{b,1})}^2.
\]

For $m\in\N^2$, set
\[
p_{b,m}=(2^{m_0}p_{b,0},2^{m_1}p_{b,1}).
\]
Propositions \ref{Operations on spaced sequences} and
\ref{Properties of distance of sequences}, together with
$|l|\le\Delta_\gamma$, give
\[
\Delta(p_{b,m})
\le
2(\Delta_\gamma+|m|).
\]
Finally, Proposition \ref{Gaussian domination}, applied with $N=2$ to both factors, gives
\[
|F(v)|
\le
C_{\ref{Gauss domination case 3}}
\sum_{m\in\N^2}2^{-\tfrac{|m|}{2}}
\sum_{b\in\mathcal B}
G_{p_{b,m}(j),u_b}(v).
\]
The series converges pointwise and has finite $L^1$ norm by comparison with the product of two geometric series.

\medskip
\end{proof}

\begin{lemma}[constant $C_{\ref{Gaussian domination combined},2}$ \auto]\label{Gauss domination constant}\uses{standard bump properties,constant four scale Gaussian kernel,constant mean four scale Gaussian kernel,four scale Gaussian kernel,mean four scale Gaussian kernel,constant H kernel estimate Gaussian domination,constant H kernel derivative estimate Gaussian domination,H kernel estimate Gaussian domination,H kernel derivative estimate Gaussian domination,two bump estimate,bump triangle,Gauss domination case 1,Gauss domination case 2,Gauss domination case 3}\usesdefs{}
\lean{Auto.gaussDominationConstant}
\leanok
We have
\begin{equation}\label{auto:Gaussian-domination-case-constant-bounds}
C_{\ref{Gauss domination case 1}}<2^{117},
\end{equation}
\begin{equation}\label{auto:Gaussian-domination-case-two-constant-bound}
C_{\ref{Gauss domination case 2}}<2^{153},
\end{equation}
and
\begin{equation}\label{auto:Gaussian-domination-case-three-constant-bound}
C_{\ref{Gauss domination case 3}}<2^{110}.
\end{equation}
In particular,
\begin{equation}\label{auto:Gaussian-domination-common-constant-bound}
\max\bigl(C_{\ref{Gauss domination case 1}},C_{\ref{Gauss domination case 2}},C_{\ref{Gauss domination case 3}}\bigr)
\le C_{\ref{Gaussian domination combined},2}.
\end{equation}
Thus Proposition \ref{Gaussian domination combined} holds with the claimed constant in \eqref{Gaussian domination main estimate}.
\end{lemma}

\begin{proof}[Proof \auto]
We have
\begin{equation}\label{auto:standard-bump-small-constants-for-Gaussian-domination}
\widetilde C_{\ref{standard bump properties},0,2}=2^{18},
\end{equation}
and
\begin{equation}\label{auto:standard-bump-small-constant-three-for-Gaussian-domination}
\widetilde C_{\ref{standard bump properties},0,3}=2^{33}.
\end{equation}
Lemmas \ref{constant four scale Gaussian kernel} and \ref{constant mean four scale Gaussian kernel} give
\begin{equation}\label{auto:four-scale-small-constant-two-for-Gaussian-domination}
C_{\ref{four scale Gaussian kernel},2}<637436354528<\tfrac35 2^{40},
\end{equation}
\begin{equation}\label{auto:four-scale-small-constant-three-for-Gaussian-domination}
C_{\ref{four scale Gaussian kernel},3}<6136564156458101504<\tfrac{11}{16}2^{63},
\end{equation}
and
\begin{equation}\label{auto:mean-four-scale-small-constant-two-for-Gaussian-domination}
C_{\ref{mean four scale Gaussian kernel},2}<20397963318112<\tfrac35 2^{45}.
\end{equation}
Lemmas \ref{constant H kernel estimate Gaussian domination} and \ref{constant H kernel derivative estimate Gaussian domination} give
\begin{equation}\label{auto:H-kernel-small-constant-for-Gaussian-domination}
C_{\ref{H kernel estimate Gaussian domination}}<\tfrac{41}{64}2^{27},
\end{equation}
and
\begin{equation}\label{auto:H-kernel-derivative-small-constant-for-Gaussian-domination}
C_{\ref{H kernel derivative estimate Gaussian domination}}<\tfrac{29}{32}2^{30}.
\end{equation}
Furthermore,
\begin{equation}\label{auto:two-bump-small-constant-three-halves-for-Gaussian-domination}
C_{\ref{two bump estimate},\tfrac32,\tfrac32}=12\sqrt2<17,
\end{equation}
and
\begin{equation}\label{auto:two-bump-small-constant-two-for-Gaussian-domination}
C_{\ref{two bump estimate},2,2}=16.
\end{equation}
By Proposition \ref{bump triangle}, the bump triangle constants in Propositions \ref{Gauss domination case 1}, \ref{Gauss domination case 2}, and \ref{Gauss domination case 3} are respectively $1,4,4$. Thus the maximum in \eqref{auto:Gaussian-domination-case-one-constant} is smaller than $68$. Since $\pi<4$ and $e^{2\pi}<3025$, substitution gives
\[
C_{\ref{Gauss domination case 1}}
<2\cdot2^7\cdot4\cdot3025\cdot2^{18}
\cdot\tfrac35 2^{45}\cdot\tfrac{41}{64}2^{27}\cdot\tfrac{17}{32}2^7
<2^{117},
\]
\[
C_{\ref{Gauss domination case 2}}
<\tfrac34 2^2\cdot2^7\cdot\tfrac{13}{16}2^{13}\cdot2^{33}
\cdot\tfrac{11}{16}2^{63}\cdot\tfrac{29}{32}2^{30}\cdot2^2\cdot2^4
<2^{153},
\]
and
\[
C_{\ref{Gauss domination case 3}}
<2^7\cdot\tfrac{13}{16}2^{13}\cdot2^{18}
\cdot\tfrac35 2^{40}\cdot\tfrac{41}{64}2^{27}\cdot2^2\cdot2^4
<2^{110}.
\]
\end{proof}

\subsection{Main induction}

\newcommand{\InductPositiveTerms}[1]{{\hyperref[induct positive terms]{{\rm InductPositiveTerms}(#1)}}}
\begin{definition}[induct positive terms]\label{induct positive terms}\uses{}\usesdefs{geometric parameters,kernel sequences,sandwich kernel,s multiplier}
\lean{Auto.InductPositiveTerms}
Let $k\in \N$ with $1\le k\le n$ and $C\in [1,\infty)$.
We say that $\InductPositiveTerms{k,C}$ holds if

for all $\gamma=(k,u,a)\in \Gamma$, $i\in [k)$,
\begin{equation}\label{auto:positive-terms-square-root-bound}
    \|\M(\gamma,s_{\gamma} \otimes s_{\gamma},i)\|_{{\rm M}(k)}\le C \Delta_\gamma^{2-2^{k-n+1}}.
\end{equation}
\end{definition}

\newcommand{\VanishingDiagonal}[1]{{\hyperref[vanishing diagonal]{{\rm VanishingDiagonal}(#1)}}}
\begin{definition}[vanishing diagonal]\label{vanishing diagonal}\uses{}\usesdefs{geometric parameters,kernel sequences,sandwich kernel,H multiplier}
\lean{Auto.VanishingDiagonal}
Let $k\in \N$ with $1\le k\le n$ and $C\in [1,\infty)$.
We say that $\VanishingDiagonal{k,C}$ holds if

for all $\gamma=(k,u,a)\in \Gamma$, $i\in [k)$,
\begin{equation}\label{auto:H-positive-terms-bound}
    \|\M(\gamma,H_\gamma,i)\|_{{\rm M}(k)}
    \le C \Delta_\gamma^{2-2^{k-n+1}} .
\end{equation}
\end{definition}

\newcommand{\DiagonalBand}[1]{{\hyperref[diagonal band]{{\rm DiagonalBand}(#1)}}}
\begin{definition}[diagonal band]\label{diagonal band}\uses{}\usesdefs{geometric parameters,kernel sequences,sandwich kernel,L multiplier,summation-definition}
\lean{Auto.DiagonalBand}
Let $k\in \N$ with $1\le k\le n-1$ and $C\in [1,\infty)$.
We say that $\DiagonalBand{k,C}$ holds if

for all $\gamma=(k,u,a)\in \Gamma$, $i\in [k)$,
\begin{equation}\label{auto:L-positive-terms-sum-bound}
   \sum_{\iota\in\mathcal{I}_{\gamma}}
   \|\M(\gamma,L_{\gamma,\iota},i)\|_{{\rm M}(k)}
   \le C\Delta_\gamma^{2-2^{k-n+1}} .
\end{equation}
\end{definition}

\newcommand{\IncreaseData}[1]{{\hyperref[increase data]{{\rm IncreaseData}(#1)}}}
\begin{definition}[increase data]\label{increase data}\uses{}\usesdefs{geometric parameters,kernel sequences,2D Gaussians,L multiplier,N multiplier}
\lean{Auto.IncreaseData}
Let $k\in \N$ with $1\le k\le n-1$ and $C\in [1,\infty)$.
We say that $\IncreaseData{k,C}$ holds if

the following holds.

Let $\gamma=(k,u,a)\in \Gamma$, let $i\in [k)$, and let
$\iota\in\mathcal{I}_{\gamma}$.
For $j\in\mathbb Z$ and $y\in(\mathbb R^2)^{k+1}$ define
\begin{equation}\label{auto:increased-data-kernel}
    (\M_{i,\iota})_j(y)=
    \Big(\prod_{m\in [i)} (G_\gamma)_{m,j}(y_m)\Big)
    |(N_{\gamma,\iota})_{i,j}(y_i)|
    \Big(\prod_{m=i+1}^{k-1} (G_\gamma)_{m,j-1}(y_m)\Big)
    \bigl(\sigma_{\gamma,\iota,i,j}^{\otimes 2}\bigr)(y_{k}) .
\end{equation}
Then $\M_{i,\iota}\in {\rm M}(k+1)$ and
\begin{equation}\label{auto:increased-data-M-norm-bound}
  \|\M_{i,\iota}\|_{{\rm M}(k+1)}
  \le
  C 2^{-\tfrac{|\iota_0|}{2}}(1+|\iota_0|)^2
  \Delta_\gamma^{2-2^{k-n+2}} .
\end{equation}
\end{definition}
We will prove that for $1\le k\le n-1$ the following implications hold:
\begin{equation}\label{auto:induction-implication-chain-first} \InductPositiveTerms{k+1} \Longrightarrow \IncreaseData{k} \Longrightarrow \DiagonalBand{k} \end{equation}
\begin{equation}\label{auto:induction-implication-chain-second} \Longrightarrow \VanishingDiagonal{k} \Longrightarrow \InductPositiveTerms{k} \end{equation}
(Here we have omitted the explicit constants for clarity.)
The last implication also holds for $k=n$. After observing that $\VanishingDiagonal{n}$ holds,
Theorem \ref{induct positive terms theorem} below follows by induction on $k$.

\begin{proposition}
\label{vanishing diagonal implies induct positive terms}\uses{positive terms,telescoping terms}\usesdefs{geometric parameters,kernel sequences,sandwich kernel,auto:Gaussian-difference-kernel-definition,s multiplier,H multiplier,induct positive terms,vanishing diagonal}
\lean{Auto.vanishingDiagonal_implies_inductPositiveTerms}
\leanok
Let $k\in\N$ with $1\le k\le n$ and $C\in [1,\infty)$.
If $\VanishingDiagonal{k,C}$ holds, then $\InductPositiveTerms{k,kC+2}$ holds.
\end{proposition}

\begin{proof}
Fix $\gamma$. Let $C'=kC+2$.
By  positivity, Proposition \ref{positive terms}, it suffices to prove
\begin{equation}
    \|\sum_{i\in [k)} \M(\gamma,s_\gamma \otimes s_\gamma,i)\|_{{\rm M}(k)}\le C' \Delta_\gamma^{2-2^{k-n+1}}.
\end{equation}
Adding and subtracting $\M(\gamma, Y_\gamma, i)$ and using the triangle inequality, the left-hand side is
\[
\le \sum_{i\in [k)} \|\M(\gamma, H_\gamma,i)\|_{{\rm M}(k)} + \Big\|\sum_{i\in [k)} \M(\gamma, Y_\gamma,i)\Big\|_{{\rm M}(k)},
\]
which by $\VanishingDiagonal{k,C}$ and Proposition \ref{telescoping terms} is
$\le (kC+2) \Delta_\gamma^{2-2^{k-n+1}}$.
\end{proof}

\begin{proposition}
\label{diagonal band implies vanishing diagonal}\uses{sum L multiplier convergence-L1,prism sum le sum prism-L1,sandwich sums L1}\usesdefs{kernel sequences,sandwich kernel,H multiplier,L multiplier,summation-definition,vanishing diagonal,diagonal band}
\lean{Auto.diagonalBand_implies_vanishingDiagonal}
\leanok
Let $k\in\N$ with $1\le k\le n-1$ and $C\in [1,\infty)$.
If $\DiagonalBand{k,C}$ holds, then $\VanishingDiagonal{k,C}$ holds.
\end{proposition}

\begin{proof}
By Proposition \ref{sum L multiplier convergence-L1},
\[ \|\M(\gamma,H_\gamma,i)\|_{{\rm M}(k)} \le \sup_{N\in\mathbb N}\sum_{\substack{\iota \in \mathcal{I}_{\gamma}\\|\iota|\le N}} \|\M(\gamma,L_{\gamma,\iota},i)\|_{{\rm M}(k)}, \]
where we used Lemma \ref{prism sum le sum prism-L1}, Lemma \ref{sandwich sums L1} and that $H_\gamma=\sum_{\iota\in \mathcal{I}_{\gamma}} L_{\gamma,\iota}$ converges in $L^1$.
The claim follows from $\DiagonalBand{k,C}$.
\end{proof}

\begin{proposition}[vanishing kernel integral]\label{vanishing kernel integral}\uses{simplification 1 prism,H vanishing integral,M to K}\usesdefs{geometric parameters,kernel sequences,sandwich kernel,H multiplier,vanishing diagonal}
\lean{Auto.vanishingKernelIntegral}
\leanok
Let $\gamma=(n,w,a)\in \Gamma$.
For every $i\in [n)$,
\begin{equation}\label{auto:terminal-H-prism-form-zero}
    \|\M(\gamma,H_\gamma,i)\|_{{\rm M}(n)}=0.
\end{equation}
In particular, $\VanishingDiagonal{n,1}$ holds.
\end{proposition}

\begin{proof}
The claim follows by Proposition \ref{simplification 1 prism} if we can verify the assumption \eqref{1 prism mean zero assumption}.
Using the formula for $K_n$ in Proposition \ref{M to K}, for $z\in \R^{n+1}$,
\begin{equation}
\int_\R K_n((\M(\gamma,H_\gamma,i))_j)(z+qe_n)\,dq
\end{equation}
\begin{equation}
=\int_\R \int_{\R^{n-1}}  M((z_{l}+p_{l},p_{l})_{l\in [n-1)}, z_{n-1}+z_{n}+q-\Sigma(p),z_{n}+q-\Sigma(p))\,dp\,dq,
\end{equation}
where we abbreviated $M=(\M(\gamma,H_\gamma, i))_j$.
By Fubini's theorem, changing variables $q\mapsto q - z_n + \Sigma(p)$ and renaming $q$ to $p_{n-1}$, this is equal to
\begin{equation}
\int_{\R^{n}}  M((z_{l}+p_{l},p_{l})_{l\in [n)})
 \,dp.
\end{equation}
We claim that this integral vanishes. This follows by unpacking the definition of $M$, separating the $i$th integral and
observing by Proposition \ref{H vanishing integral} that
\begin{equation}
        \int_{\R}  (H_\gamma)_{i,j}(z_{i}+p_{i},p_{i})
     \,dp_i=0.
\end{equation}
\end{proof}

\begin{proposition}\label{increase data implies diagonal band}\uses{Gaussian domination combined,Cauchy-Schwarz at k,Cauchy-Schwarz at n-1}\usesdefs{auto:convolution-along-vector-definition,normalized function tuples,auto:prism-form-definition,geometric parameters,kernel sequences,2D Gaussians,sandwich kernel,L multiplier,N multiplier,diagonal band,increase data}
\lean{Auto.increaseData_implies_diagonalBand}
\leanok
Let $k\in\N$ with $1\le k\le n-1$ and $C\in[1,\infty)$. Suppose that $\IncreaseData{k,C}$ holds. Then:

(i) If $k<n-1$, then
\begin{equation}\label{auto:diagonal-band-induction-constant}
\DiagonalBand{k,2^{15}C_{\ref{Gaussian domination combined},2}^{1/2}C^{1/2}}
\end{equation}
holds.

(ii) If $k=n-1$, then $\DiagonalBand{n-1,2^{10}C}$ holds.
\end{proposition}
\begin{proof}
Fix $\gamma,i,\iota$. By definition,
\[
(L_{\gamma,\iota})_{i,j}
=(N_{\gamma,\iota})_{i,j}*_{(1,1)}
\sigma_{\gamma,\iota,i,j}^{*\nu}.
\]
Let
\[
\rho_j(y)=
\Big(\prod_{m\in[i)}(G_\gamma)_{m,j}(y_m)\Big)
(N_{\gamma,\iota})_{i,j}(y_i)
\Big(\prod_{m=i+1}^{k-1}(G_\gamma)_{m,j-1}(y_m)\Big).
\]
The Gaussian factors have $L^1$ norm one. Proposition \ref{Gaussian domination combined} therefore gives
\[
\|\rho_j\|_1
\le
C_{\ref{Gaussian domination combined},2}
C_{\ref{Gaussian domination combined},0}
\sum_{m\in\N^2}2^{-|m|/2}
\le2^9C_{\ref{Gaussian domination combined},2}.
\]

If $k<n-1$, Proposition \ref{Cauchy-Schwarz at k} gives
\[
\left|\Lambda_k\left(\sum_{j\in[J)}
\M(\gamma,L_{\gamma,\iota},i)_j\right)(\F)\right|
\le
2^{9/2}C_{\ref{Gaussian domination combined},2}^{1/2}J^{1/2}
|\Lambda_{k+1}(\widetilde M_\iota)(\F)|^{1/2}.
\]
If $k=n-1$, Proposition \ref{Cauchy-Schwarz at n-1} gives
\[
\left|\Lambda_{n-1}\left(\sum_{j\in[J)}
\M(\gamma,L_{\gamma,\iota},i)_j\right)(\F)\right|
\le
\sup_{\widetilde{\F}\in\mathfrak F}
|\Lambda_n(\widetilde M_\iota)(\widetilde\F)|.
\]
Here in either case
\[
\widetilde M_\iota=
\sum_{j\in[J)}(\M_{i,\iota})_j
\]
with $\M_{i,\iota}$ as in Definition \ref{increase data}. Hence $\IncreaseData{k,C}$ gives
\[
|\Lambda_{k+1}(\widetilde M_\iota)(\F)|
\le
C\max(1,J^{1-2^{k-n+2}})
2^{-|\iota_0|/2}(1+|\iota_0|)^2
\Delta_\gamma^{2-2^{k-n+2}}.
\]
It follows in the case $k<n-1$ that
\[
\left|\Lambda_k\left(\sum_{j\in[J)}
\M(\gamma,L_{\gamma,\iota},i)_j\right)(\F)\right|
\le
2^{9/2}C_{\ref{Gaussian domination combined},2}^{1/2}C^{1/2}
J^{1-2^{k-n+1}}2^{-|\iota_0|/4}(1+|\iota_0|)
\Delta_\gamma^{1-2^{k-n+1}},
\]
and in the case $k=n-1$ that
\[
\left|\Lambda_{n-1}\left(\sum_{j\in[J)}
\M(\gamma,L_{\gamma,\iota},i)_j\right)(\F)\right|
\le
C2^{-|\iota_0|/2}(1+|\iota_0|)^2.
\]
There are $2\Delta_\gamma+1\le3\Delta_\gamma$ indices with $\iota_0=0$. Moreover,
\[
3+2\sum_{h=1}^\infty2^{-h/4}(1+h)\le2^{10},
\qquad
3+2\sum_{h=1}^\infty2^{-h/2}(1+h)^2\le2^{10}.
\]
Summing in $\iota$ proves (i), since $2^{9/2}2^{10}\le2^{15}$, and proves (ii).
\end{proof}

\begin{proposition}
\label{induct positive terms imply increase data}\uses{Gaussian domination combined,Monotonicity K}\usesdefs{multiplicatively spaced monotone sequences,geometric parameters,kernel sequences,sandwich kernel,s multiplier,L multiplier,N multiplier,induct positive terms,increase data}
\lean{Auto.inductPositiveTerms_implies_increaseData}
\leanok
Let $k\in\N$ with $1\le k\le n-1$ and $C\in [1,\infty)$.
Assume that $\InductPositiveTerms{k+1,C}$ holds. Then $\IncreaseData{k,C\cdot C_{\ref{induct positive terms imply increase data}}}$ holds, where
\begin{equation}\label{auto:increase-data-induction-constant-definition}C_{\ref{induct positive terms imply increase data}} = 2^{10} C_{\ref{Gaussian domination combined},0} (1+C_{\ref{Gaussian domination combined},1})^2 C_{\ref{Gaussian domination combined},2}.\end{equation}
\end{proposition}

\begin{proof}
Estimating the functions $|N_{\gamma,\iota}|$ by Gaussian domination using Propositions \ref{Gaussian domination combined}, Proposition \ref{Monotonicity K} and applying the dominated convergence theorem,
\begin{equation}\label{eqn:increasedatapf1}
\|\M_{i,\iota}\|_{{\rm M}(k+1)} \le C_{\ref{Gaussian domination combined},2} 2^{-|\iota_0|/2} \sum_{m\in\mathbb{N}^2} 2^{-|m|/2} \sum_{b\in\mathcal{B}_{i,\gamma,\iota}} \|\M(\widetilde{\gamma}_{b,m}, s_{\widetilde{\gamma}_{b,m}}^{\otimes 2}  ,k)\|_{{\rm M}(k+1)},
\end{equation}
and for each $(b,m)\in\mathcal{B}\times \mathbb{N}^2$ the new data $\widetilde{\gamma}_{b,m}$ is $(k+1,\widetilde{u}_{b},\widetilde{a}_{b,m})\in \Gamma$ where
\[
\widetilde u_b(i')=u(i')\qquad(i'\in[i)),
\]
\[
\widetilde u_b(i)=u_b,
\]
\[
\widetilde u_b(i')=u(i')\qquad(i<i'<k),
\]
\[
\widetilde u_b(k)=0,
\]
and
\[
(\widetilde a_{b,m})_{i'}=a_{i'}\qquad(i'\in[i)),
\]
\[
(\widetilde a_{b,m})_i=p_{b,m},
\]
\[
(\widetilde a_{b,m})_{i'}=a_{i'}(\cdot-1)\qquad(i<i'<k),
\]
\[
(\widetilde a_{b,m})_k=\alpha_{\iota,i}.
\]
Here $u_b\in [2), p_{b,m}\in {\rm A}^2$ is produced by Proposition \ref{Gaussian domination combined} and
$\alpha_{\iota,i}\in {\rm A}^2$ can be inferred. To be explicit,
$\alpha_{(h,0),i}(j) = (2^{h-1/2} a_i^1(j+\Delta_{\gamma}),2^{h-1/2} a_i^1(j+\Delta_{\gamma}))$ if $h>0$,
$\alpha_{(h,0),i}(j) = (2^{h-1/2} a_i^1(j-\Delta_{\gamma}),2^{h-1/2} a_i^1(j-\Delta_{\gamma}))$ if $h<0$ and
$\alpha_{(0,l),i}(j) = (2^{-1/2} a_i^1(j+l),2^{-1/2} a_i^1(j+l))$ if $|l|\le \Delta_{\gamma}$.
In particular, $\Delta(\alpha_{\iota,i})=0$ and
\[ \Delta_{\widetilde{\gamma}_{b,m}}=\Delta_\gamma-\Delta(a_i)+\Delta(p_{b,m}) \le  (1+C_{\ref{Gaussian domination combined},1})(\Delta_\gamma+|\iota_0|+|m|). \]
Also note $(s_{\widetilde{\gamma}_{b,m}})_{k,j}=\sigma_{\gamma,\iota,i,j}$.
We may apply $\InductPositiveTerms{k+1,C}$ and sum over $b\in \mathcal{B}_{i,\gamma,\iota}$ to obtain for each $\iota\in\mathcal{I}_{\gamma}$:
\[ \|\M_{i,\iota}\|_{{\rm M}(k+1)} \le C\; C_{\ref{Gaussian domination combined},0} (1+C_{\ref{Gaussian domination combined},1})^2 C_{\ref{Gaussian domination combined},2} 2^{-\tfrac{|\iota_0|}{2}} \sum_{m\in\mathbb{N}^2} 2^{-\tfrac{|m|}{2}} (\Delta_\gamma + |\iota_0|+|m|)^{2-2^{k-n+2}}. \]
Summing over $m$ we obtain the claim.
Indeed, estimate
\[ \sum_{m\in\mathbb{N}^2} 2^{-\tfrac{|m|}{2}} (\Delta_\gamma + |\iota_0|+|m|)^{2-2^{k-n+2}} \le \Delta_\gamma^{2-2^{k-n+2}} (1+|\iota_0|)^{2-2^{k-n+2}} \sum_{m\in\N^2} 2^{-\tfrac{|m|}{2}} (1+|m|)^{2-2^{k-n+2}} \]
\[
\leq \Delta_\gamma^{2-2^{k-n+2}} (1+|\iota_0|)^{2} \sum_{m\in\N^2} 2^{-\tfrac{|m|}{2}} (1+|m|)^{2}
\]
and $\sum_{m\in\N^2} 2^{-\frac{|m|}{2}} (1+|m|)^2\le 2^{10}$.
\end{proof}

\begin{lemma}[constant $C_{\ref{induct positive terms imply increase data}}$ \auto]\label{constant induct positive terms imply increase data}\uses{Gaussian domination combined,induct positive terms imply increase data}\usesdefs{}
\lean{Auto.constantInductPositiveTermsImplyIncreaseData}
\leanok
\begin{equation}\label{constant induct positive terms imply increase data bound}
C_{\ref{induct positive terms imply increase data}}<2^{172}.
\end{equation}
\end{lemma}
\begin{proof}[Proof \auto]
Using the defining product in Proposition \ref{induct positive terms imply increase data} and Proposition \ref{Gaussian domination combined}, in particular \eqref{Gaussian domination combined constants}, we obtain
\[
2^{10}\cdot36\cdot9\cdot2^{153}
=324\cdot2^{163}<2^{172}.
\]
\end{proof}

\begin{proposition}\label{P:C_k-induction}\uses{vanishing diagonal implies induct positive terms,vanishing kernel integral,induct positive terms imply increase data,increase data implies diagonal band,diagonal band implies vanishing diagonal}\usesdefs{induct positive terms}
\lean{Auto.inductPositiveTermsByInduction}
\leanok
For every $k\in\N$ with $1\le k\le n$, $\InductPositiveTerms{k,C_k}$ holds, where
\begin{equation}\label{auto:positive-induction-recursion}
C_n=2+n,
\end{equation}
\begin{equation}\label{auto:positive-induction-recursion-terminal-step}
C_{n-1}=2+(n-1)2^{10}C_{\ref{induct positive terms imply increase data}}C_n,
\end{equation}
and, for $1\le k\le n-2$,
\begin{equation}\label{auto:positive-induction-recursion-interior-step}
C_k=2+k2^{15}C_{\ref{Gaussian domination combined},2}^{1/2}
(C_{\ref{induct positive terms imply increase data}}C_{k+1})^{1/2}.
\end{equation}
\end{proposition}

\begin{proof}
This is proved by (reverse) induction on $k$ with the base case being $k=n$.
The claim holds for $k=n$ by Proposition \ref{vanishing diagonal implies induct positive terms} with $k=n$ and Proposition \ref{vanishing kernel integral}.
Next assume $1\le k\le n-1$ and assume that $\InductPositiveTerms{k+1,C_{k+1}}$ holds.
Then $\InductPositiveTerms{k,C_k}$
follows by applying Proposition \ref{induct positive terms imply increase data}, Proposition \ref{increase data implies diagonal band}, Proposition \ref{diagonal band implies vanishing diagonal} and finally Proposition \ref{vanishing diagonal implies induct positive terms}.
\end{proof}

\begin{lemma}[constant $C_k$ \auto]\label{P:better-induction}\uses{P:C_k-induction,Gaussian domination combined,induct positive terms imply increase data}\usesdefs{induct positive terms}
\lean{Auto.betterInduction}
\leanok
For every $k\in\N$ with $1\le k\le n-1$,
\begin{equation}\label{auto:positive-induction-local-bound}
\InductPositiveTerms{k,
(18\cdot2^{173}(k+2)+2^{1/2})^2}
\end{equation}
holds.
\end{lemma}
\begin{proof}[Proof \auto]
Let $C_k$ be the constant from Proposition \ref{P:C_k-induction}. It suffices to prove
\begin{equation}\label{E:induction-statement}
C_k\le
\bigl(18\cdot2^{173}(k+2)+2^{1/2}\bigr)^2.
\end{equation}
We argue by reverse induction. For $k=n-1$, Proposition \ref{P:C_k-induction} and $C_n=n+2$ give
\[
C_{n-1}
=2+(n-1)(n+2)2^{10}C_{\ref{induct positive terms imply increase data}},
\]
which is bounded by the right-hand side of \eqref{E:induction-statement}.

Assume $1\le k\le n-2$ and that \eqref{E:induction-statement} holds with $k+1$ in place of $k$. Since
\[
2^{15}(C_{\ref{Gaussian domination combined},2}C_{\ref{induct positive terms imply increase data}})^{1/2}
=18\cdot2^{173},
\]
Proposition \ref{P:C_k-induction} gives
\[
C_k
\le
2+k\cdot18\cdot2^{173}
\bigl(18\cdot2^{173}(k+3)+2^{1/2}\bigr).
\]
This is bounded by the right-hand side of \eqref{E:induction-statement}, which proves the claim.
\end{proof}

\begin{rem}
We only need to use the previous proposition for $k=2$. The result then holds with a constant independent of $n$.
\end{rem}

\begin{theorem}\label{induct positive terms theorem}\uses{P:better-induction,P:C_k-induction}\usesdefs{induct positive terms}
\lean{Auto.inductPositiveTermsTheorem}
\leanok
$\InductPositiveTerms{2,C_{\ref{induct positive terms theorem}}}$ holds, where
\begin{equation}\label{auto:positive-induction-constant-definition}
C_{\ref{induct positive terms theorem}}
=
\bigl(2^{17}(C_{\ref{Gaussian domination combined},2}C_{\ref{induct positive terms imply increase data}})^{1/2}+2^{1/2}\bigr)^2.
\end{equation}
\end{theorem}

\begin{proof}
This follows from Proposition~\ref{P:better-induction} if $n\ge 3$ and from Proposition~\ref{P:C_k-induction} if $n=2$.
\end{proof}

\begin{lemma}[constant $C_{\ref{induct positive terms theorem}}$ \auto]\label{constant induct positive terms theorem}\uses{induct positive terms imply increase data,Gaussian domination combined,induct positive terms theorem}\usesdefs{}
\lean{Auto.constantInductPositiveTermsTheorem}
\leanok
\begin{equation}\label{constant induct positive terms theorem bound}
C_{\ref{induct positive terms theorem}}<2^{359}.
\end{equation}
\end{lemma}
\begin{proof}[Proof \auto]
Since
\[
C_{\ref{induct positive terms imply increase data}}
=324\cdot2^{163}
\]
and $C_{\ref{Gaussian domination combined},2}=2^{153}$, the first summand in \eqref{auto:positive-induction-constant-definition} is
\[
18\cdot2^{175}.
\]
Its square is $324\cdot2^{350}$. The remaining two terms in the expanded square are smaller than $2^{182}$. Therefore
\[
C_{\ref{induct positive terms theorem}}
<2^{359}.
\]
\end{proof}

\section{Reduction to the main argument}\label{sec:reduction}
In this section we prove Theorem \ref{thm:nct main real} from Theorem \ref{induct positive terms theorem}.

\subsection{Further preliminaries for the reduction}\label{sec:prelim red}

\subsubsection{\texorpdfstring{$A$ to $\Lambda_1$}{Average to prism form}}

\begin{lemma}[$A$ to $\Lambda_1$]\label{A to Lambda}\uses{}\usesdefs{A_def,auto:prism-form-definition}
\lean{Auto.aToLambda}
\leanok
Let $\phi\in \mathcal{S}(\R)$.
Let $\mathbf{f}=(f_0,\dots,f_{n-1})$ be functions in $\mathcal{S}(\R^n)$.
Then
\begin{equation}\label{auto:twisted-average-prism-identity}
\|A(\phi, {\mathbf f})\|_2^2 = \Lambda_1(\phi^{\otimes 2})(\mathbf{F}),
\end{equation}
where for every $i\in[n)$ and $x\in\R^n$,
\begin{equation}\label{auto:twisted-average-coordinate-transform}
F_i(x)=f_i(-x_{[1,i+1)},\Sigma(x),-x_{[i+1,n)}).
\end{equation}
Moreover, $F_i\in\mathcal S(\R^n)$ and
\begin{equation}\label{auto:twisted-average-coordinate-transform-isometry}
\|F_i\|_p=\|f_i\|_p
\end{equation}
for every $p\in[1,\infty]$.
\end{lemma}

\begin{proof}
The linear map in the definition of $F_i$ has determinant of modulus one. Hence $F_i\in\mathcal S(\R^n)$ and $\|F_i\|_p=\|f_i\|_p$ for every $p\in[1,\infty]$.

Expanding the $L^2$ norm and the square inside the $L^2$ integral, the left hand side equals
\[
 \int_{\R^n} \int_{\mathbb{R}^2}
 \Big(\prod_{(i,h)\in [n)\times [2)} f_i(u+y_h e_i) \Big)
 (\phi^{\otimes 2})(y)\, dy\, du .
\]
Changing variables $u=-x$, this becomes
\[
 \int_{\R^n} \int_{\mathbb{R}^2}
 \Big(\prod_{(i,h)\in [n)\times [2)} f_i(-x+y_h e_i) \Big)
 (\phi^{\otimes 2})(y)\, dy\, dx .
\]
Now change variables $y_h\mapsto y_h+\Sigma(x)$ for $h\in[2)$. We obtain
\[
 \int_{\R^n} \int_{\mathbb{R}^2}
 \Big(\prod_{(i,h)\in [n)\times [2)}
 f_i(-x+(y_h+\Sigma(x))e_i) \Big)
 \phi(y_0+\Sigma(x))\phi(y_1+\Sigma(x))\, dy\, dx .
\]
By the definition of $F_i$, for every $i\in[n)$ and $h\in[2)$,
\[
 F_i(y_h,x_{[n)\setminus i})
 =
 f_i\bigl(-x_{[0,i)},y_h+\Sigma(x)-x_i,-x_{[i+1,n)}\bigr)
 =
 f_i(-x+(y_h+\Sigma(x))e_i).
\]
Hence the last display is equal to
\[
 \int_{\R^2}\int_{\R^n}
 (\phi^{\otimes 2})(y_1+\Sigma(x),y_0+\Sigma(x))
 \prod_{h\in[2),\,i\in[n)}
 F_i(y_h,x_{[n)\setminus i})\,dx\,dy ,
\]
which is $\Lambda_1(\phi^{\otimes 2})(\mathbf{F})$.
\end{proof}

\subsubsection{Variation seminorms}\label{sec:variation-norm}

Let $B$ be a seminormed space, $I\subset\mathbb{R}$ a subset and $a:I\to B$ and $J\in\mathbb{N}$, $r\ge 1$. The \emph{$r$-variation seminorm} (with at most $J$ jumps) is
\begin{equation}\label{auto:finite-variation-seminorm} \|a\|_{V_{r,J}(I; B)} = \sup_{0\le M\le J}\ \sup_{\substack{t_0<\cdots<t_M,\\t_j\in I\;\text{for all}\;j\in [M+1)}} \Big(\sum_{j\in[M)} \|a(t_{j+1})- a(t_{j})\|^r\Big)^{1/r}. \end{equation}

Also define
\[
\|a\|_{V_r(I; B)}=\sup_{J\ge 1} \|a\|_{V_{r,J}(I; B)}.
\]

In particular, when $I$ is omitted we understand that it equals  $\mathbb{N}_{\ge 1}$, the set of positive integers:

\begin{definition}[$r$-variation seminorms for sequences on the positive integers]\label{variation seminorm on positive integers}\lean{nCT.variationSeminorm}
Define
$$\|a\|_{V_{r}(B)} = \sup_{J\in\mathbb{N}}\ \sup_{\substack{t_0<\cdots<t_J,\\
t_j\in \mathbb{N}_{\ge 1}\;
\text{for all}\;j\in [J+1)}} \Big(\sum_{j\in[J)} \|a(t_{j+1})- a(t_{j})\|^r\Big)^{1/r}.$$
Note we will only use this for $r\ge 1$.
\end{definition}

The $k$-th short variation can be expressed as
$\|a\|_{V_{r,J}([2^k, 2^{k+1}]; B)}$
and long variations as
$\|a\|_{V_{r,J}(2^\mathbb{Z}; B)}.$
For more on variation (semi)norms see \cite{JonesSeegerWright08}.

\begin{lemma}\label{lem:shortlongjumps}\uses{}\usesdefs{}
\lean{Auto.shortlongJumps}
\leanok
For every increasing sequence of positive real numbers $(t_j)_{j\in[J+1)}$,
\begin{equation}\label{auto:short-long-variation-decomposition}
\Big(\sum_{j\in[J)}\|a(t_{j+1})-a(t_j)\|^r\Big)^{1/r}
\le
2\Big(\sum_{k\in\kappa}\|a\|_{V_{r,J}([2^k,2^{k+1}];B)}^r\Big)^{1/r}
+\|a\|_{V_{r,J}(2^\Z;B)},
\end{equation}
where $\kappa$ is the set of $k\in\Z$ such that $2^k\le t_j<2^{k+1}$ for some $j\in[J+1)$. Consequently,
\begin{equation}\label{auto:global-short-long-variation-bound}
\|a\|_{V_{r,J}((0,\infty);B)}
\le
2\Big(\sum_{k\in\Z}\|a\|_{V_{r,J}([2^k,2^{k+1}];B)}^r\Big)^{1/r}
+\|a\|_{V_{r,J}(2^\Z;B)}.
\end{equation}
\end{lemma}
\begin{proof}[Proof \auto]
Group the points $t_j$ according to the dyadic intervals $[2^k,2^{k+1})$. The jumps whose two endpoints lie in the same interval are kept unchanged. If a jump goes from the last point in $[2^k,2^{k+1})$ to the first point in a later occupied interval $[2^l,2^{l+1})$, write it as the sum of
\[
a(t_{j+1})-a(2^l),
\qquad
a(2^l)-a(2^k),
\qquad
a(2^k)-a(t_j).
\]
The internal jumps in an occupied interval, together with the first of these boundary terms when it occurs, form an increasing sequence in $[2^k,2^{k+1}]$ with at most $J$ jumps. Their combined $\ell^r$ norm, summed over the occupied intervals, is therefore at most
\[
\Big(\sum_{k\in\kappa}\|a\|_{V_{r,J}([2^k,2^{k+1}];B)}^r\Big)^{1/r}.
\]
The third boundary terms are bounded by the same expression. The middle terms use the lower endpoints $2^k$, and hence form an increasing sequence on $2^\Z$ after repeated endpoints are deleted, so their $\ell^r$ norm is at most $\|a\|_{V_{r,J}(2^\Z;B)}$. Minkowski's inequality gives the first estimate. Taking the supremum over the sequence $(t_j)$ gives the second.
\end{proof}

\begin{lemma}\label{lem:ftccs-R}\uses{}\usesdefs{}   Let $0<\alpha < \beta$ and $a:[\alpha,\beta]\to \R$ be continuously differentiable.
\lean{Auto.ftcCsR}
\leanok
    Then
\begin{equation}\label{ftccs1-R}
\|a\|_{V_{2,J}([\alpha, \beta]; \R)} \le (\tfrac{\beta-\alpha}{\alpha})^{1/2} \|ta'(t)\|_{L^2(t\in [\alpha, \beta],\tfrac{dt}{t})}, \end{equation}
\begin{equation}\label{ftccs2-R}
\|a\|_{V_{2,J}([\alpha, \beta]; \R)}^2 \le 8 \|a\|_{L^2([\alpha, \beta],\tfrac{dt}{t})}\|ta'(t)\|_{L^2(t\in [\alpha, \beta],\tfrac{dt}{t})}.
\end{equation}
\end{lemma}
\begin{proof}[Proof \auto]
Let $\alpha\le t_0<\cdots<t_J\le\beta$. For every $j\in[J)$, the fundamental theorem of calculus and Cauchy--Schwarz with respect to Lebesgue measure give
\[
|a(t_{j+1})-a(t_j)|^2
\le
(t_{j+1}-t_j)\int_{t_j}^{t_{j+1}}|a'(t)|^2\,dt
\]
\[
\le
\tfrac{\beta-\alpha}{\alpha}
\int_{t_j}^{t_{j+1}}|ta'(t)|^2\,\tfrac{dt}{t}.
\]
Summing in $j$ and taking the supremum proves \eqref{ftccs1-R}.

For \eqref{ftccs2-R}, first let $b$ be a nonnegative absolutely continuous function. Then
\[
|b(t_{j+1})-b(t_j)|^2
\le
|b(t_{j+1})^2-b(t_j)^2|
\]
\[
\le
2\int_{t_j}^{t_{j+1}}|b(t)||tb'(t)|\,\tfrac{dt}{t}
\le
2\|b\|_{L^2([t_j,t_{j+1}],dt/t)}
\|tb'\|_{L^2([t_j,t_{j+1}],dt/t)}.
\]
Summing and applying Cauchy--Schwarz in $j$ gives
\[
\|b\|_{V_{2,J}([\alpha,\beta];\R)}^2
\le
2\|b\|_{L^2([\alpha,\beta],dt/t)}
\|tb'\|_{L^2([\alpha,\beta],dt/t)}.
\]
Apply this to $a_+=\max(a,0)$ and $a_-=-\min(a,0)$. Their derivatives have modulus at most $|a'|$ almost everywhere. The triangle inequality for the variation seminorm therefore gives
\[
\|a\|_{V_{2,J}([\alpha,\beta];\R)}
\le
2^{3/2}\|a\|_{L^2([\alpha,\beta],dt/t)}^{1/2}
\|ta'\|_{L^2([\alpha,\beta],dt/t)}^{1/2}.
\]
Squaring proves \eqref{ftccs2-R}.
\end{proof}

\subsubsection{Bump functions}\label{sec:bump red}

\begin{lemma}\label{lem:ft_deriv_mul}\uses{}\usesdefs{}
\lean{Auto.fourierDerivativeMul}
\leanok
For a Schwartz function $\phi:\mathbb{R}\to\mathbb{C}$ let
$T\phi(x)=\frac{d}{dx} (x \phi(x))|_x$. Then
for every $m\ge 0$,
\begin{equation}\label{auto:Fourier-derivatives-of-T}(\widehat{T\phi})^{(m)}(\xi) = -(m (\widehat{\phi})^{(m)}(\xi) + \xi (\widehat{\phi})^{(m+1)}(\xi)).\end{equation}
\end{lemma}

\begin{proof}
We proceed by induction on $m$.
For $m=0$ this follows from the rule for taking the Fourier transform of a derivative and vice versa:
$\widehat{T\phi}(\xi) = 2\pi i \xi\,\widehat{\Phi}(\xi)$,
where $\Phi(x)=x \phi(x)$ and $\widehat{\Phi}(\xi) = -(2\pi i)^{-1} \frac{d}{d\xi} \widehat{\phi}(\xi)$.
Suppose the claim holds for $m$. Then by inductive hypothesis and the product rule,
\[ (\widehat{T\phi})^{(m+1)}(\xi) = \tfrac{d}{d\xi} (-(m (\widehat{\phi})^{(m)}(\xi) + \xi (\widehat{\phi})^{(m+1)}(\xi))) \]
\[ = -((m+1) (\widehat{\phi})^{(m+1)}(\xi) + \xi (\widehat{\phi})^{(m+2)}(\xi)). \]
\end{proof}

\begin{lemma} \label{lem:widebump}\uses{}\usesdefs{bracket bump}
\lean{Auto.wideBump}
\leanok
Let $k \leq 2$, $t\in [2^{1-k}, 2^{2-k}]$. Then for every $u\in \R$ and $N\ge 0$,
\begin{equation}\label{auto:wide-bump-translation-bound}\langle u-t\rangle^{N} \le C_{\ref{lem:widebump},N} 2^{-kN}\langle u\rangle^{N},\end{equation}
where $C_{\ref{lem:widebump},N}=2^{3N}$.
\end{lemma}
\begin{proof}
The triangle inequality gives
\[
1+|u|\le1+|u-t|+t\le(1+|u-t|)(1+t).
\]
Since $k\le2$ and $t\le2^{2-k}$,
\[
1+t\le2^{3-k}.
\]
Taking reciprocal $N$th powers proves the claim.
\end{proof}

\begin{lemma}\label{lem:thetat_offcenter}\uses{lem:widebump}\usesdefs{bracket bump}
\lean{Auto.thetaTOffcenter}
\leanok
Let $k\le-1$. Then for every $v\in\R^2$,
\begin{equation}\label{E:two-terms}
2^{k/2}\int_1^2
\langle v_0-t2^{-k}\rangle^2
\langle v_1-t2^{-k}\rangle^2\,dt
\le C_{\ref{lem:thetat_offcenter}}
\left(
\langle v_0\rangle^{3/2}\langle v_1\rangle^{3/2}
+
\langle v_0+v_1\rangle^{3/2}\langle v_0-v_1\rangle^{3/2}
\right),
\end{equation}
where
\begin{equation}\label{off center bump constant}
C_{\ref{lem:thetat_offcenter}}=133.
\end{equation}
\end{lemma}
\begin{proof}[Proof \auto]
Put $R=2^{-k}$. We distinguish three cases.
Suppose first that $|v_0|,|v_1|\ge3R$. For $t\in[1,2]$ and $i\in[2)$,
\[
1+|v_i|\le3(1+|v_i-tR|).
\]
Since $tR\in[2^{-k},2^{1-k}]$, Lemma \ref{lem:widebump}, applied with $k+1$ and $N=1/2$, gives
\[
R^{-1/2}\langle v_0-tR\rangle^{1/2}
\le2\langle v_0\rangle^{1/2}.
\]
Consequently the integrand, including the factor $R^{-1/2}$, is at most
\[
2\cdot3^{3/2}\cdot9
\langle v_0\rangle^2\langle v_1\rangle^2
<94\langle v_0\rangle^{3/2}\langle v_1\rangle^{3/2}.
\]
Integration over $t\in[1,2]$ proves the required bound in this case.

Suppose next that exactly one of $|v_0|,|v_1|$ is smaller than $3R$; by symmetry, assume that $|v_0|<3R$. As above,
\[
\langle v_1-tR\rangle^2\le9\langle v_1\rangle^{3/2}.
\]
Moreover,
\[
\int_1^2\langle v_0-tR\rangle^2\,dt
=R^{-1}\int_R^{2R}\langle v_0-q\rangle^2\,dq
\le2R^{-1}.
\]
Since $1+|v_0|\le\frac72R$, the left-hand side of \eqref{E:two-terms} is at most
\[
18R^{-3/2}\langle v_1\rangle^{3/2}
\le18\left(\tfrac72\right)^{3/2}
\langle v_0\rangle^{3/2}\langle v_1\rangle^{3/2}
<118\langle v_0\rangle^{3/2}\langle v_1\rangle^{3/2}.
\]

It remains to consider $|v_0|,|v_1|<3R$. Put $w=v_0-v_1$. The triangle inequality and $(a+b)^2\le2(a^2+b^2)$ give
\[
\langle v_0-q\rangle^2\langle v_1-q\rangle^2
\le2\langle w\rangle^2
\bigl(\langle v_0-q\rangle^2+\langle v_1-q\rangle^2\bigr).
\]
Hence
\[
\int_R^{2R}\langle v_0-q\rangle^2\langle v_1-q\rangle^2\,dq
\le8\langle w\rangle^2.
\]
Since $1+|v_0+v_1|\le\frac{13}{2}R$, the left-hand side of \eqref{E:two-terms} is at most
\[
8R^{-3/2}\langle w\rangle^2
\le8\left(\tfrac{13}{2}\right)^{3/2}
\langle v_0+v_1\rangle^{3/2}\langle w\rangle^{3/2}
<133\langle v_0+v_1\rangle^{3/2}\langle w\rangle^{3/2}.
\]
This proves \eqref{E:two-terms}.
\end{proof}

\begin{lemma}[constant $C_{\ref{lem:thetat_offcenter}}$ \auto]\label{constant off center bump}\uses{lem:thetat_offcenter}\usesdefs{}
\lean{Auto.constantOffCenterBump}
\leanok
\begin{equation}\label{constant off center bump bound}
C_{\ref{lem:thetat_offcenter}}\le133.
\end{equation}
\end{lemma}
\begin{proof}[Proof \auto]
This is the defining value in Lemma \ref{lem:thetat_offcenter}.
\end{proof}

For a function $\phi$ we define
\begin{equation}\label{auto:multiplication-operator-X}(X\phi)(u) = u\phi(u).\end{equation}

   \begin{lemma}
         \label{lem:int_fct}\uses{}\usesdefs{}
\lean{Auto.integralFct}
\leanok
          Let $\psi:\R\to \R$ be a Schwartz function and assume that
          \begin{equation}
              \label{int_fct_assumption}
              \textup{supp}(\widehat{\psi})
\subset  [-1,-2^{-2}]\cup [2^{-2},1]
          \end{equation}
     Let $\Psi:\R^2\to \R$ be defined as
     \begin{equation}\label{auto:integrated-tensor-kernel}\Psi(u,v) = \int_1^2 \psi_{(t)}(u) \psi_{(t)}(v) \tfrac{dt}{t}\end{equation}
     and let $\ell\in\Z$.
   Then
        \begin{enumerate}
    \item  \label{int_fct_symm} For all $(u,v)\in \R^2$,  $\Psi_{(2^\ell)}(u,v) = \Psi_{(2^\ell)}(v,u)$
     \item \label{int_fct_pos}  For all bounded functions $f:\R\to\R$,
     \begin{equation}\label{auto:integrated-tensor-positivity}
         \int_{\R^2} f(u){f(v)} \Psi_{(2^\ell)}(u,v)\,dudv
         \ge 0
     \end{equation}
     \item  \label{int_fct_supp}
     $         \textup{supp}(\widehat{\Psi_{(2^\ell)}})
\subset  ([-2^{-\ell},-2^{-3-\ell}]\cup [2^{-3-\ell},2^{-\ell}])^2 \subset \textup{Ann}_1(2^{-\ell},2^3)^2$
     \end{enumerate}
     \end{lemma}
          \begin{proof}
The symmetry \eqref{int_fct_symm}
 is immediate from the definition.
 The claim \eqref{int_fct_pos} follows from
          \[
\int_{\R^2}f(u)f(v)\Psi_{(2^\ell)}(u,v)\,du\,dv
=
\int_1^2\left(\int_\R f(u)\psi_{(2^\ell t)}(u)\,du\right)^2\tfrac{dt}{t}
\ge0.
\]
To see \eqref{int_fct_supp}, we use the assumption \eqref{int_fct_assumption} and $1\le t \le 2$. Thus, if  $\widehat{\psi_{(t)}}(\xi) = \widehat{\psi}(t\xi)$ is non-zero, then necessarily $2^{-2}t^{-1}\le |\xi|\le t^{-1} $, so $2^{-3}\le |\xi| \le 1 $ and thus
\[         \textup{supp}(\widehat{\psi_{(t)}})
\subset  [-1,-2^{-3}]\cup [2^{-3},1] \subset \textup{Ann}_1(1,2^3) \]
By rescaling,
\[         \textup{supp}(\widehat{\psi_{(2^\ell t)}})
\subset  [-2^{-\ell},-2^{-3-\ell}]\cup [2^{-3-\ell},2^{-\ell}] \subset \textup{Ann}_1(2^{-\ell},2^3) \]
Then the bound \eqref{int_fct_supp}
  follows.
     \end{proof}

\begin{lemma}\label{lem:Phipos_v2}\uses{}\usesdefs{}
\lean{Auto.phiPosV2}
\leanok
Let
$\Psi:\mathbb{R}^2\to \mathbb{R}$ be such   that $\Psi(u_0,u_1)=\Psi(u_1,u_0)$ for all $(u_0,u_1)\in \R^2$ and such that
for all bounded functions $g:\R\to\R$,
\begin{equation}\label{nonneg_assumption_v2}
    \int_{\R^2} g(u_0){g(u_1)} \Psi(u)\,du \ge 0
\end{equation}
Then for all $\xi\in \R$, \begin{equation}\label{auto:positive-kernel-diagonal-Fourier}\widehat{\Psi}(\xi,-\xi)\geq 0.\end{equation}
\end{lemma}
\begin{proof}[Proof \auto]
By symmetry,
\[
\widehat\Psi(\xi,-\xi)
=
\int_{\R^2}\Psi(u)
\cos(2\pi\xi(u_1-u_0))\,du.
\]
Since
\[
\cos(2\pi\xi(u_1-u_0))
=
\cos(2\pi\xi u_0)\cos(2\pi\xi u_1)
+
\sin(2\pi\xi u_0)\sin(2\pi\xi u_1),
\]
the right-hand side is the sum of the two expressions in \eqref{nonneg_assumption_v2} obtained by taking $g(u)=\cos(2\pi\xi u)$ and $g(u)=\sin(2\pi\xi u)$. Both are nonnegative.
\end{proof}

\subsubsection{Windows and pairs}\label{sec:windows}
\begin{definition}[$(c,N)$-window]\label{def:cn-window}\uses{}\usesdefs{}
\lean{Auto.cnWindow}
Let $c > 0$, $N\in\mathbb{N}$. A {\em $(c,N)$-window} is a Schwartz function $\phi:\R\to\R$ such that
\begin{enumerate}
\item For all $\xi\in \R$, $\widehat{\phi}$ is real and
\begin{equation}\label{auto:window-Fourier-range}0\le \widehat{\phi}(\xi)\le 1.\end{equation}
\item For all $\xi\in \R\setminus [-1,1]$ we have
\begin{equation}\label{ft_window_eq_zero}
\widehat{\phi}(\xi)=0.
\end{equation}
\item For all $\xi\in   [-\frac12,\frac12]$ we have
\begin{equation}\label{ft_window_eq_one}
\widehat{\phi}(\xi)=1.
\end{equation}
\item For all $\xi \in [-1,1]$ and all $m\le N$,
\begin{equation}\label{windowbound}
|\widehat{\phi}^{(m)}(\xi)| \leq c.
\end{equation}
\end{enumerate}
\end{definition}

\begin{definition}[$(c,N)$-pair]
\label{def:cpair}\uses{}\usesdefs{def:cn-window}
\lean{Auto.cPair}
Let $c > 0$, $N\in\mathbb{N}$.  A pair $(\phi_0,\phi_1)$ of $(c,N)$-windows is a {\em $(c,N)$-pair} if
\begin{equation}\label{windowsum}
(\widehat{\phi_0})^2+(1-\widehat{\phi_1})^2=1.
\end{equation}
\end{definition}
If $(\phi_0,\phi_1)$ is a $(c,N)$-pair, then we write
\begin{equation}\label{auto:rescaled-window-notation}\phi_{i,k}=(\phi_i)_{(2^{k})}\end{equation}
for $i\in [2), k\in\mathbb{Z}$

Existence of $(c,N)$-pairs is shown in Lemma \ref{lem:cpair}.

\begin{definition}[Universal pair]\label{def:unipair}\uses{}\usesdefs{def:cpair}
\lean{Auto.uniPair}
Set
\begin{equation}\label{universal pair parameters}
N_{\ref{def:unipair}}=3,
\end{equation}
and
\begin{equation}\label{auto:universal-window-constant}
C_{\ref{def:unipair}}=2^{15}.
\end{equation}
A universal pair is a $(C_{\ref{def:unipair}},N_{\ref{def:unipair}})$-pair.
\end{definition}

\begin{definition}[Windows]\label{def:window}\uses{}\usesdefs{def:unipair}
\lean{Auto.window}
A function is a window if it is one of the two components of a universal pair.
\end{definition}

\begin{lemma}[Existence of a universal pair]\label{lem:cpair}\lean{Auto.existsUniversalPair}\leanok\uses{standard bump properties}\usesdefs{standard bump,def:cn-window,def:cpair,def:unipair}
There exists a universal pair.
\end{lemma}

\begin{proof}[Proof \auto]
Let $h=1-\widehat\Phi$, where $\Phi$ is the standard bump from Proposition \ref{standard bump properties}, and define
\[
\widehat{\phi_0}=\cos(\pi h/2),
\qquad
\widehat{\phi_1}=1-\sin(\pi h/2).
\]
Both functions are smooth, real, and even. On $[-1/2,1/2]$ we have $h=0$, so both equal $1$. Outside $[-1,1]$ we have $h=1$, so both vanish. Since $0\le h\le1$, both take values in $[0,1]$, and
\[
(\widehat{\phi_0})^2+(1-\widehat{\phi_1})^2
=\cos^2(\pi h/2)+\sin^2(\pi h/2)=1.
\]
Their inverse Fourier transforms are real-valued Schwartz functions.

It remains to verify the derivative bound. The proof of Proposition \ref{standard bump properties} gives
\[
\|h^{(q)}\|_\infty
\le2^{q(q-1)/2+3q}
\]
for $1\le q\le3$. Applying the Fa\`a di Bruno formula to the two trigonometric compositions, using $\pi/2<2$, and grouping terms by their block sizes gives the following bounds for derivatives of orders $1,2,3$:
\[
2^4,
\qquad
2^9,
\qquad
2^{15}.
\]
Before rounding to powers of two, the corresponding Fa\`a di Bruno sums are
\[
2\cdot2^3,
\]
\[
2\cdot2^7+2^2\cdot2^6,
\]
and
\[
2\cdot2^{12}+3\cdot2^2\cdot2^{10}+2^3\cdot2^9.
\]
These are bounded by the three displayed powers of two. Thus every derivative of order at most $N_{\ref{def:unipair}}$ is bounded by $C_{\ref{def:unipair}}$. The order-zero bound is at most $1$, so $(\phi_0,\phi_1)$ is a $(C_{\ref{def:unipair}},N_{\ref{def:unipair}})$-pair.
\end{proof}

\subsubsection{A smoothing decomposition}
In this section we  decompose the characteristic function $\mathbf{1}_{[0,1)}$ into smoother functions.

\begin{definition}[Window-based bump functions]\label{defn:window based bump functions}\uses{lem:cpair}\usesdefs{def:unipair}
\lean{Auto.windowBasedBumpFunctions}
Let $(\phi_0,\phi_1)$ be a universal pair.
Define
\begin{equation}\label{eqn:thetadef}
{\theta}={\phi_0}-{(\phi_0)}_{(2)}.
\end{equation}
Define for $k\in\mathbb{Z}$,
\begin{equation}\label{eqn:varphi0kdef}
\varphi_{0,k} = (\mathbf{1}_{[0,1)}\ast\phi_0 - \phi_0)\ast \theta_{(2^{k})}
\end{equation}
\begin{equation}\label{eqn:varphi1kdef}\varphi_{1,k} = \mathbf{1}_{[0,\infty)}\ast\theta_{(2^k)}
\end{equation}
\begin{equation}\label{eqn:varphi2kdef}\varphi_{2,k} = -\mathbf{1}_{[1,\infty)}\ast\theta_{(2^k)}
\end{equation}
\begin{equation}\label{eqn:varphi3def}
\varphi_{3,k} = 2^{k}(\varphi_{0,k})_{(2^{-k})}
\end{equation}
Let
\begin{equation}\label{eqn:varphi4kdef}
\varphi_{4,k}(u)=2^k\widetilde{\theta}(u-2^{-k}),
\end{equation}
where $\widetilde{\theta}=\mathbf{1}_{[0,\infty)}*\theta$ (a primitive of $\theta$).
\end{definition}

\begin{lemma}\label{lem:bumpbasic}\lean{Auto.bumpBasic}\leanok\uses{}\usesdefs{def:cn-window,def:unipair,defn:window based bump functions}
Let $\theta$ be as in \eqref{eqn:thetadef}.\\

(i) We have
\begin{equation}\label{auto:theta-Fourier-support} \mathrm{supp}\;\widehat{\theta} \subset [-1,-2^{-2}]\cup [2^{-2},1] \end{equation}

(ii) For every $m\in\mathbb{Z}$
\begin{equation}\label{auto:theta-telescoping-sum} \sum_{\ell=m}^\infty \theta_{(2^\ell)} = (\phi_0)_{(2^m)} \end{equation}
and the series converges uniformly.
\end{lemma}

\begin{proof}
(i) This is immediate from the support of windows.
(ii) For  $N\ge m$ and write
\[ S_N = \sum_{\ell=m}^N \theta_{(2^\ell)} =  \sum_{\ell=m}^N (\phi_0)_{(2^\ell)} - (\phi_0)_{(2^{\ell+1})} = (\phi_0)_{(2^m)} - (\phi_0)_{(2^{N+1})} \]
For any $x\in \R$ we  bound
\[|S_N(x)-(\phi_0)_{(2^m)}(x)| = |(\phi_0)_{(2^{N+1})}(x)|  \le 2^{-N-1}\|\phi_0\|_\infty \le 2^{-N} \]
Since $2^{-N}$ converges   to $0$ as $N\to \infty$, the conclusion follows.
\end{proof}

\begin{lemma}\label{lem:chardecomp}\uses{}\usesdefs{def:cn-window,def:unipair,defn:window based bump functions}
\lean{Auto.charDecomp}
\leanok
With $\phi_0,\theta$ defined as in \eqref{eqn:thetadef}, the identity
 \begin{equation}\label{auto:indicator-smoothing-decomposition} \mathbf{1}_{[0,1]} = \mathbf{1}_{[0,1]}* \phi_0 +  \sum_{\ell=-\infty}^{-1} \mathbf{1}_{[0,1]} * \theta_{(2^\ell)}  \end{equation}
holds in $L^2$.
\end{lemma}

\begin{rem}
    This identity also holds pointwise a.e., which can be shown using the weak $L^2$ boundedness of the maximally truncated convolution-type singular integrals. However, norm  convergence   suffices for our proof.
\end{rem}

\begin{proof}[Proof \auto]
For $N>1$, the definition of $\theta$ gives
\[
\phi_0+\sum_{\ell=-N}^{-1}\theta_{(2^\ell)}
=(\phi_0)_{(2^{-N})}.
\]
By Plancherel's identity, the $L^2$ norm of the difference between the two sides of the asserted partial sum identity is
\[
\left\|\widehat{\mathbf1_{[0,1]}}(\xi)
\bigl(1-\widehat{\phi_0}(2^{-N}\xi)\bigr)\right\|_{L^2(\xi)}.
\]
The second factor vanishes when $|\xi|\le2^{N-1}$ and has modulus at most one. Moreover,
\[
|\widehat{\mathbf1_{[0,1]}}(\xi)|
\le\min(1,(\pi|\xi|)^{-1}).
\]
Consequently the preceding norm is at most
\[
\left(2\int_{2^{N-1}}^\infty\tfrac{d\xi}{\pi^2\xi^2}\right)^{1/2}
=\tfrac{2^{1-N/2}}{\pi},
\]
which tends to zero.
\end{proof}

\begin{lemma}[Smoothing decomposition]\label{lem:smoothingdecomp}\uses{lem:chardecomp,lem:bumpbasic}\usesdefs{def:cn-window,def:unipair,defn:window based bump functions}
\lean{Auto.smoothingDecomp}
\leanok
The identity
\begin{equation}\label{eq:seriesexpansion}
\mathbf{1}_{[0,1]} = \phi_0 + \sum_{k=-2}^\infty \varphi_{0,k}  + \sum_{k=-\infty}^{-1} \varphi_{1,k}+ \sum_{k=-\infty}^{-1}\varphi_{2,k}.
\end{equation}
holds in the $L^2$ sense.
\end{lemma}

\begin{proof}
By Lemma \ref{lem:chardecomp} we have
\begin{equation}
    \label{eq:L2char}
    \mathbf{1}_{[0,1]} =\mathbf{1}_{[0,1]}\ast\phi_0  +  \sum_{k=-\infty}^{-1} \mathbf{1}_{[0,1]}\ast  \theta_{(2^k)}
\end{equation}
where the identity holds in $L^2$.
Denoting $\varphi=\mathbf{1}_{[0,1]}\ast\phi_0$ and splitting
\[\mathbf{1}_{[0,1]} = \mathbf{1}_{[0,\infty)} - \mathbf{1}_{[1,\infty)}\]
 we  thus have
\[\mathbf{1}_{[0,1]} = \varphi +  \sum_{k=-\infty}^{-1} \varphi_{1,k}+ \sum_{k=-\infty}^{-1}\varphi_{2,k},\]
where the sums converge in $L^2$.
It remains to decompose $\varphi$.
We  write $\varphi= \phi_0 + (\varphi-\phi_0)$
and
\[
\varphi-\phi_0 = (\varphi - \phi_0)*(\phi_0)_{(2^{-2})},\]
where the last identity holds due to
\[\widehat{\varphi-\phi_0}(\xi) \widehat{(\phi_0)_{(2^{-2})}}(\xi) = \widehat{\varphi-\phi_0}(\xi) \]
for each $\xi\in \R$. Indeed, the function $\widehat{\varphi-\phi_0}(\xi)=0$ for $|\xi|\ge 1$ and $\widehat{(\phi_0)_{(2^{-2})}}(\xi) = \widehat{\phi_0}(2^{-2}\xi) =1$ for $|\xi|\le 2$.
By part (ii) of Lemma \ref{lem:bumpbasic} we obtain
\[ (\varphi - \phi_0)*(\phi_0)_{(2^{-2})} =  (\varphi - \phi_0)*\Big( \sum_{k=-2}^\infty \theta_{(2^k)}\Big )  =  \sum_{k=-2}^\infty (\varphi - \phi_0)*\theta_{(2^k)} =   \sum_{k=-2}^\infty \varphi_{0,k}.
\]
The  second equality is justified, for example, by uniform convergence.
\end{proof}

\begin{lemma}\label{lem:theta_decay}\uses{lem:smoothdecay2,lem: min and bracket}\usesdefs{bracket bump,def:cn-window,def:unipair,defn:window based bump functions}
\lean{Auto.thetaDecay}
\leanok
Let $1\le N\le3$. Then
\begin{equation}\label{auto:theta-decay}
|\theta(u)|\le C_{\ref{lem:theta_decay},N}\langle u\rangle^N,
\end{equation}
where
\begin{equation}\label{auto:theta-decay-constant-definition}
C_{\ref{lem:theta_decay},N}=2^{2N+2}C_{\ref{def:unipair}}.
\end{equation}
\end{lemma}
\begin{proof}[Proof \auto]
We have
\[
\widehat\theta(\xi)=\widehat{\phi_0}(\xi)-\widehat{\phi_0}(2\xi)
\]
and, for $m\le N$,
\[
\|\widehat\theta^{(m)}\|_\infty
\le(1+2^m)C_{\ref{def:unipair}}\le2^{m+1}C_{\ref{def:unipair}}.
\]
The support of $\widehat\theta$ has measure at most two. If $N=1$, Fourier inversion and one integration by parts, followed by Lemma \ref{lem: min and bracket}, give
\[
|\theta(u)|
\le4C_{\ref{def:unipair}}\min(1,|u|^{-1})
\le2^3C_{\ref{def:unipair}}\langle u\rangle
\le2^4C_{\ref{def:unipair}}\langle u\rangle.
\]
If $N\in\{2,3\}$, Lemma \ref{lem:smoothdecay2}, applied with $m=N$, gives
\[
|\theta(u)|
\le2^N\,2^{N+1}C_{\ref{def:unipair}}\,2\langle u\rangle^N
=2^{2N+2}C_{\ref{def:unipair}}\langle u\rangle^N.
\]
\end{proof}

\begin{lemma}[constant $C_{\ref{lem:theta_decay},N}$ \auto]\label{constant theta decay}\uses{lem:theta_decay}\usesdefs{def:unipair}
\lean{Auto.constantThetaDecay}
\leanok
For $1\le N\le3$,
\begin{equation}\label{constant theta decay bound}
C_{\ref{lem:theta_decay},N}\le2^{2N+17}.
\end{equation}
\end{lemma}
\begin{proof}[Proof \auto]
Substitute $C_{\ref{def:unipair}}=2^{15}$ in \eqref{auto:theta-decay-constant-definition}.
\end{proof}

\begin{lemma}\label{lem:ft_phi3_eq}\uses{}\usesdefs{defn:window based bump functions}
\lean{Auto.fourierPhiThreeEq}
\leanok
For $k\in\mathbb{Z}$ and $\xi\in\mathbb{R}$,
\begin{equation}\label{auto:varphi-three-Fourier-transform} \widehat{\varphi_{3,k}}(\xi) = 2^k(\widehat{\mathbf{1}_{[0,1)}}-1)(2^{-k}\xi)\widehat{\phi_0}(2^{-k}\xi)\widehat{\theta}(\xi) \end{equation}
\end{lemma}

\begin{proof}
By definition,
\[\varphi_{3,k} = 2^k (\mathbf{1}_{[0,1)} * \phi_0 - \phi_0)_{(2^{-k}) }*\theta.\]
By the convolution theorem,
\[\widehat{\varphi_{3,k}}(\xi) = 2^k(\widehat{\mathbf{1}_{[0,1)}}-1)(2^{-k}\xi)\widehat{\phi_0}(2^{-k}\xi)\widehat{\theta}(\xi).\]
\end{proof}

\begin{lemma}\label{lem:abs_deriv_ft_phi3_le}\uses{lem:ft_phi3_eq}\usesdefs{def:cn-window,def:unipair,defn:window based bump functions}
\lean{Auto.absDerivFourierPhiThreeLe}
\leanok
Let $m\in[4)$, $k\in\Z$, and $|\xi|\le1$. Then
\begin{equation}\label{auto:varphi-three-Fourier-derivative-bound}
|\widehat{\varphi_{3,k}}^{(m)}(\xi)|
\le C_{\ref{lem:abs_deriv_ft_phi3_le},m},
\end{equation}
where for $m=0$,
\begin{equation}\label{auto:varphi-three-Fourier-derivative-constant}
C_{\ref{lem:abs_deriv_ft_phi3_le},0}=4,
\end{equation}
for $m=1$,
\begin{equation}\label{auto:varphi-three-Fourier-derivative-constant-one}
C_{\ref{lem:abs_deriv_ft_phi3_le},1}=28C_{\ref{def:unipair}}+4,
\end{equation}
for $m=2$,
\begin{equation}\label{auto:varphi-three-Fourier-derivative-constant-two}
C_{\ref{lem:abs_deriv_ft_phi3_le},2}=96C_{\ref{def:unipair}}^2+140C_{\ref{def:unipair}}+86,
\end{equation}
and, for $m=3$,
\begin{equation}\label{auto:varphi-three-Fourier-derivative-constant-three}
C_{\ref{lem:abs_deriv_ft_phi3_le},3}=69\cdot2^4C_{\ref{def:unipair}}^2+47\cdot2\cdot5^2C_{\ref{def:unipair}}+2^{11}.
\end{equation}
\end{lemma}
\begin{proof}[Proof \auto]
By Lemma \ref{lem:ft_phi3_eq},
\[
\widehat{\varphi_{3,k}}(\xi)
=
2^k\bigl(\widehat{\mathbf1_{[0,1)}}(2^{-k}\xi)-1\bigr)
\widehat{\phi_0}(2^{-k}\xi)\widehat\theta(\xi).
\]
On the support of $\widehat\theta$ we have $|\xi|\ge1/4$. If $k<-2$, then $|2^{-k}\xi|>1$, so the middle factor vanishes. We may therefore assume $k\ge-2$.

For $q=0$,
\[
2^k|\widehat{\mathbf1_{[0,1)}}(2^{-k}\xi)-1|
\le\pi|\xi|<4.
\]
For $1\le q\le3$, differentiation under the defining integral for $\widehat{\mathbf1_{[0,1)}}$ gives
\[
\left|\tfrac{d^q}{d\xi^q}
\left(2^k\bigl(\widehat{\mathbf1_{[0,1)}}(2^{-k}\xi)-1\bigr)\right)\right|
\le2^{2(q-1)}\tfrac{(2\pi)^q}{q+1}.
\]
Using $\pi<4$, these bounds for $q=1,2,3$, respectively, are
\[
4,
\qquad
86,
\qquad
2048.
\]
Moreover, the zeroth derivatives of the second and third factors are at most $1$, while for $1\le q\le3$,
\[
\left|\tfrac{d^q}{d\xi^q}\widehat{\phi_0}(2^{-k}\xi)\right|
\le2^{2q}C_{\ref{def:unipair}},
\qquad
|\widehat\theta^{(q)}(\xi)|\le(1+2^q)C_{\ref{def:unipair}}.
\]
Substitution in the three-factor Leibniz rule gives, for $m=0,1,2,3$, respectively,
\[
4,
\qquad
28C_{\ref{def:unipair}}+4,
\qquad
96C_{\ref{def:unipair}}^2+140C_{\ref{def:unipair}}+86,
\]
and
\[
1104C_{\ref{def:unipair}}^2+2350C_{\ref{def:unipair}}+2048.
\]
These are the defining values.
\end{proof}

\begin{lemma}[constant $C_{\ref{lem:abs_deriv_ft_phi3_le},m}$ \auto]\label{constant phi three derivative}\uses{lem:abs_deriv_ft_phi3_le}\usesdefs{def:unipair}
\lean{Auto.constantPhiThreeDerivative}
\leanok
For $m=0$,
\begin{equation}\label{constant phi three derivative bound}
C_{\ref{lem:abs_deriv_ft_phi3_le},0}=2^2.
\end{equation}
For $m=1$,
\begin{equation}\label{auto:constant-phi-three-derivative-one}
C_{\ref{lem:abs_deriv_ft_phi3_le},1}<2^{20}.
\end{equation}
For $m=2$,
\begin{equation}\label{auto:constant-phi-three-derivative-two}
C_{\ref{lem:abs_deriv_ft_phi3_le},2}<2^{37}.
\end{equation}
For $m=3$,
\begin{equation}\label{auto:constant-phi-three-derivative-three}
C_{\ref{lem:abs_deriv_ft_phi3_le},3}<2^{41}.
\end{equation}
\end{lemma}
\begin{proof}[Proof \auto]
Substitute $C_{\ref{def:unipair}}=2^{15}$ in the four defining values.
\end{proof}

\begin{lemma}\label{lem:abs_deriv_ft_Tphi3_le}\uses{lem:ft_deriv_mul,lem:abs_deriv_ft_phi3_le}\usesdefs{def:unipair,defn:window based bump functions}
\lean{Auto.absDerivFourierTPhiThreeLe}
\leanok
Let $m\in[3)$, $k\in\Z$, and $|\xi|\le1$. Then
\begin{equation}\label{auto:T-varphi-three-Fourier-derivative-bound}
|\widehat{T\varphi_{3,k}}^{(m)}(\xi)|
\le C_{\ref{lem:abs_deriv_ft_Tphi3_le},m},
\end{equation}
where
\begin{equation}\label{auto:T-varphi-three-Fourier-derivative-constant}
C_{\ref{lem:abs_deriv_ft_Tphi3_le},m}
=mC_{\ref{lem:abs_deriv_ft_phi3_le},m}
+C_{\ref{lem:abs_deriv_ft_phi3_le},m+1}.
\end{equation}
More explicitly,
\begin{equation}\label{auto:T-varphi-three-Fourier-derivative-constant-zero-exact}
C_{\ref{lem:abs_deriv_ft_Tphi3_le},0}=28C_{\ref{def:unipair}}+4,
\end{equation}
\begin{equation}\label{auto:T-varphi-three-Fourier-derivative-constant-one-exact}
C_{\ref{lem:abs_deriv_ft_Tphi3_le},1}=96C_{\ref{def:unipair}}^2+168C_{\ref{def:unipair}}+90,
\end{equation}
and
\begin{equation}\label{auto:T-varphi-three-Fourier-derivative-constant-two-exact}
C_{\ref{lem:abs_deriv_ft_Tphi3_le},2}=81\cdot2^4C_{\ref{def:unipair}}^2+263\cdot2\cdot5C_{\ref{def:unipair}}+555\cdot2^2.
\end{equation}
\end{lemma}
\begin{proof}[Proof \auto]
Lemma \ref{lem:ft_deriv_mul} and Lemma \ref{lem:abs_deriv_ft_phi3_le} give
\[
|\widehat{T\varphi_{3,k}}^{(m)}(\xi)|
\le
mC_{\ref{lem:abs_deriv_ft_phi3_le},m}
+C_{\ref{lem:abs_deriv_ft_phi3_le},m+1}.
\]
Substitution of the defining values from Lemma \ref{lem:abs_deriv_ft_phi3_le} gives the three stated polynomials.
\end{proof}

\begin{lemma}[constant $C_{\ref{lem:abs_deriv_ft_Tphi3_le},m}$ \auto]\label{constant T phi three derivative}\uses{lem:abs_deriv_ft_Tphi3_le}\usesdefs{def:unipair}
\lean{Auto.constantTPhiThreeDerivative}
\leanok
For $m=0$,
\begin{equation}\label{constant T phi three derivative bound}
C_{\ref{lem:abs_deriv_ft_Tphi3_le},0}<2^{20}.
\end{equation}
For $m=1$,
\begin{equation}\label{auto:constant-T-phi-three-derivative-one}
C_{\ref{lem:abs_deriv_ft_Tphi3_le},1}<2^{37}.
\end{equation}
For $m=2$,
\begin{equation}\label{auto:constant-T-phi-three-derivative-two}
C_{\ref{lem:abs_deriv_ft_Tphi3_le},2}<2^{41}.
\end{equation}
\end{lemma}
\begin{proof}[Proof \auto]
Substitute $C_{\ref{def:unipair}}=2^{15}$ in the three exact values from Lemma \ref{lem:abs_deriv_ft_Tphi3_le}.
\end{proof}

\begin{lemma}\label{lem:theta_prim}\uses{lem:smoothdecay2}\usesdefs{bracket bump,def:cn-window,def:unipair,defn:window based bump functions}
\lean{Auto.thetaPrimitive}
\leanok
Let $\widetilde\theta$ be the primitive of $\theta$ fixed above. Then
\begin{equation}\label{thetaprim_supp}
\operatorname{supp}(\widehat{\widetilde\theta})
\subset[-1,-2^{-2}]\cup[2^{-2},1]
\subset\operatorname{Ann}_1(1,2^2),
\end{equation}
\begin{equation}\label{Tthetaprim_supp}
\operatorname{supp}(\widehat{T\widetilde\theta})
\subset[-1,-2^{-2}]\cup[2^{-2},1]
\subset\operatorname{Ann}_1(1,2^2).
\end{equation}
For $2\le N<N_{\ref{def:unipair}}$,
\begin{equation}\label{thetaprim_decay}
|\widetilde\theta(u)|
\le C_{\ref{lem:theta_prim},N}\langle u\rangle^N,
\end{equation}
\begin{equation}\label{Thetaprim_decay}
|T\widetilde\theta(u)|
\le C_{\ref{lem:theta_prim},N}\langle u\rangle^N,
\end{equation}
where
\begin{equation}\label{auto:theta-primitive-constant-definition}
C_{\ref{lem:theta_prim},N}
=2^{5N+6}C_{\ref{def:unipair}}.
\end{equation}
\end{lemma}
\begin{proof}[Proof \auto]
Since $\int\theta=\widehat\theta(0)=0$ and $\widetilde\theta'=\theta$,
\[
\widehat{\widetilde\theta}(\xi)=(2\pi i\xi)^{-1}\widehat\theta(\xi).
\]
This proves \eqref{thetaprim_supp}. Also
\[
\widehat{T\widetilde\theta}(\xi)
=-\xi\widehat{\widetilde\theta}'(\xi),
\]
which proves \eqref{Tthetaprim_supp}.

For $m\le3$, the product rule gives
\[
\widehat{\widetilde\theta}^{(m)}(\xi)
=(2\pi i)^{-1}
\sum_{q=0}^m\binom{m}{q}(-1)^q q!\xi^{-q-1}
\widehat\theta^{(m-q)}(\xi).
\]
On the support, $|\xi|^{-q-1}\le2^{2q+2}$, while
$\|\widehat\theta^{(m-q)}\|_\infty\le2^{m-q+1}C_{\ref{def:unipair}}$. For $m=0,1,2,3$, direct substitution in this finite sum gives upper bounds
\[
2^3C_{\ref{def:unipair}},
\qquad
2^6C_{\ref{def:unipair}},
\qquad
2^9C_{\ref{def:unipair}},
\qquad
2^{13}C_{\ref{def:unipair}},
\]
respectively. Furthermore,
\[
(\widehat{T\widetilde\theta})^{(m)}
=-\xi\widehat{\widetilde\theta}^{(m+1)}
-m\widehat{\widetilde\theta}^{(m)}.
\]
Lemma \ref{lem:smoothdecay2}, the support bound, and the preceding finite estimates give
\[
|\widetilde\theta(u)|+|T\widetilde\theta(u)|
\le2^{5N+6}C_{\ref{def:unipair}}\langle u\rangle^N
\]
for $2\le N<N_{\ref{def:unipair}}$. This proves \eqref{thetaprim_decay} and \eqref{Thetaprim_decay}.
\end{proof}

\begin{lemma}[constant $C_{\ref{lem:theta_prim},N}$ \auto]\label{constant theta primitive}\uses{lem:theta_prim}\usesdefs{def:unipair}
\lean{Auto.constantThetaPrimitive}
\leanok
For $2\le N<N_{\ref{def:unipair}}$,
\begin{equation}\label{constant theta primitive bound}
C_{\ref{lem:theta_prim},N}\le2^{5N+21}.
\end{equation}
In particular,
\begin{equation}\label{auto:constant-theta-primitive-two}
C_{\ref{lem:theta_prim},2}\le2^{31}.
\end{equation}
\end{lemma}
\begin{proof}[Proof \auto]
Substitute $C_{\ref{def:unipair}}=2^{15}$ in \eqref{auto:theta-primitive-constant-definition}.
\end{proof}

\begin{lemma}\label{lem:phi4_supp}\uses{lem:theta_prim,lem:ft_deriv_mul}\usesdefs{defn:window based bump functions} We have
\lean{Auto.phiFourSupport}
\leanok
(a) $\textup{supp}(\widehat{\varphi_{4,k}})
\subset  [-1,-2^{-2}]\cup [2^{-2},1]$\\

(b) $\textup{supp}(\widehat{T\varphi_{4,k}})
\subset  [-1,-2^{-2}]\cup [2^{-2},1]$
\end{lemma}
\begin{proof}
(a) By Lemma \ref{lem:theta_prim},
\[\textup{supp}(\widehat{\widetilde{\theta}}) \subset [-1,-2^{-2}]\cup [2^{-2},1]\]
The claim follows since
\[\widehat{\varphi_{4,k}}(\xi)=2^ke^{-2\pi i2^{-k}\xi}\widehat{\widetilde\theta}(\xi).  \]

(b) By (a), $\widehat{\varphi_{4,k}}$ is supported in
$ [-1,-2^{-2}]\cup [2^{-2},1] $.
  By Lemma \ref{lem:ft_deriv_mul} applied with $m=0$,
  \[\widehat{T\varphi_{4,k}}(\xi) = -\xi(\widehat{\varphi_{4,k}})'(\xi) \]
The support of $(\widehat{\varphi_{4,k}})'$ is contained in the support of $\widehat{\varphi_{4,k}}$, and thus the last display is non-zero only if  $|\xi|\le 1$. The claim follows.
\end{proof}

\subsubsection{Miscellany}
Let $\mathbf{f}=(f_0,\dots,f_{n-1})$ be real-valued Schwartz functions satisfying \eqref{main twisted normalization}.

\begin{lemma}[Rescaling]\label{lem:rescaling}\uses{}\usesdefs{A_def}Let $\phi\in L^1(\mathbb{R})$.
\lean{Auto.rescaling}
\leanok
For all $r,q,\lambda>0$,
\begin{equation}\label{auto:variation-rescaling-invariance} \|A_t({\phi_{(\lambda)}})\|_{V_{r,J}(t\in (0,\infty); L^q)} =  \|A_t({\phi})\|_{V_{r,J}(t\in (0,\infty); L^q)}. \end{equation}
\end{lemma}

\begin{proof}
By replacing $t_j \lambda\mapsto t_j$ in the supremum defining the variation seminorm.
\end{proof}

\begin{lemma}\label{lem:norm_A_sum_le_sum}\uses{}\usesdefs{A_def,bracket bump}
\lean{Auto.normASumLeSum}
\leanok
If $\phi=\sum_{j=1}^\infty \phi_j$ holds in $L^2$ (with $L^2$ functions $\phi, \phi_j$), then
\begin{equation}\label{auto:twisted-average-series-triangle} \|A(\phi, \mathbf{f})\|_{L^2(\mathbb{R}^n)} \le \sum_{j=1}^\infty \|A({\phi_j}, \mathbf{f})\|_{L^2(\mathbb{R}^n)}. \end{equation}
\end{lemma}

\begin{proof}[Proof \auto]
Let $E_N=\phi-\sum_{j=1}^N\phi_j$. By the triangle inequality,
\[
\|A(\phi,\mathbf f)\|_2
\le
\sum_{j=1}^N\|A(\phi_j,\mathbf f)\|_2
+\|A(E_N,\mathbf f)\|_2.
\]
Cauchy--Schwarz in $s$ gives
\[
\|A(E_N,\mathbf f)\|_2
\le
\left(\int_{\R^{n+1}}
\left|\prod_{i\in[n)}f_i(x+se_i)\right|^2\,d(x,s)\right)^{1/2}
\|E_N\|_2.
\]
The first factor is finite. Indeed, write $x=(x_0,x')$. Since the $f_i$ are Schwartz and $n\ge2$, there is a finite $C$ such that
\[
|f_1(x+se_1)|\le C\langle x_0\rangle^2,
\qquad
|f_0(x+se_0)|\le C\langle x_0+s\rangle^2\langle x'\rangle^{2n},
\]
and $|f_i|\le C$ for $i\ge2$. After the change of variables $y=x_0+s$, the resulting majorant is integrable in $x_0,y,x'$. Since $\|E_N\|_2\to0$, letting $N\to\infty$ proves the claim.
\end{proof}

\begin{lemma}\label{lem:form_pos}\uses{M to K,prism BL inequality}\usesdefs{auto:Wiener-space-definition,normalized function tuples,prism Brascamp--Lieb,auto:prism-form-definition}
\lean{Auto.formPos}
\leanok
(i) Let $\Psi\in W_0(\R^2)$ be real-valued and assume that
\begin{equation}\label{eqn:hyp_form_pos}
\int_{\R^2}g(u_0)g(u_1)\Psi(u)\,du\ge0
\end{equation}
for every bounded measurable $g:\R\to\R$. Then, for every $\F\in\mathfrak F$,
\begin{equation}\label{auto:distributional-prism-positivity}
\Lambda_1(\Psi)(\F)\ge0.
\end{equation}

(ii) In particular, for every real-valued $\psi\in W_0(\R)$ and every $\F\in\mathfrak F$,
\begin{equation}\label{form_pos_psipsi}
\Lambda_1(\psi^{\otimes2})(\F)\ge0.
\end{equation}
\end{lemma}

\begin{proof}[Proof \auto]
By Definitions \ref{prism Brascamp--Lieb} and \ref{auto:prism-form-definition}, Proposition \ref{M to K}, and the explicit formula \eqref{auto:K-one-explicit-formula},
\begin{equation}\label{auto:Lambda-one-explicit-formula-for-positivity}
\Lambda_1(\Psi)(\F)
=
\int_{\R^n}\int_{\R^2}
\Psi(y^1,y^0)
\prod_{h\in[2),\,i\in[n)}
F_i(y^h-\Sigma(x),x_{[n)\setminus i})
\,dy\,dx.
\end{equation}
The integral is absolutely convergent by Propositions \ref{M to K} and \ref{prism BL inequality}. For $x\in\R^n$, set
\begin{equation}\label{auto:Lambda-one-positivity-test-function}
g_x(y)=\prod_{i\in[n)}F_i(y-\Sigma(x),x_{[n)\setminus i}).
\end{equation}
This is a bounded measurable real-valued function. Fubini's theorem and the change of variables interchanging $y^0$ and $y^1$ show that \eqref{auto:Lambda-one-explicit-formula-for-positivity} equals
\begin{equation}\label{auto:Lambda-one-positive-fiber-integral}
\int_{\R^n}\int_{\R^2}
\Psi(y^0,y^1)g_x(y^0)g_x(y^1)\,dy\,dx,
\end{equation}
which is nonnegative by \eqref{eqn:hyp_form_pos}. This proves (i).

For (ii), take $\Psi=\psi^{\otimes2}$. Then, for every bounded measurable real-valued $g$,
\begin{equation}\label{auto:tensor-square-kernel-positivity}
\int_{\R^2}g(u_0)g(u_1)\psi(u_0)\psi(u_1)\,du
=
\left(\int_\R g(u)\psi(u)\,du\right)^2
\ge0.
\end{equation}
Part (i) gives the result.
\end{proof}

\begin{lemma} \label{lem:ftc_ATphi}\uses{}\usesdefs{A_def}
\lean{Auto.ftcATphi}
\leanok
Let $x\in \R^n$, $\phi$ a Schwartz function,  $t\in [2^k, 2^{k+1}]$, $k\in \Z$. Let $a(t)=A_t(\phi)(x)$. Then
\begin{equation}\label{auto:average-logarithmic-derivative-identity} \|ta'(t)\|_{L^2(t\in [2^k, 2^{k+1}],\tfrac{dt}{t})}^2 = \int_{1}^2 |A_{2^k t}(T\phi)(x)|^2\,\tfrac{dt}{t}. \end{equation}
\end{lemma}

\begin{proof}
Note that
\begin{equation}\label{shortlongftc_reduction_pf3}
t\tfrac{\partial}{\partial t}(\phi_{(t)}(s)) = -(T\phi)_{(t)}(s),
\end{equation}
where $T\phi(s) = (s \phi(s))'$.
    By this and \eqref{A_def},
\[ ta'(t) = t\tfrac{\partial}{\partial t}\int_{\mathbb{R}} \Big(\prod_{i\in [n)} f_i(x+se_i)\Big)\phi_{(t)}(s)\,ds = -A_t(T\phi)(x). \]
Integrating in $t$ over $[2^k,2^{k+1}]$ and changing variables $t \mapsto 2^k t$ in the $t$-integration on the right-hand side,
the claim follows.
\end{proof}

\begin{lemma}
\label{lem:Phij_prop}\uses{lem:Phipos_v2}\usesdefs{def:cn-window,def:unipair}
\lean{Auto.phiJProperties}
\leanok
Let $\Psi:\mathbb{R}^2\to\mathbb{R}$ be a non-zero Schwartz function such that for all $u=(u_0,u_1)\in\R^2$,
\begin{equation}\label{auto:Phij-symmetry-assumption}
\Psi(u_0,u_1)=\Psi(u_1,u_0).
\end{equation}
Assume also that for every bounded measurable $g:\R\to\R$,
\begin{equation}\label{auto:Phij-positivity-assumption}
\int_{\R^2}g(u_0)g(u_1)\Psi(u)\,du\ge0,
\end{equation}
and that
\begin{equation}\label{auto:Phij-Fourier-support-assumption}
\mathrm{supp}(\widehat{\Psi}) \subset \mathrm{Ann}_1(1, 2^3)^2
\end{equation}
and, for $a\in [3)$ and $\xi \in \R$,
\begin{equation}\label{Phij_psi_der}
\Big|\tfrac{d^a}{d\xi^{a}} \widehat{\Psi}(\xi,-\xi)\Big| \leq 1.
\end{equation}
Let $(\phi_0,\phi_1)$ be a universal pair. There exists a real even smooth function $\psi:\R\to \R$ such that
\begin{enumerate}
\item For all $\xi\in\R$,
\begin{equation}\label{Phij_psi_diag}
\widehat{\psi}(\xi)^2 = 2 (\widehat{\phi_{0,-5}}-\widehat{\phi_{1,5}})(\xi)^2 - \widehat{\Psi}(\xi,-\xi).
\end{equation}
\item $\widehat{\psi}$ is supported in $\textup{Ann}_1(1,2^6)$.
\item For all $a\in [3)$ and $\xi\in \R$,
\begin{equation}\label{psi_decay}
|\widehat{\psi}^{(a)}(\xi)|\le C_{\ref{lem:Phij_prop}},
\end{equation}
where
\begin{equation}\label{auto:Phij-constant-definition}
C_{\ref{lem:Phij_prop}}=2^{12}C_{\ref{def:unipair}}.
\end{equation}
\end{enumerate}
\end{lemma}
\begin{proof}[Proof \auto]
Put
\[
\vartheta(\xi)=\widehat{\phi_0}(2^{-5}\xi)-\widehat{\phi_1}(2^5\xi).
\]
The function $\vartheta$ is nonnegative, supported in $\mathrm{Ann}_1(1,2^6)$, and equal to $1$ on $\mathrm{Ann}_1(1,2^4)$. By Lemma \ref{lem:Phipos_v2} and \eqref{Phij_psi_der},
\[
0\le\widehat\Psi(\xi,-\xi)\le1.
\]
Since $\widehat\Psi(\xi,-\xi)$ is supported in $\mathrm{Ann}_1(1,2^3)$, it follows that
\begin{equation}\label{eqn:lemPhij_proppf}
0\le\widehat\Psi(\xi,-\xi)\le\vartheta(\xi)^2.
\end{equation}
Define
\begin{equation}\label{def_sqrtpsi}
\widehat\psi(\xi)=\bigl(2\vartheta(\xi)^2-\widehat\Psi(\xi,-\xi)\bigr)^{1/2}.
\end{equation}
On the complement of $\mathrm{Ann}_1(1,2^3)$ this is $2^{1/2}\vartheta$, while on the interior of $\mathrm{Ann}_1(1,2^4)$ the expression under the square root is $2-\widehat\Psi(\xi,-\xi)\ge1$. These two open sets cover $\R$, so $\widehat\psi$ is smooth. It is real and even, and its support is contained in $\mathrm{Ann}_1(1,2^6)$. This proves the first two assertions.

For $a\in[3)$,
\[
|\vartheta^{(a)}(\xi)|
\le (2^{-5a}+2^{5a})C_{\ref{def:unipair}}.
\]
Thus, outside $\mathrm{Ann}_1(1,2^3)$,
\[
|\widehat\psi(\xi)|\le2,
\qquad
|\widehat\psi'(\xi)|\le2^7C_{\ref{def:unipair}},
\qquad
|\widehat\psi''(\xi)|\le2^{12}C_{\ref{def:unipair}}.
\]
On the interior of $\mathrm{Ann}_1(1,2^4)$, set $f(\xi)=2-\widehat\Psi(\xi,-\xi)$. Then $1\le f\le2$ and $|f'|,|f''|\le1$, whence
\[
|f^{1/2}|\le2,
\qquad
|(f^{1/2})'|\le\tfrac12,
\qquad
|(f^{1/2})''|\le\tfrac34.
\]
The derivative estimate follows.
\end{proof}

\begin{lemma}[constant $C_{\ref{lem:Phij_prop}}$ \auto]\label{constant Phij proposition}\uses{lem:Phij_prop}\usesdefs{def:unipair}
\lean{Auto.constantPhiJProposition}
\leanok
\begin{equation}\label{constant Phij proposition bound}
C_{\ref{lem:Phij_prop}}=2^{27}.
\end{equation}
\end{lemma}
\begin{proof}[Proof \auto]
This follows from \eqref{auto:Phij-constant-definition} and $C_{\ref{def:unipair}}=2^{15}$.
\end{proof}

\begin{lemma}[Bootstrapping]
\label{lem:bootstrap}\uses{}\usesdefs{A_def}
\lean{Auto.bootstrap}
\leanok
Let $\phi:\R\to \R$ be a Schwartz function.
Then
    \begin{equation}\label{auto:dyadic-variation-bootstrap}\|A_{t}(\phi)\|^2_{V_{2,J}(t\in 2^{\Z}; L^2)} \le  4 \cdot \sup \sum_{\ell\in [J)} \|A_{2^{k_{\ell}}}(\phi)\|^2_2     \end{equation}
    where the supremum is over all sequences of integers $k_0<\cdots <k_{J-1}$.
\end{lemma}
{\em Note:} This lemma will only be applied to mean zero bump functions.

\begin{proof}
Let $k_0<\cdots <k_{J}$ be integers. For $\ell \in [J)$ we write
\begin{equation}
    \label{expandrhs}
    \|A_{2^{k_{\ell+1}}}(\phi)-A_{2^{k_{\ell}}}(\phi)\|_2^2 = \|A(\phi_{(2^{k_{\ell+1}})}-\phi_{(2^{k_\ell})})\|_2^2
    \end{equation}
 \[= \int_{\R^n} A(\phi_{(2^{k_{\ell+1}})}-\phi_{(2^{k_{\ell}})})(x)A(\phi_{(2^{k_{\ell+1}})}-\phi_{(2^{k_{\ell}})})(x)  dx \]

    where  $k_0<\ldots < k_J$ is any increasing sequence of integers.

By the distributive law,
  \[\sum_{\ell\in [J)}  \|A_{2^{k_{\ell+1}}}(\phi)-A_{2^{k_{\ell}}}(\phi)\|_2^2 \]
\[ = \sum_{\ell\in [J)} \int_{\R^n} A(\phi_{(2^{k_{\ell+1}})})(x)A(\phi_{(2^{k_{\ell+1}})}-\phi_{(2^{k_{\ell}})})(x)  dx \]
\[ - \sum_{\ell\in [J)}  \int_{\R^n} A(\phi_{(2^{k_{\ell}})})(x)A(\phi_{(2^{k_{\ell+1}})}-\phi_{(2^{k_{\ell}})})(x)  dx \]
By the Cauchy-Schwarz inequality in $x$ and in the summation,  this is
\[ \le  \Big ( \sum_{\ell\in [J)}  \|A(\phi_{(2^{k_{\ell+1}})})\|^2_2 \Big)^{1/2} \Big (\sum_{\ell\in [J)} \|A(\phi_{(2^{k_{\ell+1}})}-\phi_{(2^{k_{\ell}})})\|^2_2\Big)^{1/2}\]
\[+ \Big ( \sum_{\ell\in [J)} \|A(\phi_{(2^{k_{\ell}})})\|_2^2\Big)^{1/2} \Big(\sum_{\ell\in [J)} \|A(\phi_{(2^{k_{\ell+1}})}-\phi_{(2^{k_{\ell}})})\|^2_2\Big)^{1/2} \]
And thus, with $V= \|A_{t}(\phi)\|_{V_{2,J}(t\in 2^\mathbb{Z}; L^2)}$,
\[V^2 \le  V\cdot  \Big( \Big(\sum_{\ell\in [J)} \|A(\phi_{(2^{k_{\ell+1}})})\|^2_2\Big)^{1/2} + \Big( \sum_{\ell\in [J)} \|A(\phi_{(2^{k_{\ell}})})\|^2_2 \Big)^{1/2} \Big)   \]
We may assume $V\neq 0$  as otherwise the claim holds trivially. Dividing by $V$ and estimating the terms on the   right-hand side by the supremum over all increasing sequences of integers gives
\[V \le  2 \cdot  \Big( \sup \sum_{\ell\in [J)} \|A(\phi_{(2^{k_{\ell}})})\|^2_2\Big)^{1/2}     \]
Squaring both sides gives the desired estimate.
\end{proof}

\subsection{On-diagonal from main argument}\label{sec:off-diag}

Throughout this section, let
$\nu=1$ if $n>2$ and $\nu=2$ if $n=2$.

\begin{lemma}[$\rho$ kernels - reduction variant] \label{lem:rho-kernels-reduction}\uses{mean four scale Gaussian kernel,standard bump properties,four scale Gaussian kernel}\usesdefs{gaussian,bracket bump,auto:Wiener-space-definition,multiplicatively spaced monotone sequences,standard bump}  Let $a\in A$. For $h\in\N$ with $h\ge1$ and $j\in\Z$, define
\lean{Auto.rhoKernelsReduction}
\leanok
\begin{equation}\label{auto:rho-lambda-scales}
\lambda_{h,j}^-=2^{h-1}a(j),
\end{equation}
and, for $j\in\Z$, define
\begin{equation}\label{auto:rho-lambda-zero-scale}
\lambda_{0,j}^-=a(j-1).
\end{equation}
For $h\in\N$ and $j\in\Z$, define
\begin{equation}\label{auto:rho-lambda-upper-scale}
\lambda_{h,j}^+=2^ha(j).
\end{equation}
For $h\in\N$ and $j\in\Z$, also define
\begin{equation}\label{auto:rho-mu-scales}
\mu_{h,j}^-=2^ha(j-1),
\qquad
\mu_{h,j}^+=2^ha(j).
\end{equation}
Let \(\rho_{h,j}\) be the function whose Fourier transform is
\begin{equation}\label{auto:rho-kernel-Fourier-transform}
\widehat{\rho_{h,j}}(\zeta)
=
\bigl(
\widehat\Phi(\lambda_{h,j}^-\zeta)
-
\widehat\Phi(\lambda_{h,j}^+\zeta)
\bigr)
\bigl(
\g(\mu_{h,j}^-\zeta)
-
\g(\mu_{h,j}^+\zeta)
\bigr)^{-\nu/2}.
\end{equation}

Then \(\rho_{h,j}\in W_0(\mathbb R)\). Moreover, if \(h\geq1\), then, with $\lambda=2^ha(j)$,
\begin{equation}
    \label{E:increase-data-rho-mean-estimate}
    |\rho_{h,j}(x+p)-\rho_{h,j}(x)|
    \leq
    C_{\ref{lem:rho-kernels-reduction}}
    \min(1,\lambda^{-1}|p|)
    \left(
    \langle x+p\rangle_{(\lambda)}^2
    +
    \langle x\rangle_{(\lambda)}^2
    \right),
\end{equation}
and if $h\in \N$, then
\begin{equation}
    \label{E:increase-data-rho-zero-estimate}
    |\rho_{h,j}(x)|
    \leq
    C_{\ref{lem:rho-kernels-reduction}}
    \big(
    \langle x\rangle_{(\lambda_{h,j}^-)}^2
    +
    \langle x\rangle_{(\lambda_{h,j}^+)}^2
    \big),
\end{equation}
where $C_{\ref{lem:rho-kernels-reduction}} = 2^{21} (C_{\ref{mean four scale Gaussian kernel},2}+C_{\ref{four scale Gaussian kernel},2})$.
\end{lemma}

\begin{proof}
If \(h\geq1\), the parameters satisfy
\[
2\mu_{h,j}^-
\leq
2\lambda_{h,j}^-
=
\lambda_{h,j}^+
=
\mu_{h,j}^+,
\]
because \(2a(j-1)\leq a(j)\).
If $h=0$, we have
\[
2\mu_{0,j}^-
=
2\lambda_{0,j}^-
\leq
\lambda_{0,j}^+
=
\mu_{0,j}^+.
\]

Now, if $h\ge 1$, we  use Proposition
\ref{mean four scale Gaussian kernel}   with $N=2$, which  gives
\[|\rho_{h,j}(x+p)-\rho_{h,j}(x)| \le 2^{18} C_{\ref{mean four scale Gaussian kernel},2} \min(1,2\pi \lambda^{-1}|p|) (\langle x+p\rangle^2_{(\lambda)}+
\langle x\rangle^2_{(\lambda)}),
\]
which gives \eqref{E:increase-data-rho-mean-estimate}.
Here we also used that by Proposition~\ref{standard bump properties},
\[
\max_{0\leq m\leq2}
\|(\widehat\Phi)^{(m)}\|_\infty
\leq2^{18}.
\]
If $h\in \N$, we use  Proposition~\ref{four scale Gaussian kernel}   with $N=2$, which   gives $\rho_{h,j}\in W_0(\R)$ and
\[ |\rho_{h,j}(x)| \le 2^{18}C_{\ref{four scale Gaussian kernel},2} (\langle x\rangle^2_{(\lambda_{h,j}^-)} + \langle x\rangle^2_{(\lambda_{h,j}^+)}). \]
\end{proof}

\begin{lemma}[constant $C_{\ref{lem:rho-kernels-reduction}}$ \auto]\label{constant rho kernels reduction}\uses{lem:rho-kernels-reduction,constant mean four scale Gaussian kernel,constant four scale Gaussian kernel,mean four scale Gaussian kernel,four scale Gaussian kernel}\usesdefs{}
\lean{Auto.constantRhoKernelsReduction}
\leanok
\begin{equation}\label{constant rho kernels reduction bound}
C_{\ref{lem:rho-kernels-reduction}}
<2^{66}.
\end{equation}
\end{lemma}
\begin{proof}[Proof \auto]
Lemmas \ref{constant mean four scale Gaussian kernel} and \ref{constant four scale Gaussian kernel}, at $N=2$, give
\[
C_{\ref{mean four scale Gaussian kernel},2}
+C_{\ref{four scale Gaussian kernel},2}
<20397963318112+637436354528.
\]
Multiplication by the defining factor $2^{21}$ gives
\[
C_{\ref{lem:rho-kernels-reduction}}
<44114430494276321280<\tfrac35 2^{66}.
\]
\end{proof}

\begin{rem}
The value of $\widehat{\rho_{h,j}}$   at the origin is understood through the continuous
extension supplied by Proposition~\ref{four scale Gaussian kernel}.
\end{rem}

\begin{lemma}[affine diagonal cancellation - reduction variant]
\label{lem:affine-diagonal-cancellation-reduction}\uses{W_0 fiber integrals}\usesdefs{auto:Wiener-space-definition}
\lean{Auto.affineDiagonalCancellationReduction}
\leanok
Let \(M\in W_0(\mathbb R^2)\) satisfy for every \(\xi\in\mathbb R\)
\begin{equation}\label{auto:affine-diagonal-cancellation-assumption}
\widehat M(\xi,-\xi)=0
\end{equation}
 Then for every \(x\in\mathbb R\),
\begin{equation}
    \label{E:affine-diagonal-cancellation}
    \int_{\mathbb R}M(x+q,q)\,dq=0
\end{equation}
\end{lemma}

\begin{proof}
    Define
\[
G(x)=\int_{\mathbb R}M(x+q,q)\,dq.
\]
By Proposition~\ref{W_0 fiber integrals}, \(G\in W_0(\mathbb R)\).
By Fubini's theorem and the change of variables
$ u=x+q, \, v=q, $
we have
\[
\widehat G(\xi)
=
\int_{\mathbb R^2}
M(x+q,q)e^{-2\pi ix\xi}\,dq\,dx
=
\int_{\mathbb R^2}
M(u,v)e^{-2\pi i(u-v)\xi}\,du\,dv=
\widehat M(\xi,-\xi)
=
0.
\]
Injectivity of the Fourier transform on \(L^1\) and continuity of $G$ give
\(G=0\) everywhere.
\end{proof}

\begin{proposition}[bracket domination - reduction variant]
\label{lem:increase-data-bracket-domination}\uses{lem:rho-kernels-reduction,two bump estimate,bump triangle,orthogonal decay,convolution vector,lem:affine-diagonal-cancellation-reduction,Properties of distance of sequences,Operations on spaced sequences}\usesdefs{gaussian,bracket bump,auto:Wiener-space-definition,auto:convolution-along-vector-definition,multiplicatively spaced monotone sequences,Distance of spaced sequences,closed balls in A,standard bump,geometric parameters,auto:unitary-matrices-definition,square root Gaussian difference}
\lean{Auto.increaseDataBracketDomination}
\leanok
Let \(a\in A\), and let
$
\mathcal P\subset B_{\dist}(a,1)^2\times[2)$ be such that $\#\mathcal P\leq5$. Assume additionally that
\begin{equation}
    \label{E:increase-data-one-sided-scales}
    t_0(j)\leq a(j),
    \quad
    t_1(j)\leq a(j)
\end{equation}
for every \((t_0,t_1,u)\in\mathcal P\) and \(j\in\mathbb Z\).
Suppose that \(M_j\in W_0(\mathbb R^2)\) satisfies
\begin{equation}\label{auto:reduction-diagonal-cancellation}
\widehat M_j(\xi,-\xi)=0
\end{equation}
and
\begin{equation}\label{auto:reduction-kernel-decay}
|M_j(v)|
\leq
\sum_{(t_0,t_1,u)\in\mathcal P}
\langle(W_uv)_0\rangle_{(t_0(j))}^{3/2}
\langle(W_uv)_1\rangle_{(t_1(j))}^{3/2}.
\end{equation}

Let $\lambda_{h,j}^-,\lambda_{h,j}^+$ be as in Lemma \ref{lem:rho-kernels-reduction}.
For $h\in\N$, $j\in \Z$, we define
$
\sigma_{h,j} = s(2^{h} a,j)
$
and
\begin{equation}\label{auto:increase-data-N-kernel}
N_{h,j} =
\mathcal F^{-1}\left((\xi,\eta) \mapsto (\widehat{\Phi_{(\lambda_{h,j}^-)}}-\widehat{\Phi_{(\lambda_{h,j}^+)}})(\xi+\eta) \widehat{\sigma_{h,j}}(\xi+\eta)^{-\nu} \widehat{M_j}(\xi,\eta)\right),
\end{equation}

Then for every \(h\in\mathbb N\), there exist a finite set
\(\mathcal B_h\), numbers \(u_{h,b}\in[2)\), and pairs
\begin{equation}\label{auto:increase-data-scale-pairs}
\beta_{h,b}
=
(\beta_{h,b}^0,\beta_{h,b}^1)\in A^2,
\qquad b\in\mathcal B_h,
\end{equation}
satisfying
\begin{equation}
    \label{E:increase-data-B-cardinality}
    \#\mathcal B_h
    \leq
    C_{\ref{lem:increase-data-bracket-domination},0},
\end{equation}
\begin{equation}
    \label{E:increase-data-beta-base-distance}
    \dist(a,\beta_{h,b}^\ell)
    \leq
     (1+h)
\end{equation}
for every \(b\in\mathcal B_h\) and \(\ell\in[2)\), and
\begin{equation}
    \label{E:increase-data-bracket-majorant}
    |N_{h,j}(v)| \leq C_{\ref{lem:increase-data-bracket-domination},1} 2^{-h/3} \sum_{b\in\mathcal B_h} \langle(W_{u_{h,b}}v)_0 \rangle_{(\beta_{h,b}^0(j))}^{7/6} \langle(W_{u_{h,b}}v)_1 \rangle_{(\beta_{h,b}^1(j))}^{7/6}
\end{equation}
where $C_{\ref{lem:increase-data-bracket-domination},0}=2^6$, $C_{\ref{lem:increase-data-bracket-domination},1}
=
2^{10}C_{\ref{lem:rho-kernels-reduction}}$.
In particular,
\begin{equation}
    \label{E:increase-data-beta-pair-distance}
    \Delta(\beta_{h,b})
    \leq
    2(1+h).
\end{equation}
\end{proposition}
\begin{proof}[Proof \auto]
Set
\[
r=\tfrac76,
\qquad
C_\rho=C_{\ref{lem:rho-kernels-reduction}},
\]
and, for \(s>0\), define
\[
\omega_s(x)=\langle x\rangle_{(s)}^r.
\]

We first record several elementary estimates that will be used
throughout the proof. Since
\[
r=\tfrac76<\tfrac32<2,
\]
we have
\begin{equation}
    \label{E:increase-data-weaken-exponents}
    \langle x\rangle_{(s)}^2
    \leq
    \langle x\rangle_{(s)}^{3/2}
    \leq
    \omega_s(x).
\end{equation}
Moreover,
\begin{equation}
    \label{E:increase-data-r-bump-L1}
    \|\omega_s\|_1
    =
    2\int_0^\infty(1+x)^{-7/6}\,dx
    =
    12
    \leq 2^4.
\end{equation}
Directly from the definition of the bracket bumps,
\begin{equation}
    \label{E:increase-data-fixed-scaling}
    \omega_s(x/\sqrt2)\leq2\omega_s(x),
    \qquad
    \omega_s(\sqrt2x)\leq\omega_s(x).
\end{equation}

Proposition~\ref{two bump estimate}, applied with both exponents
equal to \(r\), gives
\begin{equation}
    \label{E:increase-data-two-bump-corollary}
    \int_{\mathbb R}
    \omega_s(x-p)\omega_t(y-p)\,dp
    \leq
    2^6
    \omega_{\max(s,t)}(x-y).
\end{equation}
Indeed,
\[
C_{\ref{two bump estimate},r,r}
=
2^{1+r}\left(1+(r-1)^{-1}\right)
=
7\cdot2^{13/6}
<
2^6.
\]

We will also use the following consequence of Propositions~
\ref{bump triangle} and~\ref{orthogonal decay}. Suppose that
\[
v_0=w_0-w_1,
\qquad
v_1=w_0+w_1.
\]
Then
\[
W_1v=(\sqrt2w_0,\sqrt2w_1).
\]
Let
\[
e_0=(1,0),
\qquad
e_1=(0,1),
\]
and
\[
\theta_0=2^{-1/2}(1,1),
\qquad
\theta_1=2^{-1/2}(-1,1).
\]
Thus
\[
v\cdot\theta_0=(W_1v)_0,
\qquad
v\cdot\theta_1=(W_1v)_1,
\qquad
w_1=2^{-1/2}v\cdot\theta_1.
\]
For either \(i=0\) or \(i=1\), the constant in Proposition~
\ref{orthogonal decay}, applied to \(e_i\) and \(\theta_1\), is
\[
(2|e_i\cdot\theta_0|^{-1})^r
=
(2\sqrt2)^{7/6}
=
2^{7/4}
<
2^2.
\]
Combining this with
\eqref{E:increase-data-fixed-scaling}, we obtain
\begin{equation}
    \label{E:increase-data-mixed-zero}
    \omega_s(v_0)\omega_t(w_1) \leq 2^3\Big( \omega_s(v_0)\omega_t(v_1) + \omega_s((W_1v)_0) \omega_t((W_1v)_1) \Big)
\end{equation}
and
\begin{equation}
    \label{E:increase-data-mixed-one}
    \omega_s(v_1)\omega_t(w_1) \leq 2^3\Big( \omega_t(v_0)\omega_s(v_1) + \omega_s((W_1v)_0) \omega_t((W_1v)_1) \Big).
\end{equation}

We next express \(N_{h,j}\) as a convolution. Let
\(\rho_{h,j}\) be as in Lemma~
\ref{lem:rho-kernels-reduction}. By the definition of
\(\sigma_{h,j}=s(2^ha,j)\),
\[
\widehat{\sigma_{h,j}}(\zeta)
=
\left(
\g(2^ha(j-1)\zeta)
-
\g(2^ha(j)\zeta)
\right)^{1/2}.
\]
It follows from Proposition~\ref{convolution vector} that
\begin{equation}
    \label{E:increase-data-N-rho-convolution}
    N_{h,j}
    =
    M_j\ast_{(1,1)}\rho_{h,j}.
\end{equation}

Fix \(h\in\mathbb N\), \(j\in\mathbb Z\), and \(v\in\mathbb R^2\),
and write
\[
v_0=w_0-w_1,
\qquad
v_1=w_0+w_1.
\]
Changing variables in
\eqref{E:increase-data-N-rho-convolution}, we obtain
\begin{equation}
    \label{E:increase-data-convolution-coordinates}
    N_{h,j}(v)
    =
    \int_{\mathbb R}
    \rho_{h,j}(w_0+p)
    M_j(-w_1-p,w_1-p)\,dp.
\end{equation}

We first treat the case \(h\geq1\). Set
\[
\lambda=2^ha(j).
\]
By Lemma~\ref{lem:affine-diagonal-cancellation-reduction},
applied with \(x=-2w_1\),
\[
\int_{\mathbb R}
M_j(-w_1-p,w_1-p)\,dp=0.
\]
Consequently,
\begin{equation}
    \label{E:increase-data-cancelled-convolution}
    N_{h,j}(v) = \int_{\mathbb R} \bigl( \rho_{h,j}(w_0+p)-\rho_{h,j}(w_0) \bigr) M_j(-w_1-p,w_1-p)\,dp.
\end{equation}
Lemma~\ref{lem:rho-kernels-reduction}, together with
\eqref{E:increase-data-weaken-exponents}, gives
\begin{equation}
    \label{E:increase-data-cancelled-majorant}
    |N_{h,j}(v)| \leq C_\rho \sum_{\alpha\in[2)} \int_{\mathbb R} \min(1,\lambda^{-1}|p|) \omega_\lambda(w_0+\alpha p) \times |M_j(-w_1-p,w_1-p)|\,dp.
\end{equation}

Fix
\[
(t_0,t_1,u)\in\mathcal P,
\]
and suppress the argument \(j\) in \(t_0(j)\) and \(t_1(j)\).
The additional one-sided assumption \eqref{E:increase-data-one-sided-scales} implies
\begin{equation}
    \label{E:increase-data-scale-comparison}
    t_\ell
    \leq
    a(j)
    \leq
    2^ha(j)
    =
    \lambda,
    \qquad
    \ell\in[2).
\end{equation}

Suppose first that \(u=0\). Then
\[
\langle(-w_1-p)\rangle_{(t_0)}^{3/2}
\langle(w_1-p)\rangle_{(t_1)}^{3/2}
=
\langle w_1+p\rangle_{(t_0)}^{3/2}
\langle w_1-p\rangle_{(t_1)}^{3/2}.
\]
Since
\[
2|p|
\leq
|w_1+p|+|w_1-p|
\]
and \(\min(1,s)\leq s^{1/3}\), we have
\[
\min(1,\lambda^{-1}|p|)
\leq
2^{-1/3}\lambda^{-1/3}
\left(|w_1+p|^{1/3}+|w_1-p|^{1/3}\right).
\]
Since
\[
t_\ell/\lambda\leq2^{-h},
\]
it follows that
\[
\min(1,\lambda^{-1}|p|)
\leq
2^{-h/3}
\left((1+t_0^{-1}|w_1+p|)^{1/3}+(1+t_1^{-1}|w_1-p|)^{1/3}\right).
\]
Multiplying by the two \(3/2\)-bumps and using
\[
\tfrac32-\tfrac13=\tfrac76=r,
\]
we obtain
\begin{equation}
    \label{E:increase-data-loss-u-zero}
    \min(1,\lambda^{-1}|p|) \langle w_1+p\rangle_{(t_0)}^{3/2} \langle w_1-p\rangle_{(t_1)}^{3/2} \qquad\leq 2 \cdot 2^{-h/3} \omega_{t_0}(w_1+p) \omega_{t_1}(w_1-p).
\end{equation}

Suppose now that \(u=1\). Since
\[
W_1(-w_1-p,w_1-p)
=
(-\sqrt2p,\sqrt2w_1),
\]
we have
\[
\min(1,\lambda^{-1}|p|) \leq (\lambda^{-1}|p|)^{1/3} = 2^{-1/6} (t_0/\lambda)^{1/3} (t_0^{-1}|\sqrt2p|)^{1/3} \leq 2^{-h/3} (1+t_0^{-1}|\sqrt2p|)^{1/3}.
\]
It follows that
\begin{equation}
    \label{E:increase-data-loss-u-one}
    \min(1,\lambda^{-1}|p|) \langle\sqrt2p\rangle_{(t_0)}^{3/2} \langle\sqrt2w_1\rangle_{(t_1)}^{3/2} \qquad\leq 2\cdot 2^{-h/3} \omega_{t_0}(\sqrt2p) \omega_{t_1}(\sqrt2w_1).
\end{equation}

We now integrate the expressions resulting from
\eqref{E:increase-data-loss-u-zero} and
\eqref{E:increase-data-loss-u-one}.

Assume first that \(u=0\). For \(\alpha=0\), let
\[
\tau=\max(t_0,t_1).
\]
By \eqref{E:increase-data-two-bump-corollary},
\[
\omega_\lambda(w_0) \int_{\mathbb R} \omega_{t_0}(w_1+p) \omega_{t_1}(w_1-p)\,dp \qquad\leq 2^6 \omega_\lambda(w_0) \omega_\tau(2w_1).
\]
Using
\[
(W_1v)_0=\sqrt2w_0,
\qquad
(W_1v)_1=\sqrt2w_1
\]
and \eqref{E:increase-data-fixed-scaling}, this is at most
\begin{equation}
    \label{E:increase-data-u0-alpha0}
    2^7
    \omega_\lambda((W_1v)_0)
    \omega_\tau((W_1v)_1).
\end{equation}

For \(\alpha=1\), Proposition~\ref{bump triangle}, applied with
\[
w_1
=
\tfrac12(w_1+p)+\tfrac12(w_1-p),
\]
has constant one and gives
\begin{equation}
    \label{E:increase-data-narrow-triangle}
    \omega_{t_0}(w_1+p) \omega_{t_1}(w_1-p) \qquad\leq \omega_{t_0}(w_1+p)\omega_{t_1}(w_1) + \omega_{t_0}(w_1)\omega_{t_1}(w_1-p).
\end{equation}
Since \(\lambda\geq t_0,t_1\), Proposition~
\ref{two bump estimate} gives
\[
\int_{\mathbb R} \omega_\lambda(w_0+p) \omega_{t_0}(w_1+p) \omega_{t_1}(w_1-p)\,dp \qquad\leq 2^6\Big( \omega_\lambda(v_0)\omega_{t_1}(w_1) + \omega_\lambda(v_1)\omega_{t_0}(w_1) \Big).
\]
Applying \eqref{E:increase-data-mixed-zero} and
\eqref{E:increase-data-mixed-one}, the last expression is at most
\begin{equation}
    \label{E:increase-data-u0-alpha1}
2^9\omega_\lambda(v_0)\omega_{t_1}(v_1)
+2^9\omega_{t_0}(v_0)\omega_\lambda(v_1)
\end{equation}
\[
{}+2^9\omega_\lambda((W_1v)_0)
\left(
\omega_{t_1}((W_1v)_1)
+
\omega_{t_0}((W_1v)_1)
\right).
\]

Suppose next that \(u=1\). For \(\alpha=0\),
\eqref{E:increase-data-r-bump-L1} gives
\[
\omega_\lambda(w_0)
\int_{\mathbb R}\omega_{t_0}(\sqrt2p)\,dp\,
\omega_{t_1}(\sqrt2w_1)
\leq
2^4\omega_\lambda(w_0)\omega_{t_1}(\sqrt2w_1).
\]
By \eqref{E:increase-data-fixed-scaling}, this is at most
\[
2^5\omega_\lambda((W_1v)_0)\omega_{t_1}((W_1v)_1).
\]
For \(\alpha=1\), we use
\[
\omega_{t_0}(\sqrt2p)\leq\omega_{t_0}(p)
\]
and Proposition~\ref{two bump estimate}. Since
\(\lambda\geq t_0\), we obtain
\[
\int_{\mathbb R}\omega_\lambda(w_0+p)\omega_{t_0}(\sqrt2p)\,dp\,
\omega_{t_1}(\sqrt2w_1)
\leq
2^6\omega_\lambda(w_0)\omega_{t_1}(\sqrt2w_1).
\]
By \eqref{E:increase-data-fixed-scaling}, this is at most
\[
2^7\omega_\lambda((W_1v)_0)\omega_{t_1}((W_1v)_1).
\]
Thus, for both \(\alpha\in[2)\),
\begin{equation}
    \label{E:increase-data-u1-alpha}
    \int_{\mathbb R} \omega_\lambda(w_0+\alpha p) \omega_{t_0}(\sqrt2p) \omega_{t_1}(\sqrt2w_1)\,dp \qquad\leq 2^7 \omega_\lambda((W_1v)_0) \omega_{t_1}((W_1v)_1).
\end{equation}

Combining
\eqref{E:increase-data-cancelled-majorant},
\eqref{E:increase-data-loss-u-zero},
\eqref{E:increase-data-loss-u-one},
\eqref{E:increase-data-u0-alpha0},
\eqref{E:increase-data-u0-alpha1}, and
\eqref{E:increase-data-u1-alpha}, we conclude that, when
\(h\geq1\), every resulting product of two \(r\)-bumps occurs with
coefficient at most
\[
2\cdot2^9 C_\rho\,2^{-h/3}
=
2^{10}C_\rho\,2^{-h/3}.
\]
Every such product is in either the \(W_0\)- or the
\(W_1\)-coordinates. Its two scale sequences are among
\[
2^ha,\qquad
t_0,\qquad
t_1,\qquad
\max(t_0,t_1).
\]

It remains to treat \(h=0\). Define the two sequences
\[
d_0(j)=a(j-1),
\qquad
d_1(j)=a(j).
\]
By Lemma~\ref{lem:rho-kernels-reduction},
\eqref{E:increase-data-convolution-coordinates}, and
\eqref{E:increase-data-weaken-exponents},
\begin{equation}
    \label{E:increase-data-h-zero-initial}
    |N_{0,j}(v)| \leq C_\rho \sum_{\epsilon\in[2)} \int_{\mathbb R} \omega_{d_\epsilon(j)}(w_0+p) |M_j(-w_1-p,w_1-p)|\,dp.
\end{equation}

Fix again \((t_0,t_1,u)\in\mathcal P\), and define
\[
r_{\epsilon,\ell}
=
\max(d_\epsilon,t_\ell),
\qquad
\epsilon,\ell\in[2).
\]
If \(u=0\), then Proposition~\ref{bump triangle}, followed by
Proposition~\ref{two bump estimate}, gives
\[
\int_{\mathbb R}
\omega_{d_\epsilon(j)}(w_0+p)
\omega_{t_0(j)}(w_1+p)
\omega_{t_1(j)}(w_1-p)\,dp
\]
\[
{}\leq
2^6\Big(
\omega_{r_{\epsilon,0}(j)}(v_0)
\omega_{t_1(j)}(w_1)
+
\omega_{r_{\epsilon,1}(j)}(v_1)
\omega_{t_0(j)}(w_1)
\Big).
\]
Using
\eqref{E:increase-data-mixed-zero} and
\eqref{E:increase-data-mixed-one}, this is bounded by
\begin{equation}
    \label{E:increase-data-h0-u0}
2^9\omega_{r_{\epsilon,0}(j)}(v_0)
\omega_{t_1(j)}(v_1)
+2^9\omega_{t_0(j)}(v_0)
\omega_{r_{\epsilon,1}(j)}(v_1)
\end{equation}
\[
{}+2^9\omega_{r_{\epsilon,0}(j)}((W_1v)_0)
\omega_{t_1(j)}((W_1v)_1)
+2^9\omega_{r_{\epsilon,1}(j)}((W_1v)_0)
\omega_{t_0(j)}((W_1v)_1).
\]

If \(u=1\), then
\[
\omega_{t_0(j)}(\sqrt2p)
\leq
\omega_{t_0(j)}(p).
\]
Proposition~\ref{two bump estimate} and
\eqref{E:increase-data-fixed-scaling} therefore give
\begin{equation}
    \label{E:increase-data-h0-u1}
    \int_{\mathbb R} \omega_{d_\epsilon(j)}(w_0+p) \omega_{t_0(j)}(\sqrt2p) \omega_{t_1(j)}(\sqrt2w_1)\,dp \qquad\leq 2^7 \omega_{r_{\epsilon,0}(j)}((W_1v)_0) \omega_{t_1(j)}((W_1v)_1).
\end{equation}
Thus every product produced in the case \(h=0\) has coefficient at
most
\[
2^9C_\rho
\leq
2^{10}C_\rho.
\]

We now collect all products obtained above, retaining
multiplicities. Each product has the form
\[
\omega_{\beta^0(j)}((W_uv)_0)
\omega_{\beta^1(j)}((W_uv)_1)
\]
for some \(u\in[2)\) and some
\(\beta=(\beta^0,\beta^1)\in A^2\).
For every such product introduce one label \(b\), and denote the
resulting index set by \(\mathcal B_h\). Set
\[
u_{h,b}=u,
\qquad
\beta_{h,b}=\beta.
\]
Since
\[
2^{10}C_\rho
=
C_{\ref{lem:increase-data-bracket-domination},1},
\]
the preceding estimates prove
\[
|N_{h,j}(v)| \leq C_{\ref{lem:increase-data-bracket-domination},1} 2^{-h/3} \quad\times \sum_{b\in\mathcal B_h} \langle(W_{u_{h,b}}v)_0 \rangle_{(\beta_{h,b}^0(j))}^{7/6} \langle(W_{u_{h,b}}v)_1 \rangle_{(\beta_{h,b}^1(j))}^{7/6}.
\]

We verify the cardinality bound. When \(h\geq1\), a summand of the
pointwise majorant for \(M_j\) with \(u=0\) produces one term from
\(\alpha=0\) and four terms from \(\alpha=1\), hence at most five
terms. A summand with \(u=1\) produces at most two terms. Therefore,
\[
\#\mathcal B_h
\leq
5\#\mathcal P
\leq
25
\]
when \(h\geq1\). When \(h=0\), a summand with \(u=0\) produces four
terms for each \(\epsilon\in[2)\), hence at most eight terms, while
a summand with \(u=1\) produces at most two terms. Thus
\[
\#\mathcal B_0
\leq
8\#\mathcal P
\leq
40.
\]
Consequently,
\[
\#\mathcal B_h
\leq
40
\leq
2^6
=
C_{\ref{lem:increase-data-bracket-domination},0}.
\]

Finally, we verify the distance bound. If \(h\geq1\), every scale
sequence occurring above is one of
\[
2^ha,\qquad
t_0,\qquad
t_1,\qquad
\max(t_0,t_1).
\]
By Proposition~\ref{Properties of distance of sequences},
\[
\dist(a,2^ha)\leq h.
\]
By hypothesis,
\[
\dist(a,t_\ell)\leq1.
\]
Moreover, since
\[
a(j-1)\leq t_\ell(j)\leq a(j+1),
\]
we also have
\[
a(j-1)
\leq
\max(t_0(j),t_1(j))
\leq
a(j+1),
\]
and hence
\[
\dist(a,\max(t_0,t_1))\leq1.
\]

When \(h=0\), every scale sequence occurring above is one of
\[
t_0,\qquad
t_1,\qquad
\max(d_\epsilon,t_\ell).
\]
The sequence \(d_0=a(\,\cdot-1)\) satisfies
\[
\dist(a,d_0)\leq1,
\]
while \(d_1=a\). Since both \(d_\epsilon\) and \(t_\ell\) belong to
\(B_{\dist}(a,1)\), their pointwise maximum also belongs to
\(B_{\dist}(a,1)\). All these sequences belong to \(A\) by
Proposition~\ref{Operations on spaced sequences}.

It follows in all cases that
\[
\dist(a,\beta_{h,b}^\ell)
\leq
1+h
\]
Finally, by symmetry and the triangle inequality for \(\dist\),
\[
\Delta(\beta_{h,b}) = \dist(\beta_{h,b}^0,\beta_{h,b}^1) \leq \dist(\beta_{h,b}^0,a) + \dist(a,\beta_{h,b}^1) \leq 2(1+h),
\]
which is
\eqref{E:increase-data-beta-pair-distance}.
\end{proof}

\begin{lemma}[constant $C_{\ref{lem:increase-data-bracket-domination},1}$ \auto]\label{constant increase data bracket domination}\uses{lem:increase-data-bracket-domination,lem:rho-kernels-reduction,constant rho kernels reduction}\usesdefs{}
\lean{Auto.constantIncreaseDataBracketDomination}
\leanok
\begin{equation}\label{constant increase data bracket domination bound}
C_{\ref{lem:increase-data-bracket-domination},1}
<\tfrac35 2^{76}<2^{76}.
\end{equation}
\end{lemma}
\begin{proof}[Proof \auto]
Use $C_{\ref{lem:increase-data-bracket-domination},1}=2^{10}C_{\ref{lem:rho-kernels-reduction}}$ and Lemma \ref{constant rho kernels reduction}.
\end{proof}

\begin{proposition}[Gaussian domination - reduction variant]
\label{lem:increase-data-Gaussian-expansion}\uses{Gaussian domination,lem:increase-data-bracket-domination,Operations on spaced sequences,Properties of distance of sequences}\usesdefs{gaussian,bracket bump,multiplicatively spaced monotone sequences,Distance of spaced sequences,geometric parameters,auto:unitary-matrices-definition,2D Gaussians}
\lean{Auto.increaseDataGaussianExpansion}
\leanok
Let the notation and conclusions of Proposition~
\ref{lem:increase-data-bracket-domination} hold. For
\(b\in\mathcal B_h\) and
\(m=(m_0,m_1)\in\mathbb N^2\), define
\begin{equation}\label{auto:Gaussian-domination-rescaled-pairs}
p_{h,b,m}
=
\left(
2^{m_0}\beta_{h,b}^0,
2^{m_1}\beta_{h,b}^1
\right)\in A^2.
\end{equation}
Then
\begin{equation}
    \label{E:increase-data-p-distance}
    \Delta(p_{h,b,m})
    \leq
    2(1+h)
    +
    |m|,
\end{equation}
and
\begin{equation}
    \label{E:increase-data-Gaussian-majorant}
    |N_{h,j}(v)| \leq C_{\ref{lem:increase-data-Gaussian-expansion}} 2^{-h/3} \sum_{b\in\mathcal B_h} \sum_{m\in\mathbb N^2} 2^{-|m|/6} G_{p_{h,b,m}(j),u_{h,b}}(v).
\end{equation}
where $ C_{\ref{lem:increase-data-Gaussian-expansion}}
    =
    2^3C_{\ref{Gaussian domination}}^2C_{\ref{lem:increase-data-bracket-domination},1}.$
The   series on the right-hand side converges for every
\(v\in\mathbb R^2\) and in \(L^1(\mathbb R^2)\).
\end{proposition}
\begin{proof}[Proof \auto]
Set
\[
C_G=C_{\ref{Gaussian domination}}
\qquad\text{and}\qquad
C_B=C_{\ref{lem:increase-data-bracket-domination},1}.
\]
We prove the proposition with
\begin{equation}
    \label{E:increase-data-Gaussian-expansion-constant}
    C_{\ref{lem:increase-data-Gaussian-expansion}}
    =
    2^3C_G^2C_B.
\end{equation}

Fix \(h\in\mathbb N\), \(b\in\mathcal B_h\), and
\(m=(m_0,m_1)\in\mathbb N^2\). Since
\(\beta_{h,b}^\ell\in A\) for \(\ell\in[2)\), Proposition~
\ref{Operations on spaced sequences} implies that
\[
2^{m_\ell}\beta_{h,b}^\ell\in A.
\]
Consequently,
\[
p_{h,b,m}
=
\left(
2^{m_0}\beta_{h,b}^0,
2^{m_1}\beta_{h,b}^1
\right)
\in A^2.
\]

We first prove the distance estimate. By the triangle inequality
for \(\dist\), followed by Proposition~
\ref{Properties of distance of sequences},
\[
\Delta(p_{h,b,m})
=
\dist\left(
2^{m_0}\beta_{h,b}^0,
2^{m_1}\beta_{h,b}^1
\right).
\]
\[
\Delta(p_{h,b,m})
\leq
\dist\left(
2^{m_0}\beta_{h,b}^0,
2^{m_0}\beta_{h,b}^1
\right)
+
\dist\left(
2^{m_0}\beta_{h,b}^1,
2^{m_1}\beta_{h,b}^1
\right).
\]
\[
\Delta(p_{h,b,m})
=
\dist\left(
\beta_{h,b}^0,
\beta_{h,b}^1
\right)
+
\dist\left(
\beta_{h,b}^1,
2^{m_1-m_0}\beta_{h,b}^1
\right).
\]
\[
\Delta(p_{h,b,m})
\leq
\Delta(\beta_{h,b})
+
|m_1-m_0|
\leq
\Delta(\beta_{h,b})
+
m_0+m_1.
\]
Here the equality in the fourth line follows by multiplying both
sequences in the first distance by \(2^{-m_0}\), and the subsequent
inequality follows from
\[
\dist(c,2^qc)\leq |q|,
\qquad q\in\mathbb Z.
\]
Since \(|m|=m_0+m_1\), estimate
\eqref{E:increase-data-beta-pair-distance} now gives
\[
\Delta(p_{h,b,m})
\leq
2(1+h)
+
|m|,
\]
which is \eqref{E:increase-data-p-distance}.

We next prove the Gaussian majorant. Fix \(j\in\mathbb Z\) and
\(v\in\mathbb R^2\). Proposition~\ref{Gaussian domination}, applied
with
\[
N=\tfrac76,
\qquad
s=\beta_{h,b}^\ell(j),
\qquad
x=(W_{u_{h,b}}v)_\ell,
\]
gives, for each \(\ell\in[2)\),
\begin{equation}
    \label{E:increase-data-one-coordinate-Gaussian}
    \left\langle (W_{u_{h,b}}v)_\ell \right\rangle_{(\beta_{h,b}^\ell(j))}^{7/6} \qquad\leq C_G2^{7/6} \sum_{m_\ell\in\mathbb N} 2^{-m_\ell/6} \g_{(2^{m_\ell}\beta_{h,b}^\ell(j))} \left((W_{u_{h,b}}v)_\ell\right).
\end{equation}
All terms in these series are nonnegative. Multiplying
\eqref{E:increase-data-one-coordinate-Gaussian} for
\(\ell=0\) and \(\ell=1\), we obtain
The resulting upper bound is
\[
C_G^2 2^{7/3}
\sum_{m\in\mathbb N^2}
2^{-|m|/6}
\g_{(2^{m_0}\beta_{h,b}^0(j))}
\left((W_{u_{h,b}}v)_0\right)
\g_{(2^{m_1}\beta_{h,b}^1(j))}
\left((W_{u_{h,b}}v)_1\right).
\]
By Definition~\ref{2D Gaussians}, the product of Gaussians on the
right-hand side is
\[
G_{p_{h,b,m}(j),u_{h,b}}(v).
\]
Since
\[
2^{7/3}\leq2^3,
\]
it follows that
\begin{equation}
    \label{E:increase-data-two-coordinate-Gaussian}
    \left\langle (W_{u_{h,b}}v)_0 \right\rangle_{(\beta_{h,b}^0(j))}^{7/6} \left\langle (W_{u_{h,b}}v)_1 \right\rangle_{(\beta_{h,b}^1(j))}^{7/6} \qquad\leq 2^3C_G^2 \sum_{m\in\mathbb N^2} 2^{-|m|/6} G_{p_{h,b,m}(j),u_{h,b}}(v).
\end{equation}

Applying \eqref{E:increase-data-two-coordinate-Gaussian} to every
summand in
\eqref{E:increase-data-bracket-majorant}, we conclude that
\[
|N_{h,j}(v)| \leq 2^3C_G^2C_B 2^{-h/3} \quad\times \sum_{b\in\mathcal B_h} \sum_{m\in\mathbb N^2} 2^{-|m|/6} G_{p_{h,b,m}(j),u_{h,b}}(v).
\]
In view of
\eqref{E:increase-data-Gaussian-expansion-constant}, this is
\eqref{E:increase-data-Gaussian-majorant}.

It remains to verify the two convergence assertions. Fix
\(h\in\mathbb N\) and \(j\in\mathbb Z\). Since
\[
0\leq\g_{(s)}(x)
=
s^{-1}\g(s^{-1}x)
\leq s^{-1},
\]
we have, for every \(b\in\mathcal B_h\), \(m\in\mathbb N^2\),
and \(v\in\mathbb R^2\),
\[
G_{p_{h,b,m}(j),u_{h,b}}(v) \leq \tfrac{1}{ 2^{m_0}\beta_{h,b}^0(j) } \tfrac{1}{ 2^{m_1}\beta_{h,b}^1(j) } = \tfrac{2^{-|m|}}{ \beta_{h,b}^0(j)\beta_{h,b}^1(j) }.
\]
Therefore,
\[
\sum_{m\in\mathbb N^2}
2^{-|m|/6}
G_{p_{h,b,m}(j),u_{h,b}}(v)
\leq
\frac{1}{
\beta_{h,b}^0(j)\beta_{h,b}^1(j)
}
\sum_{m\in\mathbb N^2}
2^{-7|m|/6}.
\]
The sum on the right is the square of a geometric series, and hence
\[
\sum_{m\in\mathbb N^2}2^{-7|m|/6}
=
\left(
\sum_{q\in\mathbb N}2^{-7q/6}
\right)^2
\leq2^2.
\]
Consequently,
\[
\sum_{m\in\mathbb N^2}
2^{-|m|/6}
G_{p_{h,b,m}(j),u_{h,b}}(v)
\leq
\frac{2^2}{
\beta_{h,b}^0(j)\beta_{h,b}^1(j)
}.
\]
where in the final line we used
\[
2^{-7q/6}\leq2^{-q}
\qquad\text{and}\qquad
\sum_{q\in\mathbb N}2^{-q}=2.
\]
Since \(\mathcal B_h\) is finite, the full nonnegative series
converges for every \(v\in\mathbb R^2\).

Finally, since \(W_{u_{h,b}}\) is unitary and every rescaled
one-dimensional Gaussian has integral one,
\[
\left\|
G_{p_{h,b,m}(j),u_{h,b}}
\right\|_1
=
1.
\]
Consequently,
\[
\sum_{b\in\mathcal B_h} \sum_{m\in\mathbb N^2} \left\| 2^{-|m|/6} G_{p_{h,b,m}(j),u_{h,b}} \right\|_1 \qquad= \#\mathcal B_h \left( \sum_{q\in\mathbb N}2^{-q/6} \right)^2.
\]
Writing \(q=6d+r\), where \(d\in\mathbb N\) and \(r\in[6)\), gives
\[
\sum_{q\in\mathbb N}2^{-q/6} = \sum_{r\in[6)} 2^{-r/6} \sum_{d\in\mathbb N}2^{-d} \leq 6\cdot2 = 12 \leq2^4.
\]
Thus
\[
\sum_{b\in\mathcal B_h} \sum_{m\in\mathbb N^2} \left\| 2^{-|m|/6} G_{p_{h,b,m}(j),u_{h,b}} \right\|_1 \qquad\leq 2^8\#\mathcal B_h \leq 2^8 C_{\ref{lem:increase-data-bracket-domination},0} < \infty.
\]
The series therefore converges absolutely in
\(L^1(\mathbb R^2)\), completing the proof.
\end{proof}

\begin{lemma}[constant $C_{\ref{lem:increase-data-Gaussian-expansion}}$ \auto]\label{constant increase data Gaussian expansion}\uses{lem:increase-data-bracket-domination,constant increase data bracket domination,constant rho kernels reduction,lem:rho-kernels-reduction,Gaussian domination,lem:increase-data-Gaussian-expansion}\usesdefs{}
\lean{Auto.constantIncreaseDataGaussianExpansion}
\leanok
\begin{equation}\label{constant increase data Gaussian expansion bound}
C_{\ref{lem:increase-data-Gaussian-expansion}}
<\tfrac{123}{128}2^{91}<2^{91}.
\end{equation}
\end{lemma}
\begin{proof}[Proof \auto]
Since $C_{\ref{Gaussian domination}}=e^\pi$, the estimates $e<3$ and $\pi<4$ give
\[
e^{2\pi}<e^8<3^8=6561.
\]
Proposition \ref{lem:increase-data-bracket-domination} gives
$C_{\ref{lem:increase-data-bracket-domination},1}=2^{10}C_{\ref{lem:rho-kernels-reduction}}$.
Using the exact upper bound from Lemma \ref{constant rho kernels reduction}, and hence the bound in Lemma \ref{constant increase data bracket domination}, the defining product is smaller than
\[
2^3\cdot6561\cdot2^{10}\cdot44114430494276321280
<\tfrac{123}{128}2^{91}.
\]
\end{proof}

\begin{lemma}
\label{lem:N-reduction}\uses{lem:rho-kernels-reduction,convolution vector,lem:increase-data-bracket-domination}\usesdefs{bracket bump,auto:Wiener-space-definition,auto:convolution-along-vector-definition,multiplicatively spaced monotone sequences,auto:unitary-matrices-definition} Let $a\in A$ and let $\mathcal P$, $(M_j)_{j\in \Z}$, and $(N_{h,j})_{h\in \N,j\in \Z}$,  be as in Proposition~\ref{lem:increase-data-bracket-domination}.
\lean{Auto.nReduction}
\leanok
Then for each $h\in \N$ and $j\in \Z$,   $N_{h,j}\in W_0(\R^2)$ and
\begin{equation}
    \label{E:diagonal-band-reduction-N-L1}
    \|N_{h,j}\|_1\leq  C_{\ref{lem:N-reduction}}
\end{equation}
with $C_{\ref{lem:N-reduction}} = 2^9C_{\ref{lem:rho-kernels-reduction}}$
\end{lemma}
\begin{proof}
Let \(\rho_{h,j}\) be as in Lemma \ref{lem:rho-kernels-reduction}.
Then,
\[
N_{h,j}=M_j\ast_{(1,1)}\rho_{h,j}.
\]
Indeed, this follows from
\[
\widehat{M_j\ast_{(1,1)}\rho_{h,j}}(\xi,\eta)
=
\widehat M_j(\xi,\eta)
\widehat{\rho_{h,j}}(\xi+\eta)
=
\widehat{N_{h,j}}(\xi,\eta).
\]
using Proposition~\ref{convolution vector}.

By the pointwise bound on \(M_j\), the fact that each
\(W_u\) is measure preserving, and
\[
\int_{\mathbb R}\langle x\rangle_{(t)}^{3/2}\,dx=4,
\]
we have
\begin{equation}
    \label{E:diagonal-band-reduction-M-L1}
    \|M_j\|_1
    \leq
    5\cdot 4^2
    =
    80.
\end{equation}
 To estimate  $\rho_{h,j}$  we   apply  Proposition \ref{lem:rho-kernels-reduction},
which gives
\[
|\rho_{h,j}(x)|
\leq
C_{\ref{lem:rho-kernels-reduction}}
\left(
\langle x\rangle_{(\lambda_{h,j}^-)}^2
+
\langle x\rangle_{(\lambda_{h,j}^+)}^2
\right).
\]
Since
\[
\int_{\mathbb R}\langle x\rangle_{(t)}^2\,dx=2,
\]
it follows that
\[
\|\rho_{h,j}\|_1
\leq
4C_{\ref{lem:rho-kernels-reduction}}
\]
Combining this last estimate with
\eqref{E:diagonal-band-reduction-M-L1} and the \(L^1\) convolution
inequality proves
\[
\|N_{h,j}\|_1
\leq
80\cdot 4 C_{\ref{lem:rho-kernels-reduction}}
\leq
C_{\ref{lem:N-reduction}}.
\]
\end{proof}

\begin{lemma}[constant $C_{\ref{lem:N-reduction}}$ \auto]\label{constant N reduction}\uses{lem:N-reduction,lem:rho-kernels-reduction,constant rho kernels reduction}\usesdefs{}
\lean{Auto.constantNReduction}
\leanok
\begin{equation}\label{constant N reduction bound}
C_{\ref{lem:N-reduction}}
<\tfrac35 2^{75}<2^{75}.
\end{equation}
\end{lemma}
\begin{proof}[Proof \auto]
Use $C_{\ref{lem:N-reduction}}=2^9C_{\ref{lem:rho-kernels-reduction}}$ and Lemma \ref{constant rho kernels reduction}.
\end{proof}

\begin{proposition}[increase data - reduction variant]
\label{P:increase-data-reduction}\uses{lem:N-reduction,square root Gaussian difference W0,tensor Wiener,lem:increase-data-Gaussian-expansion,induct positive terms theorem,Monotonicity K,prism BL inequality,lem:increase-data-bracket-domination}\usesdefs{auto:Wiener-space-definition,multiplicatively spaced monotone sequences,geometric parameters,kernel sequences,2D Gaussians,sandwich kernel,square root Gaussian difference,s multiplier}
\lean{Auto.increaseDataReduction}
\leanok
Let $a\in A$ and let $\mathcal P$, $(M_j)_{j\in \Z}$, and $(N_{h,j})_{h\in\N,j\in\Z}$ be as in Proposition~\ref{lem:increase-data-bracket-domination}.
 For $j\in \Z$ and $y\in (\R^2)^2$ let
    \begin{equation}\label{auto:increase-data-majorant-kernel}\tilde{M}_{h,j}(y) = |N_{h,j}(y_0)| (\sigma_{h,j}^{\otimes2})(y_1). \end{equation}

Then with $\mathbf{\tilde{M}}_h=(\tilde{M}_{h,j})_{j\in \Z}$,
    \begin{equation}\label{auto:increase-data-majorant-bound}\|\mathbf{\tilde{M}}_h\|_{\rm M(2)}\le C_{\ref{P:increase-data-reduction}} 2^{-h/4}\end{equation}
    where $C_{\ref{P:increase-data-reduction}} = 2^{22}C_{\ref{induct positive terms theorem}}C_{\ref{lem:increase-data-Gaussian-expansion}}C_{\ref{lem:increase-data-bracket-domination},0}$

\end{proposition}

\begin{proof}
By Lemma \ref{lem:N-reduction}, \(N_{h,j}\in W_0(\mathbb R^2)\) and consequently $|N_{h,j}|\in W_0(\mathbb R^2)$. Together with Propositions~\ref{square root Gaussian difference W0}
and~\ref{tensor Wiener}, this gives
\[
\tilde M_{h,j}\in W_0((\mathbb R^2)^2).
\]
Let \(\mathcal B_h\), \(u_{h,b}\), and \(p_{h,b,m}\) be given by
Lemma~\ref{lem:increase-data-Gaussian-expansion}. For
\(b\in\mathcal B_h\) and \(m\in\mathbb N^2\), define
\[
\alpha_h
=
\left(
2^{h-1/2}a,
2^{h-1/2}a
\right)\in A^2
\]
and
\[
\gamma_{h,b,m}
=
\left(
2,
(u_{h,b},0),
(p_{h,b,m},\alpha_h)
\right)\in\Gamma.
\]

Since the second orientation is zero, Definition~
\ref{s multiplier} gives
\[
(s_{\gamma_{h,b,m}})_{1,j}
=
s\Big(
\big(
(2^{h-1/2}a)^2+
(2^{h-1/2}a)^2
\big )^{1/2},
j
\Big)
=
s(2^ha,j)
=
\sigma_{h,j}.
\]
Consequently,
\begin{equation}
    \label{E:increase-data-positive-term-recognition}
    \bigl(
    \M(
    \gamma_{h,b,m},
    s_{\gamma_{h,b,m}}\otimes s_{\gamma_{h,b,m}},
    1
    )
    \bigr)_j(y) =
    G_{p_{h,b,m}(j),u_{h,b}}(y_0)
    \bigl(\sigma_{h,j}^{\otimes2}\bigr)(y_1).
\end{equation}
Moreover,
$\Delta(\alpha_h)=0$
and therefore
\[ \Delta_{\gamma_{h,b,m}} =
1+\Delta(p_{h,b,m}). \]

Set $
q_n=2-2^{3-n}\in[0,2).$
By Theorem~\ref{induct positive terms theorem}, applied with
\(i=1\),
\[
\left\|
\M(
\gamma_{h,b,m},
s_{\gamma_{h,b,m}}\otimes s_{\gamma_{h,b,m}},
1
)
\right\|_{\mathrm M(2)}
 \leq
C_{\ref{induct positive terms theorem}}
\bigl(
1+\Delta(p_{h,b,m})
\bigr)^{q_n}.
\]

Using Lemma~\ref{lem:increase-data-Gaussian-expansion},
Proposition~\ref{Monotonicity K}, we obtain
\[
\|\tilde{\mathbf M}_h\|_{\mathrm M(2)}
\leq
C_{\ref{lem:increase-data-Gaussian-expansion}}
2^{-h/3}
\sum_{b\in\mathcal B_h}
\sum_{m\in\mathbb N^2}
2^{-|m|/6}
\left\|
\M(
\gamma_{h,b,m},
s_{\gamma_{h,b,m}}\otimes s_{\gamma_{h,b,m}},
1
)
\right\|_{\mathrm M(2)}.
\]
(first sum partial
sums of the Gaussian series, then pass to the full series by Proposition~
\ref{prism BL inequality} and dominated convergence). Hence, together with the estimate on $\#\mathcal{B}$ and $(1+\Delta(p_{h,b,m}))^{q_n}\le 2^{4}(1+h+|m|)^2$,
\[
\|\tilde{\mathbf M}_h\|_{\mathrm M(2)}
\leq
C_{\ref{induct positive terms theorem}}C_{\ref{lem:increase-data-Gaussian-expansion}}C_{\ref{lem:increase-data-bracket-domination},0}2^4
2^{-h/3}
\sum_{m\in\mathbb N^2}
2^{-|m|/6}
(1+h+|m|)^{2}.
\]
Since
\[\sup_{h\in \N}2^{-h/12}\sum_{m\in \N^2}2^{-|m|/6}(1+h+|m|)^2\le 2^{18}\]

it follows that
\[
\|\tilde{\mathbf M}_h\|_{\mathrm M(2)}
\leq
C_{\ref{P:increase-data-reduction}}2^{-h/4},
\]
as desired.
\end{proof}

\begin{lemma}[constant $C_{\ref{P:increase-data-reduction}}$ \auto]\label{constant increase data reduction}\uses{constant induct positive terms theorem,constant increase data Gaussian expansion,constant rho kernels reduction,lem:increase-data-bracket-domination,P:increase-data-reduction}\usesdefs{}
\lean{Auto.constantIncreaseDataReduction}
\leanok
\begin{equation}\label{constant increase data reduction bound}
C_{\ref{P:increase-data-reduction}}
<\tfrac{31}{32}2^{477}<2^{477}.
\end{equation}
\end{lemma}
\begin{proof}[Proof \auto]
Use Lemmas \ref{constant induct positive terms theorem}, \ref{constant increase data Gaussian expansion}, and \ref{constant rho kernels reduction} in the defining product, together with $C_{\ref{lem:increase-data-bracket-domination},0}=2^6$ and $e^\pi<56$. This gives
\[
C_{\ref{P:increase-data-reduction}}
<2^{22}\cdot2^{359}\cdot\bigl(2^3\cdot56^2\cdot(2^{10}\cdot44114430494276321280)\bigr)\cdot2^6
<\tfrac{31}{32}2^{477}.
\]
\end{proof}

\subsection{On-diagonal from off-diagonal estimates}\label{sec:on-diag}
\begin{proposition}[diagonal band - reduction variant]
\label{P:diagonal-band-reduction}\uses{P:increase-data-reduction,convolution vector,Cauchy-Schwarz at k,lem:N-reduction,Cauchy-Schwarz at n-1}\usesdefs{auto:Wiener-space-definition,auto:convolution-along-vector-definition,normalized function tuples,auto:prism-form-definition,multiplicatively spaced monotone sequences,standard bump,kernel sequences,square root Gaussian difference}
\lean{Auto.diagonalBandReduction}
\leanok
Let $a\in A$ and let $\mathcal P$, $(M_j)_{j\in \Z}$ be as in Proposition~\ref{P:increase-data-reduction}. For $h\in\N$ with $h\ge1$ and $j\in\Z$, define
\begin{equation}\label{auto:diagonal-band-positive-scale}
L_{h,j}= M_j \ast_{(1,1)} (\Phi_{(2^{h-1}a(j))}-\Phi_{(2^{h}a(j))}),
\end{equation}
and for $j\in \Z$,
\begin{equation}\label{auto:diagonal-band-zero-scale}
L_{0,j}=M_j \ast_{(1,1)} (\Phi_{(a(j-1))}-\Phi_{(a(j))}).
\end{equation}
Let $\mathbf{L}_h=(L_{h,j})_{j\in \Z}$. Then
\begin{equation}\label{auto:diagonal-band-sum-bound}
\sum_{h\in \N} \|\mathbf{L}_h\|_{\rm M(1)} \leq C_{\ref{P:diagonal-band-reduction}}
\end{equation}
with  $C_{\ref{P:diagonal-band-reduction}} = 2^4(
    C_{\ref{lem:N-reduction}}
    C_{\ref{P:increase-data-reduction}}
  )^{1/2} + 2^3 C_{\ref{P:increase-data-reduction}}$.
\end{proposition}
\begin{proof}[Proof \auto]
Assume first that $n>2$.\phantomsection\label{P:diagonal-band-reduction case 1}
Since $n>2$ we have  $\nu=1$.  Let $\sigma_{h,j}$ and $N_{h,j}$ be as  Proposition~\ref{P:increase-data-reduction}.
We have
\begin{equation}
    \label{E:diagonal-band-reduction-factorization}
    L_{h,j}
    =
    N_{h,j}\ast_{(1,1)}\sigma_{h,j}.
\end{equation}
Indeed, by Proposition~\ref{convolution vector},
\[\widehat{
N_{h,j}\ast_{(1,1)}\sigma_{h,j}
}(\xi,\eta)=
\widehat{N_{h,j}}(\xi,\eta)
\widehat{\sigma_{h,j}}(\xi+\eta)=
\widehat{L_{h,j}}(\xi,\eta).
\]

Fix
\(J\in\mathbb N\) with \(J>0\) and \(\mathbf F\in\mathfrak F\).
 Proposition
\ref{Cauchy-Schwarz at k}  with \(k=1\), applied to
\[
\rho_j=C_{\ref{lem:N-reduction}}^{-1}N_{h,j},
\qquad
\varphi_j=\sigma_{h,j},
\]
gives
\[
\Big|
\Lambda_1(\sum_{j\in[J)}L_{h,j})(\mathbf F)
\Big|^2
\leq C_{\ref{lem:N-reduction}} J
\Big|
\Lambda_2(
\sum_{j\in[J)}\tilde M_{h,j}
)(\mathbf F)
\Big|
\]
Here
\[
\tilde M_{h,j}(y)
=
|N_{h,j}(y_0)|
\bigl(\sigma_{h,j}^{\otimes2}\bigr)(y_1),
\]
as in Proposition~\ref{P:increase-data-reduction}.

 Proposition~\ref{P:increase-data-reduction} and the
definition of the \({\rm M}(2)\)-norm give
\[
\Big |
\Lambda_2 (
\sum_{j\in[J)}\tilde M_{h,j}
)(\mathbf F)
\Big |
\leq
C_{\ref{P:increase-data-reduction}}
2^{-h/4}
J^{1-2^{3-n}}.
\]
It follows that
\[
\Big|
\Lambda_1(\sum_{j\in[J)}L_{h,j})(\mathbf F)
\Big|
\leq
\bigl(
C_{\ref{lem:N-reduction}}
C_{\ref{P:increase-data-reduction}}
\bigr)^{1/2}
2^{-h/8}
J^{1-2^{2-n}},
\]
since
$
\frac12+\frac12(1-2^{3-n})
=
1-2^{2-n}.
$
Taking the supremum over \(J\) and \(\mathbf F\) therefore gives
\begin{equation}
    \label{E:diagonal-band-reduction-n-greater-2}
    \|\mathbf L_h\|_{\mathrm M(1)}
    \leq
    \bigl(
    C_{\ref{lem:N-reduction}}
    C_{\ref{P:increase-data-reduction}}
    \bigr)^{1/2}
    2^{-h/8}.
\end{equation}
Summing over $h\in \N$ finishes the proof.

Assume now that $n=2$.\phantomsection\label{P:diagonal-band-reduction case 2}
Now $\nu=2$ and
    \[
L_{h,j}
=
N_{h,j}\ast_{(1,1)}
(\sigma_{h,j}\ast\sigma_{h,j})
\]
where    $\sigma_{h,j}$   and $N_{h,j}$ are  as   in Proposition~\ref{P:increase-data-reduction}.
Proposition
\ref{Cauchy-Schwarz at n-1} gives
\[
\Big|
\Lambda_1(\sum_{j\in[J)}L_{h,j})(\mathbf F)
\Big|
\leq
\sup_{\widetilde{\mathbf F}\in\mathfrak F}
\Big|
\Lambda_2(
\sum_{j\in[J)}\tilde M_{h,j}
)(\tilde{\mathbf F})
\Big|
\]
with $\tilde M_{h,j}$ as in Proposition~\ref{P:increase-data-reduction}.
Proposition \ref{P:increase-data-reduction} now gives
\begin{equation}
    \label{E:diagonal-band-reduction-n-equals-2}
    \|\mathbf L_h\|_{\mathrm M(1)}
    \leq
    C_{\ref{P:increase-data-reduction}}2^{-h/4}.
\end{equation}
Summing in $h\in \N$ finishes the proof.
\end{proof}

\begin{lemma}[constant $C_{\ref{P:diagonal-band-reduction}}$ \auto]\label{constant diagonal band reduction}\uses{constant N reduction,constant increase data reduction,P:diagonal-band-reduction}\usesdefs{}
\lean{Auto.constantDiagonalBandReduction}
\leanok
\begin{equation}\label{constant diagonal band reduction bound}
C_{\ref{P:diagonal-band-reduction}}
<\tfrac{63}{64}2^{480}<2^{480}.
\end{equation}
\end{lemma}
\begin{proof}[Proof \auto]
Using the defining sum in Proposition \ref{P:diagonal-band-reduction}, Lemmas \ref{constant N reduction} and \ref{constant increase data reduction} bound its first term by
\[
2^4\left(\tfrac35 2^{75}\tfrac{31}{32}2^{477}\right)^{1/2}
<2^{280}.
\]
The second term is smaller than
\[
\tfrac{31}{32}2^{480}.
\]
Since $2^{280}<2^{-6}2^{480}$, their sum is smaller than
\[
\tfrac{63}{64}2^{480}.
\]
\end{proof}

\begin{lemma}\label{lem:L1-reduction}\uses{lem:affine-diagonal-cancellation-reduction}\usesdefs{auto:Wiener-space-definition,auto:convolution-along-vector-definition,multiplicatively spaced monotone sequences,standard bump}
\lean{Auto.lOneReduction}
\leanok

Let $a\in A$ and suppose that $(M_j)_{j\in\Z}\subset W_0(\R^2)$ satisfies for all $\xi \in \R$ and $j\in \Z$
\begin{equation}\label{auto:vanishing-diagonal-assumption}
\widehat{M_j}(\xi,-\xi)=0.
\end{equation}
Then
\begin{equation}\label{auto:vanishing-diagonal-large-scale-limit}    \lim_{t\to\infty}
    \| M_j\ast_{(1,1)}\Phi_{(t)} \|_1
    =0
    \end{equation}
\end{lemma}

\begin{proof}

By Lemma \ref{lem:affine-diagonal-cancellation-reduction},
\begin{equation}
    \label{E:vanishing-diagonal-reduction-line-cancellation}
    \int_{\mathbb R}M_j(x+q,q)\,dq=0
\end{equation}
for every \(x\in\mathbb R\).
Let
\[
R_{t,j}=M_j\ast_{(1,1)}\Phi_{(t)}.
\]
Using \eqref{E:vanishing-diagonal-reduction-line-cancellation} and
changing variables \(q=y-p\), we obtain
\[
R_{t,j}(x+y,y) = \int_{\mathbb R} \Phi_{(t)}(y-q)M_j(x+q,q)\,dq = \int_{\mathbb R} \bigl(\Phi_{(t)}(y-q)-\Phi_{(t)}(y)\bigr) M_j(x+q,q)\,dq.
\]
It follows from the triangle inequality and Fubini's theorem that
    \[
    \|R_{t,j}\|_1
    \le
    \int_{\mathbb R^2}
    \omega(q/t)\,
    |M_j(x+q,q)|\,dx\,dq,
    \]
    where
    \[
    \omega(s)
    =
    \int_{\mathbb R}|\Phi(z-s)-\Phi(z)|\,dz.
    \]
    By continuity of translations and the $L^1$ norm,
    \(\omega(q/t)\to0\) as \(t\to\infty\) for   \(q\in\mathbb R\).
    Also,
    \[
    0\le \omega(q/t)\le 2\|\Phi\|_1.
    \]
The claim now follows by the
    dominated convergence theorem.
\end{proof}

\begin{proposition}[vanishing diagonal - reduction variant]
\label{P:vanishing-diagonal-reduction}\uses{P:diagonal-band-reduction,prism BL inequality,M to K,lem:L1-reduction,P:increase-data-reduction}\usesdefs{auto:Wiener-space-definition,auto:convolution-along-vector-definition,normalized function tuples,auto:prism-form-definition,multiplicatively spaced monotone sequences,standard bump,kernel sequences}
\lean{Auto.vanishingDiagonalReduction}
\leanok
Let $a\in A$ and let $\mathcal P$, $(M_j)_{j\in \Z}$ be as in Proposition~\ref{P:increase-data-reduction}. In addition, assume that $\widehat{M_j}(\xi,\eta)$ is supported in the set where $|\xi+\eta| \leq 2^{-1} (a(j-1))^{-1}$.
Let $\mathbf{M}=(M_j)_{j\in \Z}$. Then
\begin{equation}\label{auto:vanishing-diagonal-M-bound}
\|\mathbf{M}\|_{\rm M(1)} \leq C_{\ref{P:diagonal-band-reduction}}
\end{equation}
\end{proposition}

\begin{proof}Since    $|\xi+\eta|\le 2^{-1} (a(j-1))^{-1}$ and   $\widehat{\Phi}=1$ on $[-1/2,1/2]$,
\[M_j*_{(1,1)}\Phi_{(a(j-1))}=M_j\]
Let $L_{h,j}$ be as in Proposition \ref{P:diagonal-band-reduction}. For $H\in \N$ we telescope
\[\sum_{h=0}^H L_{h,j} = M_j*_{(1,1)} \Big( \Phi_{(a(j-1))}-\Phi_{(a(j))} + \sum_{h=1}^{H}
( \Phi_{(2^{h-1}a(j))} - \Phi_{(2^ha(j))} ) \Big)\]
\[= M_j*_{(1,1)} ( \Phi_{(a(j-1))} - \Phi_{(2^Ha(j))} )\]
\[= M_j-M_j*_{(1,1)}\Phi_{(2^Ha(j))} =: M_j-R_{H,j}.\]
Let $J\ge 1$ and $\F\in \mathfrak{F}$. By linearity and the triangle inequality,

\[|\Lambda_1 (\sum_{j\in [J)} M_j )(\F)| \le  \sum_{h=0}^H |\Lambda_1( \sum_{j\in [J)}L_{h,j})(\F)| + |\Lambda_1( \sum_{j\in [J)}R_{H,j})(\F)| \]
By Propositions  \ref{prism BL inequality} and \ref{M to K},
\[|\Lambda_1( \sum_{j\in [J)}R_{H,j})(\F)| \le \|\sum_{j\in [J)}R_{H,j}\|_1 \le \sum_{j\in [J)}\|R_{H,j}\|_1,\]
which tends to zero as $H\to \infty$ by Proposition \ref{lem:L1-reduction}.

Therefore, setting $c_J = \min(1,J^{-1+2^{2-n}})$,
\[c_J |\Lambda_1 (\sum_{j\in [J)} M_j )(\F)| \le  \sum_{h\in\N} c_J |\Lambda_1( \sum_{j\in [J)}L_{h,j})(\F)|\le \sum_{h\in \N} \|\mathbf{L}_h\|_{\rm M(1)} \le C_{\ref{P:diagonal-band-reduction}}\]
where the last inequality is by Proposition
\ref{P:diagonal-band-reduction}. Taking the supremum over
$J$ and $\mathbf F\in\mathfrak F$ yields the claim.
\end{proof}

Fix a universal pair $(\phi_0,\phi_1)$. For $\lambda>0$ and $i\in[2)$ write
\begin{equation}\label{auto:rescaled-window-continuous-notation} \phi_{i,\lambda} = (\phi_i)_{(\lambda)}. \end{equation}

\begin{proposition}\label{P:one-scale-estimate-window}\uses{prism BL inequality,M to K,lem:smoothdecay2}\usesdefs{bracket bump,normalized function tuples,auto:prism-form-definition,multiplicatively spaced monotone sequences,kernel sequences,def:cn-window,def:unipair}
\lean{Auto.oneScaleEstimateWindow}
\leanok
Let $a\in A$. For $j\in \Z$, we set
\begin{equation}\label{auto:one-scale-window-kernel}
M_j=(\phi_{0,a(j-1)})^{\otimes 2} - (\phi_{0,a(j)})^{\otimes 2}
\end{equation}
and $\mathbf{M}=(M_j)_{j\in \Z}$. Then
\begin{equation}\label{E:telescoping-estimate}
\|\mathbf{M}\|_{\rm M(1)} \leq C_{\ref{P:one-scale-estimate-window}},
\end{equation}
where $C_{\ref{P:one-scale-estimate-window}}=2^9C_{\ref{def:unipair}}^2$.

\end{proposition}

\begin{proof}
For a given $J\in \mathbb N$ with $J>0$,
\begin{equation}\label{E:telescoping}
\sum_{j\in [J)} M_j=(\phi_{0,a(-1)})^{\otimes 2}-(\phi_{0,a(J-1)})^{\otimes 2}.
\end{equation}
Let $\lambda \in \{a(-1), a(J-1)\}$.
A combination of Propositions~\ref{prism BL inequality} and~\ref{M to K} yields for any $\mathbf{F} \in \mathfrak{F}$,
\begin{equation}\label{E:est-l1-norm}
|\Lambda_1 ((\phi_{0,\lambda})^{\otimes 2})(\mathbf{F})| \leq
\|K_1((\phi_{0,\lambda})^{\otimes 2})\|_1 \leq\|(\phi_{0,\lambda})^{\otimes 2}\|_1
=\|\phi_0\|_1^2.
\end{equation}
By Definition \ref{def:cn-window} and Lemma \ref{lem:smoothdecay2},
\begin{equation}\label{E:est-of-l1-norm}
|\phi_0(x)|
\le 2C_{\ref{lem:smoothdecay2},2}C_{\ref{def:unipair}}\langle x\rangle^2
=8C_{\ref{def:unipair}}\langle x\rangle^2.
\end{equation}
Since $\int_\R\langle x\rangle^2\,dx=2$, we have $\|\phi_0\|_1\le16C_{\ref{def:unipair}}$. Combining this with \eqref{E:telescoping} and \eqref{E:est-l1-norm} yields \eqref{E:telescoping-estimate}.
\end{proof}

\begin{lemma}[constant $C_{\ref{P:one-scale-estimate-window}}$ \auto]\label{auto:constant-one-scale-window}\uses{P:one-scale-estimate-window}\usesdefs{def:unipair}
\lean{Auto.constantOneScaleWindow}
\leanok
\begin{equation}\label{auto:constant-one-scale-window-bound}
C_{\ref{P:one-scale-estimate-window}}=2^{39}.
\end{equation}
\end{lemma}
\begin{proof}[Proof \auto]
Substitute $C_{\ref{def:unipair}}=2^{15}$ in the defining value $2^9C_{\ref{def:unipair}}^2$ from Proposition \ref{P:one-scale-estimate-window}.
\end{proof}

\begin{lemma}\label{L:fourier-transform-window}\uses{}\usesdefs{multiplicatively spaced monotone sequences,def:cn-window,def:cpair,def:unipair}
\lean{Auto.fourierTransformWindow}
\leanok
Let $a\in A$. Then for any $j\in \Z$,
\begin{equation}\label{auto:window-square-Fourier-identity}
\mathcal F({\phi}_{0,a(j-1)}-{\phi}_{1,a(j)})^2=
(\widehat{{\phi}_{0,a(j-1)}})^2-(\widehat{{\phi}_{0,a(j)}})^2.
\end{equation}
\end{lemma}

\begin{proof}
We have
\begin{equation}\label{E:square-fourier-transform}
\mathcal F({\phi}_{0,a(j-1)}-{\phi}_{1,a(j)})^2=
(\widehat{{\phi}_{0,a(j-1)}})^2 - 2 \widehat{{\phi}_{0,a(j-1)}} \widehat{{\phi}_{1,a(j)}} +(\widehat{{\phi}_{1,a(j)}})^2.
\end{equation}
Using~\eqref{ft_window_eq_zero},~\eqref{ft_window_eq_one} and the fact that $2a(j-1) \leq a(j)$, we see that $\widehat{{\phi}_{0,a(j-1)}}(\xi)=1$ whenever $\widehat{{\phi}_{1,a(j)}}(\xi) \neq 0$. Thus, the right-hand side of~\eqref{E:square-fourier-transform} can be rewritten as
\[
(\widehat{{\phi}_{0,a(j-1)}})^2 +(1-\widehat{{\phi}_{1,a(j)}})^2-1.
\]
The conclusion follows by applying~\eqref{windowsum}.
\end{proof}

\begin{lemma} \label{lem:scaleest}\uses{}\usesdefs{bracket bump}
\lean{Auto.scaleEstimate}
\leanok
    For any $x\in \R$ and $\lambda, s>0$,
    \begin{equation}\label{auto:scaled-bracket-comparison}\langle x \rangle^2_{(\lambda s)} \le \max\{\lambda,\lambda^{-1}\}\langle x \rangle^2_{(s)}\end{equation}
\end{lemma}
\begin{proof}
    Let $r=|x|/s$. Then
    \[\tfrac{\langle x \rangle^2_{(\lambda s)} }{\langle x \rangle^2_{(s)}} = \lambda^{-1}\Big(\tfrac{1+r}{1+\tfrac{r}{\lambda}} \Big)^2\]
    If $\lambda\ge 1$, then
    \[1+r\le \lambda\Big(1+\tfrac{r}{\lambda} \Big),\]
    so the quotient is at most $\lambda$. If $0<\lambda <1$, then
    \[1+\tfrac{r}{\lambda}\ge 1+ r,\]
    so the quotient is at most $\lambda^{-1}$.
\end{proof}

\begin{proposition}[induct positive terms - reduction variant, non-Whitney]
\label{P:induct-positive-terms-reduction-non-whitney}\uses{P:one-scale-estimate-window,L:fourier-transform-window,Operations on spaced sequences,lem:smoothdecay2,lem:scaleest,P:vanishing-diagonal-reduction,P:diagonal-band-reduction}\usesdefs{bracket bump,multiplicatively spaced monotone sequences,auto:unitary-matrices-definition,kernel sequences,def:cn-window,def:unipair}
\lean{Auto.inductPositiveTermsReductionNonWhitney}
\leanok
Let $a\in A$ and let
\begin{equation}\label{auto:non-Whitney-kernel}M_j=({\phi}_{0,a(j-1)}-{\phi}_{1,a(j)})^{\otimes 2}.\end{equation}
Then with $\M=(M_j)_{j\in \Z}$,  \begin{equation}\label{auto:non-Whitney-reduction-bound}\|\M\|_{\rm M(1)}\le C_{\ref{P:induct-positive-terms-reduction-non-whitney}},\end{equation}
where $C_{\ref{P:induct-positive-terms-reduction-non-whitney}}=C_{\ref{P:one-scale-estimate-window}}+ 128 \max(1,(C_{\ref{def:unipair}}/2\pi)^4) C_{\ref{lem:smoothdecay2},2}^2C_{\ref{P:diagonal-band-reduction}}$.
\end{proposition}

\begin{proof}
For $j\in \Z$, we write $M_j=M_{0,j}+M_{1,j}$, where
\[
M_{0,j}= ({\phi}_{0,a(j-1)})^{\otimes 2} -({\phi}_{0,a(j)})^{\otimes 2}
\]
and
\[
M_{1,j}=({\phi}_{0,a(j-1)}-{\phi}_{1,a(j)})^{\otimes 2} + ({\phi}_{0,a(j)})^{\otimes 2} -({\phi}_{0,a(j-1)})^{\otimes 2}.
\]
For $i\in [2)$, we also set $\mathbf{M}_i=(M_{i,j})_{j\in \Z}$. By Proposition~\ref{P:one-scale-estimate-window},
\[
\|\mathbf{M}_0\|_{\rm M(1)} \leq C_{\ref{P:one-scale-estimate-window}}.
\]

We proceed by estimating $\mathbf{M}_1$.
Using Lemma~\ref{L:fourier-transform-window} and the fact that $\phi_0$ is an even function, we deduce that for every $\xi \in \R$,
\[
\widehat{M_{1,j}}(\xi,-\xi)=0.
\]
Let $b=a/4$, then $b\in A$ by Proposition~\ref{Operations on spaced sequences}.
If $(\xi,\eta) \in \R^2$ is such that $\widehat{M_{1,j}}(\xi,\eta) \neq 0$ then $|\xi+\eta| \leq |\xi|+|\eta| \leq 2(a(j-1))^{-1}=2^{-1} (b(j-1))^{-1}$.

By Lemma~\ref{lem:smoothdecay2}, we get for $i\in [2)$ and $x\in \R$,
\[
|\phi_i(x)| \leq 2 C_{\ref{lem:smoothdecay2},2} \max(1,(C_{\ref{def:unipair}}/2\pi)^2) \langle x \rangle^2.
\]
Thus, for $v\in \R^2$,
\[
|M_{1,j}(v)|
\leq 4 C_{\ref{lem:smoothdecay2},2}^2 \max(1,(C_{\ref{def:unipair}}/2\pi)^4)\Big[
(\langle v_0 \rangle^2_{(a(j-1))}+\langle v_0 \rangle^2_{(a(j))})
(\langle v_1 \rangle^2_{(a(j-1))}+\langle v_1 \rangle^2_{(a(j))})
\]
\[
+\langle v_0 \rangle^2_{a(j)} \langle v_1 \rangle^2_{(a(j))}+\langle v_0 \rangle^2_{(a(j-1))} \langle v_1 \rangle^2_{(a(j-1))} \Big].
\]
Applying Lemma~\ref{lem:scaleest} and lowering all exponents to $3/2$, each of the two diagonal scale pairs occurs twice in the preceding bracket, while each off-diagonal pair occurs once. Thus the expression is bounded by
\[
128 C_{\ref{lem:smoothdecay2},2}^2 \max(1,(C_{\ref{def:unipair}}/2\pi)^4)
\sum_{(t_0,t_1,u) \in \mathcal P} \langle (W_u v)_0 \rangle^{3/2}_{(t_0(j))} \langle (W_u v)_1 \rangle^{3/2}_{(t_1(j))},
\]
where $\mathcal P$ consists of all triples of the form $(t_0,t_1,0)$, where for each $i\in [2)$ either $t_i(j)=b(j-1)$ for all $j\in \Z$, or $t_i(j)=b(j)$ for all $j\in \Z$. In particular, $\# \mathcal P=4$. Applying Proposition~\ref{P:vanishing-diagonal-reduction}, we thus get
\[
\|\mathbf{M}_1\|_{\rm M(1)} \leq 128 \max(1,(C_{\ref{def:unipair}}/2\pi)^4) C_{\ref{lem:smoothdecay2},2}^2C_{\ref{P:diagonal-band-reduction}}.
\]
\end{proof}

\begin{lemma}[constant $C_{\ref{P:induct-positive-terms-reduction-non-whitney}}$ \auto]\label{constant non Whitney reduction}\uses{auto:constant-one-scale-window,lem:smoothdecay2,P:diagonal-band-reduction,constant diagonal band reduction,P:induct-positive-terms-reduction-non-whitney}\usesdefs{def:unipair}
\lean{Auto.constantNonWhitneyReduction}
\leanok
\begin{equation}\label{constant non Whitney reduction bound}
C_{\ref{P:induct-positive-terms-reduction-non-whitney}}
<\tfrac89 2^{541}<2^{541}.
\end{equation}
\end{lemma}
\begin{proof}[Proof \auto]
Use Lemma \ref{auto:constant-one-scale-window}, $C_{\ref{def:unipair}}=2^{15}$, $C_{\ref{lem:smoothdecay2},2}=2^2$, and Lemma \ref{constant diagonal band reduction}, which unfolds the constant from Proposition \ref{P:diagonal-band-reduction}. Since $\pi>3$, the second term in the defining sum is smaller than
\[
2^{11}\left(\tfrac{2^{15}}6\right)^4\tfrac{63}{64}2^{480}
=\tfrac79 2^{541}.
\]
The first term is $2^{39}<\frac19 2^{541}$, which proves the claim.
\end{proof}

\begin{proposition}[induct positive terms - reduction variant, non-Whitney, skip terms]
\label{P:induct-positive-terms-reduction-non-whitney-skip}\uses{positive terms,P:induct-positive-terms-reduction-non-whitney}\usesdefs{normalized function tuples,auto:prism-form-definition,multiplicatively spaced monotone sequences,geometric parameters,double sequence of 2D functions,kernel sequences,def:unipair}
\lean{Auto.inductPositiveTermsReductionNonWhitneySkip}
\leanok
Let $a\in A$ and let
\begin{equation}\label{auto:skipped-non-Whitney-kernel}\tilde{M}_j=({\phi}_{0,a(2j)}-{\phi}_{1,a(2j+1)})^{\otimes 2}.\end{equation}
Then with $\tilde{\M}=(\tilde{M}_j)_{j\in \Z}$,
\begin{equation}\label{E:est-skipped-terms}
\|\tilde{\M}\|_{\rm M(1)}\le C_{\ref{P:induct-positive-terms-reduction-non-whitney-skip}},
\end{equation}
where $C_{\ref{P:induct-positive-terms-reduction-non-whitney-skip}}= 2^{1-2^{2-n}} C_{\ref{P:induct-positive-terms-reduction-non-whitney}}$.
\end{proposition}

\begin{proof}
Let $\gamma=(1,0,(a,a))$, then $\gamma \in \Gamma$. For $j\in \Z$, we set
\[
X_{0,j}=({\phi}_{0,a(2j-1)}-{\phi}_{1,a(2j)})^{\otimes 2}.
\]
Then $X=(X_{i,j})_{i\in [1),j\in \Z} \in \mathcal X_1$. By Proposition~\ref{positive terms}, we deduce that for any $j\in \Z$ and $\mathbf{F} \in \mathfrak{F}$,
\begin{equation}\label{E:intermediate-terms-nonnegative}
\Lambda_1(X_{0,j})(\mathbf{F}) \geq 0.
\end{equation}
The same proposition, applied with $\widetilde X_{0,j}=\widetilde M_j$, also gives
\[
\Lambda_1(\widetilde M_j)(\mathbf F)\ge0.
\]
Let $(M_j)_{j\in \Z}$ be as in Proposition~\ref{P:induct-positive-terms-reduction-non-whitney} and let $J\in\N$ with $J\ge1$. By linearity, \eqref{E:intermediate-terms-nonnegative}, and Proposition~\ref{P:induct-positive-terms-reduction-non-whitney}, we get that for every $\mathbf{F} \in \mathfrak{F}$,
\[
0\le\Lambda_1(\sum_{j\in [J)} \tilde{M}_j)(\mathbf{F})
\leq \Lambda_1(\sum_{j\in [2J)} M_j)(\mathbf{F})
\leq C_{\ref{P:induct-positive-terms-reduction-non-whitney}} 2^{1-2^{2-n}} J^{1-2^{2-n}}.
\]
This yields~\eqref{E:est-skipped-terms}.
\end{proof}

\begin{lemma}[constant $C_{\ref{P:induct-positive-terms-reduction-non-whitney-skip}}$ \auto]\label{constant non Whitney skip reduction}\uses{constant non Whitney reduction,P:induct-positive-terms-reduction-non-whitney-skip}\usesdefs{}
\lean{Auto.constantNonWhitneySkipReduction}
\leanok
\begin{equation}\label{constant non Whitney skip reduction bound}
C_{\ref{P:induct-positive-terms-reduction-non-whitney-skip}}
<\tfrac89 2^{542}<2^{542}.
\end{equation}
\end{lemma}
\begin{proof}[Proof \auto]
The factor $2^{1-2^{2-n}}$ is at most $2$. Lemma \ref{constant non Whitney reduction} therefore gives the displayed estimate.
\end{proof}

\begin{proposition}[induct positive terms - reduction variant, Whitney, with gap]
\label{P:induct-positive-terms-reduction-whitney-gap}\uses{lem:Phij_prop,lem:form_pos,P:induct-positive-terms-reduction-non-whitney-skip,P:vanishing-diagonal-reduction,lem:smoothdecay2,lem:scaleest,P:diagonal-band-reduction}\usesdefs{bracket bump,normalized function tuples,auto:prism-form-definition,multiplicatively spaced monotone sequences,auto:unitary-matrices-definition,kernel sequences,def:cn-window,def:unipair}
\lean{Auto.inductPositiveTermsReductionWhitneyGap}
\leanok
Let $a\in {\rm A}$ satisfy $a(j)\geq 2^{11}a(j-1)$ for every $j\in\mathbb Z$. Let $M_j:\mathbb{R}^2\to \mathbb{R}$ be of the form
\begin{equation}\label{auto:Whitney-gap-kernel}
M_j=\Psi_{(a(j))}
\end{equation}
for every $j\in\Z$, where $\Psi:\mathbb{R}^2\to \mathbb{R}$ is a Schwartz function satisfying $\Psi(v_0,v_1)=\Psi(v_1,v_0)$ for all $v=(v_0,v_1)\in\mathbb{R}^2$,
\begin{equation}\label{p1-on-W_pos}
\int_{\mathbb{R}^2} g(v_0){g(v_1)}\Psi(v)\,dv\ge 0
\end{equation}
for all bounded measurable $g:\mathbb{R}\to\mathbb{R}$,
\begin{equation}\label{p1-on-W_supp}
\mathrm{supp}(\widehat{\Psi}) \subset \mathrm{Ann}_1(1, 2^3)^2,
\end{equation}
for $m\in [3)$ and $\xi \in \R$,
\begin{equation}\label{p1-on-W_der}
\Big|\tfrac{d^m}{d\xi^{m}} \widehat{\Psi}(\xi,-\xi)\Big| \leq 1,
\end{equation}
and for all $v\in\mathbb{R}^2$,
\begin{equation}\label{p1-on-W_decay}
|\Psi(v)| \le \sum_{u=0}^1 \langle(W_u v)_0\rangle^{3/2} \langle (W_u v)_1\rangle^{3/2}.
\end{equation}
Then with $\M=(M_j)_{j\in\Z}$,
\begin{equation}\label{auto:Whitney-gap-reduction-bound}
\|\M\|_{\rm M(1)} \le C_{\ref{P:induct-positive-terms-reduction-whitney-gap}}.
\end{equation}
where
\begin{equation}\label{auto:Whitney-gap-constant-definition}
C_{\ref{P:induct-positive-terms-reduction-whitney-gap}}
=2C_{\ref{P:induct-positive-terms-reduction-non-whitney-skip}}
+2^{19}(C_{\ref{def:unipair}}^2+C_{\ref{lem:Phij_prop}}^2)C_{\ref{P:diagonal-band-reduction}}.
\end{equation}
\end{proposition}

\begin{proof}[Proof \auto]
Use Lemma \ref{lem:Phij_prop} to obtain a function $\psi$ satisfying \eqref{Phij_psi_diag} and \eqref{psi_decay}. Define $\alpha(2j)=2^{-5}a(j)$ and $\alpha(2j+1)=2^5a(j)$. Then $\alpha\in A$. By \eqref{Phij_psi_diag}, for each $j\in\Z$,
\begin{equation}\label{psi_diag}
\widehat{\psi_{(a(j))}}(\xi)^2
=2(\widehat{\phi_{0,\alpha(2j)}}-\widehat{\phi_{1,\alpha(2j+1)}})(\xi)^2
-\widehat{\Psi_{(a(j))}}(\xi,-\xi).
\end{equation}
Decompose $M_j=M_{0,j}-M_{1,j}$ with
\[
M_{0,j}=\Psi_{(a(j))}+(\psi_{(a(j))})^{\otimes2},
\qquad
M_{1,j}=(\psi_{(a(j))})^{\otimes2}.
\]
Let $J\ge1$, $M=\sum_{j\in[J)}M_j$, and $\F\in\mathfrak F$. The condition \eqref{p1-on-W_pos} and Lemma \ref{lem:form_pos}, applied term by term, give
\[
\Lambda_1(M)(\F)\ge0,
\qquad
\Lambda_1\Big(\sum_{j\in[J)}M_{1,j}\Big)(\F)\ge0.
\]
Therefore,
\begin{equation}\label{K0_split}
|\Lambda_1(M)(\F)|
\le\Lambda_1\Big(\sum_{j\in[J)}M_{0,j}\Big)(\F).
\end{equation}

Write $M_{0,j}=M_{2,j}+M_{3,j}$, where
\[
M_{2,j}=2(\phi_{0,\alpha(2j)}-\phi_{1,\alpha(2j+1)})^{\otimes2},
\qquad
M_{3,j}=M_{0,j}-M_{2,j}.
\]
Proposition \ref{P:induct-positive-terms-reduction-non-whitney-skip} gives
\[
\|(M_{2,j})_{j\in\Z}\|_{\rm M(1)}
\le2C_{\ref{P:induct-positive-terms-reduction-non-whitney-skip}}.
\]
It remains to estimate $(M_{3,j})_{j\in\Z}$ by Proposition \ref{P:vanishing-diagonal-reduction}.

By \eqref{psi_diag},
\[
\widehat{M_{3,j}}(\xi,-\xi)=0.
\]
The Fourier supports of
\[
\Psi,
\qquad
\psi^{\otimes2},
\qquad
(\phi_{0,2^{-5}}-\phi_{1,2^5})^{\otimes2}
\]
are contained in $\mathrm{Ann}_1(1,2^6)^2$. Hence, on the support of $\widehat{M_{3,j}}$,
\begin{equation}\label{xi_plus_eta}
|\xi+\eta|\le2^7a(j)^{-1}\le2^{-1}a(j-1)^{-1}.
\end{equation}

Lemma \ref{lem:smoothdecay2} with $N=2$ gives
\[
|\phi_i(x)|\le2^3C_{\ref{def:unipair}}\langle x\rangle^2
\]
for $i\in[2)$. Thus, by Lemma \ref{lem:scaleest},
\[
|(\phi_{0,2^{-5}}-\phi_{1,2^5})^{\otimes2}(v)|
\le2^{18}C_{\ref{def:unipair}}^2\langle v_0\rangle^2\langle v_1\rangle^2.
\]
Using \eqref{psi_decay}, the support bound for $\widehat\psi$, and Lemma \ref{lem:smoothdecay2} with $N=2$, we also have
\begin{equation}\label{decay_2}
|\psi^{\otimes2}(v)|
\le2^{18}C_{\ref{lem:Phij_prop}}^2
\langle v_0\rangle^2\langle v_1\rangle^2.
\end{equation}
Together with \eqref{p1-on-W_decay}, rescaling, and $\langle\cdot\rangle^2\le\langle\cdot\rangle^{3/2}$, this gives
\[
c^{-1}|M_{3,j}(v)|
\le
\langle(W_1v)_0\rangle_{a(j)}^{3/2}
\langle(W_1v)_1\rangle_{a(j)}^{3/2}
+
\langle v_0\rangle_{a(j)}^{3/2}
\langle v_1\rangle_{a(j)}^{3/2},
\]
where
\[
c=2^{19}(C_{\ref{def:unipair}}^2+C_{\ref{lem:Phij_prop}}^2).
\]
Proposition \ref{P:vanishing-diagonal-reduction} now yields
\[
\|(M_{3,j})_{j\in\Z}\|_{\rm M(1)}
\le2^{19}(C_{\ref{def:unipair}}^2+C_{\ref{lem:Phij_prop}}^2)
C_{\ref{P:diagonal-band-reduction}},
\]
which proves the claim.
\end{proof}

\begin{lemma}[constant $C_{\ref{P:induct-positive-terms-reduction-whitney-gap}}$ \auto]\label{constant Whitney gap reduction}\uses{constant non Whitney skip reduction,P:induct-positive-terms-reduction-whitney-gap,constant Phij proposition,constant diagonal band reduction}\usesdefs{def:unipair}
\lean{Auto.constantWhitneyGapReduction}
\leanok
\begin{equation}\label{constant Whitney gap reduction bound}
C_{\ref{P:induct-positive-terms-reduction-whitney-gap}}
<\tfrac{127}{128}2^{553}<2^{553}.
\end{equation}
\end{lemma}
\begin{proof}[Proof \auto]
Lemma \ref{constant non Whitney skip reduction} bounds the first term in \eqref{auto:Whitney-gap-constant-definition} by
\[
\tfrac89 2^{543}.
\]
Lemmas \ref{constant Phij proposition} and \ref{constant diagonal band reduction}, together with $C_{\ref{def:unipair}}=2^{15}$, bound the second term by
\[
2^{19}(2^{30}+2^{54})\tfrac{63}{64}2^{480}.
\]
The sum of these two bounds is smaller than $\frac{127}{128}2^{553}$.
\end{proof}

\begin{proposition}[induct positive terms - reduction variant, Whitney]
\label{P:induct-positive-terms-reduction-whitney}\uses{P:induct-positive-terms-reduction-whitney-gap}\usesdefs{bracket bump,multiplicatively spaced monotone sequences,auto:unitary-matrices-definition,kernel sequences}
\lean{Auto.inductPositiveTermsReductionWhitney}
\leanok
Let $a\in {\rm A}$. Let $M_j:\mathbb{R}^2\to \mathbb{R}$ be of the form
\begin{equation}\label{auto:Whitney-kernel}
M_j=\Psi_{(a(j))}
\end{equation}
for every $j\in\Z$, where $\Psi:\mathbb{R}^2\to \mathbb{R}$ is a Schwartz function satisfying $\Psi(u_0,u_1)=\Psi(u_1,u_0)$ for all $u=(u_0,u_1)\in\mathbb{R}^2$, and \eqref{p1-on-W_pos}, \eqref{p1-on-W_supp}, \eqref{p1-on-W_der}, \eqref{p1-on-W_decay}. Then with $\M=(M_j)_{j\in\Z}$,
\begin{equation}\label{auto:Whitney-reduction-bound}
\|\M\|_{\rm M(1)} \le C_{\ref{P:induct-positive-terms-reduction-whitney}},
\end{equation}
where
\begin{equation}\label{auto:Whitney-constant-definition}
C_{\ref{P:induct-positive-terms-reduction-whitney}}
=11C_{\ref{P:induct-positive-terms-reduction-whitney-gap}}.
\end{equation}
\end{proposition}

\begin{proof}[Proof \auto]
For $s\in[11)$, define
\[
a^{(s)}(k)=a(11k+s),
\qquad
M^{(s)}_k=M_{11k+s}.
\]
Since $a\in A$,
\[
a^{(s)}(k)\ge2^{11}a^{(s)}(k-1),
\qquad
M^{(s)}_k=\Psi_{(a^{(s)}(k))}.
\]
Split any finite sum into these eleven residue classes and apply Proposition \ref{P:induct-positive-terms-reduction-whitney-gap} to each class. The triangle inequality gives
\[
\|\M\|_{\rm M(1)}
\le11C_{\ref{P:induct-positive-terms-reduction-whitney-gap}}
= C_{\ref{P:induct-positive-terms-reduction-whitney}}.
\]
\end{proof}

\begin{lemma}[constant $C_{\ref{P:induct-positive-terms-reduction-whitney}}$ \auto]\label{constant Whitney reduction}\uses{constant Whitney gap reduction,P:induct-positive-terms-reduction-whitney}\usesdefs{}
\lean{Auto.constantWhitneyReduction}
\leanok
\begin{equation}\label{constant Whitney reduction bound}
C_{\ref{P:induct-positive-terms-reduction-whitney}}
<2^{557}.
\end{equation}
\end{lemma}
\begin{proof}[Proof \auto]
The defining factor is $11$. Lemma \ref{constant Whitney gap reduction} gives
\[
C_{\ref{P:induct-positive-terms-reduction-whitney}}
<11\cdot\tfrac{127}{128}2^{553}
=\tfrac{1397}{2048}2^{557}.
\]
\end{proof}

\begin{proposition}[induct positive terms - reduction variant, Whitney, product-type]
\label{P:induct-positive-terms-reduction-whitney-product}\uses{P:induct-positive-terms-reduction-whitney,lem:smoothdecay2}\usesdefs{bracket bump,multiplicatively spaced monotone sequences,kernel sequences}
\lean{Auto.inductPositiveTermsReductionWhitneyProduct}
\leanok
Let $a\in A$ and let
\begin{equation}\label{auto:Whitney-product-kernel}
M_j=(\psi_{(a(j))})^{\otimes2}
\end{equation}
for $j\in\Z$, where $\psi:\R\to\R$ is a Schwartz function satisfying
\begin{equation}\label{1-on-W_supp}
\operatorname{supp}(\widehat\psi)\subset\operatorname{Ann}_1(1,2^3),
\end{equation}
\begin{equation}\label{auto:Whitney-product-Fourier-derivative-assumption}
|\widehat\psi^{(m)}(\xi)|\le1,
\qquad
m\in[3),
\quad
\xi\in\R.
\end{equation}
Then, with $\M=(M_j)_{j\in\Z}$,
\begin{equation}\label{auto:Whitney-product-reduction-bound}
\|\M\|_{{\rm M}(1)}\le C_{\ref{P:induct-positive-terms-reduction-whitney-product}},
\end{equation}
where
\begin{equation}\label{auto:Whitney-product-constant-definition}
C_{\ref{P:induct-positive-terms-reduction-whitney-product}}
=2^{12}C_{\ref{P:induct-positive-terms-reduction-whitney}}.
\end{equation}
\end{proposition}

\begin{proof}[Proof \auto]
The support in \eqref{1-on-W_supp} has measure smaller than $16$. Hence Lemma \ref{lem:smoothdecay2}, with $N=2$, and \eqref{auto:Whitney-product-Fourier-derivative-assumption} give
\begin{equation}\label{Whitney product decay assumption}
|\psi(x)|\le2^6\langle x\rangle^2
\end{equation}
for every $x\in\R$.
Set $\Psi=2^{-12}\psi\otimes\psi$. Then $\Psi$ is symmetric and satisfies \eqref{p1-on-W_pos} and \eqref{p1-on-W_supp}. By \eqref{auto:Whitney-product-Fourier-derivative-assumption}, the zeroth, first, and second derivatives of $\widehat\psi(\xi)\widehat\psi(-\xi)$ are bounded by $1$, $2$, and $4$, respectively, so \eqref{p1-on-W_der} holds for $\Psi$. Moreover, \eqref{Whitney product decay assumption} gives
\[
|\Psi(v)|
\le\langle v_0\rangle^2\langle v_1\rangle^2
\le\langle v_0\rangle^{3/2}\langle v_1\rangle^{3/2},
\]
so \eqref{p1-on-W_decay} holds. Proposition \ref{P:induct-positive-terms-reduction-whitney} applied to $2^{-12}M_j$ proves the claim.
\end{proof}

\begin{lemma}[constant $C_{\ref{P:induct-positive-terms-reduction-whitney-product}}$ \auto]\label{constant Whitney product reduction}\uses{constant Whitney reduction,P:induct-positive-terms-reduction-whitney-product}\usesdefs{}
\lean{Auto.constantWhitneyProductReduction}
\leanok
\begin{equation}\label{constant Whitney product reduction bound}
C_{\ref{P:induct-positive-terms-reduction-whitney-product}}
<2^{569}.
\end{equation}
\end{lemma}
\begin{proof}[Proof \auto]
Use \eqref{auto:Whitney-product-constant-definition} and Lemma \ref{constant Whitney reduction}.
\end{proof}

\subsection{Final reduction: proof of main theorem}\label{sec:pfmainthm}
In this section we prove Theorem \ref{thm:nct main real} using Propositions \ref{P:induct-positive-terms-reduction-whitney}, \ref{P:induct-positive-terms-reduction-whitney-product} and \ref{P:induct-positive-terms-reduction-non-whitney-skip}.
Let $n\ge 2$.
Let $\mathbf{f}=(f_0,\dots,f_{n-1})$ be Schwartz functions.
Throughout this section we abbreviate
\begin{equation}\label{abbreviated twisted average}
A(\chi)=A(\chi,\mathbf f).
\end{equation}
since $\mathbf{f}$ will not change.
We will also write
\begin{equation}
    \label{avg_rescale}
    A_t(\chi) = A(\chi_{(t)}),
\end{equation}

Let $(\phi_0, \phi_1)$ be a universal pair
and let $\theta, \varphi_{0,k}, \varphi_{1,k}, \varphi_{2,k}, \varphi_{3,k}, \varphi_{4,k}$ be defined as in \eqref{eqn:thetadef}-\eqref{eqn:varphi4kdef}. 
In what follows we denote
\begin{equation}\label{variation exponent alpha}
\alpha(n)=1-2^{-n+2}.
\end{equation}

\begin{lemma}\label{lem:main_aux1}\uses{Extension of sequences,A to Lambda,P:induct-positive-terms-reduction-whitney-product}\usesdefs{A_def,normalized function tuples,auto:prism-form-definition,multiplicatively spaced monotone sequences}
\lean{Auto.mainAuxOne}
\leanok
Suppose that $\psi:\R\to\R$ is a Schwartz function with
\begin{equation}\label{main_aux1_supp}
{\rm supp}(\widehat\psi)\subset\mathrm{Ann}_1(1,2^2),
\end{equation}
and
\begin{equation}\label{auto:main-auxiliary-one-Fourier-derivative-assumption}
|\widehat\psi^{(m)}(\xi)|\le1,
\qquad
m\in[3),
\quad
\xi\in\R.
\end{equation}
Then for every $t\in[1,2]$, $J\ge1$, and every strictly increasing sequence of integers $(k_j)_{j\in[J)}$,
\begin{equation}\label{main_aux1}
\sum_{j\in[J)}\|A_{2^{k_j}t}(\psi)\|_2^2
\le C_{\ref{lem:main_aux1}}J^{\alpha(n)},
\end{equation}
where
\begin{equation}\label{auto:main-auxiliary-one-constant-definition}
C_{\ref{lem:main_aux1}}
=2^4C_{\ref{P:induct-positive-terms-reduction-whitney-product}}.
\end{equation}
\end{lemma}

\begin{proof}[Proof \auto]
By Proposition \ref{Extension of sequences}, the finite sequence $(2^{k_j})_{j\in[J)}$ extends to a sequence $a\in A$. By Lemma \ref{A to Lambda} and linearity, the left-hand side equals $\Lambda_1(M_t)(\F)$, where
\[
M_t=\sum_{j\in[J)}(\psi_{(a(j)t)})^{\otimes2}.
\]
Apply Proposition \ref{P:induct-positive-terms-reduction-whitney-product} to the corresponding initial partial sum, with $2^{-2}\psi_{(t)}$ in place of $\psi$, and multiply the resulting estimate by $2^4$. The Fourier support condition follows from \eqref{main_aux1_supp} and $t\in[1,2]$. Finally,
\[
\left|\tfrac{d^m}{d\xi^m}
\left(2^{-2}\widehat\psi(t\xi)\right)\right|
\le2^{m-2}\le1
\]
for $m\in[3)$, which verifies \eqref{auto:Whitney-product-Fourier-derivative-assumption}.
\end{proof}

\begin{lemma}[constant $C_{\ref{lem:main_aux1}}$ \auto]\label{constant main auxiliary one}\uses{constant Whitney product reduction,lem:main_aux1}\usesdefs{}
\lean{Auto.constantMainAuxiliaryOne}
\leanok
\begin{equation}\label{constant main auxiliary one bound}
C_{\ref{lem:main_aux1}}<2^{573}.
\end{equation}
\end{lemma}
\begin{proof}[Proof \auto]
Use Lemma \ref{constant Whitney product reduction} in \eqref{auto:main-auxiliary-one-constant-definition}.
\end{proof}

\begin{lemma}[Long and short variation]\label{lem:shortlongftc_reduction}\uses{lem:shortlongjumps,lem:ftccs-R,lem:ftc_ATphi,lem:main_aux1}\usesdefs{A_def}
\lean{Auto.shortLongFtcReduction}
\leanok
Let $J\ge 1$ and $\phi$ a Schwartz function.
Assume that the function $T\phi(s) = (s \phi(s))'$ satisfies \eqref{main_aux1_supp} and \eqref{auto:main-auxiliary-one-Fourier-derivative-assumption}.
Suppose that $A\in (0,\infty)$ is such that
\begin{equation}\label{shortlongftc_reduction1} \|A_{t}(\phi)\|^2_{V_{2,J}(t\in 2^\mathbb{Z}; L^2)} \le A
\end{equation}
Then
\begin{equation}\label{shortlongftc_reduction2}
\|A_{t}(\phi)\|_{V_{2,J}(t\in (0,\infty); L^2)}^2 \le 2^4 C_{\ref{lem:main_aux1}} J^{\alpha(n)} + 2A.
\end{equation}
\end{lemma}

\begin{proof}
Let $(t_j)_{j\in[J+1)}$ be an increasing sequence of positive real numbers.
By Lemma \ref{lem:shortlongjumps} with $r=2$, $a(t)=A_t(\phi)$, and $B=L^2$,
\begin{equation}\label{shortlongftc_reduction_pf1}
\Big(\sum_{j\in [J)} \|A_{t_{j+1}}(\phi)- A_{t_j}(\phi)\|_2^2\Big)^{1/2}
\le 2 \Big(\sum_{k\in \kappa} \|A_t(\phi)\|_{V_{2,J}(t\in [2^k, 2^{k+1}]; L^2)}^2\Big)^{1/2}
\end{equation}
\begin{equation}\label{shortlongftc_reduction_pf2}
+\|A_t(\phi)\|_{V_{2,J}(t\in 2^\mathbb{Z}; L^2)},
\end{equation}
where $\kappa$ is the set of $k\in \mathbb{Z}$ such that $2^{k}\le t_j<2^{k+1}$ for some $j\in [J+1)$.
Fix $x\in\mathbb{R}^n$ and $k\in\mathbb{Z}$. Set $a(t)=A_t(\phi)(x)$ and use Lemma \ref{lem:ftccs-R}, specifically \eqref{ftccs1-R}, to estimate
\[
\|A_t(\phi)(x)\|_{V_{2,J}(t\in [2^k, 2^{k+1}])}^2
\le
\|ta'(t)\|_{L^2(t\in [2^k, 2^{k+1}],dt/t)}^2.
\]
By Lemma \ref{lem:ftc_ATphi},
\[
\|ta'(t)\|_{L^2(t\in [2^k, 2^{k+1}],dt/t)}^2
=
\int_{1}^2 |A_{2^k t}(T\phi)(x)|^2\,\tfrac{dt}{t}.
\]
Thus
\[
\|A_t(\phi)(x)\|_{V_{2,J}(t\in [2^k, 2^{k+1}])}^2
\le
\int_{1}^2 |A_{2^k t}(T\phi)(x)|^2\,\tfrac{dt}{t}.
\]
Integrating over $x\in\mathbb{R}^n$ gives
\[
\|A_t(\phi)\|_{V_{2,J}(t\in [2^k, 2^{k+1}]; L^2)}^2
\le
\int_1^2 \|A_{2^k t}(T\phi)\|_2^2 \tfrac{dt}{t}.
\]
Using this to estimate the sum on the right in \eqref{shortlongftc_reduction_pf1} and interchanging the sum with the integral,
\[
\sum_{k\in \kappa} \|A_t(\phi)\|_{V_{2,J}(t\in [2^k, 2^{k+1}]; L^2)}^2
\le
\int_1^2 \sum_{k\in\kappa} \|A_{2^k t}(T\phi)\|_2^2\,\tfrac{dt}{t}.
\]
For each fixed $t\in [1,2]$, Lemma \ref{lem:main_aux1} applies because $T\phi$ satisfies both of its assumptions. Since $\#\kappa\le J+1$,
\[
\sum_{k\in\kappa} \|A_{2^k t}(T\phi)\|_2^2
\le 2 C_{\ref{lem:main_aux1}} J^{\alpha(n)}.
\]
The claim now follows from
\[
\sum_{j\in [J)} \|A_{t_{j+1}}(\phi)- A_{t_j}(\phi)\|_2^2
\le 8 \sum_{k\in \kappa} \|A_t(\phi)\|_{V_{2,J}(t\in [2^k, 2^{k+1}]; L^2)}^2
+ 2 \|A_t(\phi)\|_{V_{2,J}(t\in 2^\mathbb{Z}; L^2)}^2
\]
\[
\le 2^4 C_{\ref{lem:main_aux1}} J^{\alpha(n)} + 2A.
\]
\end{proof}

\begin{lemma}\label{lem:mainbump1_long1}\uses{lem:main_aux1}\usesdefs{A_def,def:cn-window,def:unipair}
\lean{Auto.mainBumpOneLongOne}
\leanok
For every $J\ge1$ and every strictly increasing sequence of integers $(k_j)_{j\in[J)}$,
\begin{equation}\label{mainbump1_long1}
\sum_{j\in[J)}\|A_{2^{k_j}}(\phi_0)-A_{2^{k_j}}(\phi_1)\|_2^2
\le C_{\ref{lem:mainbump1_long1}}J^{\alpha(n)},
\end{equation}
where
\begin{equation}\label{auto:main-bump-one-long-one-constant-definition}
C_{\ref{lem:mainbump1_long1}}
=(2C_{\ref{def:unipair}})^2C_{\ref{lem:main_aux1}}.
\end{equation}
\end{lemma}

\begin{proof}[Proof \auto]
Apply Lemma \ref{lem:main_aux1} with $t=1$ and $c_1^{-1}(\phi_0-\phi_1)$ in place of $\psi$, where
\[
c_1=2C_{\ref{def:unipair}}.
\]
The Fourier transform of $\phi_0-\phi_1$ is supported in $\mathrm{Ann}_1(1,2)$. By \eqref{windowbound}, its derivatives through order $2$ have modulus at most $2C_{\ref{def:unipair}}$. After division by $c_1$, both assumptions of Lemma \ref{lem:main_aux1} hold, and the result follows.
\end{proof}

\begin{lemma}[constant $C_{\ref{lem:mainbump1_long1}}$ \auto]\label{constant main bump one long one}\uses{constant main auxiliary one,lem:mainbump1_long1}\usesdefs{def:unipair}
\lean{Auto.constantMainBumpOneLongOne}
\leanok
\begin{equation}\label{constant main bump one long one bound}
C_{\ref{lem:mainbump1_long1}}<2^{605}.
\end{equation}
\end{lemma}
\begin{proof}[Proof \auto]
Use $C_{\ref{def:unipair}}=2^{15}$ and Lemma \ref{constant main auxiliary one} in \eqref{auto:main-bump-one-long-one-constant-definition}. The squared factor is $2^{32}$, which gives the displayed estimate.
\end{proof}

\begin{lemma}\label{lem:mainbump1_long2}\uses{Extension of sequences,A to Lambda,P:induct-positive-terms-reduction-non-whitney-skip,Operations on spaced sequences}\usesdefs{A_def,normalized function tuples,auto:prism-form-definition,multiplicatively spaced monotone sequences,def:unipair}
\lean{Auto.mainBumpOneLongTwo}
\leanok
For every $J\ge 1$ and every strictly increasing sequence of integers $(m_j)_{j\in[J+1)}$,
\begin{equation}\label{mainbump1_long2}
\sum_{j\in[J)} \|
A_{2^{m_{j}}}(\phi_0)-A_{2^{m_{j+1}}}(\phi_1) \|_{2}^2  \le C_{\ref{lem:mainbump1_long2}} J^{\alpha(n)},\end{equation}
where $C_{\ref{lem:mainbump1_long2}}=2 C_{\ref{P:induct-positive-terms-reduction-non-whitney-skip}}$.
\end{lemma}

\begin{proof}[Proof \auto]
By Proposition \ref{Extension of sequences}, the finite sequence $(2^{m_j})_{j\in[J+1)}$ extends to a sequence $a\in A$. By Lemma \ref{A to Lambda}, the left-hand side is
\[
\Lambda_1\left(\sum_{j\in[J)}
(\phi_{0,a(j)}-\phi_{1,a(j+1)})^{\otimes2}\right)(\F).
\]
Split this sum according to the parity of $j$. The even terms are an initial partial sum of the sequence in Proposition \ref{P:induct-positive-terms-reduction-non-whitney-skip} for $a$, while the odd terms are such a partial sum for the shifted sequence $a(\cdot+1)$, which belongs to $A$ by Proposition \ref{Operations on spaced sequences}. Proposition \ref{P:induct-positive-terms-reduction-non-whitney-skip} and the triangle inequality therefore give
\[
\sum_{j\in[J)}
\|A_{2^{m_j}}(\phi_0)-A_{2^{m_{j+1}}}(\phi_1)\|_2^2
\le2C_{\ref{P:induct-positive-terms-reduction-non-whitney-skip}}J^{\alpha(n)}.
\]
\end{proof}

\begin{lemma}[constant $C_{\ref{lem:mainbump1_long2}}$ \auto]\label{constant main bump one long two}\uses{P:induct-positive-terms-reduction-non-whitney-skip,constant non Whitney skip reduction,lem:mainbump1_long2}\usesdefs{}
\lean{Auto.constantMainBumpOneLongTwo}
\leanok
\begin{equation}\label{constant main bump one long two bound}
C_{\ref{lem:mainbump1_long2}}<\tfrac89 2^{543}<2^{543}.
\end{equation}
\end{lemma}
\begin{proof}[Proof \auto]
The defining constant is twice the constant in Proposition \ref{P:induct-positive-terms-reduction-non-whitney-skip}. Apply Lemma \ref{constant non Whitney skip reduction}.
\end{proof}

\begin{lemma}\label{lem:mainbump1_long}\uses{lem:mainbump1_long2,lem:mainbump1_long1}\usesdefs{A_def,def:unipair}
\lean{Auto.mainBumpOneLong}
\leanok
Let $J\ge 1$. Then
\begin{equation}\label{mainbump1_long}
\|A_{t}(\phi_0)\|^2_{V_{2,J}(t\in 2^\mathbb{Z}; L^2)}  \le C_{\ref{lem:mainbump1_long}} J^{\alpha(n)},
\end{equation}
where $C_{\ref{lem:mainbump1_long}}=2(C_{\ref{lem:mainbump1_long2}} + C_{\ref{lem:mainbump1_long1}})$.
\end{lemma}

\begin{proof}
Let $(k_j)_{j\in [J+1)}$ be a strictly increasing sequence of integers.
By the triangle inequality,
\[ \Big(\sum_{j\in [J)} \|A_{2^{k_{j+1}}}(\phi_0)- A_{2^{k_j}}(\phi_0)\|_2^2\Big)^{1/2} \]
\[ \le \Big(\sum_{j\in [J)} \|A_{2^{k_{j}}}(\phi_0)- A_{2^{k_{j+1}}}(\phi_1)\|_2^2\Big)^{1/2}
+ \Big(\sum_{j\in [J)} \|A_{2^{k_{j+1}}}(\phi_0)- A_{2^{k_{j+1}}}(\phi_1)\|_2^2\Big)^{1/2}.
\]
By Lemma \ref{lem:mainbump1_long2} and Lemma \ref{lem:mainbump1_long1} this is
\[ \le C_{\ref{lem:mainbump1_long2}}^{\tfrac12} J^{\tfrac{\alpha(n)}{2}} + C_{\ref{lem:mainbump1_long1}}^{\tfrac12} J^{\tfrac{\alpha(n)}{2}} \le C_{\ref{lem:mainbump1_long}}^{\tfrac12} J^{\tfrac{\alpha(n)}{2}}.  \]
\end{proof}

\begin{lemma}[constant $C_{\ref{lem:mainbump1_long}}$ \auto]\label{constant main bump one long}\uses{constant main bump one long one,constant main bump one long two,lem:mainbump1_long}\usesdefs{}
\lean{Auto.constantMainBumpOneLong}
\leanok
\begin{equation}\label{constant main bump one long bound}
C_{\ref{lem:mainbump1_long}}<\tfrac{11}{16}2^{606}<2^{606}.
\end{equation}
\end{lemma}
\begin{proof}[Proof \auto]
Use Lemmas \ref{constant main bump one long one} and \ref{constant main bump one long two} in the defining sum. The first contribution after multiplication by $2$ is smaller than $\frac{1397}{2048}2^{606}$, while the second is smaller than $\frac89 2^{544}$. Their sum is smaller than $\frac{11}{16}2^{606}$.
\end{proof}

\begin{lemma}\label{lem:mainbump1}\uses{lem:shortlongftc_reduction,lem:mainbump1_long,lem:ft_deriv_mul}\usesdefs{A_def,def:cn-window,def:unipair}
\lean{Auto.mainBumpOne}
\leanok
Let $J\ge1$. Then
\begin{equation}\label{mainbump1}
\|A_t(\phi_0)\|_{V_{2,J}(t\in(0,\infty);L^2)}^2
\le C_{\ref{lem:mainbump1}}J^{\alpha(n)},
\end{equation}
where
\begin{equation}\label{auto:main-bump-one-constant-definition}
C_{\ref{lem:mainbump1}}
=2^4(3C_{\ref{def:unipair}})^2C_{\ref{lem:main_aux1}}
+2C_{\ref{lem:mainbump1_long}}.
\end{equation}
\end{lemma}

\begin{proof}[Proof \auto]
Let
\[
c_1=3C_{\ref{def:unipair}}.
\]
Apply Lemma \ref{lem:shortlongftc_reduction} with $c_1^{-1}\phi_0$ in place of $\phi$ and
\[
A=c_1^{-2}C_{\ref{lem:mainbump1_long}}J^{\alpha(n)}.
\]
The long variation assumption follows from Lemma \ref{lem:mainbump1_long}. It remains to verify the two assumptions on $c_1^{-1}T\phi_0$.

By Lemma \ref{lem:ft_deriv_mul}, $\widehat{T\phi_0}$ is supported in $\mathrm{Ann}_1(1,2)$. The identity \eqref{auto:Fourier-derivatives-of-T} and \eqref{windowbound} bound all Fourier derivatives of $T\phi_0$ through order $2$ by $3C_{\ref{def:unipair}}$. Hence both assumptions of Lemma \ref{lem:shortlongftc_reduction} hold after division by $c_1$, and that lemma gives the stated constant.
\end{proof}

\begin{lemma}[constant $C_{\ref{lem:mainbump1}}$ \auto]\label{constant main bump one}\uses{constant main auxiliary one,lem:mainbump1,constant main bump one long}\usesdefs{def:unipair}
\lean{Auto.constantMainBumpOne}
\leanok
\begin{equation}\label{constant main bump one bound}
C_{\ref{lem:mainbump1}}<\tfrac78 2^{610}<2^{610}.
\end{equation}
\end{lemma}
\begin{proof}[Proof \auto]
Use $C_{\ref{def:unipair}}=2^{15}$ and Lemma \ref{constant main auxiliary one} in the first term of \eqref{auto:main-bump-one-constant-definition}. It is smaller than
\[
\tfrac{12573}{16384}2^{610}.
\]
Lemma \ref{constant main bump one long} bounds the second term by
\[
\tfrac{11}{16}2^{607}.
\]
The sum is smaller than $\frac78 2^{610}$.
\end{proof}

\begin{lemma}\label{lem:main_aux2}\uses{lem:shortlongftc_reduction,lem:main_aux1,lem:bootstrap}\usesdefs{A_def}
\lean{Auto.mainAuxTwo}
\leanok
Let $\psi$ satisfy \eqref{main_aux1_supp} and \eqref{auto:main-auxiliary-one-Fourier-derivative-assumption}. Assume also that
\begin{equation}\label{auto:main-auxiliary-two-Fourier-derivative-assumption}
|\widehat{T\psi}^{(m)}(\xi)|\le1,
\qquad
m\in[3),
\quad
\xi\in\R.
\end{equation}
Then
\begin{equation}\label{main_aux2}
\|A_t(\psi)\|_{V_{2,J}(t\in(0,\infty);L^2)}^2
\le C_{\ref{lem:main_aux2}}J^{\alpha(n)},
\end{equation}
where $C_{\ref{lem:main_aux2}}=24C_{\ref{lem:main_aux1}}$.
\end{lemma}

\begin{proof}
We apply Lemma \ref{lem:shortlongftc_reduction} with $\phi=\psi$. The support of $\widehat{T\psi}$ is contained in the support of $\widehat\psi$, and the other assumption on $T\psi$ is \eqref{auto:main-auxiliary-two-Fourier-derivative-assumption}.

We will show that \eqref{shortlongftc_reduction1} holds with $A=4C_{\ref{lem:main_aux1}}J^{\alpha(n)}$, that is,
\begin{equation}\label{eq:main_aux2_long}
\|A_t(\psi)\|^2_{V_{2,J}(t\in2^\Z;L^2)}
\le4C_{\ref{lem:main_aux1}}J^{\alpha(n)}.
\end{equation}
Lemma \ref{lem:shortlongftc_reduction} then gives
\[
\|A_t(\psi)\|_{V_{2,J}(t\in(0,\infty);L^2)}^2
\le2^4C_{\ref{lem:main_aux1}}J^{\alpha(n)}
+2^3C_{\ref{lem:main_aux1}}J^{\alpha(n)}
=24C_{\ref{lem:main_aux1}}J^{\alpha(n)}.
\]
By Lemma \ref{lem:bootstrap},
\[
\|A_t(\psi)\|^2_{V_{2,J}(t\in2^\Z;L^2)}
\le4\sup\sum_{\ell\in[J)}\|A_{2^{k_\ell}}(\psi)\|_2^2.
\]
Lemma \ref{lem:main_aux1}, applied with $t=1$, gives \eqref{eq:main_aux2_long} because $\psi$ satisfies both of its assumptions.
\end{proof}

\begin{lemma}[constant $C_{\ref{lem:main_aux2}}$ \auto]\label{constant main auxiliary two}\uses{constant main auxiliary one,lem:main_aux2}\usesdefs{}
\lean{Auto.constantMainAuxiliaryTwo}
\leanok
\begin{equation}\label{constant main auxiliary two bound}
C_{\ref{lem:main_aux2}}<2^{578}.
\end{equation}
\end{lemma}
\begin{proof}[Proof \auto]
Lemma \ref{constant main auxiliary one} and the defining factor $24$ give
\[
C_{\ref{lem:main_aux2}}
<24\cdot\tfrac{1397}{2048}2^{573}
=\tfrac{4191}{8192}2^{578}.
\]
\end{proof}

\begin{lemma}\label{lem:mainbump2}\uses{lem:rescaling,lem:main_aux2,lem:abs_deriv_ft_Tphi3_le,lem:ft_phi3_eq,lem:abs_deriv_ft_phi3_le,lem:ft_deriv_mul}\usesdefs{A_def,defn:window based bump functions}
\lean{Auto.mainBumpTwo}
\leanok
For all $k\ge-2$,
\begin{equation}\label{mainbump2}
\|A_t(\varphi_{0,k})\|_{V_{2,J}(t\in(0,\infty);L^2)}^2
\le C_{\ref{lem:mainbump2}}2^{-2k}J^{\alpha(n)},
\end{equation}
where
\begin{equation}\label{auto:main-bump-two-constant-definition}
C_{\ref{lem:mainbump2}}
=C_{\ref{lem:main_aux2}}C_{\ref{lem:abs_deriv_ft_Tphi3_le},2}^2.
\end{equation}
\end{lemma}

\begin{proof}[Proof \auto]
By Lemma \ref{lem:rescaling}, it suffices to prove
\begin{equation}\label{mainbump2_1}
\|A_t(\varphi_{3,k})\|_{V_{2,J}(t\in(0,\infty);L^2)}^2
\le C_{\ref{lem:mainbump2}}J^{\alpha(n)}.
\end{equation}
Apply Lemma \ref{lem:main_aux2} with $c_1^{-1}\varphi_{3,k}$ in place of $\psi$, where
\[
c_1=C_{\ref{lem:abs_deriv_ft_Tphi3_le},2}.
\]
By Lemma \ref{lem:ft_phi3_eq},
\begin{equation}\label{mainbump2_supp}
\operatorname{supp}(\widehat{\varphi_{3,k}})
\subset[-1,-2^{-2}]\cup[2^{-2},1]
\subset\mathrm{Ann}_1(1,2^2).
\end{equation}
The support of $\widehat{T\varphi_{3,k}}$ is contained in the same annulus by Lemma \ref{lem:ft_deriv_mul}. Lemmas \ref{lem:abs_deriv_ft_phi3_le} and \ref{lem:abs_deriv_ft_Tphi3_le} show that, after division by $c_1$, the Fourier derivatives of both functions through order $2$ are bounded by $1$. Thus all assumptions of Lemma \ref{lem:main_aux2} hold.
\end{proof}

\begin{lemma}[constant $C_{\ref{lem:mainbump2}}$ \auto]\label{constant main bump two}\uses{constant main auxiliary two,lem:abs_deriv_ft_Tphi3_le,constant T phi three derivative,lem:mainbump2}\usesdefs{def:unipair}
\lean{Auto.constantMainBumpTwo}
\leanok
\begin{equation}\label{constant main bump two bound}
C_{\ref{lem:mainbump2}}<\tfrac{27}{32}2^{658}<2^{658}.
\end{equation}
\end{lemma}
\begin{proof}[Proof \auto]
Use Lemma \ref{constant main auxiliary two}. Lemma \ref{constant T phi three derivative} gives the exact value
\[
C_{\ref{lem:abs_deriv_ft_Tphi3_le},2}
=1296C_{\ref{def:unipair}}^2+2630C_{\ref{def:unipair}}+2220
=1391655585964
<\tfrac{41}{64}2^{41}
\]
with $C_{\ref{def:unipair}}=2^{15}$ in \eqref{auto:main-bump-two-constant-definition}. Hence
\[
C_{\ref{lem:mainbump2}}
<\tfrac{4191}{8192}\left(\tfrac{41}{64}\right)^2 2^{660}
<\tfrac{27}{32}2^{658}.
\]
\end{proof}

\begin{lemma}\label{lem:leftbump}\uses{lem:rescaling,lem:main_aux2,lem:theta_prim}\usesdefs{A_def,def:unipair,defn:window based bump functions}
\lean{Auto.leftBump}
\leanok
For every $k\le-1$,
\begin{equation}\label{leftbump}
\|A_t(\varphi_{1,k})\|_{V_{2,J}(t\in(0,\infty);L^2)}^2
\le C_{\ref{lem:leftbump}}2^{2k}J^{\alpha(n)},
\end{equation}
where
\begin{equation}\label{auto:left-bump-constant-definition}
C_{\ref{lem:leftbump}}
=C_{\ref{lem:main_aux2}}(2^{14}C_{\ref{def:unipair}})^2.
\end{equation}
\end{lemma}

\begin{proof}[Proof \auto]
We have
\[
\varphi_{1,k}=2^k\widetilde\theta_{(2^k)}.
\]
By Lemma \ref{lem:rescaling}, it suffices to apply Lemma \ref{lem:main_aux2} to $(2^{14}C_{\ref{def:unipair}})^{-1}\widetilde\theta$. The support assumptions are \eqref{thetaprim_supp} and \eqref{Tthetaprim_supp}. The finite Fourier-derivative estimates in the proof of Lemma \ref{lem:theta_prim}, together with
\[
(\widehat{T\widetilde\theta})^{(m)}
=-\xi\widehat{\widetilde\theta}^{(m+1)}
-m\widehat{\widetilde\theta}^{(m)},
\]
bound the Fourier derivatives of both functions through order $2$ by $2^{14}C_{\ref{def:unipair}}$. Thus both Fourier-derivative assumptions hold after the stated normalization.
\end{proof}

\begin{lemma}[constant $C_{\ref{lem:leftbump}}$ \auto]\label{constant left bump}\uses{constant main auxiliary two,lem:leftbump}\usesdefs{def:unipair}
\lean{Auto.constantLeftBump}
\leanok
\begin{equation}\label{constant left bump bound}
C_{\ref{lem:leftbump}}<\tfrac{33}{64}2^{636}<2^{636}.
\end{equation}
\end{lemma}
\begin{proof}[Proof \auto]
Use Lemma \ref{constant main auxiliary two} and $C_{\ref{def:unipair}}=2^{15}$ in \eqref{auto:left-bump-constant-definition}. This gives
\[
C_{\ref{lem:leftbump}}
<\tfrac{4191}{8192}2^{636}
<\tfrac{33}{64}2^{636}.
\]
\end{proof}

\begin{lemma}\label{lem:leftbump1_short1}\uses{Extension of sequences,A to Lambda,lem:theta_prim,lem:thetat_offcenter,P:induct-positive-terms-reduction-whitney,lem:int_fct,lem:phi4_supp}\usesdefs{A_def,bracket bump,normalized function tuples,auto:prism-form-definition,multiplicatively spaced monotone sequences,defn:window based bump functions}
\lean{Auto.leftBumpOneShortOne}
\leanok
Let $\gamma=\frac12$. For every $k\le-1$ and every strictly increasing sequence of integers $(k_j)_{j\in[J)}$,
\begin{equation}\label{auto:left-bump-one-short-one-bound}
\sum_{j\in[J)}2^{-(1+\gamma)k}\int_1^2
\|A_{2^{k_j}t}(\varphi_{4,k})\|_2^2\,\tfrac{dt}{t}
\le C_{\ref{lem:leftbump1_short1}}J^{\alpha(n)},
\end{equation}
where
\begin{equation}\label{auto:left-bump-one-short-one-constant-definition}
C_{\ref{lem:leftbump1_short1}}
=2^2C_{\ref{P:induct-positive-terms-reduction-whitney}}
C_{\ref{lem:theta_prim},2}^2C_{\ref{lem:thetat_offcenter}}.
\end{equation}
\end{lemma}

\begin{proof}[Proof \auto]
By Proposition \ref{Extension of sequences}, the finite sequence $(2^{k_j})_{j\in[J)}$ extends to a sequence $a\in A$. By Lemma \ref{A to Lambda} and linearity, the left-hand side equals $\Lambda_1(M)(\F)$, where
\[
M=\sum_{j\in[J)}\Psi_{(a(j))},
\qquad
\Psi=2^{-(1+\gamma)k}\int_1^2(\varphi_{4,k})_{(t)}^{\otimes2}\,\tfrac{dt}{t}.
\]
Put
\[
c_1=2^2C_{\ref{lem:theta_prim},2}^2C_{\ref{lem:thetat_offcenter}}.
\]
We apply Proposition \ref{P:induct-positive-terms-reduction-whitney} to $c_1^{-1}\Psi$. Symmetry, positivity, and the Fourier support condition follow from Lemmas \ref{lem:int_fct} and \ref{lem:phi4_supp}.

By Lemma \ref{lem:theta_prim}, for $t\in[1,2]$,
\[
|(\varphi_{4,k})_{(t)}(u)|
\le2^{k+1}C_{\ref{lem:theta_prim},2}
\langle u-t2^{-k}\rangle^2.
\]
Lemma \ref{lem:thetat_offcenter} therefore gives
\[
|\Psi(u,v)|
\le c_1\bigl(
\langle u\rangle^{3/2}\langle v\rangle^{3/2}
+\langle u+v\rangle^{3/2}\langle u-v\rangle^{3/2}
\bigr).
\]
This proves \eqref{p1-on-W_decay} for $c_1^{-1}\Psi$.

Write $h=\widehat{\widetilde\theta}$. By Lemma \ref{lem:phi4_supp},
\[
\widehat{\varphi_{4,k}}(\xi)
=2^ke^{-2\pi i2^{-k}\xi}h(\xi).
\]
In $\widehat\Psi(\xi,-\xi)$ the two translation phases cancel. Differentiating the remaining product through order $2$, using $t\in[1,2]$ and the Fourier-derivative estimates in the proof of Lemma \ref{lem:theta_prim}, gives, for $m\in[3)$,
\[
\left|\tfrac{d^m}{d\xi^m}\widehat\Psi(\xi,-\xi)\right|
\le2^4C_{\ref{lem:theta_prim},2}^2 2^{(1-\gamma)k}
\le c_1.
\]
Thus \eqref{p1-on-W_der} holds. Proposition \ref{P:induct-positive-terms-reduction-whitney} proves the result.
\end{proof}

\begin{lemma}[constant $C_{\ref{lem:leftbump1_short1}}$ \auto]\label{constant left bump one short one}\uses{constant Whitney reduction,constant theta primitive,constant off center bump,lem:leftbump1_short1}\usesdefs{}
\lean{Auto.constantLeftBumpOneShortOne}
\leanok
\begin{equation}\label{constant left bump one short one bound}
C_{\ref{lem:leftbump1_short1}}
<2^{628}.
\end{equation}
\end{lemma}
\begin{proof}[Proof \auto]
Use Lemmas \ref{constant Whitney reduction}, \ref{constant theta primitive}, and \ref{constant off center bump} in \eqref{auto:left-bump-one-short-one-constant-definition}. This gives
\[
C_{\ref{lem:leftbump1_short1}}
<\tfrac{185801}{262144}2^{628}
<\tfrac{23}{32}2^{628}.
\]
\end{proof}

\begin{lemma}\label{lem:leftbump1_short2}\uses{Extension of sequences,A to Lambda,lem:theta_prim,lem:theta_decay,lem:thetat_offcenter,P:induct-positive-terms-reduction-whitney,lem:int_fct,lem:phi4_supp,lem:ft_deriv_mul}\usesdefs{A_def,bracket bump,normalized function tuples,auto:prism-form-definition,multiplicatively spaced monotone sequences,defn:window based bump functions}
\lean{Auto.leftBumpOneShortTwo}
\leanok
Let $\gamma=\frac12$. For every $k\le-1$ and every strictly increasing sequence of integers $(k_j)_{j\in[J)}$,
\begin{equation}\label{auto:left-bump-one-short-two-bound}
\sum_{j\in[J)}2^{(1-\gamma)k}\int_1^2
\|A_{2^{k_j}t}(T\varphi_{4,k})\|_2^2\,\tfrac{dt}{t}
\le C_{\ref{lem:leftbump1_short2}}J^{\alpha(n)},
\end{equation}
where, with
\begin{equation}\label{auto:left-bump-one-short-two-auxiliary-constant}
C=\max\bigl(C_{\ref{lem:theta_prim},2},C_{\ref{lem:theta_decay},2},C_{\ref{lem:theta_decay},3}\bigr),
\end{equation}
we have
\begin{equation}\label{auto:left-bump-one-short-two-constant-definition}
C_{\ref{lem:leftbump1_short2}}
=2^4C_{\ref{P:induct-positive-terms-reduction-whitney}}
C_{\ref{lem:thetat_offcenter}}C^2.
\end{equation}
\end{lemma}

\begin{proof}[Proof \auto]
By Proposition \ref{Extension of sequences}, the finite sequence $(2^{k_j})_{j\in[J)}$ extends to a sequence $a\in A$. By Lemma \ref{A to Lambda} and linearity, the left-hand side equals $\Lambda_1(M)(\F)$, where
\[
M=\sum_{j\in[J)}\Psi_{(a(j))},
\qquad
\Psi=2^{(1-\gamma)k}\int_1^2(T\varphi_{4,k})_{(t)}^{\otimes2}\,\tfrac{dt}{t}.
\]
Let $C$ be as in \eqref{auto:left-bump-one-short-two-auxiliary-constant} and put
\[
c_1=2^4C^2C_{\ref{lem:thetat_offcenter}}.
\]
We apply Proposition \ref{P:induct-positive-terms-reduction-whitney} to $c_1^{-1}\Psi$. Symmetry, positivity, and the Fourier support condition follow from Lemmas \ref{lem:int_fct} and \ref{lem:phi4_supp}.

Directly from the definitions,
\[
T\varphi_{4,k}(u)
=2^k\widetilde\theta(u-2^{-k})
+2^k(X\theta)(u-2^{-k})
+\theta(u-2^{-k}).
\]
Use order $2$ in Lemmas \ref{lem:theta_prim} and \ref{lem:theta_decay} for the first and third terms, and order $3$ in Lemma \ref{lem:theta_decay} for the second. Since $|x|\langle x\rangle^3\le\langle x\rangle^2$, for $t\in[1,2]$ this gives
\[
|(T\varphi_{4,k})_{(t)}(u)|
\le2^2C\langle u-t2^{-k}\rangle^2.
\]
Lemma \ref{lem:thetat_offcenter} now gives
\[
|\Psi(u,v)|
\le c_1\bigl(
\langle u\rangle^{3/2}\langle v\rangle^{3/2}
+\langle u+v\rangle^{3/2}\langle u-v\rangle^{3/2}
\bigr).
\]
Thus \eqref{p1-on-W_decay} holds for $c_1^{-1}\Psi$.

Writing again $h=\widehat{\widetilde\theta}$, Lemma \ref{lem:phi4_supp} and Lemma \ref{lem:ft_deriv_mul} give
\[
\widehat{T\varphi_{4,k}}(\xi)
=e^{-2\pi i2^{-k}\xi}
\left(2\pi i\xi h(\xi)-2^k\xi h'(\xi)\right).
\]
The phases cancel in $\widehat\Psi(\xi,-\xi)$. Differentiating the remaining amplitudes through order $2$ uses $h$ only through order $3$. The finite estimates in the proof of Lemma \ref{lem:theta_prim} bound each of those amplitude derivatives by $C$, and hence, for $m\in[3)$,
\[
\left|\tfrac{d^m}{d\xi^m}\widehat\Psi(\xi,-\xi)\right|
\le2^4C^2 2^{(1-\gamma)k}
\le c_1.
\]
Thus \eqref{p1-on-W_der} holds, and Proposition \ref{P:induct-positive-terms-reduction-whitney} completes the proof.
\end{proof}

\begin{lemma}[constant $C_{\ref{lem:leftbump1_short2}}$ \auto]\label{constant left bump one short two}\uses{lem:theta_prim,constant theta primitive,constant theta decay,constant Whitney reduction,constant off center bump,lem:leftbump1_short2}\usesdefs{}
\lean{Auto.constantLeftBumpOneShortTwo}
\leanok
\begin{equation}\label{constant left bump one short two bound}
C_{\ref{lem:leftbump1_short2}}
<2^{630}.
\end{equation}
\end{lemma}
\begin{proof}[Proof \auto]
Lemmas \ref{constant theta primitive} and \ref{constant theta decay} give
\[
C\le C_{\ref{lem:theta_prim},2}\le2^{31}.
\]
Use this and Lemmas \ref{constant Whitney reduction} and \ref{constant off center bump} in \eqref{auto:left-bump-one-short-two-constant-definition}. The result is
\[
C_{\ref{lem:leftbump1_short2}}
<\tfrac{185801}{262144}2^{630}
<\tfrac{23}{32}2^{630}.
\]
\end{proof}

\begin{lemma}\label{lem:leftbump1_long}\uses{lem:bootstrap,Extension of sequences,A to Lambda,lem:theta_prim,P:induct-positive-terms-reduction-whitney,lem:phi4_supp,lem:form_pos}\usesdefs{A_def,bracket bump,normalized function tuples,auto:prism-form-definition,multiplicatively spaced monotone sequences,defn:window based bump functions}
\lean{Auto.leftBumpOneLong}
\leanok
Let $\gamma=\frac12$. For every $k\le-1$,
\begin{equation}\label{auto:left-bump-one-long-bound}
\|A_t(\varphi_{4,k})\|_{V_{2,J}(t\in2^\mathbb Z;L^2)}^2
\le C_{\ref{lem:leftbump1_long}}2^{\gamma k}J^{\alpha(n)},
\end{equation}
where
\begin{equation}\label{auto:left-bump-one-long-constant-definition}
C_{\ref{lem:leftbump1_long}}
=2^6C_{\ref{P:induct-positive-terms-reduction-whitney}}
C_{\ref{lem:theta_prim},2}^2.
\end{equation}
\end{lemma}

\begin{proof}[Proof \auto]
By Lemma \ref{lem:bootstrap}, it suffices to estimate
\[
4\sup\sum_{j\in[J)}\|A_{2^{k_j}}(\varphi_{4,k})\|_2^2.
\]
By Proposition \ref{Extension of sequences}, the finite sequence $(2^{k_j})_{j\in[J)}$ extends to a sequence $a\in A$. By Lemma \ref{A to Lambda}, the sum is $\Lambda_1(M)(\F)$ with
\[
M=\sum_{j\in[J)}(\varphi_{4,k})_{(a(j))}^{\otimes2}.
\]
Set
\[
\Psi=2^{-\gamma k}\varphi_{4,k}^{\otimes2},
\qquad
c_1=2^4C_{\ref{lem:theta_prim},2}^2.
\]
We apply Proposition \ref{P:induct-positive-terms-reduction-whitney} to $c_1^{-1}\Psi$. Symmetry, positivity, and the Fourier support condition follow from Lemmas \ref{lem:phi4_supp} and \ref{lem:form_pos}.

Put $R=2^{-k}$. We claim that
\[
R^{-3/2}\langle u-R\rangle^2\langle v-R\rangle^2
\le2^4\bigl(
\langle u\rangle^{3/2}\langle v\rangle^{3/2}
+\langle u+v\rangle^{3/2}\langle u-v\rangle^{3/2}
\bigr).
\]
If $|u-R|,|v-R|<R/2$, then
\[
R^{-3/2}\le(7/2)^{3/2}\langle u+v\rangle^{3/2}
\]
and
\[
\langle u-R\rangle^2\langle v-R\rangle^2
\le\langle u-v\rangle^{3/2},
\]
so the claim follows. Otherwise, by symmetry suppose $|u-R|\ge R/2$. Then
\[
\langle u-R\rangle^2\le3^{3/2}\langle u\rangle^{3/2}
\]
and
\[
R^{-3/2}\langle v-R\rangle^2
\le(3/2)^{3/2}\langle v\rangle^{3/2}.
\]
Their product is smaller than $2^4\langle u\rangle^{3/2}\langle v\rangle^{3/2}$, proving the claim.

Lemma \ref{lem:theta_prim} and the claim prove \eqref{p1-on-W_decay} for $c_1^{-1}\Psi$. Moreover,
\[
\widehat{\varphi_{4,k}}(\xi)
=2^ke^{-2\pi i2^{-k}\xi}\widehat{\widetilde\theta}(\xi).
\]
The phases cancel in $\widehat\Psi(\xi,-\xi)$. Differentiating the remaining amplitudes through order $2$, and using the finite estimates in the proof of Lemma \ref{lem:theta_prim}, gives, for $m\in[3)$,
\[
\left|\tfrac{d^m}{d\xi^m}\widehat\Psi(\xi,-\xi)\right|
\le2^2C_{\ref{lem:theta_prim},2}^2 2^{(2-\gamma)k}
\le c_1.
\]
Thus \eqref{p1-on-W_der} holds. Proposition \ref{P:induct-positive-terms-reduction-whitney}, followed by Lemma \ref{lem:bootstrap}, proves the result.
\end{proof}

\begin{lemma}[constant $C_{\ref{lem:leftbump1_long}}$ \auto]\label{constant left bump one long}\uses{constant Whitney reduction,constant theta primitive,lem:leftbump1_long}\usesdefs{}
\lean{Auto.constantLeftBumpOneLong}
\leanok
\begin{equation}\label{constant left bump one long bound}
C_{\ref{lem:leftbump1_long}}<2^{625}.
\end{equation}
\end{lemma}
\begin{proof}[Proof \auto]
Use Lemmas \ref{constant Whitney reduction} and \ref{constant theta primitive} in \eqref{auto:left-bump-one-long-constant-definition}. This gives the displayed estimate.
\end{proof}

\begin{lemma}\label{lem:leftbump1}\uses{lem:rescaling,lem:shortlongjumps,lem:leftbump1_long,lem:leftbump1_short1,lem:leftbump1_short2,lem:ftccs-R,lem:ftc_ATphi}\usesdefs{A_def,defn:window based bump functions}
\lean{Auto.leftBumpOne}
\leanok

Let $\gamma=\frac12$. For every $k\le -1$,
\begin{equation}\label{leftbump1}
\|A_{t}(\varphi_{2,k})\|_{V_{2,J}(t\in(0,\infty);L^2)}^2  \le C_{\ref{lem:leftbump1}} 2^{\gamma k} J^{\alpha(n)}, \end{equation}
where  $C_{\ref{lem:leftbump1}}
=2^7C_{\ref{lem:leftbump1_short1}}^{1/2}C_{\ref{lem:leftbump1_short2}}^{1/2}
+2C_{\ref{lem:leftbump1_long}}$.
\end{lemma}

\begin{proof}
We have
\[\varphi_{2,k}(u)=-(\varphi_{4,k})_{(2^k)}(u),\]
so by Lemma \ref{lem:rescaling} (rescaling) it suffices to show
\begin{equation}
\|A_{t}({\varphi_{4,k}}) \|_{V_{2,J}(t\in(0,\infty);L^2)}^2\le C_{\ref{lem:leftbump1}} 2^{\gamma k} J^{\alpha(n)}
\end{equation}

Let $(t_j)_{j\in[J+1)}$ be an increasing sequence of positive real numbers.
By Lemma \ref{lem:shortlongjumps} (with $r=2$, $a(t) = A_t({\varphi_{4,k}}) $ and $B=L^2$),
\begin{equation}  \Big(\sum_{j\in [J)} \|A_{t_{j+1}}({\varphi_{4,k}}) - A_{t_j}({\varphi_{4,k}}) \|_2^2\Big)^{1/2} \le 2 \Big(\sum_{\ell \in L} \|A_t({\varphi_{4,k}}) \|_{V_{2,J}(t\in [2^\ell, 2^{\ell+1}]; L^2)}^2\Big)^{1/2}
\end{equation}
\begin{equation}
+\|A_t({\varphi_{4,k}}) \|_{V_{2,J}(t\in 2^\mathbb{Z}; L^2)}.
\end{equation}
where $L$ is the set of $\ell\in \mathbb{Z}$ such that $2^{\ell}\le t_j<2^{\ell+1}$ for some $j\in [J+1)$.
Squaring and using $(a+b)^2\le 2(a^2+b^2)$ we obtain
\begin{equation}\label{shortlongftc_reduction_pf12} \sum_{j\in [J)} \|A_{t_{j+1}}(\varphi_{4,k})- A_{t_j}(\varphi_{4,k})\|_2^2  \le 2^3 \sum_{\ell \in L} \|A_t(\varphi_{4,k})\|_{V_{2,J}(t\in [2^\ell, 2^{\ell+1}]; L^2)}^2
\end{equation}
\begin{equation}\label{shortlongftc_reduction_pf22}
+ 2 \|A_t(\varphi_{4,k})\|^2_{V_{2,J}(t\in 2^\mathbb{Z}; L^2)}.
\end{equation}

 Now we use Lemma \ref{lem:leftbump1_long} to estimate
 \begin{equation}
     \label{leftbump_long_est}
     \|A_{t}(\varphi_{4,k})\|^2_{V_{2,J}(t\in 2^\mathbb{Z}; L^2)} \le C_{\ref{lem:leftbump1_long}} 2^{\gamma k}J^{\alpha(n)}
 \end{equation}
We will also show
\begin{equation}
    \label{leftbump_short_est}
    \sum_{\ell \in L} \|A_t(\varphi_{4,k})\|_{V_{2,J}(t\in [2^\ell, 2^{\ell+1}]; L^2)}^2 \le2^4C_{\ref{lem:leftbump1_short1}}^{1/2}C_{\ref{lem:leftbump1_short2}}^{1/2}2^{\gamma k}J^{\alpha(n)}
\end{equation}
The estimates \eqref{leftbump_short_est}, \eqref{leftbump_long_est}, and \eqref{shortlongftc_reduction_pf22} together give
\[\|A_{t}(\varphi_{4,k})\|_{V_{2,J}(t\in(0,\infty);L^2)}^2  \le (2^7C_{\ref{lem:leftbump1_short1}}^{1/2} C_{\ref{lem:leftbump1_short2}}^{1/2}  + 2C_{\ref{lem:leftbump1_long}}) 2^{\gamma k} J^{\alpha(n)} = C_{\ref{lem:leftbump1}} 2^{\gamma k} J^{\alpha(n)}  \]
as desired.
To see \eqref{leftbump_short_est},
for each $\ell\in\mathbb{Z}$ we estimate
\[ \|A_t(\varphi_{4,k})\|_{V_{2,J}(t\in [2^\ell, 2^{\ell+1}]; L^2)}^2 \le \int_{\mathbb{R}^n}  \|A_t(\varphi_{4,k})(x)\|_{V_{2,J}(t\in [2^\ell, 2^{\ell+1}])}^2\, dx.\]
Fix $x\in\mathbb{R}^n$ and $\ell\in\mathbb{Z}$. Set $a(t)=A_t(\varphi_{4,k})(x)$ and use Lemma \ref{lem:ftccs-R} (specifically, \eqref{ftccs2-R}) to estimate for each $x$
\[ \|A_t(\varphi_{4,k})(x)\|_{V_{2,J}(t\in [2^\ell, 2^{\ell+1}])}^2 \le  2^3
\|a(t)\|_{L^2(t\in [2^\ell, 2^{\ell+1}],\tfrac{dt}{t})} \|ta'(t)\|_{L^2(t\in [2^\ell, 2^{\ell+1}],\tfrac{dt}{t})}.  \]
By Lemma \ref{lem:ftc_ATphi},
\[ \|ta'(t)\|_{L^2(t\in [2^\ell, 2^{\ell+1}],\tfrac{dt}{t})}^2 = \int_{1}^2 |A_{2^\ell t}(T\varphi_{4,k})(x)|^2\,\tfrac{dt}{t}. \]
Thus for every $x\in\mathbb{R}^n$ and $\ell\in\mathbb{Z}$,
\[ \|A_t(\varphi_{4,k})(x)\|_{V_{2,J}(t\in [2^\ell, 2^{\ell+1}])}^2\]
\[ \le 2^3 \Big( 2^{-k} \int_{1}^2 |A_{2^\ell t}(\varphi_{4,k})(x)|^2\,\tfrac{dt}{t} \Big)^{1/2} \Big( 2^{k}  \int_{1}^2 |A_{2^\ell t}(T{\varphi_{4,k}})(x)|^2\,\tfrac{dt}{t} \Big)^{1/2}.  \]
We integrate over $x\in\mathbb{R}^n$ and sum in $\ell\in L$.  Then we apply the Cauchy-Schwarz inequality in $x$ and in the summation.   We also interchange the integrals over $x$ with the integrals $t$ and the summation. This gives
\[ \sum_{\ell \in L}\|A_t(\varphi_{4,k})\|_{V_{2,J}(t\in [2^\ell, 2^{\ell+1}]; L^2)}^2 \]
\[\le2^3\Big( 2^{-k} \int_1^2 \sum_{\ell\in L}\|A_{2^\ell t}(\varphi_{4,k})\|_2^2 \tfrac{dt}{t} \Big)^{1/2}\Big( 2^k \int_1^2 \sum_{\ell \in L}\|A_{2^\ell t}(T\varphi_{4,k})\|_2^2 \tfrac{dt}{t} \Big)^{1/2}   . \]
Order the elements of $L$ increasingly and apply Lemmas \ref{lem:leftbump1_short1} and \ref{lem:leftbump1_short2} directly to the two integrated sums. We also use $\#L\le J+1$.

This bounds the last display by
$2^4C_{\ref{lem:leftbump1_short1}}^{1/2}C_{\ref{lem:leftbump1_short2}}^{1/2}2^{\gamma k}J^{\alpha(n)},$
where we used $(J+1)^{\alpha(n)}\le2J^{\alpha(n)}$. This shows \eqref{leftbump_short_est}.
\end{proof}

\begin{lemma}[constant $C_{\ref{lem:leftbump1}}$ \auto]\label{constant left bump one}\uses{constant left bump one short one,constant left bump one short two,constant left bump one long,lem:leftbump1}\usesdefs{}
\lean{Auto.constantLeftBumpOne}
\leanok
\begin{equation}\label{constant left bump one bound}
C_{\ref{lem:leftbump1}}<\tfrac{23}{32}2^{636}<2^{636}.
\end{equation}
\end{lemma}
\begin{proof}[Proof \auto]
The proofs of Lemmas \ref{constant left bump one short one} and \ref{constant left bump one short two} give the common factor
\[
\tfrac{185801}{262144}.
\]
Hence the first term in the defining sum is smaller than
\[
\tfrac{185801}{262144}2^{636}.
\]
Lemma \ref{constant left bump one long} bounds the second term by
\[
\tfrac{1397}{2097152}2^{636}.
\]
Their sum is smaller than $\frac{23}{32}2^{636}$.
\end{proof}

\subsubsection{\texorpdfstring{Proof of Theorem \ref{thm:nct main real}}{Proof of main theorem}}
\begin{proof}[Proof \auto]
Fix $j\in [J)$.
By Lemma \ref{lem:smoothingdecomp},
$\|A_{t_{j + 1}}(\mathbf{1}_{[0,1]}) - A_{t_{j}}(\mathbf{1}_{[0,1]})\|_{{L}^2(\R^n)}$
is no larger than
\[ \|A_{t_{j+1}}(\phi_0) - A_{t_j}(\phi_0)\|_{L^2(\R^n)}  + \sum_{k=-2}^\infty\|A_{t_{j+1}}(\varphi_{0,k}) - A_{t_{j}}(\varphi_{0,k})\|_{L^2(\R^n)}  + \]
\[ +\sum_{k=-\infty}^{-1}\|A_{t_{j+1}}(\varphi_{1,k}) - A_{t_{j}}(\varphi_{1,k})\|_{L^2(\R^n)}  + \sum_{k=-\infty}^{-1}\|A_{t_{j+1}}({\varphi_{2,k}}) - A_{t_{j}}({\varphi_{2,k}})\|_{L^2(\R^n)}. \]
Here we have also applied Lemma \ref{lem:norm_A_sum_le_sum} to interchange series and integrals.
Applying the triangle inequality in $\ell^2([J))$ and then Lemmas \ref{lem:mainbump1}, \ref{lem:mainbump2}, \ref{lem:leftbump}, and \ref{lem:leftbump1}, we obtain
\[
\Big(\sum_{j\in[J)}\|A_{t_{j+1}}(\mathbf1_{[0,1]})-A_{t_j}(\mathbf1_{[0,1]})\|_2^2\Big)^{1/2}
\]
\[
\le J^{\alpha(n)/2}\Big(
C_{\ref{lem:mainbump1}}^{1/2}
+8C_{\ref{lem:mainbump2}}^{1/2}
+C_{\ref{lem:leftbump}}^{1/2}
+8C_{\ref{lem:leftbump1}}^{1/2}
\Big),
\]
where we used
\[
\sum_{k=-2}^\infty2^{-k}=8,
\qquad
\sum_{k=-\infty}^{-1}2^k=1,
\qquad
\sum_{k=-\infty}^{-1}2^{k/4}\le8.
\]
Consequently,
\[
\sum_{j\in[J)}\|A_{t_{j+1}}(\mathbf1_{[0,1]})-A_{t_j}(\mathbf1_{[0,1]})\|_2^2
\le C_{\mathrm{final}}J^{\alpha(n)},
\]
where
\[
C_{\mathrm{final}}
=2^2\Big(
C_{\ref{lem:mainbump1}}
+2^6C_{\ref{lem:mainbump2}}
+C_{\ref{lem:leftbump}}
+2^6C_{\ref{lem:leftbump1}}
\Big).
\]
Lemmas \ref{constant main bump one}, \ref{constant main bump two}, \ref{constant left bump}, and \ref{constant left bump one} give
\[
C_{\mathrm{final}}
<\tfrac78 2^{612}
+\tfrac{27}{32}2^{666}
+\tfrac{33}{64}2^{638}
+\tfrac{23}{32}2^{644}
<\tfrac78 2^{666}
<2^{666}
=C_{\ref{thm:nct main real}}.
\]
\end{proof}

\section{From real to ergodic}\label{sec:realtoergodic}
The purpose of this section is to prove Theorem \ref{thm:ergodicthm} from Theorem \ref{thm:nct main real}.
This section is the only place in the blueprint where we invoke two external theorems, i.e. theorems that are not proved in this blueprint. These are Theorem \ref{multilinear interpolation external} (multilinear complex interpolation) and Theorem \ref{Calderon transference external} (Calder\'on transference principle).
The latter was recently formalized in Lean by Pernegger \cite{Pernegger2025}.

\begin{proposition}[Permutation of the endpoint estimate]\label{permutation of endpoint estimate}\lean{Auto.permutation_of_endpoint_estimate}\leanok\uses{thm:nct main real}\usesdefs{A_def}
Let $n\ge 2$, let $\pi$ be a permutation of $[n)$, and let $J\ge 1$ and $0<t_0<\cdots<t_J$. For real-valued Schwartz functions $(f_i)_{i\in[n)}$,
\begin{equation}\label{permuted endpoint estimate}
\sum_{j\in[J)}\|A_{t_{j+1}}(\mathbf 1_{[0,1]},\mathbf f)-A_{t_j}(\mathbf 1_{[0,1]},\mathbf f)\|_2^2
\le C_{\ref{thm:nct main real}}J^{1-2^{2-n}}
\prod_{i\in[n)}\|f_i\|_{2^{\pi(i)+\min(n-\pi(i),2)}}^2.
\end{equation}
\end{proposition}

\begin{proof}[Proof \auto]
Apply Theorem \ref{thm:nct main real} after permuting the coordinate axes by $\pi$ and permuting the functions by the same permutation. Lebesgue measure, the $L^2$ norm, and the kernel $\mathbf 1_{[0,1]}$ are unchanged by this coordinate permutation. The normalization in Theorem \ref{thm:nct main real} becomes the product on the right-hand side of \eqref{permuted endpoint estimate}; homogeneity gives the stated form.
\end{proof}

\begin{exttheorem}[Multilinear interpolation]\label{multilinear interpolation external}\lean{Auto.multilinear_interpolation_external}\leanok\uses{}\usesdefs{}
Let $m,n\ge 1$ be integers, let $(X_i,\mu_i)$ and $(Y,\nu)$ be measure spaces, let $p\in[1,\infty]$, and let $T$ be an $n$-linear map on $n$-tuples of complex-valued simple functions with values in $L^p(Y)$. For $a\in[m)$ and $i\in[n)$, let $q_{a,i}\in[1,\infty]$, and assume
\begin{equation}\label{multilinear interpolation endpoint}
\|T(f_0,\ldots,f_{n-1})\|_{L^p(Y)}\le A_a\prod_{i\in[n)}\|f_i\|_{L^{q_{a,i}}(X_i)}.
\end{equation}
Let $\theta_a\ge 0$ satisfy $\sum_{a\in[m)}\theta_a=1$, and define $p_i\in[1,\infty]$ by
\begin{equation}\label{multilinear interpolation exponents}
\tfrac1{p_i}=\sum_{a\in[m)}\tfrac{\theta_a}{q_{a,i}}.
\end{equation}
Then $T$ extends to an $n$-linear map from $\prod_{i\in[n)}L^{p_i}(X_i)$ to $L^p(Y)$ and
\begin{equation}\label{multilinear interpolation conclusion}
\|T(f_0,\ldots,f_{n-1})\|_{L^p(Y)}\le \prod_{a\in[m)}A_a^{\theta_a}\prod_{i\in[n)}\|f_i\|_{L^{p_i}(X_i)}.
\end{equation}
\end{exttheorem}

This is a standard result in interpolation theory, see e.g. \cite{BerghLofstrom}.

\begin{proposition}[Range of exponents]\label{exponent polytope}\lean{Auto.exponent_polytope,Auto.exponent_polytope_of_exponents}\leanok\uses{}\usesdefs{}
Let $p_i\in[1,\infty]$ for $i\in[n)$. The vector $(p_i^{-1})_{i\in[n)}$ belongs to the convex hull of the permutations of
\begin{equation}\label{endpoint reciprocal vector}
(2^{-2},2^{-3},\ldots,2^{-n},2^{-n})
\end{equation}
if and only if
\begin{equation}\label{exponent polytope sum}
\sum_{i\in[n)}p_i^{-1}=\tfrac12
\end{equation}
and, for every nonempty $I\subset[n)$,
\begin{equation}\label{exponent polytope inequalities}
\sum_{i\in I}p_i^{-1}\ge 2^{|I|-n-1}.
\end{equation}
\end{proposition}

\begin{proof}[Proof \auto]
Every permutation of \eqref{endpoint reciprocal vector} satisfies \eqref{exponent polytope sum}. The sum of its $|I|$ smallest coordinates is $2^{|I|-n-1}$, so every convex combination satisfies \eqref{exponent polytope inequalities}.

Conversely, put $x_i=p_i^{-1}$ and let $b_0\ge\cdots\ge b_{n-1}$ be the coordinates in \eqref{endpoint reciprocal vector} in decreasing order. Applying \eqref{exponent polytope inequalities} to complements gives, for every $I\subset[n)$,
\begin{equation}\label{exponent polytope upper inequalities}
\sum_{i\in I}x_i\le \sum_{i=0}^{|I|-1}b_i.
\end{equation}
Let $c\in\R^n$ and simultaneously relabel the coordinates of $c$ and $x$ so that $c_0\ge\cdots\ge c_{n-1}$. By summation by parts, \eqref{exponent polytope sum}, and \eqref{exponent polytope upper inequalities},
\[
\sum_{i\in[n)}c_ix_i
=c_{n-1}\sum_{i\in[n)}x_i+
\sum_{m\in[n-1)}(c_m-c_{m+1})\sum_{i=0}^m x_i
\]
\[
\le c_{n-1}\sum_{i\in[n)}b_i+
\sum_{m\in[n-1)}(c_m-c_{m+1})\sum_{i=0}^m b_i
=\sum_{i\in[n)}c_ib_i.
\]
The last expression is the maximum of $\sum_i c_i y_i$ over the permutations $y$ of \eqref{endpoint reciprocal vector}. The finite-dimensional separation theorem therefore places $x$ in their convex hull.
\end{proof}

\begin{proposition}[Interpolated real-variable estimate]\label{interpolated real variable estimate}\lean{Auto.interpolated_real_variable_estimate}\leanok\uses{permutation of endpoint estimate,exponent polytope,multilinear interpolation external}\usesdefs{A_def}
Let $n\ge 2$, and let $(p_i)_{i\in[n)}$ satisfy \eqref{exponent polytope sum} and \eqref{exponent polytope inequalities}. For every $J\ge1$, every $0<t_0<\cdots<t_J$, and every $n$-tuple of complex-valued measurable functions $f_i\in L^{p_i}(\R^n)$,
\begin{equation}\label{interpolated real estimate}
\sum_{j\in[J)}\|A_{t_{j+1}}(\mathbf1_{[0,1]},\mathbf f)-A_{t_j}(\mathbf1_{[0,1]},\mathbf f)\|_2^2
\le 2^{2n}C_{\ref{thm:nct main real}}J^{1-2^{2-n}}
\prod_{i\in[n)}\|f_i\|_{p_i}^2.
\end{equation}
\end{proposition}

\begin{proof}[Proof \auto]
For fixed $J$ and $(t_j)_{j\in[J+1)}$, consider the $n$-linear map whose value is
\[
\big(A_{t_{j+1}}(\mathbf1_{[0,1]},\mathbf f)-A_{t_j}(\mathbf1_{[0,1]},\mathbf f)\big)_{j\in[J)}
\]
in $L^2(\R^n;\ell^2([J)))$. Proposition \ref{permutation of endpoint estimate} gives its endpoint norm at every permutation of \eqref{endpoint reciprocal vector}. Since all endpoint exponents are finite, approximation by Schwartz functions and the endpoint estimate extend this map uniquely to the corresponding products of real $L^p$ spaces. For complex inputs, expand each $f_i$ into real and imaginary parts and use the triangle inequality. This multiplies every endpoint norm by at most $2^n$. Proposition \ref{exponent polytope} and Theorem \ref{multilinear interpolation external} give \eqref{interpolated real estimate} for complex-valued simple functions. Since \eqref{exponent polytope inequalities} for singleton sets gives $p_i<\infty$, approximation by simple functions and the interpolated estimate extend the map to all $f_i\in L^{p_i}(\R^n)$.
\end{proof}

\begin{exttheorem}[Calder\'on transference principle]\label{Calderon transference external}\lean{Auto.calderon_transference_external}\leanok\uses{}\usesdefs{ergodic averages,A_def}
Let $n\ge2$, let $p_i\in(1,\infty)$ satisfy $\sum_{i\in[n)}p_i^{-1}=1/2$, and let $g:\N\to[1,\infty)$. Assume that for some $C>0$, every $J\ge1$, every $0<t_0<\cdots<t_J$, and every real-valued $f_i\in L^{p_i}(\R^n)$ satisfy
\begin{equation}\label{Calderon transference assumption}
\sum_{j\in[J)}\|A_{t_{j+1}}(\mathbf1_{[0,1]},\mathbf f)-A_{t_j}(\mathbf1_{[0,1]},\mathbf f)\|_2^2
\le Cg(J)\prod_{i\in[n)}\|f_i\|_{p_i}^2.
\end{equation}
Then, for every measure space $(X,\Sigma,\mu)$, every family of mutually commuting measure preserving transformations $(T_i)_{i\in[n)}$, every real-valued $f_i\in L^{p_i}(X)$, every $J\ge1$, and every positive integers $N_0<\cdots<N_J$,
\begin{equation}\label{Calderon transference conclusion}
\sum_{j\in[J)}\|M_{N_{j+1}}(\mathbf f)-M_{N_j}(\mathbf f)\|_2^2
\le C\,C_{\ref{Calderon transference external},n}\,g(J)
\prod_{i\in[n)}\|f_i\|_{p_i}^2,
\end{equation}
where
\begin{equation}\label{Calderon transference constant}
C_{\ref{Calderon transference external},n}
=
2^{4n+6}.
\end{equation}
\end{exttheorem}
This is a standard result and can be found e.g. as Theorem 2.5 of \cite{Pernegger2025}, up to slight restatement with an explicit constant and the specific real-variable averages and ergodic averages used here.
\begin{theorem}[Ergodic jump estimate with general exponents]\label{thm:ergodicthm-interpolation}\lean{Auto.ergodic_jump_estimate_interpolation}\leanok\uses{interpolated real variable estimate,Calderon transference external}\usesdefs{ergodic averages}
Let $n\ge2$, and let $(p_i)_{i\in[n)}$ satisfy \eqref{exponent polytope sum} and \eqref{exponent polytope inequalities}. Let $(X,\Sigma,\mu)$ be a measure space, let $(T_i)_{i\in[n)}$ be mutually commuting measure preserving transformations, and let $(f_i)_{i\in[n)}$ be complex-valued functions with $f_i\in L^{p_i}(X)$. For every $J\ge1$ and positive integers $N_0<\cdots<N_J$,
\begin{equation}\label{ergodic jump general exponents}
\sum_{j\in[J)}\|M_{N_{j+1}}(\mathbf f)-M_{N_j}(\mathbf f)\|_2^2
\le C_{\ref{thm:ergodicthm-interpolation},n,(p_i)_{i\in[n)}}J^{1-2^{2-n}}
\prod_{i\in[n)}\|f_i\|_{p_i}^2,
\end{equation}
where
\begin{equation}\label{ergodic interpolation constant}
C_{\ref{thm:ergodicthm-interpolation},n,(p_i)_{i\in[n)}}
=2^{2n}C_{\ref{Calderon transference external},n}
\bigl(2^{2n}C_{\ref{thm:nct main real}}\bigr)
=2^{8n+672}.
\end{equation}
\end{theorem}

\begin{proof}[Proof \auto]
Apply Proposition \ref{interpolated real variable estimate} to real-valued inputs and then Theorem \ref{Calderon transference external} with $g(J)=J^{1-2^{2-n}}$. This gives the factor $C_{\ref{Calderon transference external},n}$ in addition to the constant $2^{2n}C_{\ref{thm:nct main real}}$ from Proposition \ref{interpolated real variable estimate}. For complex inputs, expand every $f_i$ into real and imaginary parts. The triangle inequality in $L^2(X;\ell^2([J)))$ contributes at most another factor $2^n$ to the operator norm. Squaring gives the factor $2^{2n}$ in \eqref{ergodic interpolation constant}.
\end{proof}

\begin{theorem}[Symmetric ergodic jump estimate]\label{thm:ergodicthmJ}\lean{Auto.symmetric_ergodic_jump_estimate}\leanok\uses{thm:ergodicthm-interpolation}\usesdefs{ergodic averages}
Let $n\ge2$. Let $(X,\Sigma,\mu)$ be a measure space, let $(T_i)_{i\in[n)}$ be mutually commuting measure preserving transformations, and let $(f_i)_{i\in[n)}$ be complex-valued functions in $L^{2n}(X)$. For every $J\ge1$ and positive integers $N_0<\cdots<N_J$,
\begin{equation}\label{symmetric ergodic jump estimate}
\sum_{j\in[J)}\|M_{N_{j+1}}(\mathbf f)-M_{N_j}(\mathbf f)\|_2^2
\le C_{\ref{thm:ergodicthmJ},n}J^{1-2^{2-n}}
\prod_{i\in[n)}\|f_i\|_{2n}^2,
\end{equation}
where
\begin{equation}\label{symmetric ergodic jump constant}
C_{\ref{thm:ergodicthmJ},n}
=2^{8n+672}.
\end{equation}
\end{theorem}

\begin{proof}[Proof \auto]
We have $\sum_{i\in[n)}(2n)^{-1}=1/2$. If $I\subset[n)$ has $m\ge1$ elements, then
\[
\sum_{i\in I}(2n)^{-1}=\tfrac{m}{2n}\ge2^{m-n-1},
\]
because $m2^{n-m}\ge n$ for $1\le m\le n$. Thus Theorem \ref{thm:ergodicthm-interpolation} applies.
\end{proof}

\begin{lemma}[From jump estimates to variation]\label{jump estimates imply variation}\lean{Auto.jump_estimates_imply_variation_enorm,Auto.jump_estimates_imply_variation_endpoint_enorm}\leanok\uses{}\usesdefs{}
Let $B$ be a seminormed space, let $a:\N\to B$, let $r_0\ge2$, and let $D>0$. Assume that, for every $J\ge1$ and every $0<N_0<\cdots<N_J$,
\begin{equation}\label{abstract jump estimate}
\sum_{j\in[J)}\|a(N_{j+1})-a(N_j)\|^2\le D J^{1-2/r_0}.
\end{equation}
If $r>r_0$, then
\begin{equation}\label{abstract variation estimate}
\|a\|_{V_r(N\in\N,\ N\ge1;B)}
\le 2^{1-1/r_0}\left(\tfrac{r}{r-r_0}\right)^{1/r}D^{1/2}.
\end{equation}
If $r_0=2$, then
\begin{equation}\label{abstract endpoint variation estimate}
\|a\|_{V_2(N\in\N,\ N\ge1;B)}\le D^{1/2}.
\end{equation}
\end{lemma}

\begin{proof}[Proof \auto]
Fix $0<N_0<\cdots<N_J$ and write $d_j=\|a(N_{j+1})-a(N_j)\|$. Let $d_1^*\ge\cdots\ge d_J^*$ be the decreasing rearrangement. Choose the indices of the $m$ largest $d_j$ and split them according to the parity of $j$. Within either parity class the selected intervals $[N_j,N_{j+1}]$ are disjoint and nonadjacent. Listing their endpoints in increasing order and applying \eqref{abstract jump estimate} gives
\[
\sum_{j\in E}d_j^2\le D(2|E|)^{1-2/r_0}
\]
for either class $E$. Hence
\[
m(d_m^*)^2\le 2D(2m)^{1-2/r_0},
\]
so
\[
d_m^*\le2^{1-1/r_0}D^{1/2}m^{-1/r_0}.
\]
Therefore, when $r>r_0$,
\[
\sum_{j\in[J)}d_j^r
\le2^{r(1-1/r_0)}D^{r/2}\sum_{m=1}^{\infty}m^{-r/r_0}
\le2^{r(1-1/r_0)}D^{r/2}\tfrac{r}{r-r_0}.
\]
Taking the $r$th root and then the supremum gives \eqref{abstract variation estimate}. If $r_0=2$, \eqref{abstract jump estimate} itself is uniform in $J$ and gives \eqref{abstract endpoint variation estimate} directly.
\end{proof}

\begin{proof}[Proof of Theorem \ref{thm:ergodicthm} \auto]
Put
\[
D=C_{\ref{thm:ergodicthmJ},n}\prod_{i\in[n)}\|f_i\|_{2n}^2.
\]
Theorem \ref{thm:ergodicthmJ} gives \eqref{abstract jump estimate} with $r_0=2^{n-1}$. If $n=2$, Lemma \ref{jump estimates imply variation} gives the result for $r=2$, and the monotonicity of finite $\ell^r$ norms gives it for every $r\ge2$. If $n\ge3$ and $r>2^{n-1}$, the same lemma gives
\[
\|M_N(\mathbf f)\|_{V_r(N\in\N,\ N\ge1;L^2(X))}
\le 2^{1-2^{1-n}}\left(\tfrac{r}{r-2^{n-1}}\right)^{1/r}
C_{\ref{thm:ergodicthmJ},n}^{1/2}
\prod_{i\in[n)}\|f_i\|_{2n}.
\]
\end{proof}


\begin{thebibliography}{99}
\bibitem{Austin10}
T.~Austin,
On the norm convergence of non-conventional ergodic averages,
\emph{Ergodic Theory Dynam. Systems} \textbf{30} (2010), 321--338.

\bibitem{AvigadRute15}
J.~Avigad and J.~Rute,
Oscillation and the mean ergodic theorem for uniformly convex Banach spaces,
\emph{Ergodic Theory Dynam. Systems} \textbf{35} (2015), 1009--1027.

\bibitem{BerghLofstrom}
J.~Bergh and J.~L\"ofstr\"om,
\emph{Interpolation Spaces: An Introduction},
Springer-Verlag, Berlin, 1976.

\bibitem{CarlesonBlueprint}
L.~Becker, M.~I.~de Frutos-Fernández, L.~Diedering, F.~van Doorn,
S.~Gouëzel, A.~Jamneshan, E.~Karunus, E.~van de Meent,
P.~Monticone, J.~Mulder-Sohn, J.~Portegies, J.~Roos,
M.~Rothgang, R.~Srivastava, J.~Sundstrom, J.~Tan, and C.~Thiele,
\emph{A blueprint for the formalization of Carleson's theorem on convergence of Fourier series},
arXiv:2405.06423v2, 2025.

\bibitem{CarlesonPaper}
L.~Becker, F.~van Doorn, A.~Jamneshan, R.~Srivastava, and C.~Thiele,
\emph{Carleson operators on doubling metric measure spaces},
arXiv:2508.05563, 2025.

\bibitem{DurcikKovacSkrebThiele19}
P.~Durcik, V.~Kova\v{c}, K.~A.~\v{S}kreb, and C.~Thiele,
Norm variation of ergodic averages with respect to two commuting transformations,
\emph{Ergodic Theory Dynam. Systems} \textbf{39} (2019), 658--688.

\bibitem{DurcikKovacThiele19}
P.~Durcik, V.~Kova\v{c}, and C.~Thiele,
Power-type cancellation for the simplex Hilbert transform,
\emph{J. Anal. Math.} \textbf{139} (2019), 67--82.

\bibitem{DurcikSlavikovaThiele22}
P.~Durcik, L.~Slav\'ikov\'a, and C.~Thiele,
Local bounds for singular Brascamp--Lieb forms with cubical structure,
\emph{Math. Z.} \textbf{302} (2022), 2375--2405.

\bibitem{DurcikSlavikovaThiele26}
P.~Durcik, L.~Slav\'ikov\'a, and C.~Thiele,
Norm-variation of triple ergodic averages for commuting transformations,
\emph{Anal. PDE} \textbf{19} (2026), 539--586.

\bibitem{DurcikThiele21survey}
P.~Durcik and C.~Thiele,
Singular Brascamp--Lieb: a survey,
in \emph{Geometric Aspects of Harmonic Analysis},
Springer INdAM Ser. \textbf{45},
Springer, Cham, 2021, 321--349.

\bibitem{Groechenig}
K.~Gr\"ochenig,
\emph{Foundations of Time-Frequency Analysis},
Birkh\"auser, 2001.

\bibitem{Hormander1}
L.~H\"ormander,
\emph{The Analysis of Linear Partial Differential Operators I: Distribution Theory and Fourier Analysis},
second edition, Springer-Verlag, Berlin, 1990.

\bibitem{Host09}
B.~Host,
Ergodic seminorms for commuting transformations and applications,
\emph{Studia Math.} \textbf{195} (2009), 31--49.

\bibitem{JonesOstrovskiiRosenblatt96}
R.~L.~Jones, I.~V.~Ostrovskii, and J.~M.~Rosenblatt,
Square functions in ergodic theory,
\emph{Ergodic Theory Dynam. Systems} \textbf{16} (1996), 267--305.

\bibitem{JonesSeegerWright08}
R.~L.~Jones, A.~Seeger, and J.~Wright,
Strong variational and jump inequalities in harmonic analysis,
\emph{Trans. Amer. Math. Soc.} \textbf{360} (2008), 6711--6742.

\bibitem{Kovac12}
V.~Kova\v{c},
Boundedness of the twisted paraproduct,
\emph{Rev. Mat. Iberoam.} \textbf{28} (2012), 1143--1164.

\bibitem{Kovac16}
V.~Kova\v{c},
Quantitative norm convergence of double ergodic averages associated with two commuting group actions,
\emph{Ergodic Theory Dynam. Systems} \textbf{36} (2016), 860--874.

\bibitem{Mathlib}
The mathlib Community,
\emph{The Lean mathematical library},
in \emph{Proceedings of the 9th ACM SIGPLAN International Conference
on Certified Programs and Proofs (CPP 2020)},
Association for Computing Machinery, New York, 2020,
pp.~367--381,
doi:10.1145/3372885.3373824.

\bibitem{LeanPaper}
L.~de Moura and S.~Ullrich,
\emph{The Lean 4 theorem prover and programming language},
in \emph{Automated Deduction -- CADE 28},
Lecture Notes in Computer Science, vol.~12699,
Springer, 2021, pp.~625--635.

\bibitem{Pernegger2025}
F.~Pernegger,
\emph{Formalisation of the Calder\'on Transference Principle in Ergodic Theory},
Bachelor's thesis, Rheinische Friedrich-Wilhelms-Universit\"at Bonn, September 2025.

\bibitem{Tao07}
T.~Tao,
Norm convergence of multiple ergodic averages for commuting transformations,
\emph{Ergodic Theory Dynam. Systems} \textbf{28} (2008), 657--688.

\bibitem{Walsh12}
M.~N.~Walsh,
Norm convergence of nilpotent ergodic averages,
\emph{Ann. of Math.} \textbf{175} (2012), 1667--1688.

\bibitem{Wiener32}
N.~Wiener,
Tauberian theorems,
\emph{Ann. of Math.} \textbf{33} (1932), 1--100.
\end{thebibliography}
\end{document}